\documentclass[12pt,a4paper,fleqn]{article}
\usepackage{a4wide,amsfonts,amsmath,latexsym,amssymb,euscript,graphicx,units,mathrsfs}

\usepackage[english]{babel}
\usepackage{graphicx}
\usepackage{color}
\usepackage{amssymb}
\usepackage{amssymb}
\usepackage[T1]{fontenc}
\usepackage{latexsym}
\usepackage{xypic}
\usepackage{eufrak}
\usepackage{euscript}
\usepackage{amsfonts,amsmath}
\usepackage{verbatim}
\usepackage{fancyhdr}
\usepackage{mathrsfs}
\usepackage{units}

\newtheorem{prop}{Proposition}[section]

\newtheorem{corollary}[prop]{Corollary}
\newtheorem{lemme}[prop]{Lemma}
\newtheorem{lemma}[prop]{Lemma}

\newtheorem{theorem}[prop]{Theorem}

\newtheorem{definition}[prop]{Definition}

\renewcommand{\geq}{\geqslant}
\def\leq{\leqslant}

\newcommand{\N}{\mathbb{N}}
\newcommand{\Z}{\mathbb{Z}}

\newcommand{\R}{\mathbb{R}}

\newcommand{\C}{\mathbb{C}}

\def\1{{\mathbf{1}}}

\def\1{{\mathbf{1}}}
\def\0.5{{\frac{1}{2}}}

\newenvironment{proof}[1]{\begin{trivlist}\item {\it
\bf Proof.}\quad} {\qed\end{trivlist}}

\newcommand{\qed}{\nopagebreak\hspace*{\fill}
{\vrule width6pt height6ptdepth0pt}\par}

\begin{document}
\thispagestyle{empty}

\begin{center}
{\bf\Large Change-of-variable formula for the bi-dimensional fractional Brownian motion in Brownian time}
\end{center}
\begin{center}
{\bf Raghid Zeineddine\footnote{Laboratoire J.A. Dieudonn\'e, UMR CNRS -UNSA  7351, Universit\'e de Nice Sophia-Antipolis; {\tt raghid.zeineddine@unice.fr}}}
\end{center}

\bigskip
\begin{abstract}
 Let $X^{1}, X^{2}$ be  two independent (two-sided) fractional Brownian motions having the same Hurst parameter $H \in (0,1)$, and let $Y$ be a standard (one-sided) Brownian motion independent of $(X^{1},X^{2})$. In dimension 2, fractional Brownian motion in Brownian motion time  (of index $H$) is, by definition, the process $Z_t:= (Z^1_t, Z^2_t)= (X^{1}_{Y_t},X^{2}_{Y_t}), \:\: t \geq 0$. 
The main result of the present paper is an It\^{o}'s type formula for $f(Z_t)$, when $f:\R^2\to\R$ is smooth and $H\in [ 1/6,1)$.
When $H>1/6$, the change-of-variable formula we obtain is similar to that of the classical calculus.
In the critical case $H=1/6$, our change-of-variable formula is in law and involves the third partial derivatives of $f$ as well as an extra Brownian motion independent of $(X^1,X^2,Y)$. We also discuss  the case $H<1/6$.
\end{abstract}
\tableofcontents

\section{Introduction}
Our aim in the present paper is to provide a change-of-variable formula for the fractional Brownian motion in Brownian time (fBmBt) in multi-dimension. For simplicity of the exposition and because the computations are rather involved, we will stick on dimension 2, which is representative of the difficulty. In dimension 1, the mathematical definition of fBmBt (together with its terminology) was introduced in our previous paper \cite{ItoD1}. Let us give an analogue definition in dimension 2. Set
\begin{equation}\label{fBmBt}
Z_t:= (Z^1_t, Z^2_t)= (X^{1}_{Y_t},X^{2}_{Y_t}), \:\: t \geq 0,
\end{equation}
where $X^{1}, X^{2}$ are  two independent (two-sided) fractional Brownian motions having the same Hurst parameter $H \in (0,1)$, and $Y$ is a standard (one-sided) Brownian motion independent of $(X^{1},X^{2})$. 

The present work may be seen a natural follow-up of \cite{ItoD1}, in which we proved a change-of-variable for fBmBt in dimension one, that is, for $Z^1$. 
Before stating the results we have obtained, let us start with some historical comments
and relationships with the existing literature.
When the Hurst index of the fractional Brownian motion is $H=1/2$, we note that $Z^{1}$ reduces to the iterated Brownian motion (iBm), a process introduced by Burdzy in \cite{burdzy ibm}. IBm is self-similar of order $\frac14$, has stationary increments, and it is neither a Dirichlet process, nor a semimartingale, nor a Markov process in its own filtration. A key question was therefore how to define a stochastic calculus with respect to it. A beautiful answer was given by Khoshnevisan and Lewis \cite{kh-lewis1}, who developed a Stratonovich-type stochastic calculus with respect to iBm. Recall that the Stratonovich integral of a continuous process $X$ with respect to another continuous process $Y$ may be defined (provided the limit exists in some suitable sense) as follows:
 \begin{equation}\label{symmetric}
\int_0^t X_sd^\circ Y_s := \lim_{n\to\infty} \sum_{k=0}^{\lfloor 2^n t\rfloor -1}
\frac12\big(X_{k2^{-n}}+X_{(k+1)2^{-n}}\big)(Y_{(k+1)2^{-n}}-Y_{k2^{-n}}).
\end{equation}
As observed in \cite{kh-lewis2}, in the iBm case it appears to be a very hard task to work directly with definition (\ref{symmetric}).
To circumvent this difficulty, a nice idea of Khoshnevisan and Lewis have consisted in modifying the definition (\ref{symmetric}) by replacing the uniform dyadic partition in the right-hand side by a suitable arrays of Brownian stopping times, relying to some classical excursion-theoretic arguments. Based on this new definition for the symmetric integral, Khoshnevisan and Lewis obtained, for the iBm (corresponding to $H=\frac12$) and in dimension 1, a change-of-variable formula having a classical form:
 \begin{equation}
 f(Z^{1}_t)= f(0) + \int_0^t f(Z^{1}_s)d^{\circ}Z^{1}_s, \: \: t \geq 0.\label{ito-ibm}
 \end{equation} 
 
A natural question was then to extend (\ref{ito-ibm}) for other values of $H$. We did it in the joint paper \cite{ItoD1} with Nourdin, by proving the following theorem.
 \begin{theorem}\label{1Dito}
Let $f:\R\to\R$ be a smooth and bounded enough function.
\begin{enumerate}
\item If $H>\frac16$ then
\[
f(Z^1_t)=f(0)+\int_0^t f'(Z^1_s)d^\circ Z^1_s,\quad t\geq 0.
\]

\item If $H=\frac16$ then,  with $\kappa_3\simeq 2.322$,
\begin{equation*}
f(Z^1_t)-f(0)+\frac{\kappa_3}{12}\int_0^t f'''(Z^1_s)d^{\circ 3}Z^1_s \overset{law}{=} \int_0^t f'(Z^1_s)d^\circ Z^1_s,\quad t\geq 0,
\end{equation*}
where $\int_0^t f'''(Z^1_s)d^{\circ 3}Z^1_s$ is a random variable equal in law to $\int_0^{Y_t} f'''(X^1_s)dW_s$, for $W$ a two-sided Brownian motion independent of the pair $(X^1,Y)$.

\item If $H<\frac16$, then
\[
\int_0^t (Z^1_s)^2 d^\circ Z^1_s \mbox{ does not exist (even stably in law)}.
\]
\end{enumerate}
\end{theorem}

Theorem \ref{1Dito} was proved by combining some techniques introduced in \cite{kh-lewis1} with a recent line of research in which, by means of Malliavin calculus, one aims to exhibit change-of-variable formulas in law with a correction term which is an It\^o integral with respect to martingale independent of the underlying Gaussian processes. Papers dealing with this problem and which are prior to our work include \cite{HN1, HN2, HN3, nourdin, NR, NRS}.

In the present paper, our main aim is to extend Theorem \ref{1Dito} to the bi-dimensional case. To reach this goal, we follow and use a strategy introduced in \cite{nourdin} and \cite{NRS}.  
We continue to let $X^1, X^2,Y,Z$ be as in (\ref{fBmBt}), and we set $X=(X^1,X^2)$. The following definition will play a pivotal role in the sequel.

\begin{definition}\label{ivan-definition} Let $f: \R^2 \to \R$ be a continuously differentiable function, and fix a time $t >0$. \\
1) Provided it exists, we define $\int_0^t\nabla f(X_s)dX_s$ to be the limit in probability, as $n \to \infty$, of
\begin{eqnarray}
&&O_n(f,t)\label{2Dintegral}\\
&=&\sum_{j=0}^{\lfloor 2^{n/2} t \rfloor -1}\frac{\partial f}{\partial x}\bigg( \frac{X^{1}_{(j+1)2^{-n/2}} + X^{1}_{j2^{-n/2}} }{2},\frac{X^{2}_{(j+1)2^{-n/2}} + X^{2}_{j2^{-n/2}}}{2} \bigg)\big(X^{1}_{(j+1)2^{-n/2}} - X^{1}_{j2^{-n/2}}\big)\notag\\
&& + \sum_{j=0}^{\lfloor 2^{n/2} t \rfloor -1}\frac{\partial f}{\partial y}\bigg( \frac{X^{1}_{(j+1)2^{-n/2}} + X^{1}_{j2^{-n/2}} }{2},\frac{X^{2}_{(j+1)2^{-n/2}} + X^{2}_{j2^{-n/2}}}{2} \bigg)\big(X^{2}_{(j+1)2^{-n/2}} - X^{2}_{j2^{-n/2}}\big). \notag
\end{eqnarray}
2)
When $O_n(f,t)$ defined by (\ref{2Dintegral}) does not converge in probability but converges stably instead, we denote the limit by $\int_0^t\nabla f(X_s)d^{\ast}X_s$.
\end{definition}

A first preliminary result, which concerns the bi-dimensional fractional Brownian motion $X$, can now be stated. An analogue
result for the fBmBt $Z$ will be the object of the forthcoming Theorem \ref{second-main}.

\begin{theorem}\label{first-main}Let $f: \R^2 \to \R$ be a function belonging to $C_b^{\infty}$, and fix a time $t > 0$.

\begin{enumerate}

\item If $H > 1/6$ then $\int_0^t\nabla f(X_s)dX_s$ is well-defined, and we have
\begin{equation}
f(X_t) =f(0) + \int_0^t\nabla f(X_s)dX_s. \label{first-main-1}
\end{equation}

\item If $H =1/6$ then $\int_0^t\nabla f(X_s)d^{\ast}X_s$ is well-defined, and we have
\begin{equation}
f(X_t) -f(0) - \int_0^t D^3f(X_s)d^3X_s \overset{law}{=}  \int_0^t\nabla f(X_s)d^{\ast}X_s \label{first-main-2}
\end{equation}
where $\int_0^t D^3f(X_s)d^3X_s$ is short-hand for
\begin{eqnarray}
 \int_0^t D^3f(X_s)d^3X_s &=& \kappa_1\int_0^t\frac{\partial^3 f}{\partial x^3}\big(X^{1}_s,X^{2}_s\big)dB^{1}_s + \kappa_2\int_0^t\frac{\partial^3 f}{\partial y^3}\big(X^{1}_s,X^{2}_s\big)dB^{2}_s \label{additional}\\
 && + \kappa_3\int_0^t\frac{\partial^3 f}{\partial x^2 \partial y}\big(X^{1}_s,X^{2}_s\big)dB^{3}_s + \kappa_4\int_0^t\frac{\partial^3 f}{\partial x \partial y^2}\big(X^{1}_s,X^{2}_s\big)dB^{4}_s \notag
\end{eqnarray}
with $B= (B^{1}, \ldots , B^{4})$  a 4-dimensional Brownian motion independent of $X$, $\kappa_1^2 = \kappa_2^2 = \frac{1}{96}\sum_{r\in \Z}\rho^3(r)$ and $\kappa_3^2 = \kappa_4^2 = \frac{1}{32}\sum_{r\in \Z}\rho^3(r)$ with $\rho$ defined in (\ref{rho}).

\item If $H< 1/6$, for $f(x,y) =x^3$ then 
\begin{equation}
\int_0^t\nabla f(X_s)d^{\ast}X_s \: \text{\:does not exist, even stably in law.}  \label{first-main-3}
\end{equation}
So, it is impossible to write an It\^o's type formula.
\end{enumerate}

\end{theorem}

Theorem \ref{first-main} together with a suitable extension of the Khoshnevisan-Lewis definition for the Stratonovich integral (see the next section for a precise statement) with respect to $Z$ then lead to the following change-of-variable formula for 2D fBmBt,
which represents the main finding of our paper.

\begin{theorem}\label{second-main}Let $f: \R^2 \to \R$ be a function belonging to $C_b^{\infty}$, and fix a time $t > 0$.

\begin{enumerate}

\item If $H > 1/6$ then $\int_0^t\nabla f(Z_s)dZ_s$ is well defined, and we have
\begin{equation}
f(Z_t) =f(0) + \int_0^t\nabla f(Z_s)dZ_s. \label{second-main-1}
\end{equation}

\item If $H =1/6$ then $\int_0^t\nabla f(Z_s)d^{\ast}Z_s$ is well defined, and we have
\begin{equation}
f(Z_t) -f(0) - \int_0^t D^3f(Z_s)d^3Z_s \overset{law}{=}  \int_0^t\nabla f(Z_s)d^{\ast}Z_s \label{second-main-2}
\end{equation}
where $\int_0^t D^3f(Z_s)d^3Z_s$ is short-hand for
\begin{eqnarray*}
 \int_0^t D^3f(Z_s)d^3Z_s &= & \kappa_1\int_0^{Y_t}\frac{\partial^3 f}{\partial x^3}\big(X^1_s, X^2_s\big)dB^{1}_s + \kappa_2\int_0^{Y_t}\frac{\partial^3 f}{\partial y^3}\big(X^1_s, X^2_s\big)dB^{2}_s \\
 && + \kappa_3\int_0^{Y_t}\frac{\partial^3 f}{\partial x^2 \partial y}\big(X^1_s, X^2_s\big)dB^{3}_s + \kappa_4\int_0^{Y_t}\frac{\partial^3 f}{\partial x \partial y^2}\big(X^1_s, X^2_s\big)dB^{4}_s,
\end{eqnarray*}
with $B= (B^{1}, \ldots , B^{4})$ is a 4-dimensional two-sided Brownian motion independent of $X$, and $\kappa_1,\ldots,\kappa_4$ as in Theorem \ref{first-main}. ($B$ is also independent from $Y$.)

\item If $H< 1/6$, for $f(x,y) =x^3$ then 
\begin{equation}
\int_0^t\nabla f(Z_s)d^{\ast}Z_s \: \text{\: does not exist, even stably in law.} \label{second-main-3}
\end{equation}
So, it is impossible to write an It\^o's type formula.
\end{enumerate}

\end{theorem}

A brief outline of the paper is as follows. In section 2, we introduce the framework and the preliminaries to prove our results, as well as the notation and some technical lemmas. In section 3, we prove Theorem \ref{first-main}. In section 4, we prove Theorem \ref{second-main}, and finally in section 5, we give the proof of a  technical lemma. 

\section{Framework, preliminaries, notation and technical lemmas}

\subsection{The framework of Theorem \ref{second-main}}

In this section, we explain and introduce the missing definitions of the mathematical objects appearing in Theorem \ref{second-main}.

\begin{enumerate}

\item \textbf{Khoshnevisan-Lewis' definition of the Stratonovich-integral with respect to the 1D fBmBt}

Since the paths of $Z^1$ are very irregular (precisely: H\"older continuous of order $\alpha$ if and only if $\alpha$ is strictly less than $H/2$), as a matter of fact we won't be able to define a stochastic integral with respect to it as the limit of Riemann sums with respect to a {\it deterministic} partition of the time axis. A winning idea, borrowed from Khoshnevisan and Lewis \cite{kh-lewis1, kh-lewis2}, is to approach deterministic partitions by means of random partitions defined in terms of hitting times of the underlying Brownian motion $Y$. As such, one can bypass the random ``time-deformation'' forced by $Y$, and perform asymptotic procedures by separating the roles of $X$ and $Y$ in the overall definition of $Z^1$.

Following Khoshnevisan and Lewis \cite{kh-lewis1, kh-lewis2}, we start by introducing the so-called intrinsic skeletal structure of $Z^1$. This structure is defined through a sequence of collections of stopping times (with
respect to the natural filtration of $Y$), noted
\begin{equation}
\mathscr{T}_n=\{T_{k,n}: k\geq0\}, \quad n\geq1, \label{TN}
\end{equation}
which are in turn expressed in terms of the subsequent hitting
times of a dyadic grid cast on the real axis. More precisely, let
$\mathscr{D}_n= \{j2^{-n/2}:\,j\in\Z\}$, $n\geq 1$, be the dyadic
partition (of $\R$) of order $n/2$. For every $n\geq 1$, the
stopping times $T_{k,n}$, appearing in (\ref{TN}), are given by
the following recursive definition: $T_{0,n}= 0$, and
\[
T_{k,n}= \inf\big\{s>T_{k-1,n}:\quad
Y(s)\in\mathscr{D}_n\setminus\{Y(T_{k-1,n})\}\big\},\quad k\geq 1.
\]
Note that the definition of $T_{k,n}$, and
therefore of $\mathscr{T}_n$, only involves the one-sided Brownian
motion $Y$. Also, for every $n\geq1$, the discrete stochastic
process
\[
\mathscr{Y}_n=\{Y(T_{k,n}):k\geq0\}
\]
defines a simple random
walk over $\mathscr{D}_n$.  As shown in
\cite[Lemma 2.2]{kh-lewis1}, as $n$ tends to
infinity the collection $\{T_{k,n}:\,1\leq k \leq 2^nt\}$ approximates the
common dyadic partition $\{k2^{-n}:\,1\leq k \leq 2^nt\}$ of order $n$ of the time interval $[0,t]$. More precisely,
\begin{equation}\label{lemma2.2}
\sup_{0\leq s\leq t} \big| T_{\lfloor 2^n s\rfloor,n}-s\big|\to 0\quad\mbox{almost surely and in $L^2(\Omega)$.}
\end{equation}
Based on this fact, one may introduce the counterpart of (\ref{symmetric}) based on $\mathscr{T}_n$, namely,
\begin{equation*}
V_n(f, t)= \sum_{k=0}^{\lfloor 2^n t \rfloor -1} f\bigg(\frac{Z^1_{T_{k,n}}+Z^1_{T_{k+1,n}}}{2}\bigg)(Z^1_{T_{k+1,n}}-Z^1_{T_{k,n}}).
\end{equation*}

So, the integral of $f(Z^1)$ with respect to $Z^1$ is defined as
\begin{equation}\label{1Dintegral}
\int_0^t f(Z^1_s)dZ^1_s := \lim_{n \to \infty}V_n(f, t),
\end{equation} 
provided the limit exists in some sense.

\item \textbf{A suitable definition for the Stratonovich-integral with respect to the 2D fBmBt}

In the light of the previous definition of the integral with respect to 1D fBmBt  and of Definition \ref{ivan-definition}, it might seem natural to introduce the following definition for the integral with respect to the 2D fBmBt based on $\mathscr{T}_n$.

\begin{definition}\label{ivan-definition'} Let $f: \R^2 \to \R$ be a continuously differentiable function, and fix a time $t >0$. Provided it exists, we define $\int_0^t\nabla f(Z_s)dZ_s$ to be the limit in probability, as $n \to \infty$, of
\begin{eqnarray}
\tilde{O}_n(f,t) &:=& \sum_{j=0}^{\lfloor 2^{n/2} t \rfloor -1}\frac{\partial f}{\partial x}\bigg( \frac{Z^{1}_{T_{j+1,n}} + Z^{1}_{T_{j,n}} }{2},\frac{Z^{2}_{T_{j+1,n}} + Z^{2}_{T_{j,n}}}{2} \bigg)\big(Z^{1}_{T_{j+1,n}} - Z^{1}_{T_{j,n}}\big)\notag\\
&& + \sum_{j=0}^{\lfloor 2^{n/2} t \rfloor -1}\frac{\partial f}{\partial y}\bigg( \frac{Z^{1}_{T_{j+1,n}} + Z^{1}_{T_{j,n}} }{2},\frac{Z^{2}_{T_{j+1,n}} + Z^{2}_{T_{j,n}}}{2} \bigg)\big(Z^{2}_{T_{j+1,n}} - Z^{2}_{T_{j,n}}\big). \notag \\\label{2Dintegral'}
\end{eqnarray}
 If $\tilde{O}_n(f,t)$ defined by (\ref{2Dintegral'}) does not converge in probability but converges stably, we denote the limit by $\int_0^t\nabla f(Z_s)d^{\ast}Z_s$.
\end{definition}

\end{enumerate}

\subsection{Some preliminary results}

We provide now a description of the tools of Malliavin calculus that we need in this article. We follow in this section the idea introduced in \cite{nourdin}. The reader in referred to \cite{NP} for details and any unexplained result.

Let $ X=(X_t^{1},X_t^{2})_{t \in \R}$ be a 2D fBm with Hurst parameter belonging to $ (0,1) $. For all $n \in \N^*$, we let $\mathscr{E}_n$ be the set of step $\R^2$-valued functions on $[-n,n]$, and $\displaystyle{\mathscr{E}:= \cup_n \mathscr{E}_n}$. Set $\varepsilon_t = \textbf{1}_{[0,t]}$ (resp. $\textbf{1}_{[t,0]}$) if $t \geq 0$ (resp. $t < 0$). Let $\mathscr{H}$ be the Hilbert space defined as the closure of $\mathscr{E}$ with respect to the inner product
   \[
    \langle (\varepsilon_{t_1}, \varepsilon_{t_2}), (\varepsilon_{s_1}, \varepsilon_{s_2}) \rangle_{\mathscr{H}} = C_{H}(t_1,s_1) + C_{H}(t_2,s_2) ,\quad s_1, s_2, t_1, t_2 \in \R,
    \]
    where $C_{H}(t,s) = \frac{1}{2}(|s|^{2H} + |t|^{2H} -|t-s|^{2H}) = E\big(X^{i}_s X^{i}_t \big)$ ($i$ equals 1 or 2). The mapping $ (\varepsilon_{t_1}, \varepsilon_{t_2})\mapsto X_{t_1}^{1} + X_{t_2}^{2} $ can be extended to an isometry between $\mathscr{H}$ and the Gaussian space associated with $X$. Also, let $\mathscr{F}_n$ denote the set of step $\R$-valued functions on $[-n,n]$, $\displaystyle{\mathscr{F}:= \cup_n \mathscr{F}_n}$ and $\mathcal{G}$ denote the Hilbert space defined as the closure of $\mathscr{F}$ with respect to the scalar product induced by 
\begin{equation}\label{inner-product}
\langle \varepsilon_t, \varepsilon_s \rangle_{\mathcal{G}} = C_{H}(t,s), \:\:\: s,t \in \R.
\end{equation}
The mapping $ \varepsilon_t \mapsto X_t^{i} $ ($i$ equals 1 or 2) can be extended to an isometry between $ \mathcal{G} $ and the Gaussian space associated with $X^{i}$.

We consider the set of smooth cylindrical random variables, i.e. of the form
\[
F= f\big(X(\rho_1), \ldots, X(\rho_m)\big), \:\:\: \rho_i \in \mathscr{H}, \: i=1,\ldots,m,
\]
where $ f \in C_b^{\infty} $ is bounded with bounded derivatives. The derivative operator $D$ of a smooth random variable of the above form is defined as the $ \mathscr{H}$-valued random variable
\begin{equation}
DF = \sum_{i=1}^m \frac{\partial f}{\partial x_i}\big(X(\rho_1), \ldots, X(\rho_m)\big)\rho_i =: \big(D_{X^{1}}F, D_{X^{2}}F \big). \label{diabolic-derivation}
\end{equation}
For example, if $F = f(X^1_t, X^2_s)$ with $f \in C_b^\infty(\R^2)$, then 
\[
DF = \frac{\partial f}{\partial x}(X^1_t, X^2_s)(\varepsilon_{t}, 0) + \frac{\partial f}{\partial y}(X^1_t, X^2_s)(0,\varepsilon_s).
\]
 So, we deduce from (\ref{diabolic-derivation}) that 
 \[
 D_{X^{1}}F = \frac{\partial f}{\partial x}(X^1_t, X^2_s)\varepsilon_{t} \:\text{\: and\:}\: D_{X^{2}}F = \frac{\partial f}{\partial y}(X^1_t, X^2_s)\varepsilon_s.
 \]
In particular,  for $j,k\in \{1,2\}$, we have 
\begin{eqnarray*}
D_{X^{j}}X_t^{k} = \left\{
\begin{array}{ccc}
\varepsilon_t && \mbox{if $j=k$}\\
0 && \mbox{if $j\neq k$}
\end{array}
\right .
\end{eqnarray*}
For any integer $k \geq 2$, one can define, by iteration, the $k$-th derivative $D^k F$ (which is a symmetric element of $L^2(\Omega, \mathscr{H}^{\otimes k})$). As usual, for any $k \geq 1$, the space $\mathbb{D}^{k,2}$ denotes the closure of the set of smooth random variables with respect to the norm $\|.\|_{k,2}$ defined by
   \[ 
   \|F\|_{k,2}^{2} = E(F^{2}) + \sum_{j=1}^{k} E[ \|D^{j}F\|_{\mathscr{H}^{\otimes j}}^{2}].
   \]
   The Malliavin derivative $D$ satisfies the chain rule. If $\varphi : \mathbb{R}^{n} \rightarrow \mathbb{R}$ is $C_{b}^{1}$ and if $F_1,\ldots,F_n$ are in $\mathbb{D}^{1,2}$, then $\varphi(F_{1},...,F_{n}) \in \mathbb{D}^{1,2}$ and we have
   \[ 
   D\varphi(F_{1},...,F_{n}) = \sum_{i=1}^{n} \frac{\partial \varphi}{\partial x_{i}}(F_{1},...,F_{n})DF_{i}.
   \]
   We have the following Leibniz formula. For any $F, G \in \mathbb{D}^{q,2}$ $(q \geq 1)$ such that $FG \in \mathbb{D}^{q,2}$, for $i \in \{1,2\}$, we have
    \begin{equation}
    D^q_{X^{i}}(FG) = \sum_{l=0}^q \binom{q}{l} (D^l_{X^{i}}(F))\tilde{\otimes}(D^{q-l}_{X^{i}}G), \label{Leibnitz0}
    \end{equation}
    where $\tilde{\otimes}$ stands for the symmetric tensor product.
   In particular, we have the following formula. Let $\varphi, \psi \in C_{b}^{q}(\R^2)$ $(q\geq 1)$, and fix $0 \leq u<v $ and $0\leq s<t .$ Then  $\varphi\big(\frac{X^{1}_{t}+X^{1}_{s}}{2}, \frac{X^{2}_{t}+X^{2}_{s}}{2}\big)\psi\big(\frac{X^{1}_{v}+X^{1}_{u}}{2}, \frac{X^{2}_{v}+X^{2}_{u}}{2}\big)\in \mathbb{D}^{q,2}$ and for $i\in \{1,2\}$ we have
   \begin{eqnarray}
&& D_{X^{i}}^{q}\bigg( \varphi\bigg(\frac{X^{1}_{t}+X^{1}_{s}}{2}, \frac{X^{2}_{t}+X^{2}_{s}}{2}\bigg)\psi\bigg(\frac{X^{1}_{v}+X^{1}_{u}}{2}, \frac{X^{2}_{v}+X^{2}_{u}}{2}\bigg)\bigg)\label{Leibnitz1}\\
& =& \sum_{l=0}^{q} \binom{q}{l} \frac{\partial^l \varphi}{\partial x_i^l} \bigg(\frac{X^{1}_{t}+X^{1}_{s}}{2}, \frac{X^{2}_{t}+X^{2}_{s}}{2}\bigg)\frac{\partial^{q-l} \psi}{\partial x_i^{q-l}} \bigg(\frac{X^{1}_{v}+X^{1}_{u}}{2}, \frac{X^{2}_{v}+X^{2}_{u}}{2}\bigg)\notag\\
&& \hspace{2cm} \times \bigg(\frac{\varepsilon_s+ \varepsilon_t}{2}\bigg)^{\otimes l}\tilde{\otimes}\bigg(\frac{\varepsilon_u+ \varepsilon_v}{2}\bigg)^{\otimes (q-l)}.\notag
\end{eqnarray}
   A similar statement holds for $ u<v \leq 0 $ and $ s<t \leq 0$.
   
   If a random element $u \in L^{2}(\Omega, \mathscr{H})$ belongs to the domain of the divergence operator, that is, if it satisfies
   \[ |E\langle DF,u\rangle_{\mathscr{H}}|\leq c_{u}\sqrt{E(F^{2})} \text{\: for  any\:} F\in \mathscr{F},\] then $I(u)$ is defined by the duality relationship
\[
E \big( FI(u)\big) = E \big( \langle DF,u\rangle_{\mathscr{H}}\big),
\]
for every $F \in \mathbb{D}^{1,2}.$

For every $n\geq 1$, let $\mathbb{H}_{n}$ be the $n$th Wiener chaos of $X$, that is, the closed linear subspace of $ L^{2}(\Omega, \mathscr{A},P)$ generated by the random variables $\lbrace H_{n}(X(h)), h \in \mathscr{H}, \|h\|_{\mathscr{H}}=1 \rbrace,$ where $H_{n}$ is the $n$th Hermite polynomial. The mapping
\begin{equation}
I_{n}(h^{\otimes n}) = H_{n}(X(h)), \label{linear-isometry}
\end{equation}
 provides a linear isometry between the symmetric tensor product $\mathscr{H}^{\odot n}$ and $\mathbb{H}_{n}$. The following duality formula holds
\begin{eqnarray*}
  E \big( FI_{n}(h)\big) = E \big( \langle D^{n}F,h\rangle_{\mathscr{H}^{\otimes n}}\big),
  \end{eqnarray*}
  for any element $ h\in \mathscr{H}^{\odot n}$ and any random variable $F \in \mathbb{D}^{n,2}.$ In particular, we have
  
  \begin{equation}\label{duality formula}
  E \big( FI_{n}^{(i)}(h)\big) = E \big( \langle D_{X^{i}}^{n}F,h\rangle_{\mathcal{G}^{\otimes n}}\big), \: \: i=1,2,
  \end{equation}
  for any $h \in \mathcal{G}^{\odot n}$ and $F \in \mathbb{D}^{n,2}$, where we write $ I_{n}^{(i)}(h) $ whenever the corresponding $n$-th multiple integral is only with respect to $X^{i}$.
  
  Finally, we mention the following particular cases (the only one we will need in the sequel): if $f,g \in \mathcal{G}$, $n,m \geq 1$ and $i \in \{1,2\}$, then we have the classical multiplication formula
  \begin{equation}\label{product formula}
  I_{n}^{(i)}(f^{\otimes n})I_{m}^{(i)}(g^{\otimes m})= \sum_{r=0}^{n\wedge m} r! \binom{n}{r}\binom{m}{r} I_{n+m-2r}^{(i)}(f^{\otimes n+m -r}\otimes g^{\otimes n+m-r})\langle f,g \rangle^r_{\mathcal{G}}.
  \end{equation}
  We have also the following isometric property,
  \begin{equation}\label{isometry}
  E[I_{n}^{(i)}(f^{\otimes n})^2] = n! \langle f,f \rangle^n_{\mathcal{G}},
  \end{equation}
  and, for $j\in \{1,2\}$,
  \begin{eqnarray}\label{derivative- multiple,integral}
D_{X^{j}}\big(I_{n}^{(i)}(f^{\otimes n})\big) = \left\{
\begin{array}{ccc}
n I_{n-1}^{(i)}(f^{\otimes n-1}) && \mbox{if $i=j$}\\
0 && \mbox{if $i\neq j$}
\end{array}
\right .
\end{eqnarray}
  \subsection{Notation}

  Throughout all the forthcoming proofs, we shall use the following notation.
For all $k,n \in \N$ we write 
\begin{eqnarray*}
\varepsilon_{k2^{-n/2}} &=&\textbf{1}_{[0,k2^{-n/2}]},\quad \delta_{k2^{-n/2}} = \textbf{1}_{[(k-1)2^{-n/2},k2^{-n/2}]}.
\end{eqnarray*}
For all $k \in \Z$, $H \in (0,1)$, we write 
\begin{equation}\label{rho}
\rho(k) = \frac{1}{2}(|k+1|^{2H} + |k-1|^{2H} -2|k|^{2H}).
\end{equation}
For any sufficiently smooth function $f : \R^2 \to \R$, the notation $\partial_{1 \ldots 1 2 \ldots 2}^{k,l}f$ (where the index 1 is repeated $k$ times and the index 2 is repeated $l$ times) means that $f$ is differentiated $k$ times with respect to the first component and $l$ times with respect to the second one. \\
We denote for any $j \in \Z$ , $\Delta_{j,n} f(X^{1},X^{2}):= f \bigg( \frac{X^{1}_{(j+1)2^{-n/2}} + X^{1}_{j2^{-n/2}} }{2},\frac{X^{2}_{(j+1)2^{-n/2}} + X^{2}_{j2^{-n/2}}}{2} \bigg) $.\\
 For $i \in \{1,2\}$, $H \in (0,1)$, $X^{i,n}_j := 2^{\frac{nH}{2}}X^{i}_{j2^{-\frac{n}{2}}}$.

\begin{definition}\label{definition Kni} For any $ t \in \R^+$ and any $n \in \N$, we define :
\begin{eqnarray*}
K^{(1)}_n(f,t) & :=& \frac{1}{24} \sum_{j=0}^{\lfloor 2^{\frac{n}{2}} t \rfloor -1} \Delta_{j,n}\partial_{111} f(X^{1},X^{2}) I^{(1)}_3 \big(\delta_{(j+1)2^{-n/2}}^{\otimes 3} \big)\\
K^{(2)}_n(f,t) &:=& \frac{1}{24} \sum_{j=0}^{\lfloor 2^{\frac{n}{2}} t \rfloor -1} \Delta_{j,n}\partial_{222} f(X^{1},X^{2}) I^{(2)}_3 \big(\delta_{(j+1)2^{-n/2}}^{\otimes 3} \big)\\
\end{eqnarray*}
\begin{eqnarray*}
K^{(3)}_n(f,t) &:=& \frac{1}{8}  \sum_{j=0}^{\lfloor 2^{\frac{n}{2}} t \rfloor -1} \Delta_{j,n}\partial_{122} f(X^{1},X^{2}) I^{(1)}_1 \big(\delta_{(j+1)2^{-n/2}} \big)I^{(2)}_2 \big(\delta_{(j+1)2^{-n/2}}^{\otimes 2} \big)\\
K^{(4)}_n(f,t) &:=& \frac{1}{8}  \sum_{j=0}^{\lfloor 2^{\frac{n}{2}} t \rfloor -1} \Delta_{j,n}\partial_{112} f(X^{1},X^{2}) I^{(1)}_2 \big(\delta_{(j+1)2^{-n/2}}^{\otimes 2} \big)I^{(2)}_1 \big(\delta_{(j+1)2^{-n/2}} \big).
\end{eqnarray*}
\end{definition}
For any $r \in \N^*$ and $\psi \in C_b^{\infty}( \R^{2r}, \R)$, we define $\xi$ as follows :
\begin{equation}
\xi = \psi\big( X^{1}_{s_1},X^{2}_{s_1}, \ldots, X^{1}_{s_r},X^{2}_{s_r} \big), \label{xi}
\end{equation}
where $s_1,\ldots, s_r \in \R$.
In the proofs contained in this paper, $C$ shall denote a positive, finite constant that may change value from line to line.

 \subsection{Some technical lemmas}
  
  A key tool in our analysis will be the next lemma, which can be deduced from  the following Taylor's theorem with remainder.
 
 \begin{theorem}\label{theorem-Taylor} Let $n$ be a nonnegative integer. If $g \in C^n(\R^2)$, then
 \begin{equation}
 g(k) = \sum_{|\alpha| < n} \partial^{\alpha}g(l)\frac{(k-l)^\alpha}{\alpha!} + R_n(l,k),\label{taylor1}
 \end{equation}
 where
 \[
 R_n(l,k) = n \sum_{|\alpha|=n}\frac{(k-l)^\alpha}{\alpha!}\int_0^1 (1-u)^{n-1}[ \partial^\alpha g(l +u(k-l)) -\partial^\alpha g(l) ] du
 \]
 if $n \geq 1$, and $R_0(l,k) = g(k) - g(l)$. In particular, $R_n(l,k) = \sum_{|\alpha|=n}h_{\alpha}(l,k)(k-l)^\alpha$, where $h_{\alpha}$ is a continuous function with $h_{\alpha}(l,l)=0$ for all l. Moreover,
 \[
 \big|R_n(l,k)\big| \leq (n \vee 1)\sum_{|\alpha|=n}M_{\alpha}\big|(k-l)^\alpha\big|,
 \]
 where $M_{\alpha} = \sup \{ |\partial^\alpha g(l +u(k-l)) -\partial^\alpha g(l)| : 0 \leq u \leq 1 \}.$
 
 \end{theorem}
 Thanks to the previous theorem we deduce the following lemma.
 
 \begin{lemma}\label{taylor-expansion} Let $f\in C_b^{13}(\R^2)$, then
 \begin{eqnarray*}
 f(b,d)= f(a,c) &+& \sum_{i=1}^6\sum_{\alpha_1 +\alpha_2 = 2i-1}C(\alpha_1,\alpha_2)\: \partial_{1 \ldots 1 2 \ldots 2}^{\alpha_1,\alpha_2} f\bigg(\frac{a+b}{2}, \frac{c+d}{2}\bigg)(b-a)^{\alpha_1}(d-c)^{\alpha_2}\\
  && + R_{13}\big( (b,d), (a,c) \big),
 \end{eqnarray*}
 where $\alpha_1, \alpha_2 \in \N$, and 
 \[
 |R_{13}\big( (b,d), (a,c) \big)| \leq C_f \sum_{\alpha_1 + \alpha_2 =13}|b-a|^{\alpha_1}|d-c|^{\alpha_2},
 \]
 with $C_f$ is a constant depending only on $f$. On the other hand, we have $C(1,0)=C(0,1)=1$, $C(3,0)=C(0,3)=\frac{1}{24}$ and $C(2,1)=C(1,2)=\frac18$. The other constants:  $C(\alpha_1,\alpha_2)$ could also be determined explicitly, but won't need their explicit values.
 \end{lemma}
  
  \begin{proof}{}
  By applying (\ref{taylor1}) to $f$, we get
  \begin{equation}\label{specific-taylor1}
 f(b,d) = f(a,c) + \sum_{\alpha_1 + \alpha_2 \leq 12}\partial_{1 \ldots 1 2 \ldots 2}^{\alpha_1,\alpha_2} f(a,c)\frac{(b-a)^{\alpha_1}}{\alpha_1 !}\frac{(d-c)^{\alpha_2}}{\alpha_2 !} + \tilde{R}_{13}\big( (b,d), (a,c) \big),
 \end{equation}
 where 
 \[
 |\tilde{R}_{13}\big( (b,d), (a,c) \big)| \leq C_f \sum_{\alpha_1 + \alpha_2 =13}|b-a|^{\alpha_1}|d-c|^{\alpha_2}.
 \]

  For each $i \in \{1, \ldots , 12 \}$ and each $\alpha_1 , \alpha_2 \in \N$ such that $\alpha_1 + \alpha_2 =i$, we define $g_{\alpha_1, \alpha_2,i}$ as $g_{\alpha_1, \alpha_2,i}:= \partial_{1 \ldots 1 2 \ldots 2}^{\alpha_1,\alpha_2} f$ and we set $k:= (a,c)$, $l:=\big(\frac{a+b}{2},\frac{c+d}{2}\big)$. So, by applying (\ref{taylor1}) to $g_{\alpha_1, \alpha_2,i}$ with $k:= (a,c)$, $l:=\big(\frac{a+b}{2},\frac{c+d}{2}\big)$ and $n=13-i$, we get
  \begin{eqnarray*}
  g_{\alpha_1, \alpha_2,i}(a,c) &=& g_{\alpha_1, \alpha_2,i}\big(\frac{a+b}{2}, \frac{c+d}{2} \big) + \sum_{\beta_1 + \beta_2 < 13-i}\partial_{1 \ldots 1 2 \ldots 2}^{\beta_1,\beta_2} g_{\alpha_1, \alpha_2,i}\big(\frac{a+b}{2}, \frac{c+d}{2} \big)\\
  && \times \frac{(a-b)^{\beta_1}}{2^{\beta_1}\beta_1 !}\frac{(c-d)^{\beta_2}}{2^{\beta_2}\beta_2 !}+ \bar{R}_{13-i}\big( (a,c) , \big(\frac{a+b}{2}, \frac{c+d}{2}\big)\big)\\
   &=& g_{\alpha_1, \alpha_2,i}\big(\frac{a+b}{2}, \frac{c+d}{2} \big)+ \sum_{\beta_1 + \beta_2 < 13-i}C(\beta_1,\beta_2)\partial_{1 \ldots 1 2 \ldots 2}^{\beta_1,\beta_2}g_{\alpha_1, \alpha_2,i}\big(\frac{a+b}{2}, \frac{c+d}{2} \big)\\
   && \times(b-a)^{\beta_1} (d-c)^{\beta_2} + \bar{R}_{13-i}\big( (a,c) , \big(\frac{a+b}{2}, \frac{c+d}{2}\big)\big).
  \end{eqnarray*}
 By replacing $g_{\alpha_1, \alpha_2,i}(a,c)$ in (\ref{specific-taylor1}) we get
  \begin{eqnarray}
  f(b,d) = f(a,c) &+& \sum_{\alpha_1 + \alpha_2 \leq 12}\tilde{C}(\alpha_1, \alpha_2)\partial_{1 \ldots 1 2 \ldots 2}^{\alpha_1,\alpha_2}f\big(\frac{a+b}{2}, \frac{c+d}{2}\big)(b-a)^{\alpha_1}(d-c)^{\alpha_2}\notag\\
  &&  + R_{13}\big( (a,c) , \big(\frac{a+b}{2}, \frac{c+d}{2}\big)\big), \label{specific-taylor2}
  \end{eqnarray}
  where 
  \[
  \big| R_{13}\big( (a,c) , \big(\frac{a+b}{2}, \frac{c+d}{2}\big)\big)\big| \leq C_f \sum_{\alpha_1 + \alpha_2 =13}|b-a|^{\alpha_1}|d-c|^{\alpha_2}.
  \]
  
  Let us prove that  $ \forall i \in \{1, \ldots , 6\}$ and  $\forall \alpha_1, \alpha_2 \in \N \: \text{\:such that} \: \alpha_1 + \alpha_2 =2i$, we have
  \begin{equation}
   \tilde{C}(\alpha_1, \alpha_2)=0. \label{proof-taylor}
   \end{equation}
  Indeed, let $f(x,y) = x^{\alpha_1}y^{\alpha_2}$. Thanks to (\ref{specific-taylor2}), we get
  \begin{eqnarray}
  b^{\alpha_1}d^{\alpha_2} = a^{\alpha_1}c^{\alpha_2} &+& \sum_{\beta_1 + \beta_2 \leq 2i-1}\tilde{C}(\beta_1, \beta_2)\partial_{1 \ldots 1 2 \ldots 2}^{\beta_1,\beta_2}f\big(\frac{a+b}{2}, \frac{c+d}{2}\big)(b-a)^{\beta_1}(d-c)^{\beta_2} \notag\\
  && + \alpha_1! \alpha_2!\tilde{C}(\alpha_1, \alpha_2)(b-a)^{\alpha_1}(d-c)^{\alpha_2}.\label{proof-taylor1}
  \end{eqnarray}
  Let us now change  $a$ into $b$ and $c$ into $d$ in the previous formula, so to get
  \begin{eqnarray}
  a^{\alpha_1}c^{\alpha_2} = b^{\alpha_1}d^{\alpha_2} &+& \sum_{\beta_1 + \beta_2 \leq 2i-1}\tilde{C}(\beta_1, \beta_2)\partial_{1 \ldots 1 2 \ldots 2}^{\beta_1,\beta_2}f\big(\frac{b+a}{2}, \frac{d+c}{2}\big)(a-b)^{\beta_1}(c-d)^{\beta_2} \notag\\
  && + \alpha_1! \alpha_2!\tilde{C}(\alpha_1, \alpha_2)(a-b)^{\alpha_1}(c-d)^{\alpha_2}.\label{proof-taylor2}
  \end{eqnarray}
  Observe that if in (\ref{proof-taylor2}) $\beta_1 + \beta_2$ is odd (resp. is even) then $(a-b)^{\beta_1}(c-d)^{\beta_2} = (-1)^{\beta_1 + \beta_2}(b-a)^{\beta_1}(d-c)^{\beta_2}= -(b-a)^{\beta_1}(d-c)^{\beta_2}$ (resp. $(b-a)^{\beta_1}(d-c)^{\beta_2}$). So, by taking the sum of (\ref{proof-taylor1}) and (\ref{proof-taylor2}) we get
  \begin{eqnarray*}
  b^{\alpha_1}d^{\alpha_2} + a^{\alpha_1}c^{\alpha_2}&=& a^{\alpha_1}c^{\alpha_2}+ b^{\alpha_1}d^{\alpha_2} + \sum_{k=1}^{i-1}\sum_{\beta_1 + \beta_2 = 2k}2 \tilde{C}(\beta_1, \beta_2)\partial_{1 \ldots 1 2 \ldots 2}^{\beta_1,\beta_2}f\big(\frac{b+a}{2}, \frac{d+c}{2}\big)(b-a)^{\beta_1}\\
 && \times (d-c)^{\beta_2} + 2 \alpha_1! \alpha_2!\tilde{C}(\alpha_1, \alpha_2)(b-a)^{\alpha_1}(d-c)^{\alpha_2},  
  \end{eqnarray*}
  leading to
  \begin{eqnarray}
  0&=&  \sum_{k=1}^{i-1}\sum_{\beta_1 + \beta_2 = 2k}2 \tilde{C}(\beta_1, \beta_2)\partial_{1 \ldots 1 2 \ldots 2}^{\beta_1,\beta_2}f\big(\frac{b+a}{2}, \frac{d+c}{2}\big)(b-a)^{\beta_1}(d-c)^{\beta_2} \label{proof-taylor3}\\
  && + 2 \alpha_1! \alpha_2!\tilde{C}(\alpha_1, \alpha_2)(b-a)^{\alpha_1}(d-c)^{\alpha_2}. \notag 
  \end{eqnarray}
  We deduce thanks to (\ref{proof-taylor3}) that, for $i=1$ and each $\alpha_1, \alpha_2 \in \N$ satisfying $\alpha_1+ \alpha_2 = 2$, we have $ \forall a, b, c, d \in \R$,
  \[
  2 \alpha_1! \alpha_2!\tilde{C}(\alpha_1, \alpha_2)(b-a)^{\alpha_1}(d-c)^{\alpha_2} =0,
  \]
implying in turn $\tilde{C}(\alpha_1, \alpha_2)=0$. Then, a simple recursive argument shows that for all $ \forall i \in \{1, \ldots , 6 \}$ and $ \forall \alpha_1, \alpha_2 \in \N \: / \: \alpha_1 + \alpha_2 =2i$ we have $\tilde{C}(\alpha_1, \alpha_2)=0$. As a result, (\ref{proof-taylor}) holds true. 
  
  It remains to prove that $\tilde{C}(1,0)=\tilde{C}(0,1)=1$, $\tilde{C}(3,0)=\tilde{C}(0,3)=\frac{1}{24}$ and $\tilde{C}(2,1)=\tilde{C}(1,2)=\frac18$. Thanks to (\ref{specific-taylor2}) and (\ref{proof-taylor}), by taking $f(x,y)=x$ (resp. $f(x,y)=y$) we deduce immediately that $\tilde{C}(1,0)$ (resp. $\tilde{C}(0,1)$) equals 1. By taking $f(x,y)=x^3$ (resp. $f(x,y)=y^3$) we deduce that $\tilde{C}(3,0)$ (resp. $\tilde{C}(0,3)$) equals $\frac{1}{24}$. Finally, by taking $f(x,y)=x^2y$ (resp. $f(x,y)=xy^2$) we deduce that $\tilde{C}(2,1)$ (resp. $\tilde{C}(1,2)$) equals $\frac18$. The proof of  Lemma \ref{taylor-expansion} is complete.
  \end{proof}

The following lemma gathers several estimates that will be needed while completing the proof of our theorems .

\begin{lemma} Suppose that $H < 1/2$. Then
\begin{enumerate}
 
\item  For all $j , k \in \N$ and $u\in\R$,
\begin{eqnarray}
 |\langle \varepsilon_u, \delta_{(j+1)2^{-n/2}} \rangle_\mathcal{G} | &\leq & 2^{-nH},\label{12}
 \end{eqnarray}

\item For all integers $r,n \geq 1$ and all $t\in\R_+$, and with $C_{H,r}$ a constant depending  only on $H$  and $r$ (but independent of $t$ and $n$),
 \begin{eqnarray}\label{13}
 \sum_{k,l=0}^{\lfloor 2^{n/2} t \rfloor - 1}|\langle \delta_{(k+1)2^{-n/2}} ; \delta_{(l+1)2^{-n/2}}\rangle_\mathcal{G}|^r &\leq& C_{H,r}  \: t\:  2^{n(\frac12-rH)}.
 \end{eqnarray}

\item For all integer $n \geq 1$ and all $t\in\R_+$,
\begin{eqnarray}\label{14}
 \sum_{k,l=0}^{\lfloor 2^{n/2} t \rfloor - 1}|\langle \varepsilon_{k2^{-n/2}} ; \delta_{(l+1)2^{-n/2}}\rangle_\mathcal{G}| &\leq& 2^{-nH-1} + 2^{1+n/2} t^{2H + 1},\\
\label{15}
 \sum_{k,l=0}^{\lfloor 2^{n/2} t \rfloor - 1}|\langle \varepsilon_{(k+1)2^{-n/2}} ; \delta_{(l+1)2^{-n/2}}\rangle_\mathcal{G}| &\leq& 2^{-nH-1} + 2^{1+n/2} t^{2H + 1}.
\end{eqnarray}
\end{enumerate}
\end{lemma}
\begin{proof}{ }
\begin{enumerate}
\item[1)]
 We have, for all $0 \leq s \leq t$ and $i \in \{1,2\}$,
 \[ 
 E\big( X^{i}_{u}(X^{i}_{t}-X^{i}_{s})\big) = \frac{1}{2}\big( t^{2H} -s^{2H}\big) + \frac{1}{2}\big(|s-u|^{2H} -|t-u|^{2H}\big).
 \]
Thanks to (\ref{inner-product}), we have 
\[
 \langle \varepsilon_u, \delta_{(j+1)2^{-n/2}} \rangle_\mathcal{G} = E\big( X^{i}_{u}(X^{i}_{(j+1)2^{-n/2}}-X^{i}_{j2^{-n/2}})\big).
 \]
 Since for $H < 1/2$  one has $|b^{2H} - a^{2H}|\leq |b-a|^{2H}$ for any $a, b \in \mathbb{R}_{+}$,  we immediately deduce (\ref{12}).

\item[2)]  Thanks to (\ref{inner-product}), for $i \in \{1,2\}$, we have that 
\[\langle\delta_{(k+1)2^{-n/2}};\delta_{(l+1)2^{-n/2}}\rangle_\mathcal{G}^r = (E[(X^{i}_{(k+1)2^{-n/2}}-X^{i}_{k2^{-n/2}})(X^{i}_{(l+1)2^{-n/2}}-X^{i}_{l2^{-n/2}})])^r.\]
 Thus,
 \begin{eqnarray*}
 &&\sum_{k,l=0}^{\lfloor 2^{n/2} t \rfloor -1}|\langle \delta_{(k+1)2^{-n/2}};\delta_{(l+1)2^{-n/2}}\rangle_\mathcal{G}|^r\\
 & =& 2^{-nrH-r}\sum_{k,l=0}^{\lfloor 2^{n/2} t \rfloor -1}\big| |k-l+1|^{2H} + |k-l-1|^{2H} -2|k-l|^{2H} \big|^r.
 \end{eqnarray*} 
 By first setting $p=k-l$ and then applying Fubini, we get that the latter quantity is equal to: 
\begin{eqnarray*}  
&& 2^{-nrH-r}\sum_{p=1-\lfloor 2^{n/2} t \rfloor}^{\lfloor 2^{n/2} t \rfloor -1}\big| |p+1|^{2H} + |p-1|^{2H} -2|p|^{2H} \big|^r\big( (p+\lfloor 2^{n/2} t \rfloor)\wedge \lfloor 2^{n/2} t \rfloor - p\vee 0 \big) \\
&\leq & 2^{-nrH-r} \lfloor 2^{n/2} t \rfloor \sum_{p=1-\lfloor 2^{n/2} t \rfloor}^{\lfloor 2^{n/2} t \rfloor -1}\big| |p+1|^{2H} + |p-1|^{2H} -2|p|^{2H} \big|^r \\
&\leq & C_{H,r}\: t \:2^{\frac{n}{2} \: - nrH},
\end{eqnarray*}
where $C_{H,r}:=\frac{1}{2^r} \displaystyle{\sum_{p=-\infty}^{+\infty}\big| |p+1|^{2H} + |p-1|^{2H} -2|p|^{2H} \big|^r}$. Observe that $C_{H,r}$ is finite because $H< \frac12$.
This shows (\ref{13}).
\item[3)] Thanks to (\ref{inner-product}), for $i \in \{1,2 \}$, we have that 
\[
\langle \varepsilon_{k2^{-n/2}} ; \delta_{(l+1)2^{-n/2}}\rangle_\mathcal{G} = E[ X^{i}_{k2^{-n/2}}(X^{i}_{(l+1)2^{-n/2}}-X^{i}_{l2^{-n/2}})].
\] 
Thus,
\begin{eqnarray*}
&&\sum_{k,l=0}^{\lfloor 2^{n/2} t \rfloor - 1}|\langle \varepsilon_{k2^{-n/2}} ; \delta_{(l+1)2^{-n/2}}\rangle_\mathcal{G}|\\
 &=& 2^{-nH-1}\sum_{k,l=0}^{\lfloor 2^{n/2} t \rfloor - 1}|(l+1)^{2H} - l^{2H} + |k-l|^{2H} -|k-l-1|^{2H} |\\
 &\leq & 2^{-nH-1}\sum_{k,l=0}^{\lfloor 2^{n/2} t \rfloor - 1}|(l+1)^{2H} - l^{2H}| + 2^{-nH-1}\sum_{k,l=0}^{\lfloor 2^{n/2} t \rfloor - 1} \big| |k-l|^{2H} - |k-l-1|^{2H}\big|.
\end{eqnarray*}
 We have 
 \begin{eqnarray}\label{lem1-3} 
 && 2^{-nH-1}\sum_{k,l=0}^{\lfloor 2^{n/2} t \rfloor - 1}|(l+1)^{2H} - l^{2H}|=2^{-nH-1}\lfloor 2^{n/2} t \rfloor \sum_{l=0}^{\lfloor 2^{n/2} t \rfloor - 1}((l+1)^{2H} - l^{2H})\notag\\
 &=& 2^{-nH-1}\lfloor 2^{n/2} t \rfloor (\lfloor 2^{n/2} t \rfloor)^{2H}\leq \frac{1}{2}2^{n/2}t^{2H+1}.
\end{eqnarray}
On the other hand, a telescoping sum argument leads to
\begin{eqnarray}\label{lem1-3'} 
 2^{-nH-1}\sum_{k,l=0}^{\lfloor 2^{n/2} t \rfloor - 1} \big| |k-l|^{2H} - |k-l-1|^{2H}\big| \leq 2^{-nH-1} + 2^{n/2}t^{2H+1}.
 \end{eqnarray}
 By combining (\ref{lem1-3}) and (\ref{lem1-3'}) we deduce (\ref{14}).
The proof of (\ref{15}) may be done similarly.
\end{enumerate}
\end{proof}

\begin{lemma} Suppose that $H < 1/2$. Then

\begin{enumerate}

\item For $t \geq 0$
\begin{eqnarray}
&&\sum_{j=0}^{\lfloor 2^{\frac{n}{2}} t \rfloor -1}\bigg|\bigg\langle \frac{\varepsilon_{j2^{-n/2}} + \varepsilon_{(j+1)2^{-n/2} }}{2} , \delta_{(j+1)2^{-n/2}} \bigg\rangle_\mathcal{G}\bigg| \leq \frac{1}{2}t^{2H}.\label{lemma2-1}
\end{eqnarray}

\item For $s \in \R$ and $t \geq 0$
\begin{eqnarray}
&&\sum_{j=0}^{\lfloor 2^{\frac{n}{2}} t \rfloor -1}\big|\big\langle \varepsilon_s , \delta_{(j+1)2^{-n/2}} \big\rangle_\mathcal{G}\big| \leq 2t^{2H}.\label{lemma2-3}
\end{eqnarray}
\end{enumerate}
\end{lemma}
\begin{proof}{}

\item[1)] Thanks to (\ref{inner-product}) we have
\begin{eqnarray*}
 \big\langle \varepsilon_{j2^{-n/2}} , \delta_{(j+1)2^{-n/2}} \big\rangle_\mathcal{G} &=& \frac{1}{2}2^{-nH}( |j+1|^{2H} - |j|^{2H} - 1),\\
 \big\langle \varepsilon_{(j+1)2^{-n/2}} , \delta_{(j+1)2^{-n/2}} \big\rangle_\mathcal{G} &=& \frac{1}{2}2^{-nH}( |j+1|^{2H} - |j|^{2H} + 1).
\end{eqnarray*}
Hence, we get that
\begin{equation*}
 \bigg\langle \frac{\varepsilon_{j2^{-n/2}} + \varepsilon_{(j+1)2^{-n/2}}}{2} , \delta_{(j+1)2^{-n/2}} \bigg\rangle_\mathcal{G} = \frac{1}{2}2^{-nH}(|j+1|^{2H} -|j|^{2H}),
 \end{equation*}
and, by a telescoping argument, it yields
\begin{equation*}
 \sum_{j=0}^{\lfloor 2^{\frac{n}{2}} t \rfloor -1} \bigg|\bigg\langle \frac{\varepsilon_{j2^{-n/2}} + \varepsilon_{(j+1)2^{-n/2}}}{2} , \delta_{(j+1)2^{-n/2}} \bigg\rangle_\mathcal{G}\bigg| \leq \frac{1}{2} t^{2H},
\end{equation*}
which proves (\ref{lemma2-1}).

\item[2)] Thanks to (\ref{inner-product}), for $i \in \{1,2\}$, we have
\begin{eqnarray*}
&& \big\langle \varepsilon_s , \delta_{(j+1)2^{-n/2}} \big\rangle_\mathcal{G} = E[X_s^{i}(X^{i}_{(j+1)2^{-n/2}} -X^{i}_{j2^{-n/2}})] \\
&=& \frac{1}{2}2^{-nH}\big( |j+1|^{2H} - |j|^{2H} \big) + \frac{1}{2}2^{-nH}\big( |s2^{-n/2} -j|^{2H} - |s2^{-n/2} - j -1|^{2H}\big).
\end{eqnarray*}
We deduce that 
\begin{enumerate}

\item If $s \geq 0$ : 

\begin{enumerate}

\item if $s \leq t$ : 
\begin{eqnarray*}
 &&\sum_{j=0}^{\lfloor 2^{\frac{n}{2}} t \rfloor -1}\big|\big\langle \varepsilon_s , \delta_{(j+1)2^{-n/2}} \big\rangle_\mathcal{G}\big|\\
 &\leq & \frac{1}{2}t^{2H} + \frac{1}{2}2^{-nH}\sum_{j=0}^{\lfloor 2^{\frac{n}{2}} s \rfloor -1}\big( (s2^{n/2} - j)^{2H}  - (s2^{n/2} - j -1)^{2H} \big)\\
&& + \frac{1}{2}2^{-nH}\big| |s2^{n/2} - \lfloor s2^{\frac{n}{2}}  \rfloor|^{2H} - |s2^{n/2} - \lfloor s2^{\frac{n}{2}}  \rfloor -1|^{2H} \big|\\
&& + \frac{1}{2}2^{-nH}\sum_{j=\lfloor 2^{\frac{n}{2}} s \rfloor +1}^{\lfloor 2^{\frac{n}{2}} t \rfloor -1} \big( (j+1 -s2^{n/2})^{2H} - (j -s2^{n/2})^{2H} \big)\\
&=& \frac{1}{2}t^{2H} + \frac{1}{2}2^{-nH}\big( (s2^{n/2})^{2H} - (s2^{n/2} - \lfloor s2^{\frac{n}{2}}  \rfloor)^{2H} \big)\\
&& +\frac{1}{2}2^{-nH}\big| |s2^{n/2} - \lfloor s2^{\frac{n}{2}}  \rfloor|^{2H} - |s2^{n/2} - \lfloor s2^{\frac{n}{2}}  \rfloor -1|^{2H} \big|\\
&& + \frac{1}{2}2^{-nH}\big( (\lfloor 2^{\frac{n}{2}} t \rfloor -s2^{n/2} )^{2H} - (\lfloor s2^{\frac{n}{2}}  \rfloor -s2^{n/2} +1 )^{2H} \big),
\end{eqnarray*}
where we have obtained the first equality by a telescoping argument. As a consequence of the previous calculation, and  since $|b^{2H} - a^{2H}|\leq |b-a|^{2H}$ for $H <1/2$ and $a, b \in \mathbb{R}_{+}$, we deduce that
\begin{eqnarray*}
&&\sum_{j=0}^{\lfloor 2^{\frac{n}{2}} t \rfloor -1}\big|\big\langle \varepsilon_s , \delta_{(j+1)2^{-n/2}} \big\rangle_\mathcal{G}\big|\\
& \leq &  \frac{1}{2}t^{2H} + \frac{1}{2}2^{-nH} ( s^{2H}2^{nH}) + \frac{1}{2}2^{-nH}1^{2H} + \frac{1}{2}2^{-nH} (\lfloor 2^{\frac{n}{2}} t \rfloor - \lfloor 2^{\frac{n}{2}} s \rfloor -1)^{2H}\\
\end{eqnarray*}
\begin{eqnarray*}
& \leq & \frac{1}{2}t^{2H} + \frac{1}{2}t^{2H} + \frac{1}{2}2^{-nH}(\lfloor 2^{\frac{n}{2}} t \rfloor )^{2H} + \frac{1}{2}2^{-nH} (\lfloor 2^{\frac{n}{2}} t \rfloor )^{2H}\\
&\leq & 2t^{2H},
\end{eqnarray*}
meaning that (\ref{lemma2-3}) holds true.

\item if $s >t$ : by the same argument that was used in the previous case, we have
\begin{eqnarray*}
&&\sum_{j=0}^{\lfloor 2^{\frac{n}{2}} t \rfloor -1}\big|\big\langle \varepsilon_s , \delta_{(j+1)2^{-n/2}} \big\rangle_\mathcal{G}\big|\\
& \leq & \frac{1}{2}t^{2H} + \frac{1}{2}2^{-nH}\sum_{j=0}^{\lfloor 2^{\frac{n}{2}} t \rfloor -1}\big( (s2^{n/2} - j)^{2H}  - (s2^{n/2} - j -1)^{2H} \big)\\
\end{eqnarray*}
\begin{eqnarray*}
&=& \frac{1}{2}t^{2H} + \frac{1}{2}2^{-nH}\big( (s2^{n/2})^{2H} - (s2^{n/2} - \lfloor 2^{\frac{n}{2}} t \rfloor )^{2H} \big)\\
& \leq & \frac{1}{2}t^{2H} + \frac{1}{2}2^{-nH}(\lfloor 2^{\frac{n}{2}} t \rfloor)^{2H} \leq t^{2H},
\end{eqnarray*}
and (\ref{lemma2-3}) holds true.
\end{enumerate}

\item If $ s < 0$ :
\begin{eqnarray*}
&&\sum_{j=0}^{\lfloor 2^{\frac{n}{2}} t \rfloor -1}\big|\big\langle \varepsilon_s , \delta_{(j+1)2^{-n/2}} \big\rangle_\mathcal{G}\big|\\
&\leq & \frac{1}{2} t^{2H} + \frac{1}{2}2^{-nH}\sum_{j=0}^{\lfloor 2^{\frac{n}{2}} t \rfloor -1} \big| |s2^{n/2} -j|^{2H} - |s2^{n/2} - j -1|^{2H}\big|\\
&=& \frac{1}{2} t^{2H} + \frac{1}{2}2^{-nH}\sum_{j=0}^{\lfloor 2^{\frac{n}{2}} t \rfloor -1} \big((j +1 + (-s)2^{n/2} )^{2H} - (j + (-s)2^{n/2} )^{2H} \big)\\
&=& \frac{1}{2} t^{2H} + \frac{1}{2}2^{-nH} \big( (\lfloor 2^{\frac{n}{2}} t \rfloor + (-s)2^{n/2})^{2H} - (-s2^{n/2})^{2H}\big)\\
& \leq & \frac{1}{2} t^{2H} + \frac{1}{2} \big( (t +(-s))^{2H} - (-s)^{2H} \big)\\
& \leq & \frac{1}{2} t^{2H} + \frac{1}{2} t^{2H},
\end{eqnarray*}
where the second equality follows from a telescoping argument and the last inequality follows from the relation $|b^{2H} - a^{2H}|\leq |b-a|^{2H}$ for any $a, b \in \mathbb{R}_{+}$. The last inequality shows that (\ref{lemma2-3}) holds true.
\end{enumerate} 
 \end{proof}
\begin{lemma}\label{biglemma}
Suppose that $f_1, f_2, f_3, f_4 \in C_b^{\infty}(\R^2)$ and set $H=1/6$. Fix $ t \geq 0$. Then 
\begin{enumerate}
\item For $(i,j) \in \{ (1,2) , (2,1)\},$ we have
 \begin{eqnarray}
 &&\sup_{n \geq 1} \sup_{i_1,i_2 \in \{0, \ldots, \lfloor 2^{\frac{n}{2}} t \rfloor -1\}} \sum_{i_3,i_4=0}^{\lfloor 2^{\frac{n}{2}} t \rfloor -1}\bigg|E\bigg(\Delta_{i_1,n}f_1(X^{1},X^{2})\Delta_{i_2,n}f_2(X^{1},X^{2})\Delta_{i_3,n}f_3(X^{1},X^{2})\notag\\
 && \times \Delta_{i_4,n}f_4(X^{1},X^{2})I^{(i)}_1 \big(\delta_{(i_3+1)2^{-n/2}} \big)I^{(j)}_2 \big(\delta_{(i_3+1)2^{-n/2}}^{\otimes 2} \big)I^{(i)}_1 \big(\delta_{(i_4+1)2^{-n/2}} \big)I^{(j)}_2 \big(\delta_{(i_4+1)2^{-n/2}}^{\otimes 2} \big)\bigg)\bigg|\notag\\
 &\leq & C(t +t^2), \label{lemma3}
 \end{eqnarray}
 
 \item For $i \in \{1,2\},$ we have 
 \begin{eqnarray}
 &&\sup_{n \geq 1} \sup_{i_1,i_2 \in \{0, \ldots, \lfloor 2^{\frac{n}{2}} t \rfloor -1\}} \sum_{i_3,i_4=0}^{\lfloor 2^{\frac{n}{2}} t \rfloor -1}\bigg|E\bigg(\Delta_{i_1,n}f_1(X^{1},X^{2})\Delta_{i_2,n}f_2(X^{1},X^{2})\Delta_{i_3,n}f_3(X^{1},X^{2}) \notag\\
 && \hspace{4cm} \times \Delta_{i_4,n}f_4(X^{1},X^{2})I^{(i)}_3 \big(\delta_{(i_3+1)2^{-n/2}}^{\otimes 3} \big)I^{(i)}_3 \big(\delta_{(i_4+1)2^{-n/2}}^{\otimes 3} \big)\bigg)\bigg|\notag\\
 & \leq & C(t + t^2), \label{lemma4}
 \end{eqnarray}
 
 \item For $(i,j) \in \{ (1,2) , (2,1)\},$ we have
 \begin{eqnarray}
 && \sup_{n \geq 1}\sum_{i_1,i_2,i_3,i_4=0}^{\lfloor 2^{\frac{n}{2}} t \rfloor -1}\bigg|E\bigg(\prod_{a=1}^4 \Delta_{i_a,n}f_a(X^{1},X^{2})I^{(i)}_1 \big(\delta_{(i_a+1)2^{-n/2}} \big)I^{(j)}_2 \big(\delta_{(i_a+1)2^{-n/2}}^{\otimes 2} \big)\bigg)\bigg|\notag\\
 &\leq & C(t +t^2+t^3+t^4), \label{lemma4'}
 \end{eqnarray}
 
 \item For $i \in \{1,2\}$, we have
 \begin{eqnarray}
 && \sup_{n \geq 1}\sum_{i_1,i_2,i_3,i_4=0}^{\lfloor 2^{\frac{n}{2}} t \rfloor -1}\bigg|E\bigg(\prod_{a=1}^4 \Delta_{i_a,n}f_a(X^{1},X^{2})I^{(i)}_3 \big(\delta_{(i_a+1)2^{-n/2}}^{\otimes 3} \big)\bigg)\bigg|\notag\\
 &\leq & C(t +t^2+t^3 + t^4), \label{lemma4''}
 \end{eqnarray}
 \end{enumerate}
 where $C$ is a positive constant depending only on $f$ (but independent of $n$ and $t$).
 \end{lemma}
 \begin{proof}{}
 The proof, which is quite long and rather technical, is postponed in Section 5.
 \end{proof}
 
\begin{lemma}\label{lemma5} Suppose $H=1/6$. Then
\begin{enumerate}

\item For all $j \in \N$, for $l\in \{1,2\}$
\begin{eqnarray}
  \big|\big\langle D^l_{X^{2}}\xi , \delta_{(j+1)2^{-n/2}}^{\otimes l} \big\rangle\big| & \leq & C 2^{-nl/6},\label{lemma5-1}
\end{eqnarray}
where $C$ in a positive constant depending only on $\psi$ introduced in (\ref{xi}).

\item For all $j \in \N$, for $i \in \{1,2,3,4\}$, for $l\in \{1,2\}$
\begin{equation}
E[\big\langle D^l_{X^{2}} \big(K^{(i)}_n(f,t)\big) , \delta^{\otimes l}_{(j+1)2^{-n/2}} \big\rangle^2] \leq C 2^{-nl/3}(t^2 +t+1),\label{lemma5-2}
\end{equation}
where $C$ in a positive constant depending only on $f$.

\item For all $j \in \N$, for $i \in \{1,2,3,4\}$
\begin{equation}
E[\big\langle D_{X^{2}} \big(K^{(i)}_n(f,t)\big) , \delta_{(j+1)2^{-n/2}} \big\rangle^4] \leq C 2^{-2n/3}(t^3+t^2 +t+1),\label{lemma5-4}
\end{equation}
where $C$ in a positive constant depending only on $f$.

\end{enumerate}
\end{lemma}
\begin{proof}{}
\begin{enumerate}

\item[1)] For $l=1$, observe that, thanks to (\ref{diabolic-derivation}) (see also the example given after (\ref{diabolic-derivation})), we have
\[
 D_{X^{2}}\xi = \sum_{k=1}^r \frac{\partial \psi}{\partial x_{2k}}\big( X^{1}_{s_1},X^{2}_{s_1}, \ldots, X^{1}_{s_r},X^{2}_{s_r} \big)\varepsilon_{s_k},
 \]
 As a consequence, since $\psi$ is bounded and thanks to (\ref{12}), we deduce that 
 \begin{eqnarray*}
  \big|\big\langle D_{X^{2}}\xi , \delta_{(j+1)2^{-n/2}} \big\rangle\big| &\leq & C \sum_{k=1}^r |\langle \varepsilon_{s_k} , \delta_{(j+1)2^{-n/2}} \rangle |\\
  & \leq & C 2^{-n/6},
  \end{eqnarray*}
  which proves (\ref{lemma5-1}). On the other hand, for $l=2$, we have
  \[
 D^2_{X^{2}}\xi = \sum_{k,k'=1}^r \frac{\partial^2 \psi}{\partial x_{2k}\partial x_{2k'}}\big( X^{1}_{s_1},X^{2}_{s_1}, \ldots, X^{1}_{s_r},X^{2}_{s_r} \big)\varepsilon_{s_k}\otimes\varepsilon_{s_{k'}}.
 \]
 So, as previously, we deduce  that
 \begin{equation*}
  \big|\big\langle D^2_{X^{2}}\xi , \delta_{(j+1)2^{-n/2}}^{\otimes 2} \big\rangle\big| \leq C (2^{-n/6})^2= C2^{-n/3},
  \end{equation*}
  which proves (\ref{lemma5-1}). 
  \item [2)]  We will prove (\ref{lemma5-2}) for $l=1$ and $i =2,3$. The proof is similar for the other values of $i$ and $l$.
  \begin{enumerate}

  \item For $ i=2$ : Thanks to (\ref{Leibnitz0}) and (\ref{derivative- multiple,integral}), we have
  \begin{eqnarray*}
&&D_{X^{2}}K_n^{(2)}(f, t ) \\
&=& \frac{1}{24} \sum_{l=0}^{\lfloor 2^{\frac{n}{2}} t \rfloor -1} \Delta_{l,n}\partial_{2222} f(X^{1},X^{2}) I^{(2)}_3 \big(\delta_{(l+1)2^{-n/2}}^{\otimes 3} \big)\bigg(\frac{\varepsilon_{l2^{-n/2}} + \varepsilon_{(l+1)2^{-n/2} }}{2}\bigg)\\
&& + \frac{1}{8} \sum_{l=0}^{\lfloor 2^{\frac{n}{2}} t \rfloor -1} \Delta_{l,n}\partial_{222} f(X^{1},X^{2}) I^{(2)}_2 \big(\delta_{(l+1)2^{-n/2}}^{\otimes 2} \big)\delta_{(l+1)2^{-n/2}}.
\end{eqnarray*}
So, we deduce that
\begin{eqnarray*}
&& E\big[\big\langle D_{X^{2}}K^{(2)}_n(f,t) , \delta_{(j+1)2^{-n/2}} \big\rangle^2\big]\\
 &\leq & 2(\frac{1}{24})^2 \sum_{l,k=0}^{\lfloor 2^{\frac{n}{2}} t \rfloor -1} \big| E[\Delta_{l,n}\partial_{2222} f(X^{1},X^{2})\Delta_{k,n}\partial_{2222} f(X^{1},X^{2})  \\
 && \hspace{3cm} \times I^{(2)}_3 \big(\delta_{(l+1)2^{-n/2}}^{\otimes 3} \big)I^{(2)}_3 \big(\delta_{(k+1)2^{-n/2}}^{\otimes 3} \big)]\big|\\
 && \hspace{2.5cm}\times \bigg|\bigg\langle\bigg(\frac{\varepsilon_{l2^{-n/2}} + \varepsilon_{(l+1)2^{-n/2} }}{2}\bigg), \delta_{(j+1)2^{-n/2}} \bigg\rangle\bigg|\\
 && \hspace{2.5cm}\times \bigg|\bigg\langle\bigg(\frac{\varepsilon_{k2^{-n/2}} + \varepsilon_{(k+1)2^{-n/2} }}{2}\bigg), \delta_{(j+1)2^{-n/2}} \bigg\rangle\bigg|\\
&& + 2(\frac{1}{8})^2 \sum_{l,k=0}^{\lfloor 2^{\frac{n}{2}} t \rfloor -1}\big|E[ \Delta_{l,n}\partial_{222} f(X^{1},X^{2})\Delta_{k,n}\partial_{222} f(X^{1},X^{2})\\
&& \hspace{3cm} \times  I^{(2)}_2 \big(\delta_{(l+1)2^{-n/2}}^{\otimes 2} \big)I^{(2)}_2 \big(\delta_{(k+1)2^{-n/2}}^{\otimes 2} \big)]\big|\\
&& \hspace{2cm}\times \big|\big\langle \delta_{(l+1)2^{-n/2}}, \delta_{(j+1)2^{-n/2}} \big\rangle\big|\big|\big\langle \delta_{(k+1)2^{-n/2}}, \delta_{(j+1)2^{-n/2}} \big\rangle\big|.
\end{eqnarray*}
By (\ref{lemma4}), (\ref{12}), (\ref{inner-product}), (\ref{rho}) and since
\[
\big|E[ \Delta_{l,n}\partial_{222} f(X^{1},X^{2})\Delta_{k,n}\partial_{222} f(X^{1},X^{2})I^{(2)}_2 \big(\delta_{(l+1)2^{-n/2}}^{\otimes 2} \big)I^{(2)}_2 \big(\delta_{(k+1)2^{-n/2}}^{\otimes 2} \big)]\big|
\] 
is uniformly bounded in $n$ \big(because $f \in C_b^{\infty}(\R^2)$ and thanks to the Cauchy-Schwarz inequality and to (\ref{isometry}) \big), we deduce that
\begin{eqnarray*}
E\big[\big\langle D_{X^{2}}K^{(2)}_n(f,t) , \delta_{(j+1)2^{-n/2}} \big\rangle^2\big] &\leq & C(2^{-n/6})^2(t+t^2) + C(2^{-n/6})^2 (\sum_{n \in \Z}|\rho(n)|)^2 \\
&\leq & C2^{-n/3}(t^2+t+1),
\end{eqnarray*}
which proves (\ref{lemma5-2}) for $i=2$.

\item For $i=3$ : Thanks to (\ref{Leibnitz0}) and (\ref{derivative- multiple,integral}), we have
\begin{eqnarray*}
&& D_{X^{2}}K_n^{(3)}(f, t )\\
&=& \frac{1}{8} \sum_{l=0}^{\lfloor 2^{\frac{n}{2}} t \rfloor -1} \Delta_{l,n}\partial_{1222} f(X^{1},X^{2}) I^{(1)}_1 \big(\delta_{(l+1)2^{-n/2}} \big)I^{(2)}_2 \big(\delta^{\otimes 2}_{(l+1)2^{-n/2}} \big)\\
&& \hspace{2cm}\times \bigg(\frac{\varepsilon_{l2^{-n/2}} + \varepsilon_{(l+1)2^{-n/2} }}{2} \bigg)\\
&& + \frac{1}{4} \sum_{l=0}^{\lfloor 2^{\frac{n}{2}} t \rfloor -1} \Delta_{l,n}\partial_{122} f(X^{1},X^{2}) I^{(1)}_1 \big(\delta_{(l+1)2^{-n/2}} \big)I^{(2)}_1 \big(\delta_{(l+1)2^{-n/2}} \big)\delta_{(l+1)2^{-n/2}} .
\end{eqnarray*}
Hence, we deduce that
\begin{eqnarray*}
&& E[\big\langle D_{X^{2}}K^{(3)}_n(f,t) , \delta_{(j+1)2^{-n/2}} \big\rangle^2]\\
 & \leq & 2(\frac{1}{8})^2 \sum_{l,k=0}^{\lfloor 2^{\frac{n}{2}} t \rfloor -1}\big|E[ \Delta_{l,n}\partial_{1222} f(X^{1},X^{2})\Delta_{k,n}\partial_{1222} f(X^{1},X^{2})\\
 && \hspace{2cm} \times I^{(1)}_1 \big(\delta_{(l+1)2^{-n/2}} \big)I^{(2)}_2 \big(\delta_{(l+1)2^{-n/2}}^{\otimes 2} \big)I^{(1)}_1 \big(\delta_{(k+1)2^{-n/2}} \big)I^{(2)}_2 \big(\delta_{(k+1)2^{-n/2}}^{\otimes 2} \big)]\big|\\
 && \hspace{2cm}\times \bigg|\bigg\langle\bigg(\frac{\varepsilon_{l2^{-n/2}} + \varepsilon_{(l+1)2^{-n/2} }}{2}\bigg), \delta_{(j+1)2^{-n/2}} \bigg\rangle\bigg|\\
 && \hspace{2cm}\times\bigg|\bigg\langle\bigg(\frac{\varepsilon_{k2^{-n/2}} + \varepsilon_{(k+1)2^{-n/2} }}{2}\bigg), \delta_{(j+1)2^{-n/2}} \bigg\rangle\bigg|\\
&& + 2(\frac{1}{4})^2 \sum_{l,k=0}^{\lfloor 2^{\frac{n}{2}} t \rfloor -1} \big|E[\Delta_{l,n}\partial_{122} f(X^{1},X^{2})\Delta_{k,n}\partial_{122} f(X^{1},X^{2})\\
&& \hspace{1cm}\times I^{(1)}_1 \big(\delta_{(l+1)2^{-n/2}} \big)I^{(2)}_1 \big(\delta_{(l+1)2^{-n/2}} \big)I^{(1)}_1 \big(\delta_{(k+1)2^{-n/2}} \big)I^{(2)}_1 \big(\delta_{(k+1)2^{-n/2}} \big)]\big|\\
&& \hspace{3cm}\times \big|\big\langle \delta_{(l+1)2^{-n/2}}, \delta_{(j+1)2^{-n/2}} \big\rangle\big|\big|\big\langle \delta_{(k+1)2^{-n/2}}, \delta_{(j+1)2^{-n/2}} \big\rangle\big|.
\end{eqnarray*}
Since $f \in C_b^{\infty}(\R^2)$, since $X^{1}$ and $X^{2}$ are independent, and thanks to the Cauchy-Schwarz inequality and to (\ref{isometry}), we have
\begin{eqnarray*}
&& \big|E[\Delta_{l,n}\partial_{122} f(X^{1},X^{2})\Delta_{k,n}\partial_{122} f(X^{1},X^{2})I^{(1)}_1 \big(\delta_{(l+1)2^{-n/2}} \big)\\
&& \hspace{3cm} \times I^{(2)}_1 \big(\delta_{(l+1)2^{-n/2}} \big)I^{(1)}_1 \big(\delta_{(k+1)2^{-n/2}} \big)I^{(2)}_1 \big(\delta_{(k+1)2^{-n/2}} \big)]\big|\\
&\leq & C E\big(\big|I^{(1)}_1 \big(\delta_{(l+1)2^{-n/2}} \big)\big|\big|I^{(1)}_1 \big(\delta_{(k+1)2^{-n/2}} \big)\big|\big) E\big(\big|I^{(2)}_1 \big(\delta_{(l+1)2^{-n/2}} \big)\big|\\
&&\hspace{3cm}\times\big|I^{(2)}_1 \big(\delta_{(k+1)2^{-n/2}} \big)\big|\big)\\
&\leq & C \|I^{(1)}_1 \big(\delta_{(l+1)2^{-n/2}} \big) \|_2 \|I^{(1)}_1 \big(\delta_{(k+1)2^{-n/2}} \big) \|_2 \|I^{(2)}_1 \big(\delta_{(l+1)2^{-n/2}} \big) \|_2 \\
&&\hspace{3cm}\times \| I^{(2)}_1 \big(\delta_{(k+1)2^{-n/2}} \big)\|_2\\
& \leq& C (2^{-n/12})^4 \leq C.
\end{eqnarray*}
 Thanks to the previous estimation, to (\ref{lemma3}) and to (\ref{12}) (see also (\ref{rho}) for the definition of $\rho$), we deduce that 
\begin{eqnarray*}
 E[\big\langle D_{X^{2}}K^{(3)}_n(f,t) , \delta_{(j+1)2^{-n/2}} \big\rangle^2] &\leq & C(2^{-n/6})^2(t+t^2) + C(2^{-n/6})^2 (\sum_{n \in \Z}|\rho(n)|)^2\\
 & \leq & C2^{-n/3}(t^2+t+1), 
\end{eqnarray*}
which proves (\ref{lemma5-2}) for $i=3$.
\end{enumerate}
Finally, we have proved (\ref{lemma5-2}).

\item[3)] We will prove (\ref{lemma5-4}) for $i =2$. The proof is similar for the other values of $i$.

\underline{For $i=2$ :}
\begin{eqnarray*}
&& \langle D_{X^{2}}(K_n^{(2)}(f,t)) , \delta_{(j+1)2^{-n/2}} \rangle^4\\
&\leq & C\sum_{i_1,i_2,i_3,i_4=0}^{\lfloor 2^{\frac{n}{2}} t \rfloor -1}\prod_{a=1}^4 \Delta_{i_a,n}\partial_{2222}f(X^{1},X^{2})I^{(2)}_3 \big(\delta_{(i_a+1)2^{-n/2}}^{\otimes 3} \big)\\
&& \times \prod_{a=1}^4 \bigg\langle \frac{\varepsilon_{i_a 2^{-n/2}} +\varepsilon_{(i_a +1) 2^{-n/2}}} {2}, \delta_{(j+1)2^{-n/2}} \bigg\rangle\\
&& + C \sum_{i_1,i_2,i_3,i_4=0}^{\lfloor 2^{\frac{n}{2}} t \rfloor -1}\prod_{a=1}^4 \Delta_{i_a,n}\partial_{222}f(X^{1},X^{2})I^{(2)}_2 \big(\delta_{(i_a+1)2^{-n/2}}^{\otimes 2} \big)\\
&& \times \prod_{a=1}^4 \big\langle \delta_{(i_a+1)2^{-n/2}} , \delta_{(j+1)2^{-n/2}} \big\rangle.
\end{eqnarray*}
Thus, we have
\begin{eqnarray*}
&& E\big(\langle D_{X^{2}}(K_n^{(2)}(f,t)) , \delta_{(j+1)2^{-n/2}} \rangle^4\big)\\
&\leq & C\sum_{i_1,i_2,i_3,i_4=0}^{\lfloor 2^{\frac{n}{2}} t \rfloor -1}\bigg|E\bigg(\prod_{a=1}^4 \Delta_{i_a,n}\partial_{2222}f(X^{1},X^{2})I^{(2)}_3 \big(\delta_{(i_a+1)2^{-n/2}}^{\otimes 3} \big)\bigg)\bigg|\\
&& \times \prod_{a=1}^4 \bigg|\bigg\langle \frac{\varepsilon_{i_a 2^{-n/2}} +\varepsilon_{(i_a +1) 2^{-n/2}}} {2}, \delta_{(j+1)2^{-n/2}} \bigg\rangle\bigg|\\
&& + C \sum_{i_1,i_2,i_3,i_4=0}^{\lfloor 2^{\frac{n}{2}} t \rfloor -1}\bigg|E\bigg(\prod_{a=1}^4 \Delta_{i_a,n}\partial_{222}f(X^{1},X^{2})I^{(2)}_2 \big(\delta_{(i_a+1)2^{-n/2}}^{\otimes 2} \big)\bigg)\bigg|\\
\end{eqnarray*}
\begin{eqnarray*}
&& \times \prod_{a=1}^4 \big|\big\langle \delta_{(i_a+1)2^{-n/2}} , \delta_{(j+1)2^{-n/2}} \big\rangle\big|\\
&=& Z_{n,1}^{(2)}(t) +Z_{n,2}^{(2)}(t).
\end{eqnarray*}
Thanks to (\ref{12}) and to (\ref{lemma4''}), we have
\[
Z_{n,1}^{(2)}(t) \leq C(2^{-n/6})^4(t+t^2+t^3)=C2^{-2n/3}(t+t^2+t^3).
\]
On the other hand, 
\[
\bigg|E\bigg(\prod_{a=1}^4 \Delta_{i_a,n}\partial_{222}f(X^{1},X^{2})I^{(2)}_2 \big(\delta_{(i_a+1)2^{-n/2}}^{\otimes 2} \big)\bigg)\bigg|
\]
is uniformly bounded in $n$. In fact, since $f \in C_b^\infty$ and by applying the Cauchy-Schwarz inequality two times, we get
\begin{equation*}
\bigg|E\bigg(\prod_{a=1}^4 \Delta_{i_a,n}\partial_{222}f(X^{1},X^{2})I^{(2)}_2 \big(\delta_{(i_a+1)2^{-n/2}}^{\otimes 2} \big)\bigg)\bigg| \leq \prod_{a=1}^4 \|I^{(2)}_2 \big(\delta_{(i_a+1)2^{-n/2}}^{\otimes 2} \big)\|_{L^4}.
\end{equation*}
Thanks to the Hypercontractivity property of the $p$th multiple integral with $p \geq 1$ (see for example Theorem 2.7.2 in \cite{NP}), we have for $a \in \{1, \ldots, 4 \}$, 
\begin{eqnarray*}
\|I^{(2)}_2 \big(\delta_{(i_a+1)2^{-n/2}}^{\otimes 2} \big)\|_{L^4} \leq C \|I^{(2)}_2 \big(\delta_{(i_a+1)2^{-n/2}}^{\otimes 2} \big)\|_{L^2} \leq C 2^{-n/6},
\end{eqnarray*}
where $C$ is some positive and finite constant, and we have the last inequality thanks to (\ref{isometry}). We deduce immediately from the last inequality that $\exists C >0$ such that for all $n \in \N$
\begin{equation}
\bigg|E\bigg(\prod_{a=1}^4 \Delta_{i_a,n}\partial_{222}f(X^{1},X^{2})I^{(2)}_2 \big(\delta_{(i_a+1)2^{-n/2}}^{\otimes 2} \big)\bigg)\bigg| \leq C. \label{majoration-apres-tous}
\end{equation} 
So, we get
\begin{eqnarray*}
Z_{n,2}^{(2)}(t) &\leq & C (2^{-n/6})^4 \sum_{i_1,i_2,i_3,i_4=0}^{\lfloor 2^{\frac{n}{2}} t \rfloor -1} |\rho(i_1-j)||\rho(i_2-j)||\rho(i_3-j)||\rho(i_4-j)|\\
& \leq & C2^{-2n/3}(\sum_{r \in \Z}|\rho(r)|)^4 = C2^{-2n/3}.
\end{eqnarray*}
Consequently, we deduce that $E\big(\langle D_{X^{2}}(K_n^{(2)}(f,t)) , \delta_{(j+1)2^{-n/2}} \rangle^4\big)\leq C2^{-2n/3}(1+t+t^2+t^3).$ Hence, we have proved (\ref{lemma5-4}) for $i=2$, which ends the proof of Lemma \ref{lemma5}.
\end{enumerate}
\end{proof}

\section{Proof of Theorem \ref{first-main}}

 We suppose $H < 1/2$. The proof in the case  $H \geq 1/2$ is immediate and, consequently, is left to the reader. 
 \begin{definition} For all $p,q \in \N$ such that $p+ q$ is odd, we define $V_n^{p,q}(f,t)$ as follows:
\begin{eqnarray}
V_n^{p,q}(f,t) := \sum_{j=0}^{\lfloor 2^{\frac{n}{2}} t \rfloor -1}\Delta_{j,n} f(X^{1},X^{2})\big(X^{1}_{(j+1)2^{-n/2}} - X^{1}_{j2^{-n/2}} \big)^p\big(X^{2}_{(j+1)2^{-n/2}} - X^{2}_{j2^{-n/2}} \big)^q.\notag \\\label{pq-varition}
\end{eqnarray}
\end{definition}
We have the following proposition which will play a pivotal role in the sequel :
\begin{prop}\label{proposition, H >1/6} If ($H > 1/6$ and $p+q \geq 3$) or if ($H = 1/6$ and $p+q \geq 5$), then
\begin{equation*}
V_n^{p,q}(f,t)\overset{L^2}{\longrightarrow} 0, \:\:\text{as}\:\: n \to \infty.
\end{equation*}
\end{prop}
\begin{proof}{}
We suppose that $p$ is even and $q$ is odd (the proof when $p$ is odd and $q$ is even is exactly the same). We have, for all $k \in \N^*$, $x^{2k} =\sum_{i=1}^kb_{2k,i}H_{2i}(x)+ b_{2k,0}$ and $x^{2k-1}= \sum_{i=1}^k a_{2k-1,i}H_{2i-1}(x)$, where $H_n$ is the $n$th Hermite polynomial,  $b_{2k,i}$ and $a_{2k-1,i}$ are some explicit constants (if interested, the reader can find these explicit constants, e.g.,  in \cite[Corollary 1.2]{zeineddine}).
Set
\[
\phi(j,j'):= \Delta_{j,n} f(X^{1},X^{2})\Delta_{j',n} f(X^{1},X^{2}).
\] 
Recall that for $i \in \{1,2\}$ we denote  $X^{i,n}_j := 2^{\frac{nH}{2}}X^{i}_{j2^{-\frac{n}{2}}}$. We distinguish two cases: if $p\neq 0$ and if $p=0$\,,
\begin{enumerate}
\item \underline{If $p\neq 0$ :} Then, we have
\begin{eqnarray}
V_n^{p,q}(f,t)&=& b_{p,0}2^{-\frac{nH(p+q)}{2}}\sum_{k'=1}^{\frac{q+1}{2}}a_{q,k'}\sum_{j=0}^{\lfloor 2^{\frac{n}{2}} t \rfloor -1}\Delta_{j,n} f(X^{1},X^{2})H_{2k'-1}\big(X^{2,n}_{j+1} - X^{2,n}_{j}\big)\notag\\
&& + 2^{-\frac{nH(p+q)}{2}}\sum_{k=1}^{\frac{p}{2}}\sum_{k'=1}^{\frac{q+1}{2}} a_{q,k'}b_{p,k}\sum_{j=0}^{\lfloor 2^{\frac{n}{2}} t \rfloor -1}\Delta_{j,n} f(X^{1},X^{2})H_{2k}\big( X^{1,n}_{j+1} - X^{1,n}_{j}\big)\notag\\
&& \hspace{6cm} \times H_{2k'-1}\big(X^{2,n}_{j+1} - X^{2,n}_{j}\big)\notag\\
&=& b_{p,0}\sum_{k'=1}^{\frac{q+1}{2}}a_{q,k'} V_{n,1}(f,k',t) + \sum_{k=1}^{\frac{p}{2}}\sum_{k'=1}^{\frac{q+1}{2}} a_{q,k'}b_{p,k} V_{n,2}(f,k,k',t),\label{conclusion,prop H>1/6}
\end{eqnarray} 
with obvious notation at the last equality. Let us now prove the convergence to 0 as $n$ tends to infinity, in the $L^2$ sense, of $V_{n,1}(f,k',t)$ and $ V_{n,2}(f,k,k',t)$ for all $k \in \{1, \ldots , \frac{p}{2}\}$ and $k' \in \{1, \ldots ,\frac{q+1}{2}\}$,
\begin{enumerate}
 
 \item \underline{Convergence to 0, in $L^2$, of $V_{n,1}(f,k',t)$ :} For all $k' \in \{1, \ldots ,\frac{q+1}{2}\}$,
 \begin{eqnarray*}
 &&E\big[ \big(V_{n,1}(f,k',t)\big)^2 \big] \\
 &=& 2^{-nH(p+q)}\sum_{j,j'=0}^{\lfloor 2^{\frac{n}{2}} t \rfloor -1}E\bigg(\phi(j,j')H_{2k'-1}\big(X^{2,n}_{j+1} - X^{2,n}_{j}\big) H_{2k'-1}\big(X^{2,n}_{j'+1} - X^{2,n}_{j'}\big)\bigg)\\
 &=& 2^{-nH\big(p+q - (2k'-1)\big)}\sum_{j,j'=0}^{\lfloor 2^{\frac{n}{2}} t \rfloor -1}E\bigg(\phi(j,j')I^{(2)}_{2k'-1}\big(\delta_{(j+1)2^{-n/2}}^{\otimes 2k'-1}\big) I^{(2)}_{2k'-1}\big(\delta_{(j'+1)2^{-n/2}}^{\otimes 2k'-1}\big)\bigg)\\
 \end{eqnarray*}
 \begin{eqnarray*}
 &=& 2^{-nH\big(p+q - (2k'-1)\big)}\sum_{a=0}^{2k'-1} a! \binom{2k'-1}{a}^2\sum_{j,j'=0}^{\lfloor 2^{\frac{n}{2}} t \rfloor -1}E\bigg(\phi(j,j')\\
 && \hspace{2cm} \times I^{(2)}_{4k'-2-2a}\big(\delta_{(j+1)2^{-n/2}}^{\otimes 2k'-1-a}\otimes\delta_{(j'+1)2^{-n/2}}^{\otimes 2k'-1-a}\big)\bigg) \langle \delta_{(j+1)2^{-n/2}}, \delta_{(j'+1)2^{-n/2}} \rangle^a\\
 &=& 2^{-nH\big(p+q - (2k'-1)\big)}\sum_{a=0}^{2k'-1} a! \binom{2k'-1}{a}^2\sum_{j,j'=0}^{\lfloor 2^{\frac{n}{2}} t \rfloor -1}E\bigg(\bigg \langle D_{X^{2}}^{4k'-2-2a}\big(\phi(j,j')\big) ,\\
 && \hspace{2cm} \delta_{(j+1)2^{-n/2}}^{\otimes 2k'-1-a}\otimes\delta_{(j'+1)2^{-n/2}}^{\otimes 2k'-1-a}\bigg\rangle \bigg) \langle \delta_{(j+1)2^{-n/2}}, \delta_{(j'+1)2^{-n/2}} \rangle^a\\
 &=& \sum_{a=0}^{2k'-1}a! \binom{2k'-1}{a}^2 Q_n^{(k',a)}(t),
 \end{eqnarray*}
 with obvious notation at the last equality and  with the second equality following from (\ref{linear-isometry}), the third one from (\ref{product formula}) and the fourth one from (\ref{duality formula}). Recall that $f \in C_b^{\infty}$. We have the following estimates.
 \begin{itemize}
 
 \item \underline{Case $a= 2k'-1$} 
 \begin{eqnarray*}
 |Q_n^{(k',2k'-1)}(t)| &\leq & 2^{-nH\big(p+q - (2k'-1)\big)}\sum_{j,j'=0}^{\lfloor 2^{\frac{n}{2}} t \rfloor -1}E\big( \big|\phi(j,j')\big| \big)\\
 && \hspace{4cm} \times \big|\langle \delta_{(j+1)2^{-n/2}}, \delta_{(j'+1)2^{-n/2}} \rangle\big|^{2k'-1}\\
 & \leq & C 2^{-nH\big(p+q - (2k'-1)\big)} t 2^{n(\frac12 - (2k'-1)H)}= C t 2^{-n[H(p+q) - \frac12]},
 \end{eqnarray*}
 where we have the second inequality by (\ref{13}).

 \item \underline{Preparation to the cases where $0\leq a \leq 2k'-2$} \\
  
 Thanks to (\ref{Leibnitz1}) we have
 
 \begin{eqnarray}
 && D_{X^{2}}^{4k'-2-2a}\big(\phi(j,j')\big)=\sum_{l=0}^{4k'-2-2a}\phi_l(j,j')\label{differential, h>1/6}\\
 &&\times \bigg(\frac{\varepsilon_{j2^{-n/2}} + \varepsilon_{(j+1)2^{-n/2} }}{2}\bigg)^l \tilde{\otimes}\bigg(\frac{\varepsilon_{j'2^{-n/2}} + \varepsilon_{(j'+1)2^{-n/2} }}{2}\bigg)^{4k'-2-2a -l},\notag
 \end{eqnarray}
 where $\phi_l(j,j')$ is a quantity having a similar form as $\phi(j,j')$. So, we have
 \item \underline{Case $ 1\leq a \leq 2k'-2$} 
 \begin{eqnarray*}
 &&|Q_n^{(k',a)}(t)|\\
  &\leq & 2^{-nH\big(p+q - (2k'-1)\big)}\sum_{l=0}^{4k'-2-2a}\sum_{j,j'=0}^{\lfloor 2^{\frac{n}{2}} t \rfloor -1}E\big(\big|\phi_l(j,j')\big|\big)\times \\
  \end{eqnarray*}
  \begin{eqnarray*}
 && \bigg| \bigg \langle \bigg(\frac{\varepsilon_{j2^{-n/2}} + \varepsilon_{(j+1)2^{-n/2} }}{2}\bigg)^l \tilde{\otimes}\bigg(\frac{\varepsilon_{j'2^{-n/2}} + \varepsilon_{(j'+1)2^{-n/2} }}{2}\bigg)^{4k'-2-2a -l},\\
 &&  \delta_{(j+1)2^{-n/2}}^{\otimes 2k'-1-a}\otimes\delta_{(j'+1)2^{-n/2}}^{\otimes 2k'-1-a} \bigg \rangle \bigg|\big|\langle \delta_{(j+1)2^{-n/2}}, \delta_{(j'+1)2^{-n/2}} \rangle\big|^a\\
 &\leq & C2^{-nH\big(p+q - (2k'-1)\big)}(2^{-nH})^{4k'-2-2a}\sum_{j,j'=0}^{\lfloor 2^{\frac{n}{2}} t \rfloor -1}\big|\langle \delta_{(j+1)2^{-n/2}}, \delta_{(j'+1)2^{-n/2}} \rangle\big|^a\\
 & \leq & Ct2^{-nH\big(p+q - (2k'-1)\big)}(2^{-nH})^{4k'-2-2a}2^{n(\frac12 - aH)}\\
 &=& Ct2^{-n[H(p+q+2k'-1-a)- \frac12]},
 \end{eqnarray*}
 where we have the second inequality thanks to (\ref{12}) and the third one thanks to (\ref{13}).
 \item \underline{Case $a=0$} 
 \begin{eqnarray}
 &&|Q_n^{(k',0)}(t)|\\
  &\leq & 2^{-nH\big(p+q - (2k'-1)\big)}\sum_{l=0}^{4k'-2}\sum_{j,j'=0}^{\lfloor 2^{\frac{n}{2}} t \rfloor -1}E\big(\big|\phi_l(j,j')\big|\big)\times \notag\\
 && \bigg| \bigg \langle \bigg(\frac{\varepsilon_{j2^{-n/2}} + \varepsilon_{(j+1)2^{-n/2} }}{2}\bigg)^l \tilde{\otimes}\bigg(\frac{\varepsilon_{j'2^{-n/2}} + \varepsilon_{(j'+1)2^{-n/2} }}{2}\bigg)^{4k'-2 -l}, \notag\\
 &&  \hspace{5cm}\delta_{(j+1)2^{-n/2}}^{\otimes 2k'-1}\otimes\delta_{(j'+1)2^{-n/2}}^{\otimes 2k'-1} \bigg \rangle \bigg|\notag\\
 &\leq & C 2^{-nH\big(p+q - (2k'-1)\big)}\sum_{l=0}^{4k'-2}\sum_{j,j'=0}^{\lfloor 2^{\frac{n}{2}} t \rfloor -1} \bigg| \bigg \langle \bigg(\frac{\varepsilon_{j2^{-n/2}} + \varepsilon_{(j+1)2^{-n/2} }}{2}\bigg)^l \notag \\         
 && \tilde{\otimes}\bigg(\frac{\varepsilon_{j'2^{-n/2}} + \varepsilon_{(j'+1)2^{-n/2} }}{2}\bigg)^{4k'-2 -l},
 \delta_{(j+1)2^{-n/2}}^{\otimes 2k'-1}\otimes\delta_{(j'+1)2^{-n/2}}^{\otimes 2k'-1} \bigg \rangle \bigg|\notag\\ \label{Q_n^0}
 \end{eqnarray}
 We define 
 \begin{eqnarray*}
 E_n^{(k',l)}(j,j')&:=& \bigg| \bigg \langle \bigg(\frac{\varepsilon_{j2^{-n/2}} + \varepsilon_{(j+1)2^{-n/2} }}{2}\bigg)^l \tilde{\otimes}\bigg(\frac{\varepsilon_{j'2^{-n/2}} + \varepsilon_{(j'+1)2^{-n/2} }}{2}\bigg)^{4k'-2 -l},\\
 && \hspace{4cm} \delta_{(j+1)2^{-n/2}}^{\otimes 2k'-1}\otimes\delta_{(j'+1)2^{-n/2}}^{\otimes 2k'-1} \bigg \rangle \bigg|,
 \end{eqnarray*}
 If $l=0$, 
 observe  by (\ref{12}) that 
 \[ 
 E_n^{(k',0)}(j,j') \leq (2^{-nH})^{4k'-3}\: \bigg|\bigg \langle \bigg(\frac{\varepsilon_{j'2^{-n/2}} + \varepsilon_{(j'+1)2^{-n/2} }}{2}\bigg),\delta_{(j+1)2^{-n/2}} \bigg \rangle \bigg|.
 \]
 If $l\neq 0$ then
 \[
 E_n^{(k',l)}(j,j') \leq (2^{-nH})^{4k'-3}\: \bigg|\bigg \langle \bigg(\frac{\varepsilon_{j2^{-n/2}} + \varepsilon_{(j+1)2^{-n/2} }}{2}\bigg),\delta_{(j'+1)2^{-n/2}} \bigg \rangle \bigg|.
 \]
By combining these previous estimates with (\ref{Q_n^0}), (\ref{14}) and (\ref{15}) we deduce that
\begin{eqnarray*}
|Q_n^{(k',0)}(t)| &\leq & C 2^{-nH\big(p+q - (2k'-1)\big)}(2^{-nH})^{4k'-3}(2^{-nH-1} + 2^{1+n/2} t^{2H + 1})\\
& \leq & C2^{-nH[p+q +2k'-1]} + C2^{-n[H(p+q+2k'-2) -\frac12]}t^{2H+1}.
\end{eqnarray*} 
 \end{itemize}
 Finally, we deduce that for all $k' \in \{1, \ldots ,\frac{q+1}{2}\}$,
 \begin{eqnarray}
 E\big[ \big(V_{n,1}(f,k',t)\big)^2 \big] &\leq & Ct\big( \sum_{a=1}^{2k'-1}2^{-n[H(p+q + 2k'-1-a) - \frac12]}\big) \label{V_{n,1}}\\
 && + C2^{-n[H(p+q+2k'-2) -\frac12]}t^{2H+1} + C2^{-nH[p+q +2k'-1]}.  \notag
\end{eqnarray}     
It is clear that either for $H > 1/6$ and $p+q \geq 3$ or for $H =1/6$ and $p+q \geq 5$, we have  $V_{n,1}(f,k',t) \overset{L^2}{\longrightarrow} 0$ as $n$ tends to infinity.

 \item \underline{Convergence to 0, in $L^2$, of $ V_{n,2}(f,k,k',t)$ :} For all $k \in \{1, \ldots , \frac{p}{2}\}$ and $k' \in \{1, \ldots ,\frac{q+1}{2}\}$, thanks to (\ref{linear-isometry}) we have
 \begin{eqnarray*}
 V_{n,2}(f,k,k',t)&=& 2^{-\frac{nH(p+q)}{2}}\sum_{j=0}^{\lfloor 2^{\frac{n}{2}} t \rfloor -1}\Delta_{j,n} f(X^{1},X^{2})H_{2k}\big( X^{1,n}_{j+1} - X^{1,n}_{j}\big)\\
&& \hspace{3cm} \times H_{2k'-1}\big(X^{2,n}_{j+1} - X^{2,n}_{j}\big)\\
&=& 2^{-\frac{nH(p+q)}{2}}\sum_{j=0}^{\lfloor 2^{\frac{n}{2}} t \rfloor -1}\Delta_{j,n} f(X^{1},X^{2})I^{(1)}_{2k}\big( (2^{\frac{nH}{2}}\delta_{(j+1)2^{-n/2}})^{\otimes 2k}\big)\\
&& \hspace{3cm} \times I^{(2)}_{2k'-1}\big((2^{\frac{nH}{2}}\delta_{(j+1)2^{-n/2}})^{\otimes 2k'-1} \big).\\
 \end{eqnarray*}
We thus get
 \begin{eqnarray*}
 && E\big[ \big(V_{n,2}(f,k,k',t)\big)^2 \big]\\
  &=& 2^{-nH[p+q - (2k'-1)]}\sum_{j,j'=0}^{\lfloor 2^{\frac{n}{2}} t \rfloor -1}E\bigg( \phi(j,j')I^{(1)}_{2k}\big( (2^{\frac{nH}{2}}\delta_{(j+1)2^{-n/2}})^{\otimes 2k}\big)\\
  && \hspace{2cm}\times I^{(1)}_{2k}\big( (2^{\frac{nH}{2}}\delta_{(j'+1)2^{-n/2}})^{\otimes 2k}\big)I^{(2)}_{2k'-1}\big(\delta_{(j+1)2^{-n/2}}^{\otimes 2k'-1} \big)I^{(2)}_{2k'-1}\big(\delta_{(j'+1)2^{-n/2}}^{\otimes 2k'-1} \big)\bigg)\\
&=& 2^{-nH\big(p+q - (2k'-1)\big)}\sum_{a=0}^{2k'-1} a! \binom{2k'-1}{a}^2\sum_{j,j'=0}^{\lfloor 2^{\frac{n}{2}} t \rfloor -1}E\bigg(\phi(j,j')\\
&& \times I^{(1)}_{2k}\big( (2^{\frac{nH}{2}}\delta_{(j+1)2^{-n/2}})^{\otimes 2k}\big)I^{(1)}_{2k}\big( (2^{\frac{nH}{2}}\delta_{(j'+1)2^{-n/2}})^{\otimes 2k}\big)\\
&&  \times I^{(2)}_{4k'-2-2a}\big(\delta_{(j+1)2^{-n/2}}^{\otimes 2k'-1-a}\otimes\delta_{(j'+1)2^{-n/2}}^{\otimes 2k'-1-a}\big)\bigg)\langle \delta_{(j+1)2^{-n/2}}, \delta_{(j'+1)2^{-n/2}} \rangle^a\\
\end{eqnarray*}
\begin{eqnarray*}
&=& 2^{-nH\big(p+q - (2k'-1)\big)}\sum_{a=0}^{2k'-1} a! \binom{2k'-1}{a}^2\sum_{j,j'=0}^{\lfloor 2^{\frac{n}{2}} t \rfloor -1}E\bigg(\bigg \langle D_{X^{2}}^{4k'-2-2a}\big(\phi(j,j')\big) ,\\
 && \delta_{(j+1)2^{-n/2}}^{\otimes 2k'-1-a}\otimes\delta_{(j'+1)2^{-n/2}}^{\otimes 2k'-1-a}\bigg\rangle I^{(1)}_{2k}\big( (2^{\frac{nH}{2}}\delta_{(j+1)2^{-n/2}})^{\otimes 2k}\big)I^{(1)}_{2k}\big( (2^{\frac{nH}{2}}\delta_{(j'+1)2^{-n/2}})^{\otimes 2k}\big)\bigg)\\
 && \hspace{2cm}\times \langle \delta_{(j+1)2^{-n/2}}, \delta_{(j'+1)2^{-n/2}} \rangle^a\\
 &=& \sum_{a=0}^{2k'-1} a! \binom{2k'-1}{a}^2 \tilde{Q}_n^{(k',a)}(t),
 \end{eqnarray*}
with obvious notation at the last equality. Recall that $f \in C_b^{\infty}$. We have the following estimates.
 \begin{itemize}
 
 \item \underline{Case $a=2k'-1$}
 \begin{eqnarray*}
 &&\big|\tilde{Q}_n^{(k',2k'-1)}(t)\big|\\
  &\leq & 2^{-nH\big(p+q - (2k'-1)\big)}\sum_{j,j'=0}^{\lfloor 2^{\frac{n}{2}} t \rfloor -1}E\bigg( \big|\phi(j,j') I^{(1)}_{2k}\big( (2^{\frac{nH}{2}}\delta_{(j+1)2^{-n/2}})^{\otimes 2k}\big)\\
 && \hspace{1cm} \times I^{(1)}_{2k}\big( (2^{\frac{nH}{2}}\delta_{(j'+1)2^{-n/2}})^{\otimes 2k}\big)\bigg)\big| \big) \big|\langle \delta_{(j+1)2^{-n/2}}, \delta_{(j'+1)2^{-n/2}} \rangle\big|^{2k'-1}
 \end{eqnarray*}
 \begin{eqnarray*}
 &\leq & C 2^{-nH\big(p+q - (2k'-1)\big)}\sum_{j,j'=0}^{\lfloor 2^{\frac{n}{2}} t \rfloor -1} \|I^{(1)}_{2k}\big( (2^{\frac{nH}{2}}\delta_{(j+1)2^{-n/2}})^{\otimes 2k}\big)\|_2 \\
 && \hspace{1cm} \times \|I^{(1)}_{2k}\big( (2^{\frac{nH}{2}}\delta_{(j'+1)2^{-n/2}})^{\otimes 2k}\big)\|_2\big|\langle \delta_{(j+1)2^{-n/2}}, \delta_{(j'+1)2^{-n/2}} \rangle\big|^{2k'-1}\\
 &\leq & C2^{-nH\big(p+q - (2k'-1)\big)}\sum_{j,j'=0}^{\lfloor 2^{\frac{n}{2}} t \rfloor -1}\big|\langle \delta_{(j+1)2^{-n/2}}, \delta_{(j'+1)2^{-n/2}} \rangle\big|^{2k'-1}\\
 &\leq & C t 2^{-nH\big(p+q - (2k'-1)\big)} 2^{n(\frac12 - (2k'-1)H)} = C t 2^{-n[H(p+q) - \frac12]},
 \end{eqnarray*}
 where the third inequality is a consequence of (\ref{isometry}) and the fourth one a consequence of (\ref{13}).
 
 \item \underline{Preparation to the cases where $0 \leq a \leq 2k'-2$} \\
 Thanks to (\ref{differential, h>1/6}), we have
 \begin{eqnarray*}
 && \bigg|E\bigg(\bigg \langle D_{X^{2}}^{4k'-2-2a}\big(\phi(j,j')\big) ,\delta_{(j+1)2^{-n/2}}^{\otimes 2k'-1-a}\otimes\delta_{(j'+1)2^{-n/2}}^{\otimes 2k'-1-a}\bigg\rangle\\
 && \hspace{3cm}  I^{(1)}_{2k}\big( (2^{\frac{nH}{2}}\delta_{(j+1)2^{-n/2}})^{\otimes 2k}\big)I^{(1)}_{2k}\big( (2^{\frac{nH}{2}}\delta_{(j'+1)2^{-n/2}})^{\otimes 2k}\big)\bigg)\bigg|\\
 & \leq &\sum_{l=0}^{4k'-2-2a}E\bigg(\big|\phi_l(j,j')I^{(1)}_{2k}\big( (2^{\frac{nH}{2}}\delta_{(j+1)2^{-n/2}})^{\otimes 2k}\big)I^{(1)}_{2k}\big( (2^{\frac{nH}{2}}\delta_{(j'+1)2^{-n/2}})^{\otimes 2k}\big) \big|\bigg)\\
 \end{eqnarray*}
 \begin{eqnarray*}
 && \times\bigg| \bigg \langle \bigg(\frac{\varepsilon_{j2^{-n/2}} + \varepsilon_{(j+1)2^{-n/2} }}{2}\bigg)^l \tilde{\otimes}\bigg(\frac{\varepsilon_{j'2^{-n/2}} + \varepsilon_{(j'+1)2^{-n/2} }}{2}\bigg)^{4k'-2-2a -l} ,\\
 && \hspace{7cm} \delta_{(j+1)2^{-n/2}}^{\otimes 2k'-1-a}\otimes\delta_{(j'+1)2^{-n/2}}^{\otimes 2k'-1-a}\bigg\rangle\bigg|\\
 &\leq & C \sum_{l=0}^{4k'-2-2a}\|I^{(1)}_{2k}\big( (2^{\frac{nH}{2}}\delta_{(j+1)2^{-n/2}})^{\otimes 2k}\big)\|_2 \|I^{(1)}_{2k}\big( (2^{\frac{nH}{2}}\delta_{(j'+1)2^{-n/2}})^{\otimes 2k}\big)\|_2 \\
 && \bigg| \bigg \langle \bigg(\frac{\varepsilon_{j2^{-n/2}} + \varepsilon_{(j+1)2^{-n/2} }}{2}\bigg)^l \tilde{\otimes}\bigg(\frac{\varepsilon_{j'2^{-n/2}} + \varepsilon_{(j'+1)2^{-n/2} }}{2}\bigg)^{4k'-2-2a -l} ,\\
 && \hspace{7cm} \delta_{(j+1)2^{-n/2}}^{\otimes 2k'-1-a}\otimes\delta_{(j'+1)2^{-n/2}}^{\otimes 2k'-1-a}\bigg\rangle\bigg|\\
 \end{eqnarray*}
 \begin{eqnarray*}
 &\leq & C\sum_{l=0}^{4k'-2-2a} \bigg| \bigg \langle \bigg(\frac{\varepsilon_{j2^{-n/2}} + \varepsilon_{(j+1)2^{-n/2} }}{2}\bigg)^l \tilde{\otimes}\bigg(\frac{\varepsilon_{j'2^{-n/2}} + \varepsilon_{(j'+1)2^{-n/2} }}{2}\bigg)^{4k'-2-2a -l} ,\\
 && \hspace{7cm} \delta_{(j+1)2^{-n/2}}^{\otimes 2k'-1-a}\otimes\delta_{(j'+1)2^{-n/2}}^{\otimes 2k'-1-a}\bigg\rangle\bigg|,\\
 \end{eqnarray*}
By the same arguments that was used in the study of $ V_{n,1}(f,k',t)$, we deduce that
 
 \item \underline{Case $1\leq a \leq 2k'-2$}
 \begin{equation*}
 |\tilde{Q}_n^{(k',a)}(t)| \leq   Ct2^{-n[H(p+q+2k'-1-a)- \frac12]}.
 \end{equation*}
 
 \item \underline{Case $a=0$}
 \begin{equation*}
|\tilde{Q}_n^{(k',0)}(t)| \leq  C2^{-nH[p+q +2k'-1]} + C2^{-n[H(p+q+2k'-2) -\frac12]}t^{2H+1}.
\end{equation*} 
 \end{itemize}
 Finally, we deduce that, for all $k \in \{1, \ldots , \frac{p}{2}\}$ and $k' \in \{1, \ldots ,\frac{q+1}{2}\}$,
 \begin{eqnarray}
 E\big[ \big(V_{n,2}(f,k,k',t)\big)^2 \big] &\leq & Ct\big( \sum_{a=1}^{2k'-1}2^{-n[H(p+q + 2k'-1-a) - \frac12]}\big) \label{V_{n,2}}\\
 && + C2^{-n[H(p+q+2k'-2) -\frac12]}t^{2H+1} + C2^{-nH[p+q +2k'-1]}. \notag
\end{eqnarray}     
It is clear that either for $H > 1/6$ and $p+q \geq 3$ or for $H =1/6$ and $p+q \geq 5$, we have  $V_{n,2}(f,k,k',t) \overset{L^2}{\longrightarrow} 0$ as $n$ tends to infinity.
\end{enumerate}
Combining (\ref{conclusion,prop H>1/6}), (\ref{V_{n,1}}) and (\ref{V_{n,2}}), we deduce that 
\begin{eqnarray}
E\big[ \big(V_n^{p,q}(f,t)\big)^2 \big] &\leq & C\sum_{k'=1}^{\frac{q+1}{2}} \bigg(\big( \sum_{a=1}^{2k'-1}2^{-n[H(p+q + 2k'-1-a) - \frac12]}\big)t \label{norme2 Vnpq}\\
 && + 2^{-n[H(p+q+2k'-2) -\frac12]}t^{2H+1} + 2^{-nH[p+q +2k'-1]}\bigg). \notag 
\end{eqnarray}
Hence, the desired conclusion of Proposition \ref{proposition, H >1/6} holds true when $p\neq 0$.

\item \underline{If $p=0$ :} We have
\begin{equation*}
V_n^{0,q}(f,t)= 2^{-\frac{nHq}{2}}\sum_{k'=1}^{\frac{q+1}{2}}a_{q,k'}\sum_{j=0}^{\lfloor 2^{\frac{n}{2}} t \rfloor -1}\Delta_{j,n} f(X^{1},X^{2})H_{2k'-1}\big(X^{2,n}_{j+1} - X^{2,n}_{j}\big).
\end{equation*}
By the same calculation as in the previous case, we deduce that
\begin{eqnarray}
E\big[ \big(V_n^{0,q}(f,t)\big)^2 \big] &\leq & C\sum_{k'=1}^{\frac{q+1}{2}} \bigg(\big( \sum_{a=1}^{2k'-1}2^{-n[H(q + 2k'-1-a) - \frac12]}\big)t \label{norme2 Vn0q}\\
 && + 2^{-n[H(q+2k'-2) -\frac12]}t^{2H+1} + 2^{-nH[q +2k'-1]}\bigg). \notag 
\end{eqnarray}
Hence, the desired conclusion of Proposition \ref{proposition, H >1/6} holds true when $p= 0$ as well,  which ends up the proof of Proposition \ref{proposition, H >1/6}. 
\end{enumerate}
\end{proof}

\subsection{Proof of (\ref{first-main-1})}

Thanks to Lemma \ref{taylor-expansion}, we have 
\begin{eqnarray*}
 && f(X^{1}_{(j+1)2^{-n/2}}, X^{2}_{(j+1)2^{-n/2}})-f(X^{1}_{j2^{-n/2}},X^{2}_{j2^{-n/2}} )\\
 &=& \Delta_{j,n}\frac{\partial f}{\partial x}\big(X^{1}, X^{2}\big)\big(X^{1}_{(j+1)2^{-n/2}}-X^{1}_{j2^{-n/2}}\big) +\Delta_{j,n}\frac{\partial f}{\partial y}\big(X^{1}, X^{2}\big)\big(X^{2}_{(j+1)2^{-n/2}}-X^{2}_{j2^{-n/2}}\big) \\
&&  + \sum_{i=2}^6\sum_{\alpha_1 +\alpha_2 = 2i-1}C(\alpha_1,\alpha_2)\: \Delta_{j,n}\partial_{1 \ldots 1 2 \ldots 2}^{\alpha_1,\alpha_2} f(X^{1}, X^{2})\big(X^{1}_{(j+1)2^{-n/2}}-X^{1}_{j2^{-n/2}}\big)^{\alpha_1}\\
  && \times \big(X^{2}_{(j+1)2^{-n/2}}-X^{2}_{j2^{-n/2}}\big)^{\alpha_2}+ R_{13}\big( \big(X^{1}_{(j+1)2^{-n/2}},X^{2}_{(j+1)2^{-n/2}})\big), \big(X^{1}_{j2^{-n/2}},X^{2}_{j2^{-n/2}})\big) \big).
 \end{eqnarray*}
By Definition \ref{ivan-definition} and (\ref{pq-varition}),  we can  write
 \begin{eqnarray}
 &&f(X^{1}_{\lfloor 2^{\frac{n}{2}} t \rfloor 2^{-n/2}}, X^{2}_{\lfloor 2^{\frac{n}{2}} t \rfloor 2^{-n/2}}) -f(0,0)\label{H>1/6 taylor}\\
  &=& O_n(f,t) + \sum_{i=2}^6\sum_{\alpha_1 +\alpha_2 = 2i-1}C(\alpha_1,\alpha_2)\: V_n^{\alpha_1, \alpha_2}(\partial_{1 \ldots 1 2 \ldots 2}^{\alpha_1,\alpha_2} f,t)\notag\\
 && + \sum_{j=0}^{\lfloor 2^{\frac{n}{2}} t \rfloor -1} R_{13}\big( \big(X^{1}_{(j+1)2^{-n/2}},X^{2}_{(j+1)2^{-n/2}})\big), \big(X^{1}_{j2^{-n/2}},X^{2}_{j2^{-n/2}})\big) \big).\notag
 \end{eqnarray}
By Lemma \ref{taylor-expansion}, we have
\begin{eqnarray*}
&& \sum_{j=0}^{\lfloor 2^{\frac{n}{2}} t \rfloor -1} \bigg|R_{13}\big( \big(X^{1}_{(j+1)2^{-n/2}},X^{2}_{(j+1)2^{-n/2}})\big), \big(X^{1}_{j2^{-n/2}},X^{2}_{j2^{-n/2}})\big) \big)\bigg| \\
&\leq & C_f \sum_{\alpha_1 + \alpha_2 =13}\sum_{j=0}^{\lfloor 2^{\frac{n}{2}} t \rfloor -1} |X^{1}_{(j+1)2^{-n/2}}-X^{1}_{j2^{-n/2}}|^{\alpha_1}|X^{2}_{(j+1)2^{-n/2}}-X^{2}_{j2^{-n/2}})|^{\alpha_2}.
\end{eqnarray*} 
So, we deduce by the independence of $X^{1}$ and $X^{2}$ that 
\begin{eqnarray}
&& \sum_{j=0}^{\lfloor 2^{\frac{n}{2}} t \rfloor -1} E \bigg(\bigg|R_{13}\big( \big(X^{1}_{(j+1)2^{-n/2}},X^{2}_{(j+1)2^{-n/2}})\big), \big(X^{1}_{j2^{-n/2}},X^{2}_{j2^{-n/2}})\big) \big)\bigg|\bigg)\label{H >1/6, reste} \\
&\leq & C_f \sum_{\alpha_1 + \alpha_2 =13}\sum_{j=0}^{\lfloor 2^{\frac{n}{2}} t \rfloor -1} E\big(|X^{1}_{(j+1)2^{-n/2}}-X^{1}_{j2^{-n/2}}|^{\alpha_1}\big)E\big(|X^{2}_{(j+1)2^{-n/2}}-X^{2}_{j2^{-n/2}})|^{\alpha_2}\big)\notag
\end{eqnarray}
\begin{eqnarray}
&= & C_f \sum_{\alpha_1 + \alpha_2 =13} \sum_{j=0}^{\lfloor 2^{\frac{n}{2}} t \rfloor -1} 2^{\frac{-nH(\alpha_1 +\alpha_2)}{2}} E\big[|G|^{\alpha_1}\big]E\big[|G|^{\alpha_2}\big]\notag\\
& \leq& C_f \sum_{\alpha_1 + \alpha_2 =13}E\big[|G|^{\alpha_1}\big]E\big[|G|^{\alpha_2}\big]2^{\frac{-n[H(\alpha_1 +\alpha_2) -1]}{2}}t \to_{n \to \infty} 0,\notag
\end{eqnarray} 
with $G\sim N(0,1)$. On the other hand, by the almost sure continuity of $f\big(X^{1},X^{2}\big)$, one has, almost surely and
as $n\to\infty$,
\begin{equation}
 f(X^{1}_{\lfloor 2^{\frac{n}{2}} t \rfloor 2^{-n/2}}, X^{2}_{\lfloor 2^{\frac{n}{2}} t \rfloor 2^{-n/2}}) -f(0,0)
\to f\big(X^{1}_t, X^{2}_t \big)-f(0,0).\label{H >1/6, pathcontinuity}
\end{equation}
 Finally, the desired conclusion (\ref{first-main-1}) follows from (\ref{H >1/6, reste}), (\ref{H >1/6, pathcontinuity}) and the conclusion of Proposition \ref{proposition, H >1/6}, plugged into (\ref{H>1/6 taylor}).

\subsection{Proof of (\ref{first-main-2})}
We define $W_n(f,t)$ by 
\begin{eqnarray*}
W_n(f,t) &:=& \bigg( K_n^{(1)}(f, t) , K_n^{(2)}(f, t) ,  K_n^{(3)}(f, t) , K_n^{(4)}(f, t) \bigg),
\end{eqnarray*} 
where, for all $i \in \{1, \ldots , 4 \}$, $K_n^{(i)}(f,t)$ is given in Definition \ref{definition Kni}.
Let also define $W(f,t)$ as follows, 
\begin{eqnarray*}
W(f,t)&:=& \bigg( \kappa_1\int_0^t \partial_{111}f(X^{1}_s,X^{2}_s)dB^{1}_s, \kappa_2\int_0^t \partial_{222}f(X^{1}_s,X^{2}_s)dB^{2}_s,\\
&& \hspace{2cm} \kappa_3\int_0^t \partial_{122}f(X^{1}_s,X^{2}_s)dB^{3}_s, \kappa_4\int_0^t \partial_{112}f(X^{1}_s,X^{2}_s)dB^{4}_s \bigg),
\end{eqnarray*}
where $(B^{1}, \ldots, B^{4})$ is a 4-dimensional Brownian motion independent of $(X^{1}, X^{2})$, $\kappa_1^2 = \kappa_2^2 = \frac{1}{96}\sum_{r\in \Z}\rho^3(r)$ and $\kappa_3^2 = \kappa_4^2 = \frac{1}{32}\sum_{r\in \Z}\rho^3(r)$ with $\rho$ defined in (\ref{rho}).

The following theorem will play a crucial role in the proof of (\ref{first-main-2}).

\begin{theorem}\label{f.d.d-h=1/6} Suppose $H=1/6$ and fix $t \geq 0$. Then, as $n\to\infty$,
\begin{equation*}
\big ( X^{1}, X^{2}, W_n(f,t) \big) \to \big( X^{1}, X^{2},  W(f,t)  \big),
\end{equation*}
in $D_{\R^2}[0, \infty)\times \R^4$.

\end{theorem} 

\begin{proof}{}The proof of Theorem \ref{f.d.d-h=1/6} will be done in several steps.

  \subsubsection{\underline{Step 1:} Tightness of $ \big(X^{1},X^{2}, W_n(f,t) \big)$  in $D_{\R^2}[0, \infty )\times \R^4$.}
  
  It suffices to prove that $\big(W_n(f,t)\big)_n$ is bounded in $L^2(\R^4)$. 
  We have 
  \begin{eqnarray}
   && E\big( \|W_n(f,t)\|_{\R^4}^2 \big) = E\big( \|(K_n^{(1)}(f,t), K_n^{(2)}(f,t), K_n^{(3)}(f,t), K_n^{(4)}(f,t))\|_{\R^4}^2\big)\label{tightness}\\
   &=&  \sum_{p=1}^4 E\big( (K_n^{(p)}(f,t))^2 \big).\notag
   \end{eqnarray}
  On the other hand:
 \begin{enumerate}
  
  \item For $p= 1,2$, thanks to (\ref{lemma4}),  we have 
  \[
  \sup_{n\in \N}E\big( (K_n^{(p)}(f,t))^2 \big) \leq  C(t+t^2).
  \]
  
  \item  Thanks to (\ref{lemma3}), we have
  \[
  \sup_{n\in \N}E\big( (K_n^{(3)}(f,t))^2 \big) \leq  C(t+t^2).
  \]
  
  \item Thanks to (\ref{lemma3}),  we have also
  \[
  \sup_{n\in \N}E\big( (K_n^{(4)}(f,t))^2 \big) \leq  C(t+t^2).
  \]
  
 \end{enumerate}
 By combining these previous estimates with (\ref{tightness}), we deduce that $\exists C >0$ independent of $n$ and $t$ such that
 \[
 E\big( \|W_n(f,t)\|_{\R^4}^2 \big)\leq C(t+t^2),
 \]
 which proves the boundedness, in $L^2(\R^4)$, of $(W_n(f,t))_n$ and consequently its tightness.

  \subsubsection{\underline{Step 2.}} By Step 1, the sequence $ \big(X^{1},X^{2}, W_n(f,t) \big)$ is tight in $D_{\R^2}[0, \infty)\times \R^4$. Consider a subsequence converging in law to some limit denoted by 
  \[
   \big(X^{1},X^{2}, W_{\infty}(f,t) \big)
  \]
  (for convenience, we keep the same notation for this subsequence and for the sequence itself).
  
  We have to show in this step that, conditioned on $(X^{1}, X^{2})$, the laws of $W_{\infty}(f,t)$ and $W(f,t)$ are the same:
Let $\lambda =(\lambda_1,\ldots , \lambda_4)$ denote a generic element of $\R^4$ and, for $ \lambda, \beta \in \R^4$, write $\langle \lambda, \beta \rangle$ for $\sum_{i=1}^4 \lambda_i\beta_i$. Let us define $g_1:= \partial_{111}f$, $g_2 := \partial_{222}f$, $g_3:=\partial_{122}f$ and $g_4:= \partial_{112}f$. We consider the conditional characteristic function of $W(f,t)$ given $(X^{1}, X^{2})$:
\begin{equation*}
 \phi(\lambda) := E\big(e^{i \langle \lambda, W(f,t) \rangle}| X^{1},X^{2}\big).
\end{equation*}
Observe that $\phi(\lambda) = e^{-\frac12 \sum_{p=1}^4 \lambda_p^2 q_p}$ where, for $p \in \{1, \ldots , 4 \}$,
\[
q_p = \kappa_p^2  \big( \int_0^t g_p^2(X^{1}_s,X^{2}_s)ds \big).
 \]
 Observe that $\phi$ is the unique solution of the following system of PDEs:
 \begin{equation}\label{EDP}
 \frac{\partial \phi}{\partial \lambda_p}(\lambda)= \phi(\lambda)\big( -\lambda_p q_p \big), \:\:\: p=1, \ldots , 4,
 \end{equation}
 where the unknown function $\phi: \R^4 \to \C$ satisfies the initial condition $\phi(0)=1$.
 
 Recall that our purpose is to show that 
 \[
 E\big(e^{i \langle \lambda, W_{\infty}(f,t) \rangle}| X^{1},X^{2}\big) = E\big(e^{i \langle \lambda, W(f,t) \rangle}| X^{1},X^{2}\big).
 \]
 Thanks to (\ref{EDP}) this is equivalent to show that 
 \begin{eqnarray*}
 \frac{\partial }{\partial \lambda_p}E\big(e^{i \langle \lambda, W_{\infty}(f,t) \rangle}| X^{1},X^{2}\big)&=& E\big(e^{i \langle \lambda, W_{\infty}(f,t) \rangle}| X^{1},X^{2}\big)\big(-\lambda_p \kappa_p^2 \int_0^t g^2_p(X_s^{1},X_s^{2}) ds \big),.
 \end{eqnarray*}
Hence we have to show that, for all $p \in \{1, \ldots , 4 \}$,
 \begin{eqnarray*}
 E\big(iK_{\infty}^{(p)}(f,t)\: e^{i \langle \lambda, W_{\infty}(f,t) \rangle}| X^{1},X^{2}\big)&=& E\big(e^{i \langle \lambda, W_{\infty}(f,t) \rangle}| X^{1},X^{2}\big)\\
 && \times \big(-\lambda_p \kappa_p^2 \int_0^t g^2_p(X_s^{1},X_s^{2}) ds \big).  
 \end{eqnarray*}
 Equivalently, for every random variable $ \xi $ of the form 
 $ \psi\big( X^{1}_{s_1},X^{2}_{s_1}, \ldots, X^{1}_{s_r},X^{2}_{s_r} \big) $, with $r \in \N^*$, $\psi: \R^{2r} \to \R$ belonging to $C_b^{\infty}(\R^{2r})$ and $s_1, \ldots, s_r \in \R$, for all $p \in \{1, \ldots , 4 \}$, we have to prove that
 
 \begin{eqnarray*}
 E\big(iK_{\infty}^{(p)}(f,t)\: e^{i \langle \lambda, W_{\infty}(f,t) \rangle} \xi \big) = E\big( \xi e^{i \langle \lambda, W_{\infty}(f,t) \rangle}\times \big(-\lambda_p \kappa_p^2 \int_0^t g^2_p(X_s^{1},X_s^{2}) ds \big)\big).\notag \\
\end{eqnarray*}
 Since $ \big(X^{1},X^{2}, W_{\infty}(f,t) \big)  $ is defined as the limit in law of $ \big(X^{1},X^{2}, W_n(f,t) \big) $ and since $ W_n(f,t) $ is bounded in $L^2$, we have
 \[
  E\big(iK_{\infty}^{(p)}(f,t)\: e^{i \langle \lambda, W_{\infty}(f,t) \rangle} \xi \big) = \lim_{n \to \infty} E\big(iK_n^{(p)}(f,t)\: e^{i \langle \lambda, W_n(f,t) \rangle} \xi \big).
  \]
  Thus, we have to prove that, for all $p \in \{1, \ldots , 4 \}$,
\begin{eqnarray}
\lim_{n \to \infty}E\big(iK_n^{(p)}(f,t)\: e^{i \langle \lambda , W_n(f,t) \rangle} \xi \big)= E\big( \xi e^{i \langle \lambda , W_{\infty}(f,t) \rangle}\times \big(-\lambda_p \kappa_p^2 \int_0^t g^2_p(X_s^{1},X_s^{2}) ds \big)\big).\notag\\
\label{H=1/6,principal equation}
\end{eqnarray}
Let us compute $E\big(iK_n^{(p)}(f,t)\: e^{i \langle \lambda, W_n(f,t) \rangle} \xi \big)$ for $p=3$ (the calculations are very similar for the other values of $p$). We have
\begin{eqnarray}
&&E\big(iK_n^{(3)}(f,t)\: e^{i \langle \lambda , W_n(f,t) \rangle} \xi \big) \notag\\
&=& \frac{i}{8} 2^{-\frac{n}{4}} \sum_{j=0}^{\lfloor 2^{\frac{n}{2}} t \rfloor -1} E\bigg( \Delta_{j,n}\partial_{122} f(X^{1},X^{2}) H_1\big(X^{1,n}_{j+1} - X^{1,n}_j \big)H_2\big(X^{2,n}_{j+1} - X^{2,n}_j \big) e^{i \langle \lambda, W_n(f,t) \rangle}\xi \bigg)\notag\\
&=& \frac{i}{8} 2^{-\frac{n}{4}} \sum_{j=0}^{\lfloor 2^{\frac{n}{2}} t \rfloor -1} E\bigg( \Delta_{j,n}\partial_{122} f(X^{1},X^{2})I^{(1)}_1(2^{\frac{n}{12}}\delta_{(j+1)2^{-n/2}})I^{(2)}_2(2^{\frac{n}{6}}\delta^{\otimes 2}_{(j+1)2^{-n/2}}) e^{i \langle \lambda, W_n(f,t) \rangle}\xi \bigg)\notag\\
&=& \frac{i}{8} \sum_{j=0}^{\lfloor 2^{\frac{n}{2}} t \rfloor -1} E\bigg( \Delta_{j,n}\partial_{122} f(X^{1},X^{2})I^{(1)}_1(\delta_{(j+1)2^{-n/2}})I^{(2)}_2(\delta^{\otimes 2}_{(j+1)2^{-n/2}}) e^{i \langle \lambda, W_n(f,t) \rangle}\xi \bigg)\notag\\
&=& \frac{i}{8} \sum_{j=0}^{\lfloor 2^{\frac{n}{2}} t \rfloor -1} E\bigg( \langle D_{X^{1}}\big( \Delta_{j,n}\partial_{122} f(X^{1},X^{2}) e^{i \langle \lambda, W_n(f,t) \rangle}\xi \big), \delta_{(j+1)2^{-n/2}} \rangle I^{(2)}_2(\delta^{\otimes 2}_{(j+1)2^{-n/2}})\bigg),\notag\\
\label{first equation}
\end{eqnarray}
where the second equality follows from (\ref{linear-isometry}) and the   last one follows from (\ref{duality formula}). Thanks to (\ref{Leibnitz0}),  the first Malliavin derivative with respect to $X^{1}$ of $\Delta_{j,n}\partial_{122} f(X^{1},X^{2}) e^{i \langle \lambda, W_n(f,t) \rangle}\xi$ is given by
\begin{eqnarray*}
&& D_{X^{1}}\big( \Delta_{j,n}\partial_{122} f(X^{1},X^{2}) e^{i \langle \lambda, W_n(f,t) \rangle}\xi \big)\\
&=& \Delta_{j,n}\partial_{1122} f(X^{1},X^{2})e^{i \langle \lambda, W_n(f,t) \rangle}\xi \bigg(\frac{\varepsilon_{j2^{-n/2}} + \varepsilon_{(j+1)2^{-n/2} }}{2}\bigg)\\
&& + i\Delta_{j,n}\partial_{122} f(X^{1},X^{2})e^{i \langle \lambda, W_n(f,t) \rangle}\xi D_{X^{1}} \langle \lambda, W_n(f,t) \rangle + \Delta_{j,n}\partial_{122} f(X^{1},X^{2}) e^{i \langle \lambda, W_n(f,t) \rangle}D_{X^{1}}\xi.
\end{eqnarray*}
Thus, by (\ref{first equation}), we have
\begin{eqnarray*}
&& E\big(iK_n^{(3)}(f,t)\: e^{i \langle \lambda , W_n(f,t) \rangle} \xi \big)\\
&=& \frac{i}{8} \sum_{j=0}^{\lfloor 2^{\frac{n}{2}} t \rfloor -1} E\bigg( \Delta_{j,n}\partial_{1122} f(X^{1},X^{2})e^{i \langle \lambda, W_n(f,t) \rangle}\xi \: I^{(2)}_2(\delta^{\otimes 2}_{(j+1)2^{-n/2}})\bigg)\\
&& \hspace{2cm} \times \bigg\langle \frac{\varepsilon_{j2^{-n/2}} + \varepsilon_{(j+1)2^{-n/2} }}{2} , \delta_{(j+1)2^{-n/2}} \bigg\rangle \\
&-& \frac{1}{8} \sum_{j=0}^{\lfloor 2^{\frac{n}{2}} t \rfloor -1} E\bigg( \Delta_{j,n}\partial_{122} f(X^{1},X^{2})e^{i \langle \lambda, W_n(f,t) \rangle}\xi \: \big\langle D_{X^{1}} \langle \lambda, W_n(f,t) \rangle , \delta_{(j+1)2^{-n/2}} \big\rangle \\
&& \hspace{2cm}\times I^{(2)}_2(\delta^{\otimes 2}_{(j+1)2^{-n/2}})\bigg)\\
\end{eqnarray*}
\begin{eqnarray*}
&+& \frac{i}{8} \sum_{j=0}^{\lfloor 2^{\frac{n}{2}} t \rfloor -1} E\bigg( \Delta_{j,n}\partial_{122} f(X^{1},X^{2})e^{i \langle \lambda, W_n(f,t) \rangle} \: \big\langle D_{X^{1}} \xi , \delta_{(j+1)2^{-n/2}} \big\rangle I^{(2)}_2(\delta^{\otimes 2}_{(j+1)2^{-n/2}})\bigg)\\
&=& A_n(t) + B_n(t) +C_n(t),
\end{eqnarray*}
with obvious notation at the last equality.

In the next steps we will prove firstly that  $A_n(t) \to 0$ as $n \to \infty$, then that  $B_n(t) \to -\lambda_3 \kappa_3^2E\big( \xi e^{i \langle \lambda , W_{\infty}(f,t) \rangle}\times  \int_0^t g^2_3(X_s^{1},X_s^{2}) ds \big)$ and finally that $C_n(t) \to 0$.

\subsubsection{\underline{Step 3}: Proof of the convergence to 0 of $A_n(t)$.}

Thanks to the boundedness of $e^{i \langle \lambda, W_n(f,t) \rangle}$ and $\xi$, and since $f \in C_b^{\infty}(\R^2)$, we get

\begin{eqnarray*}
|A_n(t)| &\leq& C \sum_{j=0}^{\lfloor 2^{\frac{n}{2}} t \rfloor -1}  E\big( |I^{(2)}_2(\delta^{\otimes 2}_{(j+1)2^{-n/2}})|\big)\bigg| \bigg\langle \frac{\varepsilon_{j2^{-n/2}} + \varepsilon_{(j+1)2^{-n/2} }}{2} , \delta_{(j+1)2^{-n/2}} \bigg\rangle \bigg|\\
&\leq& C \sum_{j=0}^{\lfloor 2^{\frac{n}{2}} t \rfloor -1}  \|I^{(2)}_2(\delta^{\otimes 2}_{(j+1)2^{-n/2}})\|_2\bigg| \bigg\langle \frac{\varepsilon_{j2^{-n/2}} + \varepsilon_{(j+1)2^{-n/2} }}{2} , \delta_{(j+1)2^{-n/2}} \bigg\rangle \bigg|\\
&\leq& C 2^{-n/6}\sum_{j=0}^{\lfloor 2^{\frac{n}{2}} t \rfloor -1} \bigg| \bigg\langle \frac{\varepsilon_{j2^{-n/2}} + \varepsilon_{(j+1)2^{-n/2} }}{2} , \delta_{(j+1)2^{-n/2}} \bigg\rangle \bigg|\\
&\leq& C 2^{-n/6} t^{1/3},
\end{eqnarray*}
where the third inequality follows from (\ref{isometry}) and the last inequality follows from (\ref{lemma2-1}). Hence, $A_n(t) \to 0$ as $n \to \infty$.

\subsubsection{\underline{Step 4}: Study of the convergence of $B_n(t)$.}
\begin{eqnarray}
&& B_n(t)\notag\\
&=&- \frac{1}{8} \sum_{j=0}^{\lfloor 2^{\frac{n}{2}} t \rfloor -1} E\bigg( \Delta_{j,n}\partial_{122} f(X^{1},X^{2})e^{i \langle \lambda, W_n(f,t) \rangle}\xi \: \big\langle D_{X^{1}} \langle \lambda, W_n(f,t) \rangle , \delta_{(j+1)2^{-n/2}} \big\rangle \notag\\
&& \hspace{2cm}\times I^{(2)}_2(\delta^{\otimes 2}_{(j+1)2^{-n/2}})\bigg)\notag
\end{eqnarray}
\begin{eqnarray}
&=& -\frac18 \sum_{p=1}^4 \lambda_p \sum_{j=0}^{\lfloor 2^{\frac{n}{2}} t \rfloor -1} E\bigg( \Delta_{j,n}\partial_{122} f(X^{1},X^{2})e^{i \langle \lambda, W_n(f,t) \rangle}\xi \: \big\langle D_{X^{1}}\big(K^{(p)}_n(t)\big)  , \delta_{(j+1)2^{-n/2}} \big\rangle \notag\\
&& \hspace{2cm} \times I^{(2)}_2(\delta^{\otimes 2}_{(j+1)2^{-n/2}})\bigg)\notag\\
&=&\sum_{p=1}^4 \lambda_p B_n^{(p)}(t),\label{decomposition of Bn}
\end{eqnarray}
with obvious notation at the last equality.
We will prove that 
\begin{equation}
B_n(t) \underset{n \to \infty}{\longrightarrow} - \kappa_3^2E\bigg(e^{i \langle \lambda, W_{\infty}(f,t) \rangle}\xi\times \int_0^t g^2_3(X_s^{1},X_s^{2})ds\bigg),\label{convergence Bn}
\end{equation} 
where $g_3 := \partial_{122}f$, see the beginning  of Step 2.
The proof of (\ref{convergence Bn}) will be done in several steps. Firstly, we will prove the convergence of  $B_n^{(3)}(t)$ to $-\kappa_3^2E\bigg(e^{i \langle \lambda, W_{\infty}(f,t) \rangle}\xi\times \int_0^t \big(\partial_{122}f(X_s^{1},X_s^{2})\big)^2ds\bigg)$. Then, we will prove the convergence to 0 of $B_n^{(p)}(t)$  for the other values of $p$.

\subsubsection*{Study of the convergence of $B_n^{(3)}(t)$.}
\begin{eqnarray}
&& B_n^{(3)}(t)= -\frac18  \sum_{j=0}^{\lfloor 2^{\frac{n}{2}} t \rfloor -1} E\bigg( \Delta_{j,n}\partial_{122} f(X^{1},X^{2})e^{i \langle \lambda, W_n(f,t) \rangle}\xi \label{B_n^3}\\
&& \hspace{4cm} \times \big\langle D_{X^{1}}\big(K^{(3)}_n(t)\big)  , \delta_{(j+1)2^{-n/2}} \big\rangle I^{(2)}_2(\delta^{\otimes 2}_{(j+1)2^{-n/2}})\bigg).\notag
\end{eqnarray}
 Our purpose in this subsection is to  prove the following result :
\begin{equation}
B_n^{(3)}(t) \underset{n \to \infty}{\longrightarrow} -\kappa_3^2E\bigg(e^{i \langle \lambda, W_{\infty}(f,t) \rangle}\xi\times \int_0^t \big(\partial_{122}f(X_s^{1},X_s^{2})\big)^2ds\bigg).\label{convergence B_n^3}
\end{equation}
{\bf Proof  of (\ref{convergence B_n^3})}.
Thanks to (\ref{Leibnitz0}) and (\ref{derivative- multiple,integral}),   we have
\begin{eqnarray*}
&&\big\langle D_{X^{1}}\big(K^{(3)}_n(t)\big)  , \delta_{(j+1)2^{-n/2}} \big\rangle\\
&=& \frac18 \sum_{k=0}^{\lfloor 2^{\frac{n}{2}} t \rfloor -1} \Delta_{k,n}\partial_{1122}f(X^{1},X^{2})I_1^{(1)}(\delta_{(k+1)2^{-n/2}})I_2^{(2)}(\delta_{(k+1)2^{-n/2}}^{\otimes 2}) \\
&& \hspace{6cm} \times \bigg\langle \bigg(\frac{\varepsilon_{k2^{-n/2}} + \varepsilon_{(k+1)2^{-n/2} }}{2} \bigg), \delta_{(j+1)2^{-n/2}}\bigg\rangle\\
&& +  \frac18 \sum_{k=0}^{\lfloor 2^{\frac{n}{2}} t \rfloor -1} \Delta_{k,n}\partial_{122}f(X^{1},X^{2})I_2^{(2)}(\delta_{(k+1)2^{-n/2}}^{\otimes 2})\big\langle \delta_{(k+1)2^{-n/2}}, \delta_{(j+1)2^{-n/2}} \big\rangle.
\end{eqnarray*}
Using the duality formula (\ref{duality formula}), we deduce   that
\begin{eqnarray}
&& B_n^{(3)}(t)\notag\\
&=& -\frac{1}{64}  \sum_{j,k=0}^{\lfloor 2^{\frac{n}{2}} t \rfloor -1} E\bigg( \Delta_{j,n}\partial_{122} f(X^{1},X^{2}) \Delta_{k,n}\partial_{1122}f(X^{1},X^{2})I_1^{(1)}(\delta_{(k+1)2^{-n/2}})I_2^{(2)}(\delta_{(k+1)2^{-n/2}}^{\otimes 2}) \notag \\
&& \hspace{3cm}\times I^{(2)}_2(\delta^{\otimes 2}_{(j+1)2^{-n/2}})e^{i \langle \lambda, W_n(f,t) \rangle}\xi\bigg) \bigg\langle \bigg(\frac{\varepsilon_{k2^{-n/2}} + \varepsilon_{(k+1)2^{-n/2} }}{2} \bigg), \delta_{(j+1)2^{-n/2}}\bigg\rangle\notag\\
&& -\frac{1}{64}  \sum_{j,k=0}^{\lfloor 2^{\frac{n}{2}} t \rfloor -1} E\bigg( \Delta_{j,n}\partial_{122} f(X^{1},X^{2}) \Delta_{k,n}\partial_{122}f(X^{1},X^{2})I_2^{(2)}(\delta_{(k+1)2^{-n/2}}^{\otimes 2}) \notag \\
&& \hspace{3cm}\times I^{(2)}_2(\delta^{\otimes 2}_{(j+1)2^{-n/2}})e^{i \langle \lambda, W_n(f,t) \rangle}\xi\bigg)\big\langle \delta_{(k+1)2^{-n/2}}, \delta_{(j+1)2^{-n/2}} \big\rangle \notag
\end{eqnarray}
\begin{eqnarray}
&=&-\frac{1}{64}  \sum_{j,k=0}^{\lfloor 2^{\frac{n}{2}} t \rfloor -1} E\bigg(\bigg\langle D_{X^{2}}^2 \bigg( \Delta_{j,n}\partial_{122} f(X^{1},X^{2}) \Delta_{k,n}\partial_{1122}f(X^{1},X^{2})I_1^{(1)}(\delta_{(k+1)2^{-n/2}})\notag\\
&& \times I_2^{(2)}(\delta_{(k+1)2^{-n/2}}^{\otimes 2})e^{i \langle \lambda, W_n(f,t) \rangle}\xi\bigg),\delta^{\otimes 2}_{(j+1)2^{-n/2}} \bigg \rangle \bigg) \bigg\langle \bigg(\frac{\varepsilon_{k2^{-n/2}} + \varepsilon_{(k+1)2^{-n/2} }}{2} \bigg), \delta_{(j+1)2^{-n/2}}\bigg\rangle\notag\\
&& -\frac{1}{64}  \sum_{j,k=0}^{\lfloor 2^{\frac{n}{2}} t \rfloor -1} E\bigg(\bigg\langle D_{X^{2}}^2 \bigg( \Delta_{j,n}\partial_{122} f(X^{1},X^{2}) \Delta_{k,n}\partial_{122}f(X^{1},X^{2})I_2^{(2)}(\delta_{(k+1)2^{-n/2}}^{\otimes 2})\notag\\
&& \hspace{1cm}\times e^{i \langle \lambda, W_n(f,t) \rangle}\xi\bigg),\delta^{\otimes 2}_{(j+1)2^{-n/2}} \bigg \rangle \bigg) \big\langle \delta_{(k+1)2^{-n/2}}, \delta_{(j+1)2^{-n/2}} \big\rangle\notag\\
&=& B_n^{(3,1)}(t) + B_n^{(3,2)}(t), \label{definitionB_n^3}
\end{eqnarray}
with obvious notation at the last equality. In the following two steps, we will prove firstly that $B_n^{(3,2)}(t) \to -\kappa_3^2E\bigg(e^{i \langle \lambda, W_{\infty}(f,t) \rangle}\xi\times \int_0^t\big(\partial_{122} f(X_s^{1},X_s^{2})\big)^2ds\bigg)$, then that $B_n^{(3,1)}(t) \to 0$ as $n \to \infty$.

\begin{enumerate}
\item \underline{Convergence of $B_n^{(3,2)}(t)$ as $n \to \infty$ :} Let us prove that 
\begin{equation}
B_n^{(3,2)}(t) \underset{n \to \infty}{\longrightarrow} -\kappa_3^2E\bigg(e^{i \langle \lambda, W_{\infty}(f,t) \rangle}\xi\times \int_0^t\big(\partial_{122} f(X_s^{1},X_s^{2})\big)^2ds\bigg).\label{convergence B_n^{3,2}}
\end{equation} 

{\bf Proof of (\ref{convergence B_n^{3,2}})}.
Thanks to (\ref{Leibnitz0}) and (\ref{derivative- multiple,integral})  , we have
\begin{eqnarray}
 D_{X^{2}}^2 \bigg( \Delta_{j,n}\partial_{122} f(X^{1},X^{2}) \Delta_{k,n}\partial_{122}f(X^{(1)},X^{(2)})I_2^{(2)}(\delta_{(k+1)2^{-n/2}}^{\otimes 2}) e^{i \langle \lambda, W_n(f,t) \rangle}\xi\bigg)\notag\\\label{derivative,B_n,1}
\end{eqnarray}
\begin{eqnarray}
&=& D_{X^{2}}^2 \big(\Delta_{j,n}\partial_{122} f(X^{1},X^{2}) \Delta_{k,n}\partial_{122}f(X^{1},X^{2})\big) I_2^{(2)}(\delta_{(k+1)2^{-n/2}}^{\otimes 2}) e^{i \langle \lambda, W_n(f,t) \rangle}\xi \notag\\
&& + 2D_{X^{2}}\big(\Delta_{j,n}\partial_{122} f(X^{1},X^{2}) \Delta_{k,n}\partial_{122}f(X^{1},X^{2})\big) \tilde{\otimes}D_{X^{2}}\big(I_2^{(2)}(\delta_{(k+1)2^{-n/2}}^{\otimes 2}) e^{i \langle \lambda, W_n(f,t) \rangle}\xi\big) \notag\\
&& + \Delta_{j,n}\partial_{122} f(X^{1},X^{2}) \Delta_{k,n}\partial_{122}f(X^{1},X^{2})D_{X^{2}}^2\big(I_2^{(2)}(\delta_{(k+1)2^{-n/2}}^{\otimes 2}) e^{i \langle \lambda, W_n(f,t) \rangle}\xi\big) \notag\\
&=& \mathcal{D}_1 + \mathcal{D}_2 + \mathcal{D}_3, \label{first, 3derivatives}
\end{eqnarray}
with obvious notion at the last equality. We also have
\begin{eqnarray}
&&D_{X^{2}} \big( \Delta_{j,n}\partial_{122} f(X^{1},X^{2}) \Delta_{k,n}\partial_{122}f(X^{1},X^{2})\big)\label{derivative,B_n,2}\\
&=& \Delta_{j,n}\partial_{1222} f(X^{1},X^{2})\Delta_{k,n}\partial_{122}f(X^{1},X^{2})\bigg(\frac{\varepsilon_{j2^{-n/2}} + \varepsilon_{(j+1)2^{-n/2} }}{2} \bigg)\notag\\
&& + \Delta_{j,n}\partial_{122} f(X^{1},X^{2})\Delta_{k,n}\partial_{1222}f(X^{1},X^{2})\bigg(\frac{\varepsilon_{k2^{-n/2}} + \varepsilon_{(k+1)2^{-n/2} }}{2} \bigg),\notag
\end{eqnarray} 
and
\begin{eqnarray}
&& D_{X^{2}}^2 \big( \Delta_{j,n}\partial_{122} f(X^{1},X^{2}) \Delta_{k,n}\partial_{122}f(X^{1},X^{2})\big)\label{derivative,B_n,3}\\
&=&\Delta_{j,n}\partial_{12222} f(X^{1},X^{2})\Delta_{k,n}\partial_{122}f(X^{1},X^{2})\bigg(\frac{\varepsilon_{j2^{-n/2}} + \varepsilon_{(j+1)2^{-n/2} }}{2} \bigg)^{\otimes 2}\notag\\
&& + 2 \Delta_{j,n}\partial_{1222} f(X^{1},X^{2}) \Delta_{k,n}\partial_{1222}f(X^{1},X^{2})\bigg(\frac{\varepsilon_{j2^{-n/2}} + \varepsilon_{(j+1)2^{-n/2} }}{2} \bigg)\tilde{\otimes}\notag\\
&& \bigg(\frac{\varepsilon_{k2^{-n/2}} + \varepsilon_{(k+1)2^{-n/2} }}{2} \bigg)\notag\\
&& + \Delta_{j,n}\partial_{122} f(X^{1},X^{2}) \Delta_{k,n}\partial_{12222}f(X^{1},X^{2})\bigg(\frac{\varepsilon_{k2^{-n/2}} + \varepsilon_{(k+1)2^{-n/2} }}{2} \bigg)^{\otimes 2}.\notag
\end{eqnarray}
On the other hand, we have
\begin{eqnarray}
&& D_{X^{2}}\big(I_2^{(2)}(\delta_{(k+1)2^{-n/2}}^{\otimes 2}) e^{i \langle \lambda, W_n(f,t) \rangle}\xi\big)\label{derivative,B_n,4}\\
&=& 2I_1^{(2)}(\delta_{(k+1)2^{-n/2}}) e^{i \langle \lambda, W_n(f,t) \rangle}\xi\delta_{(k+1)2^{-n/2}} + I_2^{(2)}(\delta_{(k+1)2^{-n/2}}^{\otimes 2})D_{X^{2}}\big(e^{i \langle \lambda, W_n(f,t) \rangle}\xi\big)\notag\\
&=& 2I_1^{(2)}(\delta_{(k+1)2^{-n/2}}) e^{i \langle \lambda, W_n(f,t) \rangle}\xi\delta_{(k+1)2^{-n/2}} + i I_2^{(2)}(\delta_{(k+1)2^{-n/2}}^{\otimes 2})e^{i \langle \lambda, W_n(f,t) \rangle}\xi \notag\\
&& \times D_{X^{2}}(\langle \lambda, W_n(f,t) \rangle)+ I_2^{(2)}(\delta_{(k+1)2^{-n/2}}^{\otimes 2})e^{i \langle \lambda, W_n(f,t) \rangle} D_{X^{2}}(\xi).\notag
\end{eqnarray}
Also
\begin{eqnarray}
&&D_{X^{2}}^2\big(I_2^{(2)}(\delta_{(k+1)2^{-n/2}}^{\otimes 2}) e^{i \langle \lambda, W_n(f,t) \rangle}\xi\big)\label{derivative,B_n,5}\\
&=& 2e^{i \langle \lambda, W_n(f,t) \rangle}\xi \delta_{(k+1)2^{-n/2}}^{\otimes 2} + 4I_1^{(2)}(\delta_{(k+1)2^{-n/2}})\delta_{(k+1)2^{-n/2}}\tilde{\otimes}D_{X^{2}}(e^{i \langle \lambda, W_n(f,t) \rangle}\xi)\notag
\end{eqnarray}
\begin{eqnarray*}
&& + I_2^{(2)}(\delta_{(k+1)2^{-n/2}}^{\otimes 2})D^2_{X^{2}}(e^{i \langle \lambda, W_n(f,t) \rangle}\xi)\notag\\
&=& 2e^{i \langle \lambda, W_n(f,t) \rangle}\xi \delta_{(k+1)2^{-n/2}}^{\otimes 2} + 4iI_1^{(2)}(\delta_{(k+1)2^{-n/2}})e^{i \langle \lambda, W_n(f,t) \rangle}\xi \delta_{(k+1)2^{-n/2}}\tilde{\otimes}\\
&& D_{X^{2}}(\langle \lambda, W_n(f,t) \rangle) + 4I_1^{(2)}(\delta_{(k+1)2^{-n/2}})e^{i \langle \lambda, W_n(f,t) \rangle} \delta_{(k+1)2^{-n/2}}\tilde{\otimes}D_{X^{2}}(\xi)\\
&& + i I_2^{(2)}(\delta_{(k+1)2^{-n/2}}^{\otimes 2})e^{i \langle \lambda, W_n(f,t) \rangle}\xi  D^2_{X^{2}}(\langle \lambda, W_n(f,t) \rangle) - I_2^{(2)}(\delta_{(k+1)2^{-n/2}}^{\otimes 2})e^{i \langle \lambda, W_n(f,t) \rangle}\xi\\
&& \times \big(D_{X^{2}}(\langle \lambda, W_n(f,t) \rangle)\big)^{\otimes 2} + 2iI_2^{(2)}(\delta_{(k+1)2^{-n/2}}^{\otimes 2})e^{i \langle \lambda, W_n(f,t) \rangle}D_{X^{2}}(\langle \lambda, W_n(f,t) \rangle)\tilde{\otimes}\\
&& D_{X^{2}}(\xi) + I_2^{(2)}(\delta_{(k+1)2^{-n/2}}^{\otimes 2})e^{i \langle \lambda, W_n(f,t) \rangle}D^2_{X^{2}}(\xi).
\end{eqnarray*}
From (\ref{derivative,B_n,2}), (\ref{derivative,B_n,3}), (\ref{derivative,B_n,4}), (\ref{derivative,B_n,5}) and (\ref{first, 3derivatives}), we deduce that
%\begin{eqnarray}
%&& D_{X^{2}}^2 \bigg( \Delta_{j,n}\partial_{122} f(X^{1},X^{2}) \Delta_{k,n}\partial_{122}f(X^{1},X^{2})I_2^{(2)}(\delta_{(k+1)2^{-n/2}}^{\otimes 2}) e^{i \langle \lambda, W_n(f,t) \rangle}\xi\bigg)\label{BiggD_{X^2}1}
%\end{eqnarray}
\begin{eqnarray}
&& \mathcal{D}_1= \Delta_{j,n}\partial_{12222} f(X^{1},X^{2})\Delta_{k,n}\partial_{122}f(X^{1},X^{2})I_2^{(2)}(\delta_{(k+1)2^{-n/2}}^{\otimes 2}) e^{i \langle \lambda, W_n(f,t) \rangle}\label{D1}\\
&& \times \xi\bigg(\frac{\varepsilon_{j2^{-n/2}} + \varepsilon_{(j+1)2^{-n/2} }}{2} \bigg)^{\otimes 2} 
+ 2 \Delta_{j,n}\partial_{1222} f(X^{1},X^{2}) \Delta_{k,n}\partial_{1222}f(X^{1},X^{2})\notag\\
&& \times I_2^{(2)}(\delta_{(k+1)2^{-n/2}}^{\otimes 2}) e^{i \langle \lambda, W_n(f,t) \rangle}\xi\bigg(\frac{\varepsilon_{j2^{-n/2}} + \varepsilon_{(j+1)2^{-n/2} }}{2} \bigg)\tilde{\otimes}\bigg(\frac{\varepsilon_{k2^{-n/2}} + \varepsilon_{(k+1)2^{-n/2} }}{2} \bigg)\notag\\
&& + \Delta_{j,n}\partial_{122} f(X^{1},X^{2}) \Delta_{k,n}\partial_{12222}f(X^{1},X^{2})I_2^{(2)}(\delta_{(k+1)2^{-n/2}}^{\otimes 2}) e^{i \langle \lambda, W_n(f,t) \rangle}\xi\notag\\
&& \times \bigg(\frac{\varepsilon_{k2^{-n/2}} + \varepsilon_{(k+1)2^{-n/2} }}{2} \bigg)^{\otimes 2},\notag
\end{eqnarray}
\begin{eqnarray}
&& \mathcal{D}_2=  4\Delta_{j,n}\partial_{1222} f(X^{1},X^{2})\Delta_{k,n}\partial_{122}f(X^{1},X^{2})I_1^{(2)}(\delta_{(k+1)2^{-n/2}})\times\label{D2}
\end{eqnarray}
\begin{eqnarray}
&&  e^{i \langle \lambda, W_n(f,t) \rangle}\xi \bigg(\frac{\varepsilon_{j2^{-n/2}} + \varepsilon_{(j+1)2^{-n/2} }}{2} \bigg)\tilde{\otimes}\delta_{(k+1)2^{-n/2}}\notag\\
&& + 2i\Delta_{j,n}\partial_{1222} f(X^{1},X^{2})\Delta_{k,n}\partial_{122}f(X^{1},X^{2})I_2^{(2)}(\delta_{(k+1)2^{-n/2}}^{\otimes 2})e^{i \langle \lambda, W_n(f,t) \rangle}\xi\notag\\
&& \times \bigg(\frac{\varepsilon_{j2^{-n/2}} + \varepsilon_{(j+1)2^{-n/2} }}{2} \bigg)\tilde{\otimes}D_{X^{2}}(\langle \lambda, W_n(f,t) \rangle) + 2\Delta_{j,n}\partial_{1222} f(X^{1},X^{2})\notag\\
&& \times \Delta_{k,n}\partial_{122}f(X^{1},X^{2})I_2^{(2)}(\delta_{(k+1)2^{-n/2}}^{\otimes 2})e^{i \langle \lambda, W_n(f,t) \rangle}\bigg(\frac{\varepsilon_{j2^{-n/2}} + \varepsilon_{(j+1)2^{-n/2} }}{2} \bigg)\tilde{\otimes}D_{X^{2}}(\xi)\notag\\
&& + 4\Delta_{j,n}\partial_{122} f(X^{1},X^{2})\Delta_{k,n}\partial_{1222}f(X^{1},X^{2})I_1^{(2)}(\delta_{(k+1)2^{-n/2}}) e^{i \langle \lambda, W_n(f,t) \rangle}\xi\notag\\
&& \times \bigg(\frac{\varepsilon_{k2^{-n/2}} + \varepsilon_{(k+1)2^{-n/2} }}{2} \bigg)\tilde{\otimes}\delta_{(k+1)2^{-n/2}} + 2i \Delta_{j,n}\partial_{122} f(X^{1},X^{2})\Delta_{k,n}\partial_{1222}f(X^{1},X^{2})\notag\\
&& \times I_2^{(2)}(\delta_{(k+1)2^{-n/2}}^{\otimes 2})e^{i \langle \lambda, W_n(f,t) \rangle}\xi\bigg(\frac{\varepsilon_{k2^{-n/2}} + \varepsilon_{(k+1)2^{-n/2} }}{2} \bigg)\tilde{\otimes}D_{X^{2}}(\langle \lambda, W_n(f,t) \rangle)\notag\\
&& + 2\Delta_{j,n}\partial_{122} f(X^{1},X^{2})\Delta_{k,n}\partial_{1222}f(X^{1},X^{2})I_2^{(2)}(\delta_{(k+1)2^{-n/2}}^{\otimes 2})e^{i \langle \lambda, W_n(f,t) \rangle}\notag\\
&& \times \bigg(\frac{\varepsilon_{k2^{-n/2}}+ \varepsilon_{(k+1)2^{-n/2} }}{2} \bigg)\tilde{\otimes}D_{X^{2}}(\xi),\notag
\end{eqnarray}
\begin{eqnarray}
&& \mathcal{D}_3= 2\Delta_{j,n}\partial_{122} f(X^{1},X^{2}) \Delta_{k,n}\partial_{122}f(X^{1},X^{2})\label{D3}\\
&& \times e^{i \langle \lambda, W_n(f,t) \rangle}\xi \delta_{(k+1)2^{-n/2}}^{\otimes 2} + 4i \Delta_{j,n}\partial_{122} f(X^{1},X^{2}) \Delta_{k,n}\partial_{122}f(X^{1},X^{2})\notag\\
&& \times I_1^{(2)}(\delta_{(k+1)2^{-n/2}})e^{i \langle \lambda, W_n(f,t) \rangle}\xi \delta_{(k+1)2^{-n/2}}\tilde{\otimes}D_{X^{2}}(\langle \lambda, W_n(f,t) \rangle)+ 4 \Delta_{j,n}\partial_{122} f(X^{1},X^{2})\notag\\ 
&& \times \Delta_{k,n}\partial_{122}f(X^{1},X^{2})I_1^{(2)}(\delta_{(k+1)2^{-n/2}})e^{i \langle \lambda, W_n(f,t) \rangle} \delta_{(k+1)2^{-n/2}}\tilde{\otimes}D_{X^{2}}(\xi)\notag\\
&& +i \Delta_{j,n}\partial_{122} f(X^{1},X^{2}) \Delta_{k,n}\partial_{122}f(X^{1},X^{2})
 I_2^{(2)}(\delta_{(k+1)2^{-n/2}}^{\otimes 2})e^{i \langle \lambda, W_n(f,t) \rangle}\xi \notag\\
&&  \times D^2_{X^{2}}(\langle \lambda, W_n(f,t) \rangle) - \Delta_{j,n}\partial_{122} f(X^{1},X^{2}) \Delta_{k,n}\partial_{122}f(X^{1},X^{2})I_2^{(2)}(\delta_{(k+1)2^{-n/2}}^{\otimes 2})\notag\\
&& \times e^{i \langle \lambda, W_n(f,t) \rangle}\xi \big(D_{X^{2}}(\langle \lambda, W_n(f,t) \rangle)\big)^{\otimes 2}+2i \Delta_{j,n}\partial_{122} f(X^{1},X^{2}) \Delta_{k,n}\partial_{122}f(X^{1},X^{2})\notag\\
&& \times I_2^{(2)}(\delta_{(k+1)2^{-n/2}}^{\otimes 2})e^{i \langle \lambda, W_n(f,t) \rangle}D_{X^{2}}(\langle \lambda, W_n(f,t) \rangle)\tilde{\otimes}D_{X^{2}}(\xi)+ \Delta_{j,n}\partial_{122} f(X^{1},X^{2})\notag\\ 
&& \times \Delta_{k,n}\partial_{122}f(X^{1},X^{2})I_2^{(2)}(\delta_{(k+1)2^{-n/2}}^{\otimes 2})e^{i \langle \lambda, W_n(f,t) \rangle}D^2_{X^{2}}(\xi).\notag
\end{eqnarray}

By plugging (\ref{D1}), (\ref{D2}) and (\ref{D3}) in (\ref{first, 3derivatives}), then by plugging (\ref{derivative,B_n,1}) in $B_n^{(3,2)}(t)$, we get
\begin{equation*}
B_n^{(3,2)}(t) = B_n^{(3,2,a)}(t) + B_n^{(3,2,b)}(t) + B_n^{(3,2,c)}(t),
\end{equation*}
where
\begin{eqnarray*}
&& B_n^{(3,2,a)}(t)\\
&=& -\frac{1}{64}\sum_{j,k=0}^{\lfloor 2^{\frac{n}{2}} t \rfloor -1}E\bigg ( \Delta_{j,n}\partial_{12222} f(X^{1},X^{2})\Delta_{k,n}\partial_{122}f(X^{1},X^{2})I_2^{(2)}(\delta_{(k+1)2^{-n/2}}^{\otimes 2})\\ 
&& \times e^{i \langle \lambda, W_n(f,t) \rangle}\xi\bigg)\bigg \langle \bigg(\frac{\varepsilon_{j2^{-n/2}} + \varepsilon_{(j+1)2^{-n/2} }}{2} \bigg)^{\otimes 2} , \delta_{(j+1)2^{-n/2}}^{\otimes 2} \bigg \rangle \langle \delta_{(k+1)2^{-n/2}}, \delta_{(j+1)2^{-n/2}} \rangle\\
&& -\frac{2}{64} \sum_{j,k=0}^{\lfloor 2^{\frac{n}{2}} t \rfloor -1}E\bigg (\Delta_{j,n}\partial_{1222} f(X^{1},X^{2}) \Delta_{k,n}\partial_{1222}f(X^{1},X^{2})I_2^{(2)}(\delta_{(k+1)2^{-n/2}}^{\otimes 2})\\ 
&& \times e^{i \langle \lambda, W_n(f,t) \rangle}\xi \bigg) \bigg \langle \bigg(\frac{\varepsilon_{j2^{-n/2}} + \varepsilon_{(j+1)2^{-n/2} }}{2} \bigg)\tilde{\otimes}\bigg(\frac{\varepsilon_{k2^{-n/2}} + \varepsilon_{(k+1)2^{-n/2} }}{2} \bigg), \delta_{(j+1)2^{-n/2}}^{\otimes 2} \bigg \rangle \\
&& \times \langle \delta_{(k+1)2^{-n/2}}, \delta_{(j+1)2^{-n/2}} \rangle\\
&&  -\frac{1}{64}\sum_{j,k=0}^{\lfloor 2^{\frac{n}{2}} t \rfloor -1}E\bigg (\Delta_{j,n}\partial_{122} f(X^{1},X^{2}) \Delta_{k,n}\partial_{12222}f(X^{1},X^{2})I_2^{(2)}(\delta_{(k+1)2^{-n/2}}^{\otimes 2})\\ 
&&\times e^{i \langle \lambda, W_n(f,t) \rangle}\xi\bigg)\bigg\langle \bigg(\frac{\varepsilon_{k2^{-n/2}} + \varepsilon_{(k+1)2^{-n/2} }}{2} \bigg)^{\otimes 2} , \delta_{(j+1)2^{-n/2}}^{\otimes 2} \bigg \rangle \langle \delta_{(k+1)2^{-n/2}}, \delta_{(j+1)2^{-n/2}} \rangle\\
&& = \sum_{i=1}^{3} B_{n,i}^{(3,2)}(t),
\end{eqnarray*}
\begin{eqnarray*}
&& B_n^{(3,2,b)}(t)\\
&=& -\frac{4}{64}\sum_{j,k=0}^{\lfloor 2^{\frac{n}{2}} t \rfloor -1}E\bigg (\Delta_{j,n}\partial_{1222} f(X^{1},X^{2})\Delta_{k,n}\partial_{122}f(X^{1},X^{2})I_1^{(2)}(\delta_{(k+1)2^{-n/2}})\\
&& \times e^{i \langle \lambda, W_n(f,t) \rangle}\xi  \bigg)\bigg \langle \bigg(\frac{\varepsilon_{j2^{-n/2}} + \varepsilon_{(j+1)2^{-n/2} }}{2} \bigg)\tilde{\otimes}\delta_{(k+1)2^{-n/2}} ,\delta_{(j+1)2^{-n/2}}^{\otimes 2} \bigg \rangle \\
&&\times \langle \delta_{(k+1)2^{-n/2}}, \delta_{(j+1)2^{-n/2}} \rangle\\
&& - \frac{2i}{64}\sum_{j,k=0}^{\lfloor 2^{\frac{n}{2}} t \rfloor -1}E\bigg (\Delta_{j,n}\partial_{1222} f(X^{1},X^{2})\Delta_{k,n}\partial_{122}f(X^{1},X^{2})I_2^{(2)}(\delta_{(k+1)2^{-n/2}}^{\otimes 2})e^{i \langle \lambda, W_n(f,t) \rangle}\xi\\
&& \times \bigg \langle \bigg(\frac{\varepsilon_{j2^{-n/2}} + \varepsilon_{(j+1)2^{-n/2} }}{2} \bigg)\tilde{\otimes}D_{X^{2}}(\langle \lambda, W_n(f,t) \rangle) , \delta_{(j+1)2^{-n/2}}^{\otimes 2} \bigg \rangle\bigg) \\
&& \times \langle \delta_{(k+1)2^{-n/2}}, \delta_{(j+1)2^{-n/2}} \rangle \\
\end{eqnarray*}
\begin{eqnarray*}
&& - \frac{2}{64}\sum_{j,k=0}^{\lfloor 2^{\frac{n}{2}} t \rfloor -1}E\bigg (\Delta_{j,n}\partial_{1222} f(X^{1},X^{2})\Delta_{k,n}\partial_{122}f(X^{1},X^{2})I_2^{(2)}(\delta_{(k+1)2^{-n/2}}^{\otimes 2})e^{i \langle \lambda, W_n(f,t) \rangle}\\
&& \times \bigg \langle \bigg(\frac{\varepsilon_{j2^{-n/2}} + \varepsilon_{(j+1)2^{-n/2} }}{2} \bigg)\tilde{\otimes}D_{X^{2}}(\xi) , \delta_{(j+1)2^{-n/2}}^{\otimes 2} \bigg \rangle\bigg) \langle \delta_{(k+1)2^{-n/2}}, \delta_{(j+1)2^{-n/2}} \rangle \\
&& - \frac{4}{64}\sum_{j,k=0}^{\lfloor 2^{\frac{n}{2}} t \rfloor -1}E\bigg ( \Delta_{j,n}\partial_{122} f(X^{1},X^{2})\Delta_{k,n}\partial_{1222}f(X^{1},X^{2})I_1^{(2)}(\delta_{(k+1)2^{-n/2}}) e^{i \langle \lambda, W_n(f,t) \rangle}\xi\bigg)\\
&& \times \bigg \langle \bigg(\frac{\varepsilon_{k2^{-n/2}} + \varepsilon_{(k+1)2^{-n/2} }}{2} \bigg)\tilde{\otimes}\delta_{(k+1)2^{-n/2}} , \delta_{(j+1)2^{-n/2}}^{\otimes 2} \bigg \rangle \langle \delta_{(k+1)2^{-n/2}}, \delta_{(j+1)2^{-n/2}} \rangle \\
&& -\frac{2i}{64}\sum_{j,k=0}^{\lfloor 2^{\frac{n}{2}} t \rfloor -1}E\bigg (\Delta_{j,n}\partial_{122} f(X^{1},X^{2})\Delta_{k,n}\partial_{1222}f(X^{1},X^{2})I_2^{(2)}(\delta_{(k+1)2^{-n/2}}^{\otimes 2})e^{i \langle \lambda, W_n(f,t) \rangle}\xi\\
&& \times \bigg \langle \bigg(\frac{\varepsilon_{k2^{-n/2}} + \varepsilon_{(k+1)2^{-n/2} }}{2} \bigg)\tilde{\otimes}D_{X^{2}}(\langle \lambda, W_n(f,t) \rangle) , \delta_{(j+1)2^{-n/2}}^{\otimes 2} \bigg \rangle\bigg)\\
&&\times \langle \delta_{(k+1)2^{-n/2}}, \delta_{(j+1)2^{-n/2}} \rangle\\
&& -\frac{2}{64}\sum_{j,k=0}^{\lfloor 2^{\frac{n}{2}} t \rfloor -1}E\bigg ( \Delta_{j,n}\partial_{122} f(X^{1},X^{2})\Delta_{k,n}\partial_{1222}f(X^{1},X^{2})I_2^{(2)}(\delta_{(k+1)2^{-n/2}}^{\otimes 2})e^{i \langle \lambda, W_n(f,t) \rangle}\\
&& \times \bigg \langle \bigg(\frac{\varepsilon_{k2^{-n/2}} + \varepsilon_{(k+1)2^{-n/2} }}{2} \bigg)\tilde{\otimes}D_{X^{2}}(\xi) , \delta_{(j+1)2^{-n/2}}^{\otimes 2} \bigg \rangle \bigg)\langle \delta_{(k+1)2^{-n/2}}, \delta_{(j+1)2^{-n/2}} \rangle\\
&& = \sum_{i=4}^{9} B_{n,i}^{(3,2)}(t),
\end{eqnarray*}
\begin{eqnarray*}
&& B_n^{(3,2,c)}(t)\\
&=& -\frac{2}{64}\sum_{j,k=0}^{\lfloor 2^{\frac{n}{2}} t \rfloor -1}E\bigg ( \Delta_{j,n}\partial_{122} f(X^{1},X^{2}) \Delta_{k,n}\partial_{122}f(X^{1},X^{2})e^{i \langle \lambda, W_n(f,t) \rangle}\xi \bigg)\\
&& \times \langle \delta_{(k+1)2^{-n/2}}^{\otimes 2} , \delta_{(j+1)2^{-n/2}}^{\otimes 2} \rangle \langle \delta_{(k+1)2^{-n/2}}, \delta_{(j+1)2^{-n/2}} \rangle\\
&& - \frac{4i}{64} \sum_{j,k=0}^{\lfloor 2^{\frac{n}{2}} t \rfloor -1}E\bigg( \Delta_{j,n}\partial_{122} f(X^{1},X^{2}) \Delta_{k,n}\partial_{122}f(X^{1},X^{2})I_1^{(2)}(\delta_{(k+1)2^{-n/2}})e^{i \langle \lambda, W_n(f,t) \rangle}\xi \\
&& \times \bigg \langle \delta_{(k+1)2^{-n/2}}\tilde{\otimes}D_{X^{2}}(\langle \lambda, W_n(f,t) \rangle) , \delta_{(j+1)2^{-n/2}}^{\otimes 2} \bigg\rangle\bigg) \langle \delta_{(k+1)2^{-n/2}}, \delta_{(j+1)2^{-n/2}} \rangle\\
\end{eqnarray*}
 \begin{eqnarray*} 
&& - \frac{4}{64}\sum_{j,k=0}^{\lfloor 2^{\frac{n}{2}} t \rfloor -1}E\bigg(\Delta_{j,n}\partial_{122} f(X^{1},X^{2})\Delta_{k,n}\partial_{122}f(X^{1},X^{2})I_1^{(2)}(\delta_{(k+1)2^{-n/2}})e^{i \langle \lambda, W_n(f,t) \rangle} \\
&& \times \bigg \langle \delta_{(k+1)2^{-n/2}}\tilde{\otimes}D_{X^{2}}(\xi) ,\delta_{(j+1)2^{-n/2}}^{\otimes 2} \bigg\rangle \bigg) \langle \delta_{(k+1)2^{-n/2}}, \delta_{(j+1)2^{-n/2}} \rangle\\
&& - \frac{i}{64} \sum_{j,k=0}^{\lfloor 2^{\frac{n}{2}} t \rfloor -1}E\bigg( \Delta_{j,n}\partial_{122} f(X^{1},X^{2}) \Delta_{k,n}\partial_{122}f(X^{1},X^{2})
 I_2^{(2)}(\delta_{(k+1)2^{-n/2}}^{\otimes 2})e^{i \langle \lambda, W_n(f,t) \rangle}\xi \\
 && \times \bigg \langle D^2_{X^{2}}(\langle \lambda, W_n(f,t) \rangle) , \delta_{(j+1)2^{-n/2}}^{\otimes 2} \bigg\rangle\bigg) \langle \delta_{(k+1)2^{-n/2}}, \delta_{(j+1)2^{-n/2}} \rangle\\
 && + \frac{1}{64}\sum_{j,k=0}^{\lfloor 2^{\frac{n}{2}} t \rfloor -1}E\bigg( \Delta_{j,n}\partial_{122} f(X^{1},X^{2}) \Delta_{k,n}\partial_{122}f(X^{1},X^{2})I_2^{(2)}(\delta_{(k+1)2^{-n/2}}^{\otimes 2})e^{i \langle \lambda, W_n(f,t) \rangle}\xi \\
 && \times\bigg \langle \big(D_{X^{2}}(\langle \lambda, W_n(f,t) \rangle)\big)^{\otimes 2}, \delta_{(j+1)2^{-n/2}}^{\otimes 2} \bigg\rangle\bigg) \langle \delta_{(k+1)2^{-n/2}}, \delta_{(j+1)2^{-n/2}} \rangle\\
&& - \frac{2i}{64}\sum_{j,k=0}^{\lfloor 2^{\frac{n}{2}} t \rfloor -1}E\bigg(\Delta_{j,n}\partial_{122} f(X^{1},X^{2}) \Delta_{k,n}\partial_{122}f(X^{1},X^{2}) I_2^{(2)}(\delta_{(k+1)2^{-n/2}}^{\otimes 2})e^{i \langle \lambda, W_n(f,t) \rangle} \\
&& \times \bigg \langle D_{X^{2}}(\langle \lambda, W_n(f,t) \rangle)\tilde{\otimes}D_{X^{2}}(\xi) , \delta_{(j+1)2^{-n/2}}^{\otimes 2} \bigg\rangle\bigg) \langle \delta_{(k+1)2^{-n/2}}, \delta_{(j+1)2^{-n/2}} \rangle\\
&& - \frac{1}{64} \sum_{j,k=0}^{\lfloor 2^{\frac{n}{2}} t \rfloor -1}E\bigg(\Delta_{j,n}\partial_{122} f(X^{1},X^{2})\Delta_{k,n}\partial_{122}f(X^{1},X^{2})I_2^{(2)}(\delta_{(k+1)2^{-n/2}}^{\otimes 2})e^{i \langle \lambda, W_n(f,t) \rangle} \\
&& \times \bigg \langle D^2_{X^{2}}(\xi) , \delta_{(j+1)2^{-n/2}}^{\otimes 2} \bigg\rangle\bigg) \langle \delta_{(k+1)2^{-n/2}}, \delta_{(j+1)2^{-n/2}} \rangle\\
&& = \sum_{i=10}^{16} B_{n,i}^{(3,2)}(t).
\end{eqnarray*}

Now, we will prove firstly that $B_{n,10}^{(3,2)}(t) \to -\kappa_3^2E\bigg(e^{i \langle \lambda, W_{\infty}(f,t) \rangle}\xi\times \int_0^t\big(\partial_{122} f(X_s^{1},X_s^{2})\big)^2ds\bigg)$, then we will prove the convergence to 0 of $B_{n,i}^{(3,2)}(t)$ for $i \in \{1, \ldots , 16\} \backslash \{10\}$.
\begin{itemize}
\item \underline{Convergence of $B_{n,10}^{(3,2)}(t)$ :}

\begin{eqnarray*}
B_{n,10}^{(3,2)}(t)&=& -\frac{1}{32}\sum_{j,k=0}^{\lfloor 2^{\frac{n}{2}} t \rfloor -1}E\bigg ( \Delta_{j,n}\partial_{122} f(X^{1},X^{2}) \Delta_{k,n}\partial_{122}f(X^{1},X^{2})e^{i \langle \lambda, W_n(f,t) \rangle}\xi \bigg)\\
&& \times \langle \delta_{(k+1)2^{-n/2}} , \delta_{(j+1)2^{-n/2}} \rangle^3.
\end{eqnarray*}
Let us show that 
\begin{equation}
B_{n,10}^{(3,2)}(t) \underset{n \to \infty}{\longrightarrow} -\kappa_3^2 E\bigg(e^{i \langle \lambda, W_{\infty}(f,t) \rangle}\xi\times \int_0^t\big(\partial_{122} f(X_s^{1},X_s^{2})\big)^2ds\bigg). \label{convergenceB_n10,1}
\end{equation}
Let us prove firstly that, a.s.,
\begin{eqnarray}
&&\frac{1}{32}\sum_{j,k=0}^{\lfloor 2^{\frac{n}{2}} t \rfloor -1} \Delta_{j,n}\partial_{122} f(X^{1},X^{2}) \Delta_{k,n}\partial_{122}f(X^{1},X^{2}) \langle \delta_{(k+1)2^{-n/2}} , \delta_{(j+1)2^{-n/2}} \rangle^3 \notag \\  
&& \underset{n \to \infty}{\longrightarrow}\kappa_3^2 \int_0^t\big(\partial_{122} f(X_s^{1},X_s^{2})\big)^2ds.\label{convergence-Heinetheorem}
\end{eqnarray}
We have, see (\ref{rho}) for the definition of $\rho$ :
\begin{eqnarray}
&& \frac{1}{32}\sum_{j,k=0}^{\lfloor 2^{\frac{n}{2}} t \rfloor -1} \Delta_{j,n}\partial_{122} f(X^{1},X^{2}) \Delta_{k,n}\partial_{122}f(X^{1},X^{2}) \langle \delta_{(k+1)2^{-n/2}} , \delta_{(j+1)2^{-n/2}} \rangle^3 \notag\\
&=& \frac{2^{-n/2}}{32}\sum_{j,k=0}^{\lfloor 2^{\frac{n}{2}} t \rfloor -1} \Delta_{j,n}\partial_{122} f(X^{1},X^{2}) \Delta_{k,n}\partial_{122}f(X^{1},X^{2}) \rho(j-k)^3\notag \\
&=& \frac{2^{-n/2}}{32}\sum_{r=1-\lfloor 2^{\frac{n}{2}} t \rfloor}^{\lfloor 2^{\frac{n}{2}} t \rfloor -1}\rho(r)^3\sum_{k= 0 \vee (-r)}^{(\lfloor 2^{\frac{n}{2}} t \rfloor -1 -r)\wedge (\lfloor 2^{\frac{n}{2}} t \rfloor -1)}\Delta_{r+k,n}\partial_{122} f(X^{1},X^{2}) \label{convergence B_n10,2}\\
&& \hspace{7cm}\times  \Delta_{k,n}\partial_{122}f(X^{1},X^{2}),\notag
\end{eqnarray}
where the last equality comes from the simple change of variable $r=j-k$, together with a Fubini argument. Observe that
\begin{eqnarray*}
&&\bigg|2^{-n/2}\sum_{k= 0 \vee (-r)}^{(\lfloor 2^{\frac{n}{2}} t \rfloor -1 -r)\wedge (\lfloor 2^{\frac{n}{2}} t \rfloor -1)}\Delta_{r+k,n}\partial_{122} f(X^{1},X^{2}) \Delta_{k,n}\partial_{122}f(X^{1},X^{2})\\
&& - \int_0^t\big(\partial_{122} f(X_s^{1},X_s^{2})\big)^2ds\bigg|\\
&\leq & \bigg|2^{-n/2}\sum_{k= 0 \vee (-r)}^{(\lfloor 2^{\frac{n}{2}} t \rfloor -1 -r)\wedge (\lfloor 2^{\frac{n}{2}} t \rfloor -1)}\Delta_{r+k,n}\partial_{122} f(X^{1},X^{2}) \Delta_{k,n}\partial_{122}f(X^{1},X^{2})\\
&& - 2^{-n/2}\sum_{k= 0 \vee (-r)}^{(\lfloor 2^{\frac{n}{2}} t \rfloor -1 -r)\wedge (\lfloor 2^{\frac{n}{2}} t \rfloor -1)}\partial_{122} f(X^{1}_{k2^{-n/2}},X^{2}_{k2^{-n/2}}) \Delta_{k,n}\partial_{122}f(X^{1},X^{2}) \bigg|\\
&& + \bigg|2^{-n/2}\sum_{k= 0 \vee (-r)}^{(\lfloor 2^{\frac{n}{2}} t \rfloor -1 -r)\wedge (\lfloor 2^{\frac{n}{2}} t \rfloor -1)}\partial_{122}f(X^{1}_{k2^{-n/2}},X^{2}_{k2^{-n/2}}) \Delta_{k,n}\partial_{122}f(X^{1},X^{2})
\end{eqnarray*}
\begin{eqnarray*}
&& - 2^{-n/2}\sum_{k= 0 \vee (-r)}^{(\lfloor 2^{\frac{n}{2}} t \rfloor -1 -r)\wedge (\lfloor 2^{\frac{n}{2}} t \rfloor -1)}\big(\partial_{122} f(X^{1}_{k2^{-n/2}},X^{2}_{k2^{-n/2}})\big)^2 \bigg|\\
&& + \bigg|2^{-n/2}\sum_{k= 0 \vee (-r)}^{(\lfloor 2^{\frac{n}{2}} t \rfloor -1 -r)\wedge (\lfloor 2^{\frac{n}{2}} t \rfloor -1)}\big(\partial_{122} f(X^{1}_{k2^{-n/2}},X^{2}_{k2^{-n/2}})\big)^2 \\
&& - \int_0^t\big(\partial_{122} f(X_s^{1},X_s^{2})\big)^2ds\bigg|\\
&=& r_{1,n} + r_{2,n} +r_{3,n},
\end{eqnarray*}
with obvious notation at the last equality. Let us prove the convergence to 0 of $r_{1,n}$, $r_{2,n}$ and $r_{3,n}$.
 
For any fixed integer $r > 0$ (the case $r \leq 0$ being similar), by Theorem \ref{theorem-Taylor} and since $f \in C_b^{\infty}$,  we have
\begin{eqnarray*}
r_{1,n} &\leq & \|\partial_{122} f\|_{\infty}2^{-n/2}\sum_{k= 0 }^{\lfloor 2^{\frac{n}{2}} t \rfloor -1}\big|\Delta_{r+k,n}\partial_{122} f(X^{1},X^{2})- \partial_{122} f(X^{1}_{k2^{-n/2}},X^{2}_{k2^{-n/2}})\big|\\
&\leq & C 2^{-n/2}\sum_{k= 0 }^{\lfloor 2^{\frac{n}{2}} t \rfloor -1}\big|\partial_{1122} f(X^{1}_{k2^{-n/2}},X^{2}_{k2^{-n/2}})\big|\\
&& \hspace{2cm}\times \big| X^{1}_{(r+k+1)2^{-n/2}}+ X^{1}_{(r+k)2^{-n/2}} - 2X^{1}_{k2^{-n/2}}\big|\\
&& + C2^{-n/2}\sum_{k= 0 }^{\lfloor 2^{\frac{n}{2}} t \rfloor -1}\big|\partial_{1222} f(X^{1}_{k2^{-n/2}},X^{2}_{k2^{-n/2}})\big|\\
&& \hspace{2cm}\times \big| X^{2}_{(r+k+1)2^{-n/2}} + X^{2}_{(r+k)2^{-n/2}}- 2X^{2}_{k2^{-n/2}}\big|\\
&& + C2^{-n/2}\sum_{k= 0 }^{\lfloor 2^{\frac{n}{2}} t \rfloor -1}\bigg|R_2\bigg( (X^{1}_{k2^{-n/2}},X^{2}_{k2^{-n/2}}) ,\\
&& \hspace{2cm}\bigg(\frac{X^{1}_{(r+k+1)2^{-n/2}} + X^{1}_{(r+k)2^{-n/2}}}{2},\frac{X^{2}_{(r+k+1)2^{-n/2}} + X^{2}_{(r+k)2^{-n/2}}}{2}\bigg)  \bigg)\bigg|\\
&\leq & Ct\sup_{0\leq k \leq \lfloor 2^{\frac{n}{2}} t \rfloor -1} \big|X^{1}_{(r+k+1)2^{-n/2}}+ X^{1}_{(r+k)2^{-n/2}} - 2X^{1}_{k2^{-n/2}} \big|\\
&& + Ct\sup_{0\leq k \leq \lfloor 2^{\frac{n}{2}} t \rfloor -1} \big|X^{2}_{(r+k+1)2^{-n/2}}+ X^{2}_{(r+k)2^{-n/2}} - 2X^{2}_{k2^{-n/2}} \big|\\
&& + C2^{-n/2}\sum_{k= 0 }^{\lfloor 2^{\frac{n}{2}} t \rfloor -1}\bigg|R_2\bigg( (X^{1}_{k2^{-n/2}},X^{2}_{k2^{-n/2}}) ,\\
&& \hspace{2cm}\bigg(\frac{X^{1}_{(r+k+1)2^{-n/2}} + X^{1}_{(r+k)2^{-n/2}}}{2},\frac{X^{2}_{(r+k+1)2^{-n/2}} + X^{2}_{(r+k)2^{-n/2}}}{2}\bigg)  \bigg)\bigg|.
\end{eqnarray*}
By Theorem \ref{theorem-Taylor} and since $f \in C_b^{\infty}$, we have
\begin{eqnarray*}
&&\bigg|R_2\bigg( (X^{1}_{k2^{-n/2}},X^{2}_{k2^{-n/2}}) ,\bigg(\frac{X^{1}_{(r+k+1)2^{-n/2}} + X^{1}_{(r+k)2^{-n/2}}}{2},\\
&& \hspace{5cm}\frac{X^{2}_{(r+k+1)2^{-n/2}} + X^{2}_{(r+k)2^{-n/2}}}{2}\bigg)  \bigg)\bigg|\\
&\leq & C\bigg(\big|X^{1}_{(r+k+1)2^{-n/2}}+ X^{1}_{(r+k)2^{-n/2}} - 2X^{1}_{k2^{-n/2}} \big|^2 \\
&&+ \big|X^{1}_{(r+k+1)2^{-n/2}}+ X^{1}_{(r+k)2^{-n/2}} - 2X^{1}_{k2^{-n/2}} \big|\big|X^{2}_{(r+k+1)2^{-n/2}}+ X^{2}_{(r+k)2^{-n/2}}\\
&& - 2X^{2}_{k2^{-n/2}} \big| + \big|X^{2}_{(r+k+1)2^{-n/2}}+ X^{2}_{(r+k)2^{-n/2}} - 2X^{2}_{k2^{-n/2}} \big|^2 \bigg).
\end{eqnarray*}
We finally get
\begin{eqnarray*}
r_{1,n} &\leq & C\sup_{0\leq k \leq \lfloor 2^{\frac{n}{2}} t \rfloor -1} \big|X^{1}_{(r+k+1)2^{-n/2}}+ X^{1}_{(r+k)2^{-n/2}} - 2X^{1}_{k2^{-n/2}} \big|\\
&& +C\sup_{0\leq k \leq \lfloor 2^{\frac{n}{2}} t \rfloor -1} \big|X^{1}_{(r+k+1)2^{-n/2}}+ X^{1}_{(r+k)2^{-n/2}} - 2X^{1}_{k2^{-n/2}} \big|^2\\
&& + C\sup_{0\leq k \leq \lfloor 2^{\frac{n}{2}} t \rfloor -1} \big|X^{2}_{(r+k+1)2^{-n/2}}+ X^{2}_{(r+k)2^{-n/2}} - 2X^{2}_{k2^{-n/2}} \big|\\
&&+C\sup_{0\leq k \leq \lfloor 2^{\frac{n}{2}} t \rfloor -1} \big|X^{2}_{(r+k+1)2^{-n/2}}+ X^{2}_{(r+k)2^{-n/2}} - 2X^{2}_{k2^{-n/2}} \big|^2\\
&&+ C\sup_{0\leq k \leq \lfloor 2^{\frac{n}{2}} t \rfloor -1}\bigg( \big|X^{1}_{(r+k+1)2^{-n/2}}+ X^{1}_{(r+k)2^{-n/2}} - 2X^{1}_{k2^{-n/2}} \big|\\
&& \hspace{3cm}\times\big|X^{2}_{(r+k+1)2^{-n/2}}+ X^{2}_{(r+k)2^{-n/2}} - 2X^{2}_{k2^{-n/2}} \big|\bigg).
\end{eqnarray*}
By Heine's theorem, the last quantities converge to 0 almost surely. Thus, $r_{1,n} \to 0$ as $n \to \infty$. We can prove similarly that $r_{2,n} \to 0$ as $n \to \infty$. Finally, it is clear that $r_{3,n} \to 0$ as $n \to \infty$. Hence, we have proved that, for all fixed $r \in \Z$, a.s.,
\begin{eqnarray*}
&& 2^{-n/2}\sum_{k= 0 \vee (-r)}^{(\lfloor 2^{\frac{n}{2}} t \rfloor -1 -r)\wedge (\lfloor 2^{\frac{n}{2}} t \rfloor -1)}\Delta_{r+k,n}\partial_{122} f(X^{1},X^{2}) \Delta_{k,n}\partial_{122}f(X^{1},X^{2}) \underset{n \to \infty}{\longrightarrow}\\
&& \int_0^t\big(\partial_{122} f(X_s^{1},X_s^{2})\big)^2ds.
\end{eqnarray*} 
By combining a bounded convergence argument with (\ref{convergence B_n10,2}) (observe that $\kappa_3^2:=\frac{1}{32}\sum_{r \in \Z}\rho^3(r) < \infty$), we deduce that, a.s.,
\begin{eqnarray*}
&&\frac{1}{32}\sum_{j,k=0}^{\lfloor 2^{\frac{n}{2}} t \rfloor -1} \Delta_{j,n}\partial_{122} f(X^{1},X^{2}) \Delta_{k,n}\partial_{122}f(X^{1},X^{2}) \langle \delta_{(k+1)2^{-n/2}} , \delta_{(j+1)2^{-n/2}} \rangle^3  \\
&& \underset{n \to \infty}{\longrightarrow}\kappa_3^2 \int_0^t\big(\partial_{122} f(X_s^{1},X_s^{2})\big)^2ds.
\end{eqnarray*}
Thus (\ref{convergence-Heinetheorem}) holds true.

Since $\big ( X^{1}, X^{2}, W_n(f,t) \big) \to \big( X^{1}, X^{2},  W_{\infty}(f,t)  \big)$ in $D_{\R^2}[0, \infty)\times \R^4$, we deduce the following convergence in law in $\R^3$,
\begin{eqnarray*}
&& \bigg(e^{i \langle \lambda, W_n(f,t) \rangle},\: \xi, \: \frac{1}{32}\sum_{j,k=0}^{\lfloor 2^{\frac{n}{2}} t \rfloor -1} \Delta_{j,n}\partial_{122} f(X^{1},X^{2}) \Delta_{k,n}\partial_{122}f(X^{1},X^{2})\\
&& \times  \langle \delta_{(k+1)2^{-n/2}} , \delta_{(j+1)2^{-n/2}} \rangle^3\bigg) \overset{Law}{\longrightarrow}\bigg(e^{i \langle \lambda, W_{\infty}(f,t) \rangle},\: \xi, \: \kappa_3^2 \int_0^t\big(\partial_{122} f(X_s^{1},X_s^{2})\big)^2ds \bigg).
\end{eqnarray*}
By boundedness of $e^{i \langle \lambda, W_n(f,t) \rangle}$, $\xi$ and $\partial_{122}f$, we deduce that (\ref{convergenceB_n10,1}) follows.

%Now, it remains to prove the convergence to 0 of $B_{n,i}^{(3,2)}(t)$ for $i \in \{1, \ldots , 16\} \backslash \{10\}$:

\item \underline{Convergence to 0 of $B_{n,1}^{(3,2)}(t)$, $B_{n,2}^{(3,2)}(t)$ and $B_{n,3}^{(3,2)}(t)$.}
Let us first focus on $B_{n,1}^{(3,2)}(t)$. Since $f \in C_b^{\infty}$, $e^{i \langle \lambda, W_n(f,t) \rangle}$ and $\xi$ are bounded and by Cauchy-Schwarz inequality and (\ref{isometry}), we have
\begin{eqnarray*}
 && \bigg| E\bigg ( \Delta_{j,n}\partial_{12222} f(X^{1},X^{2})\Delta_{k,n}\partial_{122}f(X^{1},X^{2})I_2^{(2)}(\delta_{(k+1)2^{-n/2}}^{\otimes 2}) e^{i \langle \lambda, W_n(f,t) \rangle}\xi\bigg) \bigg|\\
 & \leq & C \|I_2^{(2)}(\delta_{(k+1)2^{-n/2}}^{\otimes 2})\|_2 \leq C 2^{-n/6}.
 \end{eqnarray*}
Combining this fact with (\ref{12}) and (\ref{13}), we deduce that
 \begin{eqnarray*}
 |B_{n,1}^{(3,2)}(t)| &\leq & C 2^{-n/6}(2^{-n/6})^2\sum_{j,k=0}^{\lfloor 2^{\frac{n}{2}} t \rfloor -1} \big|\langle \delta_{(k+1)2^{-n/2}}, \delta_{(j+1)2^{-n/2}} \rangle\big|\\
 && \leq C2^{-n/6}2^{-n/3}t2^{n/3} = Ct2^{-n/6}.
 \end{eqnarray*}
 
Let us now turn to $B_{n,2}^{(3,2)}(t)$ and $B_{n,3}^{(3,2)}(t)$. Relying to the same arguments, we get
 \begin{eqnarray*}
 B_{n,2}^{(3,2)}(t)& \leq & Ct2^{-n/6}\\
 B_{n,3}^{(3,2)}(t)& \leq & Ct2^{-n/6}.
\end{eqnarray*}  
It is now clear that $B_{n,1}^{(3,2)}(t)$, $B_{n,2}^{(3,2)}(t)$ and $B_{n,3}^{(3,2)}(t)$ converge to 0 as $n \to \infty$.

\item \underline{Convergence to 0 of $B_{n,4}^{(3,2)}(t)$ and $B_{n,7}^{(3,2)}(t)$.} 
Let us first focus on $B_{n,4}^{(3,2)}(t)$. Since $f \in C_b^{\infty}$, $e^{i \langle \lambda, W_n(f,t) \rangle}$ and $\xi$ are bounded and by Cauchy-Schwarz inequality and (\ref{isometry}) we have
\begin{eqnarray*}
 && \bigg| E\bigg (\Delta_{j,n}\partial_{1222} f(X^{1},X^{2})\Delta_{k,n}\partial_{122}f(X^{1},X^{2})I_1^{(2)}(\delta_{(k+1)2^{-n/2}})e^{i \langle \lambda, W_n(f,t) \rangle}\xi  \bigg)\bigg|\\
 & \leq & C \| I_1^{(2)}(\delta_{(k+1)2^{-n/2}})\|_2 \leq C 2^{-n/12}.
 \end{eqnarray*}
 Combining this fact with (\ref{12}) and (\ref{13}), we get
 \begin{eqnarray*}
 |B_{n,4}^{(3,2)}(t)| &\leq & C2^{-n/12}\sum_{j,k=0}^{\lfloor 2^{\frac{n}{2}} t \rfloor -1} \bigg|\bigg \langle \bigg(\frac{\varepsilon_{j2^{-n/2}} + \varepsilon_{(j+1)2^{-n/2} }}{2} \bigg) ,\delta_{(j+1)2^{-n/2}} \bigg \rangle\bigg| \\
 && \times \langle \delta_{(k+1)2^{-n/2}}, \delta_{(j+1)2^{-n/2}} \rangle^2\\
 &\leq & C2^{-n/12}2^{-n/6}\sum_{j,k=0}^{\lfloor 2^{\frac{n}{2}} t \rfloor -1}\langle \delta_{(k+1)2^{-n/2}}, \delta_{(j+1)2^{-n/2}} \rangle^2\\
 & \leq & C2^{-n/12}2^{-n/6}t2^{n/6} \leq Ct2^{-n/12}.
 \end{eqnarray*}
 
Let us now turn to $B_{n,7}^{(3,2)}(t)$. By the same arguments, we get
 \begin{equation*}
 |B_{n,7}^{(3,2)}(t)| \leq Ct2^{-n/12}.
\end{equation*} 
It is now clear that $B_{n,4}^{(3,2)}(t)$ and $B_{n,7}^{(3,2)}(t)$ converge to 0 as $n \to \infty$. 

\item \underline{Convergence to 0 of $B_{n,5}^{(3,2)}(t)$ and $B_{n,8}^{(3,2)}(t)$.} Let us first focus on $B_{n,5}^{(3,2)}(t)$. Since $f \in C_b^{\infty}$, $e^{i \langle \lambda, W_n(f,t) \rangle}$ and $\xi$ are bounded and thanks to (\ref{12}) and to the Cauchy-Schwarz inequality, we get
\begin{eqnarray*}
&&|B_{n,5}^{(3,2)}(t)|\\
 &\leq & C2^{-n/6}\sum_{j,k=0}^{\lfloor 2^{\frac{n}{2}} t \rfloor -1}E\bigg (\big|I_2^{(2)}(\delta_{(k+1)2^{-n/2}}^{\otimes 2})\big| \big| \big\langle D_{X^{2}}(\langle \lambda, W_n(f,t) \rangle) , \delta_{(j+1)2^{-n/2}}\big\rangle\big| \bigg)\\
&&  \times \big|\langle \delta_{(k+1)2^{-n/2}}, \delta_{(j+1)2^{-n/2}} \rangle\big|\\
& \leq & C2^{-n/6} \sum_{p=1}^4 |\lambda_p|\sum_{j,k=0}^{\lfloor 2^{\frac{n}{2}} t \rfloor -1}E\bigg (\big|I_2^{(2)}(\delta_{(k+1)2^{-n/2}}^{\otimes 2})\big| \big| \big\langle D_{X^{2}}(K_n^{(p)}(t)) , \delta_{(j+1)2^{-n/2}}\big\rangle\big| \bigg)\\
&&  \times \big|\langle \delta_{(k+1)2^{-n/2}}, \delta_{(j+1)2^{-n/2}} \rangle\big|\\
&\leq & C2^{-n/6} \sum_{p=1}^4 |\lambda_p|\sum_{j,k=0}^{\lfloor 2^{\frac{n}{2}} t \rfloor -1} \|I_2^{(2)}(\delta_{(k+1)2^{-n/2}}^{\otimes 2})\|_2 \|\big\langle D_{X^{2}}(K_n^{(p)}(t)) , \delta_{(j+1)2^{-n/2}}\big\rangle\|_2 \\
&&   \times \big|\langle \delta_{(k+1)2^{-n/2}}, \delta_{(j+1)2^{-n/2}} \rangle\big|.
\end{eqnarray*}
Due to (\ref{isometry}), (\ref{lemma5-2}) and  (\ref{13}), we have
\begin{eqnarray*}
|B_{n,5}^{(3,2)}(t)| & \leq & C2^{-n/6}2^{-n/6}2^{-n/6}\big(t^2 + t +1 \big)^{\frac12}\sum_{j,k=0}^{\lfloor 2^{\frac{n}{2}} t \rfloor -1}\big|\langle \delta_{(k+1)2^{-n/2}}, \delta_{(j+1)2^{-n/2}} \rangle\big|\\
& \leq & Ct\big(t^2 + t +1 \big)^{\frac12}2^{-n/6}.
\end{eqnarray*}

Let us now turn to $B_{n,8}^{(3,2)}(t)$. The same arguments shows that
\begin{equation*}
|B_{n,8}^{(3,2)}(t)| \leq Ct\big(t^2 + t +1 \big)^{\frac12}2^{-n/6}.
\end{equation*}
It is now clear that $B_{n,5}^{(3,2)}(t)$ and $B_{n,8}^{(3,2)}(t)$ converge to 0 as $n \to \infty$.

\item \underline{Convergence to 0 of $B_{n,6}^{(3,2)}(t)$ and $B_{n,9}^{(3,2)}(t)$.} Let us first focus on $B_{n,6}^{(3,2)}(t)$. Since $f \in C_b^{\infty}$, $e^{i \langle \lambda, W_n(f,t) \rangle}$ is bounded and due to (\ref{12}), (\ref{lemma5-1}), (\ref{isometry}) and (\ref{13}), we have
\begin{eqnarray*}
|B_{n,6}^{(3,2)}(t)| &\leq & C2^{-n/6}\sum_{j,k=0}^{\lfloor 2^{\frac{n}{2}} t \rfloor -1}E\bigg (\big|I_2^{(2)}(\delta_{(k+1)2^{-n/2}}^{\otimes 2})\big| \big|\big\langle D_{X^{2}}(\xi) , \delta_{(j+1)2^{-n/2}} \big\rangle\big|\bigg)\\
&& \times \big|\langle \delta_{(k+1)2^{-n/2}}, \delta_{(j+1)2^{-n/2}} \rangle\big|\\
&\leq & C2^{-n/6}2^{-n/6}\sum_{j,k=0}^{\lfloor 2^{\frac{n}{2}} t \rfloor -1} \|I_2^{(2)}(\delta_{(k+1)2^{-n/2}}^{\otimes 2})\|_2\big|\langle \delta_{(k+1)2^{-n/2}}, \delta_{(j+1)2^{-n/2}} \rangle\big|\\
&\leq & C2^{-n/2}t2^{n/3} =Ct2^{-n/6}.
\end{eqnarray*}

Let us now turn to $B_{n,9}^{(3,2)}(t)$. The same arguments shows that
\begin{equation*}
|B_{n,9}^{(3,2)}(t)| \leq Ct2^{-n/6}.
\end{equation*}
We deduce that $B_{n,6}^{(3,2)}(t)$ and $B_{n,9}^{(3,2)}(t)$ both converge to 0 as $n \to \infty$.

\item \underline{Convergence to 0 of $B_{n,11}^{(3,2)}(t)$ :} Since $f \in C_b^{\infty}$, $e^{i \langle \lambda, W_n(f,t) \rangle}$ and $\xi$ are bounded and due to (\ref{lemma5-2}), (\ref{isometry}) and (\ref{13}), we have
\begin{eqnarray*}
&&|B_{n,11}^{(3,2)}(t)|\\
 &\leq & C\sum_{j,k=0}^{\lfloor 2^{\frac{n}{2}} t \rfloor -1}E\bigg( \big|I_1^{(2)}(\delta_{(k+1)2^{-n/2}})\big|\big| \big\langle D_{X^{2}}(\langle \lambda, W_n(f,t) \rangle) , \delta_{(j+1)2^{-n/2}} \big\rangle\big|  \bigg)\\
&& \times \big\langle \delta_{(k+1)2^{-n/2}}, \delta_{(j+1)2^{-n/2}} \big\rangle^2\\
\end{eqnarray*}
\begin{eqnarray*}
&\leq & C \sum_{p=1}^4 |\lambda_p|\sum_{j,k=0}^{\lfloor 2^{\frac{n}{2}} t \rfloor -1}E\bigg( \big|I_1^{(2)}(\delta_{(k+1)2^{-n/2}})\big|\big| \big\langle D_{X^{2}}(K_n^{(p)}(t)) , \delta_{(j+1)2^{-n/2}} \big\rangle\big|  \bigg)\\
&& \times \big\langle \delta_{(k+1)2^{-n/2}}, \delta_{(j+1)2^{-n/2}} \big\rangle^2\\
&\leq & C \sum_{p=1}^4 |\lambda_p|\sum_{j,k=0}^{\lfloor 2^{\frac{n}{2}} t \rfloor -1}\|I_1^{(2)}(\delta_{(k+1)2^{-n/2}})\|_2 \|\big\langle D_{X^{2}}(K_n^{(p)}(t)) , \delta_{(j+1)2^{-n/2}} \big\rangle\|_2 \\
&& \times \big\langle \delta_{(k+1)2^{-n/2}}, \delta_{(j+1)2^{-n/2}} \big\rangle^2\\
\end{eqnarray*}
\begin{eqnarray*}
&\leq & C2^{-n/12}2^{-n/6}(t^2+t+1)^{\frac12}\sum_{j,k=0}^{\lfloor 2^{\frac{n}{2}} t \rfloor -1}\big\langle \delta_{(k+1)2^{-n/2}}, \delta_{(j+1)2^{-n/2}} \big\rangle^2\\
& \leq & Ct(t^2+t+1)^{\frac12} 2^{-n/12}.
\end{eqnarray*}
Hence, $B_{n,11}^{(3,2)}(t)$ converge to 0 as $n \to \infty$.

\item \underline{Convergence to 0 of $B_{n,12}^{(3,2)}(t)$ :} Since $f \in C_b^{\infty}$, $e^{i \langle \lambda, W_n(f,t) \rangle}$ is bounded and due to (\ref{lemma5-1}), (\ref{isometry}) and (\ref{13}), we have
\begin{eqnarray*}
|B_{n,12}^{(3,2)}(t)| &\leq & C \sum_{j,k=0}^{\lfloor 2^{\frac{n}{2}} t \rfloor -1}E\bigg(\big|I_1^{(2)}(\delta_{(k+1)2^{-n/2}})\big| \big|\big\langle D_{X^{2}}(\xi) ,\delta_{(j+1)2^{-n/2}}\big\rangle\big|  \bigg)\\
&& \times \langle \delta_{(k+1)2^{-n/2}}, \delta_{(j+1)2^{-n/2}} \rangle^2\\ 
&\leq & C2^{-n/6} \sum_{j,k=0}^{\lfloor 2^{\frac{n}{2}} t \rfloor -1}\|I_1^{(2)}(\delta_{(k+1)2^{-n/2}})\|_2 \langle \delta_{(k+1)2^{-n/2}}, \delta_{(j+1)2^{-n/2}} \rangle^2\\
& \leq & C 2^{-n/12}t.
\end{eqnarray*}
Thus, $B_{n,12}^{(3,2)}(t)$ converge to 0 as $n \to \infty$.

\item \underline{Convergence to 0 of $B_{n,13}^{(3,2)}(t)$ :}
\begin{eqnarray*}
|B_{n,13}^{(3,2)}(t)|&\leq & C \sum_{j,k=0}^{\lfloor 2^{\frac{n}{2}} t \rfloor -1}E\bigg( 
 \big|I_2^{(2)}(\delta_{(k+1)2^{-n/2}}^{\otimes 2})\big|\big| \langle D^2_{X^{2}}(\langle \lambda, W_n(f,t) \rangle) , \delta_{(j+1)2^{-n/2}}^{\otimes 2} \rangle\big| \bigg) \\
 && \times \big|\langle \delta_{(k+1)2^{-n/2}}, \delta_{(j+1)2^{-n/2}} \rangle\big|\\
 &\leq & C\sum_{p=1}^4 |\lambda_p|\sum_{j,k=0}^{\lfloor 2^{\frac{n}{2}} t \rfloor -1}E\bigg( 
 \big|I_2^{(2)}(\delta_{(k+1)2^{-n/2}}^{\otimes 2})\big|\big| \langle D^2_{X^{2}}(K_n^{(p)}) , \delta_{(j+1)2^{-n/2}}^{\otimes 2} \rangle\big| \bigg) \\
 && \times \big|\langle \delta_{(k+1)2^{-n/2}}, \delta_{(j+1)2^{-n/2}} \rangle\big|\\
 \end{eqnarray*}
 \begin{eqnarray*}
 &\leq & C\sum_{p=1}^4 |\lambda_p|\sum_{j,k=0}^{\lfloor 2^{\frac{n}{2}} t \rfloor -1} 
 \|I_2^{(2)}(\delta_{(k+1)2^{-n/2}}^{\otimes 2})\|_2\| \langle D^2_{X^{2}}(K_n^{(p)}) , \delta_{(j+1)2^{-n/2}}^{\otimes 2} \rangle\|_2\\
 && \times  \big|\langle \delta_{(k+1)2^{-n/2}}, \delta_{(j+1)2^{-n/2}} \rangle\big|\\
 & \leq & C2^{-n/6}2^{-n/3}(t^2+t+1)^{\frac12}2^{n/3}t =C2^{-n/6}(t^2+t+1)^{\frac12}t,
\end{eqnarray*}
where the fourth inequality is due to (\ref{isometry}), (\ref{lemma5-2}) and (\ref{13}). Hence, $B_{n,13}^{(3,2)}(t) \to 0$ as $n \to \infty$. 

\item \underline{Convergence to 0 of $B_{n,14}^{(3,2)}(t)$ :} Since $f\in C_b^{\infty}$ and thanks to (\ref{isometry}), (\ref{lemma5-4}) and (\ref{13}), we have
\begin{eqnarray*}
|B_{n,14}^{(3,2)}(t)| &\leq & C\sum_{j,k=0}^{\lfloor 2^{\frac{n}{2}} t \rfloor -1}E\bigg( \big| I_2^{(2)}(\delta_{(k+1)2^{-n/2}}^{\otimes 2})\big|\big\langle D_{X^{2}}(\langle \lambda, W_n(f,t) \rangle), \delta_{(j+1)2^{-n/2}} \big\rangle^2\bigg)\\
&& \times  \big|\langle \delta_{(k+1)2^{-n/2}}, \delta_{(j+1)2^{-n/2}} \rangle\big|\\
& \leq & C \sum_{j,k=0}^{\lfloor 2^{\frac{n}{2}} t \rfloor -1} \| I_2^{(2)}(\delta_{(k+1)2^{-n/2}}^{\otimes 2})\|_2 \|\langle D_{X^{2}}(\langle \lambda, W_n(f,t) \rangle), \delta_{(j+1)2^{-n/2}} \big\rangle^2\|_2\\
&& \times  \big|\langle \delta_{(k+1)2^{-n/2}}, \delta_{(j+1)2^{-n/2}} \rangle\big|\\
& \leq & C \sum_{p=1}^4 (\lambda_p)^2\sum_{j,k=0}^{\lfloor 2^{\frac{n}{2}} t \rfloor -1} \| I_2^{(2)}(\delta_{(k+1)2^{-n/2}}^{\otimes 2})\|_2 \|\langle D_{X^{2}}(K_n^{(p)}(t)), \delta_{(j+1)2^{-n/2}} \big\rangle^2\|_2\\
&& \times  \big|\langle \delta_{(k+1)2^{-n/2}}, \delta_{(j+1)2^{-n/2}} \rangle\big|\\
& \leq & C2^{-n/6}2^{-n/3}(1+t+t^2+t^3)^{\frac12}\sum_{j,k=0}^{\lfloor 2^{\frac{n}{2}} t \rfloor -1}\big|\langle \delta_{(k+1)2^{-n/2}}, \delta_{(j+1)2^{-n/2}} \rangle\big|\\
& \leq & C2^{-n/6}(1+t+t^2+t^3)^{\frac12}t.
 \end{eqnarray*}
 It is now clear that $B_{n,14}^{(3,2)}(t) \to 0$ as $n \to \infty$.
 
 \item \underline{Convergence to 0 of $B_{n,15}^{(3,2)}(t)$ :} Since $f \in C_b^{\infty}$ and thanks to (\ref{lemma5-1}), (\ref{lemma5-2}), (\ref{isometry}) and (\ref{13}), we obtain
 \begin{eqnarray*}
 && |B_{n,15}^{(3,2)}(t)| \leq C \sum_{j,k=0}^{\lfloor 2^{\frac{n}{2}} t \rfloor -1}E\bigg(\big|I_2^{(2)}(\delta_{(k+1)2^{-n/2}}^{\otimes 2})\big| \big|\langle D_{X^{2}}(\langle \lambda, W_n(f,t) \rangle) , \delta_{(j+1)2^{-n/2}} \big\rangle\big| \\
 && \times \big|\big\langle D_{X^{2}}(\xi),\delta_{(j+1)2^{-n/2}}\big\rangle\big|\bigg) \big|\langle \delta_{(k+1)2^{-n/2}}, \delta_{(j+1)2^{-n/2}} \rangle\big|\\
 \end{eqnarray*}
 \begin{eqnarray*}
 &\leq & C 2^{-n/6}\sum_{p=1}^4|\lambda_p |\sum_{j,k=0}^{\lfloor 2^{\frac{n}{2}} t \rfloor -1}E\bigg(\big|I_2^{(2)}(\delta_{(k+1)2^{-n/2}}^{\otimes 2})\big| \big|\langle D_{X^{2}}(K_n^{(p)}(t)) , \delta_{(j+1)2^{-n/2}} \big\rangle\big| \bigg)\\
 && \times  \big|\langle \delta_{(k+1)2^{-n/2}}, \delta_{(j+1)2^{-n/2}} \rangle\big|\\
 &\leq & C 2^{-n/6}\sum_{p=1}^4|\lambda_p |\sum_{j,k=0}^{\lfloor 2^{\frac{n}{2}} t \rfloor -1} \|I_2^{(2)}(\delta_{(k+1)2^{-n/2}}^{\otimes 2})\|_2 \|\langle D_{X^{2}}(K_n^{(p)}(t)) , \delta_{(j+1)2^{-n/2}} \big\rangle\|_2\\
 && \times  \big|\langle \delta_{(k+1)2^{-n/2}}, \delta_{(j+1)2^{-n/2}} \rangle\big|\\
 &\leq & C(2^{-n/6})^3(t^2+t+1)^{\frac12}\sum_{j,k=0}^{\lfloor 2^{\frac{n}{2}} t \rfloor -1}\big|\langle \delta_{(k+1)2^{-n/2}}, \delta_{(j+1)2^{-n/2}} \rangle\big|\\
 &\leq & C2^{-n/6}(t^2+t+1)^{\frac12}t.
\end{eqnarray*} 
Hence, $B_{n,15}^{(3,2)}(t) \to 0$ as $n \to \infty$.

\item \underline{Convergence to 0 of $B_{n,16}^{(3,2)}(t)$ :} Since $f\in C_b^{\infty}$ and thanks to (\ref{lemma5-1}), (\ref{isometry}) and (\ref{13}), we deduce that
\begin{eqnarray*}
|B_{n,16}^{(3,2)}(t)| &\leq & C \sum_{j,k=0}^{\lfloor 2^{\frac{n}{2}} t \rfloor -1}E\bigg(\big|I_2^{(2)}(\delta_{(k+1)2^{-n/2}}^{\otimes 2})\big|\big|\big\langle D^2_{X^{2}}(\xi) , \delta_{(j+1)2^{-n/2}}^{\otimes 2} \big\rangle \big|\bigg)\\
&& \times  \big|\langle \delta_{(k+1)2^{-n/2}}, \delta_{(j+1)2^{-n/2}} \rangle \big|\\
&\leq & C2^{-n/3}\sum_{j,k=0}^{\lfloor 2^{\frac{n}{2}} t \rfloor -1}\|I_2^{(2)}(\delta_{(k+1)2^{-n/2}}^{\otimes 2})\|_2\big|\langle \delta_{(k+1)2^{-n/2}}, \delta_{(j+1)2^{-n/2}} \rangle \big|\\
&\leq & C2^{-n/6}t.
\end{eqnarray*}
Thus, $B_{n,16}^{(3,2)}(t) \to 0$ as $n \to \infty$.
\end{itemize}

Finally, we have shown  that 
\[
B_n^{(3,2)}(t) \underset{n \to \infty}{\longrightarrow} -\kappa_3^2E\bigg(e^{i \langle \lambda, W_{\infty}(f,t) \rangle}\xi\times \int_0^t\big(\partial_{122} f(X_s^{1},X_s^{2})\big)^2ds\bigg),
\] 
and (\ref{convergence B_n^{3,2}}) holds true.
\qed

Recall that we are proving (\ref{convergence B_n^3}). Moreover, by (\ref{definitionB_n^3}), $B^{(3)}_n(t) = B_n^{(3,1)}(t) + B_n^{(3,2)}(t)$. So, it remains to prove  the convergence to 0 of $ B_n^{(3,1)}(t)$.

\item \underline{Convergence to 0 of $B_n^{(3,1)}(t)$ as $n \to \infty$.}
\begin{eqnarray*}
 B_n^{(3,1)}(t)&=& -\frac{1}{64}  \sum_{j,k=0}^{\lfloor 2^{\frac{n}{2}} t \rfloor -1} E\bigg(\bigg\langle D_{X^{2}}^2 \bigg( \Delta_{j,n}\partial_{122} f(X^{1},X^{2}) \Delta_{k,n}\partial_{1122}f(X^{1},X^{2})\notag\\
&& \hspace{1cm}\times I_1^{(1)}(\delta_{(k+1)2^{-n/2}}) I_2^{(2)}(\delta_{(k+1)2^{-n/2}}^{\otimes 2})e^{i \langle \lambda, W_n(f,t) \rangle}\xi\bigg),\delta^{\otimes 2}_{(j+1)2^{-n/2}} \bigg \rangle \bigg)\\
&& \hspace{3cm} \times  \bigg\langle \bigg(\frac{\varepsilon_{k2^{-n/2}} + \varepsilon_{(k+1)2^{-n/2} }}{2} \bigg), \delta_{(j+1)2^{-n/2}}\bigg\rangle.\notag\\
\end{eqnarray*}
Observe that 
\begin{eqnarray*}
&& D_{X^{2}}^2 \bigg( \Delta_{j,n}\partial_{122} f(X^{1},X^{2}) \Delta_{k,n}\partial_{1122}f(X^{1},X^{2}) I_1^{(1)}(\delta_{(k+1)2^{-n/2}}) I_2^{(2)}(\delta_{(k+1)2^{-n/2}}^{\otimes 2})\\
&& \hspace{5cm} \times e^{i \langle \lambda, W_n(f,t) \rangle}\xi\bigg)\\
&=& I_1^{(1)}(\delta_{(k+1)2^{-n/2}})  D_{X^{2}}^2 \bigg( \Delta_{j,n}\partial_{122} f(X^{1},X^{2}) \Delta_{k,n}\partial_{1122}f(X^{1},X^{2})I_2^{(2)}(\delta_{(k+1)2^{-n/2}}^{\otimes 2})\notag\\
&& \hspace{5cm}\times e^{i \langle \lambda, W_n(f,t) \rangle}\xi\bigg).
\end{eqnarray*}
As in the case of (\ref{derivative,B_n,1}), the same calculations show that
\begin{eqnarray*}
&& D_{X^{2}}^2 \bigg( \Delta_{j,n}\partial_{122} f(X^{1},X^{2}) \Delta_{k,n}\partial_{1122}f(X^{1},X^{2})I_2^{(2)}(\delta_{(k+1)2^{-n/2}}^{\otimes 2}) e^{i \langle \lambda, W_n(f,t) \rangle}\xi\bigg)\\
&=& \mathscr{D}_1 + \mathscr{D}_2 + \mathscr{D}_3,
\end{eqnarray*}
where
\begin{eqnarray*}
&& \mathscr{D}_1= \Delta_{j,n}\partial_{12222} f(X^{1},X^{2})\Delta_{k,n}\partial_{1122}f(X^{1},X^{2})I_2^{(2)}(\delta_{(k+1)2^{-n/2}}^{\otimes 2}) e^{i \langle \lambda, W_n(f,t) \rangle}\xi\\
&& \times \bigg(\frac{\varepsilon_{j2^{-n/2}} + \varepsilon_{(j+1)2^{-n/2} }}{2} \bigg)^{\otimes 2} 
+ 2 \Delta_{j,n}\partial_{1222} f(X^{1},X^{2}) \Delta_{k,n}\partial_{11222}f(X^{1},X^{2})\\
&& \times I_2^{(2)}(\delta_{(k+1)2^{-n/2}}^{\otimes 2}) e^{i \langle \lambda, W_n(f,t) \rangle}\xi\bigg(\frac{\varepsilon_{j2^{-n/2}} + \varepsilon_{(j+1)2^{-n/2} }}{2} \bigg)\tilde{\otimes}\bigg(\frac{\varepsilon_{k2^{-n/2}} + \varepsilon_{(k+1)2^{-n/2} }}{2} \bigg)\\
&& + \Delta_{j,n}\partial_{122} f(X^{1},X^{2}) \Delta_{k,n}\partial_{112222}f(X^{1},X^{2})I_2^{(2)}(\delta_{(k+1)2^{-n/2}}^{\otimes 2}) e^{i \langle \lambda, W_n(f,t) \rangle}\xi\\
&& \times \bigg(\frac{\varepsilon_{k2^{-n/2}} + \varepsilon_{(k+1)2^{-n/2} }}{2} \bigg)^{\otimes 2},
\end{eqnarray*}
\begin{eqnarray*}
&&\mathscr{D}_2 = 4\Delta_{j,n}\partial_{1222} f(X^{1},X^{2})\Delta_{k,n}\partial_{1122}f(X^{1},X^{2})\\
&& I_1^{(2)}(\delta_{(k+1)2^{-n/2}})e^{i \langle \lambda, W_n(f,t) \rangle}\xi \bigg(\frac{\varepsilon_{j2^{-n/2}} + \varepsilon_{(j+1)2^{-n/2} }}{2} \bigg)\tilde{\otimes}\delta_{(k+1)2^{-n/2}}\\
&& + 2i\Delta_{j,n}\partial_{1222} f(X^{1},X^{2})\Delta_{k,n}\partial_{1122}f(X^{1},X^{2})I_2^{(2)}(\delta_{(k+1)2^{-n/2}}^{\otimes 2})e^{i \langle \lambda, W_n(f,t) \rangle}\xi\\
&& \times \bigg(\frac{\varepsilon_{j2^{-n/2}} + \varepsilon_{(j+1)2^{-n/2} }}{2} \bigg)\tilde{\otimes}D_{X^{2}}(\langle \lambda, W_n(f,t) \rangle) + 2\Delta_{j,n}\partial_{1222} f(X^{1},X^{2})\\
&& \times \Delta_{k,n}\partial_{1122}f(X^{1},X^{2})I_2^{(2)}(\delta_{(k+1)2^{-n/2}}^{\otimes 2})e^{i \langle \lambda, W_n(f,t) \rangle}\bigg(\frac{\varepsilon_{j2^{-n/2}} + \varepsilon_{(j+1)2^{-n/2} }}{2} \bigg)\tilde{\otimes}D_{X^{2}}(\xi)\\
&& + 4\Delta_{j,n}\partial_{122} f(X^{1},X^{2})\Delta_{k,n}\partial_{11222}f(X^{1},X^{2})I_1^{(2)}(\delta_{(k+1)2^{-n/2}}) e^{i \langle \lambda, W_n(f,t) \rangle}\xi\\
&& \times \bigg(\frac{\varepsilon_{k2^{-n/2}} + \varepsilon_{(k+1)2^{-n/2} }}{2} \bigg)\tilde{\otimes}\delta_{(k+1)2^{-n/2}} + 2i \Delta_{j,n}\partial_{122} f(X^{1},X^{2})\Delta_{k,n}\partial_{11222}f(X^{1},X^{2})\\
&& \times I_2^{(2)}(\delta_{(k+1)2^{-n/2}}^{\otimes 2})e^{i \langle \lambda, W_n(f,t) \rangle}\xi\bigg(\frac{\varepsilon_{k2^{-n/2}} + \varepsilon_{(k+1)2^{-n/2} }}{2} \bigg)\tilde{\otimes}D_{X^{2}}(\langle \lambda, W_n(f,t) \rangle)\\
&& + 2\Delta_{j,n}\partial_{122} f(X^{1},X^{2})\Delta_{k,n}\partial_{11222}f(X^{1},X^{2})I_2^{(2)}(\delta_{(k+1)2^{-n/2}}^{\otimes 2})e^{i \langle \lambda, W_n(f,t) \rangle}\\
&& \times \bigg(\frac{\varepsilon_{k2^{-n/2}} + \varepsilon_{(k+1)2^{-n/2} }}{2} \bigg)\tilde{\otimes}D_{X^{2}}(\xi),
\end{eqnarray*}
\begin{eqnarray*}
&& \mathscr{D}_3 = 2\Delta_{j,n}\partial_{122} f(X^{1},X^{2}) \Delta_{k,n}\partial_{1122}f(X^{1},X^{2})\\
&& \times e^{i \langle \lambda, W_n(f,t) \rangle}\xi \delta_{(k+1)2^{-n/2}}^{\otimes 2} + 4i \Delta_{j,n}\partial_{122} f(X^{1},X^{2}) \Delta_{k,n}\partial_{1122}f(X^{1},X^{2})\\
&& \times I_1^{(2)}(\delta_{(k+1)2^{-n/2}})e^{i \langle \lambda, W_n(f,t) \rangle}\xi \delta_{(k+1)2^{-n/2}}\tilde{\otimes}D_{X^{2}}(\langle \lambda, W_n(f,t) \rangle)+ 4 \Delta_{j,n}\partial_{122} f(X^{1},X^{2})\\ && \times \Delta_{k,n}\partial_{1122}f(X^{1},X^{2})I_1^{(2)}(\delta_{(k+1)2^{-n/2}})e^{i \langle \lambda, W_n(f,t) \rangle} \delta_{(k+1)2^{-n/2}}\tilde{\otimes}D_{X^{2}}(\xi)\\
&& +i \Delta_{j,n}\partial_{122} f(X^{1},X^{2}) \Delta_{k,n}\partial_{1122}f(X^{1},X^{2})
 I_2^{(2)}(\delta_{(k+1)2^{-n/2}}^{\otimes 2})e^{i \langle \lambda, W_n(f,t) \rangle}\xi \notag\\
&&  \times D^2_{X^{2}}(\langle \lambda, W_n(f,t) \rangle) - \Delta_{j,n}\partial_{122} f(X^{1},X^{2}) \Delta_{k,n}\partial_{1122}f(X^{1},X^{2})I_2^{(2)}(\delta_{(k+1)2^{-n/2}}^{\otimes 2})\\
&& \times e^{i \langle \lambda, W_n(f,t) \rangle}\xi \big(D_{X^{2}}(\langle \lambda, W_n(f,t) \rangle)\big)^{\otimes 2}+2i \Delta_{j,n}\partial_{122} f(X^{1},X^{2}) \Delta_{k,n}\partial_{1122}f(X^{1},X^{2})\\
&& \times I_2^{(2)}(\delta_{(k+1)2^{-n/2}}^{\otimes 2})e^{i \langle \lambda, W_n(f,t) \rangle}D_{X^{2}}(\langle \lambda, W_n(f,t) \rangle)\tilde{\otimes}D_{X^{2}}(\xi) + \Delta_{j,n}\partial_{122} f(X^{1},X^{2})\\ && \times \Delta_{k,n}\partial_{1122}f(X^{1},X^{2})I_2^{(2)}(\delta_{(k+1)2^{-n/2}}^{\otimes 2})e^{i \langle \lambda, W_n(f,t) \rangle}D^2_{X^{2}}(\xi).
\end{eqnarray*}
Hence, we have

\begin{equation*}
B_n^{(3,1)}(t) = B_n^{(3,1,a)}(t) + B_n^{(3,1,b)}(t) + B_n^{(3,1,c)}(t),
\end{equation*}
where
\begin{eqnarray*}
&& B_n^{(3,1,a)}(t)\\
&=& -\frac{1}{64}\sum_{j,k=0}^{\lfloor 2^{\frac{n}{2}} t \rfloor -1}E\bigg ( \Delta_{j,n}\partial_{12222} f(X^{1},X^{2})\Delta_{k,n}\partial_{1122}f(X^{1},X^{2})I_1^{(1)}(\delta_{(k+1)2^{-n/2}})\\
&& \times I_2^{(2)}(\delta_{(k+1)2^{-n/2}}^{\otimes 2}) e^{i \langle \lambda, W_n(f,t) \rangle}\xi\bigg) \bigg \langle \bigg(\frac{\varepsilon_{j2^{-n/2}} + \varepsilon_{(j+1)2^{-n/2} }}{2} \bigg)^{\otimes 2} , \delta_{(j+1)2^{-n/2}}^{\otimes 2} \bigg \rangle\\
&& \times  \bigg\langle \frac{\varepsilon_{k2^{-n/2}}+\varepsilon_{(k+1)2^{-n/2}}}{2}, \delta_{(j+1)2^{-n/2}} \bigg\rangle\\
&& -\frac{2}{64} \sum_{j,k=0}^{\lfloor 2^{\frac{n}{2}} t \rfloor -1}E\bigg (\Delta_{j,n}\partial_{1222} f(X^{1},X^{2}) \Delta_{k,n}\partial_{11222}f(X^{1},X^{2})I_1^{(1)}(\delta_{(k+1)2^{-n/2}}) \\
&& \times I_2^{(2)}(\delta_{(k+1)2^{-n/2}}^{\otimes 2}) e^{i \langle \lambda, W_n(f,t) \rangle}\xi \bigg)  \bigg \langle \bigg(\frac{\varepsilon_{j2^{-n/2}} + \varepsilon_{(j+1)2^{-n/2} }}{2} \bigg)\tilde{\otimes}\bigg(\frac{\varepsilon_{k2^{-n/2}} + \varepsilon_{(k+1)2^{-n/2} }}{2} \bigg), \\
&& \hspace{3cm} \delta_{(j+1)2^{-n/2}}^{\otimes 2} \bigg \rangle  \bigg\langle \frac{\varepsilon_{k2^{-n/2}}+\varepsilon_{(k+1)2^{-n/2}}}{2}, \delta_{(j+1)2^{-n/2}} \bigg\rangle\\
&&  -\frac{1}{64}\sum_{j,k=0}^{\lfloor 2^{\frac{n}{2}} t \rfloor -1}E\bigg (\Delta_{j,n}\partial_{122} f(X^{1},X^{2}) \Delta_{k,n}\partial_{112222}f(X^{1},X^{2})I_1^{(1)}(\delta_{(k+1)2^{-n/2}})\\
&& \times I_2^{(2)}(\delta_{(k+1)2^{-n/2}}^{\otimes 2}) e^{i \langle \lambda, W_n(f,t) \rangle}\xi\bigg)\bigg\langle \bigg(\frac{\varepsilon_{k2^{-n/2}} + \varepsilon_{(k+1)2^{-n/2} }}{2} \bigg)^{\otimes 2} , \delta_{(j+1)2^{-n/2}}^{\otimes 2} \bigg \rangle\\
&& \times \bigg\langle \frac{\varepsilon_{k2^{-n/2}}+\varepsilon_{(k+1)2^{-n/2}}}{2}, \delta_{(j+1)2^{-n/2}} \bigg\rangle\\
&=& \sum_{i=1}^{3} B_{n,i}^{(3,1)}(t),
\end{eqnarray*}
\begin{eqnarray*}
&& B_n^{(3,1,b)}(t)\\
&=& -\frac{4}{64}\sum_{j,k=0}^{\lfloor 2^{\frac{n}{2}} t \rfloor -1}E\bigg (\Delta_{j,n}\partial_{1222} f(X^{1},X^{2})\Delta_{k,n}\partial_{1122}f(X^{1},X^{2})I_1^{(1)}(\delta_{(k+1)2^{-n/2}})\\
&& \times  I_1^{(2)}(\delta_{(k+1)2^{-n/2}})e^{i \langle \lambda, W_n(f,t) \rangle}\xi  \bigg) \bigg \langle \bigg(\frac{\varepsilon_{j2^{-n/2}} + \varepsilon_{(j+1)2^{-n/2} }}{2} \bigg)\tilde{\otimes}\delta_{(k+1)2^{-n/2}} ,\delta_{(j+1)2^{-n/2}}^{\otimes 2} \bigg \rangle \\
&& \times \bigg\langle \frac{\varepsilon_{k2^{-n/2}}+\varepsilon_{(k+1)2^{-n/2}}}{2}, \delta_{(j+1)2^{-n/2}} \bigg\rangle\\
\end{eqnarray*}
\begin{eqnarray*}
&& - \frac{2i}{64}\sum_{j,k=0}^{\lfloor 2^{\frac{n}{2}} t \rfloor -1}E\bigg (\Delta_{j,n}\partial_{1222} f(X^{1},X^{2})\Delta_{k,n}\partial_{1122}f(X^{1},X^{2})I_1^{(1)}(\delta_{(k+1)2^{-n/2}}) \\
&& \times I_2^{(2)}(\delta_{(k+1)2^{-n/2}}^{\otimes 2})e^{i \langle \lambda, W_n(f,t) \rangle}\xi \\
&& \times \bigg \langle \bigg(\frac{\varepsilon_{j2^{-n/2}} + \varepsilon_{(j+1)2^{-n/2} }}{2} \bigg)\tilde{\otimes}D_{X^{2}}(\langle \lambda, W_n(f,t) \rangle) , \delta_{(j+1)2^{-n/2}}^{\otimes 2} \bigg \rangle\bigg)\\
&& \times  \bigg\langle \frac{\varepsilon_{k2^{-n/2}}+\varepsilon_{(k+1)2^{-n/2}}}{2}, \delta_{(j+1)2^{-n/2}} \bigg\rangle \\
&& - \frac{2}{64}\sum_{j,k=0}^{\lfloor 2^{\frac{n}{2}} t \rfloor -1}E\bigg (\Delta_{j,n}\partial_{1222} f(X^{1},X^{2})\Delta_{k,n}\partial_{1122}f(X^{1},X^{2})I_1^{(1)}(\delta_{(k+1)2^{-n/2}})\\
&& \times  I_2^{(2)}(\delta_{(k+1)2^{-n/2}}^{\otimes 2})e^{i \langle \lambda, W_n(f,t) \rangle} \bigg \langle \bigg(\frac{\varepsilon_{j2^{-n/2}} + \varepsilon_{(j+1)2^{-n/2} }}{2} \bigg)\tilde{\otimes}D_{X^{2}}(\xi) , \delta_{(j+1)2^{-n/2}}^{\otimes 2} \bigg \rangle\bigg)\\
&&\times \bigg\langle \frac{\varepsilon_{k2^{-n/2}}+\varepsilon_{(k+1)2^{-n/2}}}{2}, \delta_{(j+1)2^{-n/2}} \bigg\rangle \\
&& - \frac{4}{64}\sum_{j,k=0}^{\lfloor 2^{\frac{n}{2}} t \rfloor -1}E\bigg ( \Delta_{j,n}\partial_{122} f(X^{1},X^{2})\Delta_{k,n}\partial_{11222}f(X^{1},X^{2})I_1^{(1)}(\delta_{(k+1)2^{-n/2}})\\
&& \times I_1^{(2)}(\delta_{(k+1)2^{-n/2}}) e^{i \langle \lambda, W_n(f,t) \rangle}\xi\bigg)\bigg \langle \bigg(\frac{\varepsilon_{k2^{-n/2}} + \varepsilon_{(k+1)2^{-n/2} }}{2} \bigg)\tilde{\otimes}\delta_{(k+1)2^{-n/2}} ,\\
&& \hspace{2cm}\delta_{(j+1)2^{-n/2}}^{\otimes 2} \bigg \rangle \bigg\langle \frac{\varepsilon_{k2^{-n/2}}+\varepsilon_{(k+1)2^{-n/2}}}{2}, \delta_{(j+1)2^{-n/2}} \bigg\rangle \\
&& -\frac{2i}{64}\sum_{j,k=0}^{\lfloor 2^{\frac{n}{2}} t \rfloor -1}E\bigg (\Delta_{j,n}\partial_{122} f(X^{1},X^{2})\Delta_{k,n}\partial_{11222}f(X^{1},X^{2})I_1^{(1)}(\delta_{(k+1)2^{-n/2}})\\
&& \times I_2^{(2)}(\delta_{(k+1)2^{-n/2}}^{\otimes 2})e^{i \langle \lambda, W_n(f,t) \rangle}\xi\\
&& \times \bigg \langle \bigg(\frac{\varepsilon_{k2^{-n/2}} + \varepsilon_{(k+1)2^{-n/2} }}{2} \bigg)\tilde{\otimes}D_{X^{2}}(\langle \lambda, W_n(f,t) \rangle) , \delta_{(j+1)2^{-n/2}}^{\otimes 2} \bigg \rangle\bigg) \\
&& \times \bigg\langle \frac{\varepsilon_{k2^{-n/2}}+\varepsilon_{(k+1)2^{-n/2}}}{2}, \delta_{(j+1)2^{-n/2}} \bigg\rangle\\
&& -\frac{2}{64}\sum_{j,k=0}^{\lfloor 2^{\frac{n}{2}} t \rfloor -1}E\bigg ( \Delta_{j,n}\partial_{122} f(X^{1},X^{2})\Delta_{k,n}\partial_{11222}f(X^{1},X^{2})I_1^{(1)}(\delta_{(k+1)2^{-n/2}})\\
&& \times I_2^{(2)}(\delta_{(k+1)2^{-n/2}}^{\otimes 2})e^{i \langle \lambda, W_n(f,t) \rangle} \bigg \langle \bigg(\frac{\varepsilon_{k2^{-n/2}} + \varepsilon_{(k+1)2^{-n/2} }}{2} \bigg)\tilde{\otimes}D_{X^{2}}(\xi) , \delta_{(j+1)2^{-n/2}}^{\otimes 2} \bigg \rangle \bigg)\\
&& \times \bigg\langle \frac{\varepsilon_{k2^{-n/2}}+\varepsilon_{(k+1)2^{-n/2}}}{2}, \delta_{(j+1)2^{-n/2}} \bigg\rangle\\
\end{eqnarray*}
\begin{equation*}
= \sum_{i=4}^{9} B_{n,i}^{(3,1)}(t),
\end{equation*}
\begin{eqnarray*}
&& B_n^{(3,1,c)}(t)\\
&=& -\frac{2}{64}\sum_{j,k=0}^{\lfloor 2^{\frac{n}{2}} t \rfloor -1}E\bigg ( \Delta_{j,n}\partial_{122} f(X^{1},X^{2}) \Delta_{k,n}\partial_{1122}f(X^{1},X^{2})I_1^{(1)}(\delta_{(k+1)2^{-n/2}})e^{i \langle \lambda, W_n(f,t) \rangle}\xi \bigg)\\
&& \times \langle \delta_{(k+1)2^{-n/2}}^{\otimes 2} , \delta_{(j+1)2^{-n/2}}^{\otimes 2} \rangle \bigg\langle \frac{\varepsilon_{k2^{-n/2}}+\varepsilon_{(k+1)2^{-n/2}}}{2}, \delta_{(j+1)2^{-n/2}} \bigg\rangle\\
&& - \frac{4i}{64} \sum_{j,k=0}^{\lfloor 2^{\frac{n}{2}} t \rfloor -1}E\bigg( \Delta_{j,n}\partial_{122} f(X^{1},X^{2}) \Delta_{k,n}\partial_{1122}f(X^{1},X^{2})I_1^{(1)}(\delta_{(k+1)2^{-n/2}})\\
&& \times I_1^{(2)}(\delta_{(k+1)2^{-n/2}})e^{i \langle \lambda, W_n(f,t) \rangle}\xi  \bigg \langle \delta_{(k+1)2^{-n/2}}\tilde{\otimes}D_{X^{2}}(\langle \lambda, W_n(f,t) \rangle) , \delta_{(j+1)2^{-n/2}}^{\otimes 2} \bigg\rangle\bigg)\\
&& \times  \bigg\langle \frac{\varepsilon_{k2^{-n/2}}+\varepsilon_{(k+1)2^{-n/2}}}{2}, \delta_{(j+1)2^{-n/2}} \bigg\rangle\\
&& - \frac{4}{64}\sum_{j,k=0}^{\lfloor 2^{\frac{n}{2}} t \rfloor -1}E\bigg(\Delta_{j,n}\partial_{122} f(X^{1},X^{2})\Delta_{k,n}\partial_{1122}f(X^{1},X^{2})I_1^{(1)}(\delta_{(k+1)2^{-n/2}})\\
&& \times  I_1^{(2)}(\delta_{(k+1)2^{-n/2}})e^{i \langle \lambda, W_n(f,t) \rangle}  \bigg \langle \delta_{(k+1)2^{-n/2}}\tilde{\otimes}D_{X^{2}}(\xi) ,\delta_{(j+1)2^{-n/2}}^{\otimes 2} \bigg\rangle \bigg)\\
&& \times \bigg\langle \frac{\varepsilon_{k2^{-n/2}}+\varepsilon_{(k+1)2^{-n/2}}}{2}, \delta_{(j+1)2^{-n/2}} \bigg\rangle\\ 
&& - \frac{i}{64} \sum_{j,k=0}^{\lfloor 2^{\frac{n}{2}} t \rfloor -1}E\bigg( \Delta_{j,n}\partial_{122} f(X^{1},X^{2}) \Delta_{k,n}\partial_{1122}f(X^{1},X^{2})I_1^{(1)}(\delta_{(k+1)2^{-n/2}})\\
 && \times I_2^{(2)}(\delta_{(k+1)2^{-n/2}}^{\otimes 2})e^{i \langle \lambda, W_n(f,t) \rangle}\xi  \bigg \langle D^2_{X^{2}}(\langle \lambda, W_n(f,t) \rangle) , \delta_{(j+1)2^{-n/2}}^{\otimes 2} \bigg\rangle\bigg)\\
 && \times \bigg\langle \frac{\varepsilon_{k2^{-n/2}}+\varepsilon_{(k+1)2^{-n/2}}}{2}, \delta_{(j+1)2^{-n/2}} \bigg\rangle\\
&& + \frac{1}{64}\sum_{j,k=0}^{\lfloor 2^{\frac{n}{2}} t \rfloor -1}E\bigg( \Delta_{j,n}\partial_{122} f(X^{1},X^{2}) \Delta_{k,n}\partial_{1122}f(X^{1},X^{2})I_1^{(1)}(\delta_{(k+1)2^{-n/2}})\\
 && \times I_2^{(2)}(\delta_{(k+1)2^{-n/2}}^{\otimes 2})e^{i \langle \lambda, W_n(f,t) \rangle}\xi \bigg \langle \big(D_{X^{2}}(\langle \lambda, W_n(f,t) \rangle)\big)^{\otimes 2}, \delta_{(j+1)2^{-n/2}}^{\otimes 2} \bigg\rangle\bigg) \\
 && \times \bigg\langle \frac{\varepsilon_{k2^{-n/2}}+\varepsilon_{(k+1)2^{-n/2}}}{2}, \delta_{(j+1)2^{-n/2}} \bigg\rangle\\
 \end{eqnarray*}
\begin{eqnarray*}
&& - \frac{2i}{64}\sum_{j,k=0}^{\lfloor 2^{\frac{n}{2}} t \rfloor -1}E\bigg(\Delta_{j,n}\partial_{122} f(X^{1},X^{2}) \Delta_{k,n}\partial_{1122}f(X^{1},X^{2})I_1^{(1)}(\delta_{(k+1)2^{-n/2}})\\
&&\times I_2^{(2)}(\delta_{(k+1)2^{-n/2}}^{\otimes 2})e^{i \langle \lambda, W_n(f,t) \rangle} \bigg \langle D_{X^{2}}(\langle \lambda, W_n(f,t) \rangle)\tilde{\otimes}D_{X^{2}}(\xi) , \delta_{(j+1)2^{-n/2}}^{\otimes 2} \bigg\rangle\bigg)\\
&& \times \bigg\langle \frac{\varepsilon_{k2^{-n/2}}+\varepsilon_{(k+1)2^{-n/2}}}{2}, \delta_{(j+1)2^{-n/2}} \bigg\rangle\\
&& - \frac{1}{64} \sum_{j,k=0}^{\lfloor 2^{\frac{n}{2}} t \rfloor -1}E\bigg(\Delta_{j,n}\partial_{122} f(X^{1},X^{2})\Delta_{k,n}\partial_{1122}f(X^{1},X^{2})I_1^{(1)}(\delta_{(k+1)2^{-n/2}})\\
&& \times  I_2^{(2)}(\delta_{(k+1)2^{-n/2}}^{\otimes 2})e^{i \langle \lambda, W_n(f,t) \rangle}\bigg \langle D^2_{X^{2}}(\xi) , \delta_{(j+1)2^{-n/2}}^{\otimes 2} \bigg\rangle\bigg) \\
&& \times \bigg\langle \frac{\varepsilon_{k2^{-n/2}}+\varepsilon_{(k+1)2^{-n/2}}}{2}, \delta_{(j+1)2^{-n/2}} \bigg\rangle\\
&& = \sum_{i=10}^{16} B_{n,i}^{(3,1)}(t).
\end{eqnarray*}
We claim that, as $n \to \infty$, $B_{n,i}^{(3,1)}(t) \to 0$ for $i \in \{1, \ldots , 16\}$. The proof of this claim is similar to the proof of the convergence to 0 of $B_{n,i}^{(3,2)}(t)$ for $i \in \{1, \ldots , 16\} \backslash \{10\}$ and is left  to the reader. (The reader could find the detailed proof of this claim in my PhD thesis \cite{my thesis}, Chapter 4: Proof of (4.1.4)).

Finally, we have   that $B_n^{(3,1)}(t) \to 0$ as $n \to \infty$.
\end{enumerate}
 We have proved that 
 \begin{equation*}
 B_n^{(3,2)}(t) \underset{n \to \infty}{\longrightarrow} -\kappa_3^2E\bigg(e^{i \langle \lambda, W_{\infty}(f,t) \rangle}\xi\times \int_0^t\big(\partial_{122} f(X_s^{1},X_s^{2})\big)^2ds\bigg),
\end{equation*}
and
\begin{equation*}
 B_n^{(3,1)}(t) \underset{n \to \infty}{\longrightarrow} 0. 
\end{equation*}
Taking into account that $B_n^{(3)}(t) = B_n^{(3,1)}(t) + B_n^{(3,2)}(t)$ by (\ref{definitionB_n^3}), we deduce that
\[
B_n^{(3)}(t) \underset{n \to \infty}{\longrightarrow} -\kappa_3^2E\bigg(e^{i \langle \lambda, W_{\infty}(f,t) \rangle}\xi\times \int_0^t\big(\partial_{122} f(X_s^{1},X_s^{2})\big)^2ds\bigg).
\]
Consequently, (\ref{convergence B_n^3}) holds true.
\end{proof}

In order to prove (\ref{convergence Bn}), it remains to prove the convergence to 0 of $B_n^{(p)}(t)$, defined in (\ref{decomposition of Bn}),  for $p=1,2,4$.
\subsubsection*{Proof of the convergence to 0 of $B_n^{(1)}(t)$.}
\begin{eqnarray*}
&&B_n^{(1)}(t)\\
&=& -\frac18  \sum_{j=0}^{\lfloor 2^{\frac{n}{2}} t \rfloor -1} E\bigg( \Delta_{j,n}\partial_{122} f(X^{1},X^{2})e^{i \langle \lambda, W_n(f,t) \rangle}\xi \: \big\langle D_{X^{1}}\big(K^{(1)}_n(t)\big)  , \delta_{(j+1)2^{-n/2}} \big\rangle I^{(2)}_2(\delta^{\otimes 2}_{(j+1)2^{-n/2}})\bigg).\\
\end{eqnarray*}
Recall that, by Definition \ref{definition Kni},
\[ 
K^{(1)}_n(t)  = \frac{1}{24} \sum_{l=0}^{\lfloor 2^{\frac{n}{2}} t \rfloor -1} \Delta_{l,n}\partial_{111} f(X^{1},X^{2}) I^{(1)}_3 \big(\delta_{(l+1)2^{-n/2}}^{\otimes 3} \big).
\]
We deduce that
\begin{eqnarray*}
&&B_n^{(1)}(t)\\
&=& -\frac{1}{64}\sum_{j,l=0}^{\lfloor 2^{\frac{n}{2}} t \rfloor -1} E\bigg( \Delta_{j,n}\partial_{122} f(X^{1},X^{2})\Delta_{l,n}\partial_{111} f(X^{1},X^{2}) I^{(1)}_2(\delta^{\otimes 2}_{(l+1)2^{-n/2}})\\
&& \times I^{(2)}_2(\delta^{\otimes 2}_{(j+1)2^{-n/2}})e^{i \langle \lambda, W_n(f,t) \rangle}\xi\bigg)  \big\langle \delta_{(l+1)2^{-n/2}}  , \delta_{(j+1)2^{-n/2}} \big\rangle \\
&& -\frac{1}{8 \times 24}\sum_{j,l=0}^{\lfloor 2^{\frac{n}{2}} t \rfloor -1} E\bigg( \Delta_{j,n}\partial_{122} f(X^{1},X^{2})\Delta_{l,n}\partial_{1111} f(X^{1},X^{2}) I^{(1)}_3(\delta^{\otimes 3}_{(l+1)2^{-n/2}})\\
&& \times I^{(2)}_2(\delta^{\otimes 2}_{(j+1)2^{-n/2}})e^{i \langle \lambda, W_n(f,t) \rangle}\xi\bigg)  \bigg\langle \frac{\varepsilon_{l2^{-n/2}} + \varepsilon_{(l+1)2^{-n/2}}}{2}  , \delta_{(j+1)2^{-n/2}} \bigg\rangle \\
&=& B_n^{(1,1)}(t) + B_n^{(1,2)}(t).
\end{eqnarray*}
Let us prove the convergence to 0 of $B_n^{(1,1)}(t)$ and  $B_n^{(1,2)}(t)$.

\begin{enumerate}

\item \underline{Convergence to 0 of $B_n^{(1,1)}(t)$.} Observe that, thanks to (\ref{duality formula}), we have
\begin{eqnarray}
&&B_n^{(1,1)}(t) \label{Bn(1,1)}\\
&=& -\frac{1}{64}\sum_{j,l=0}^{\lfloor 2^{\frac{n}{2}} t \rfloor -1} E\bigg(\bigg\langle D^2_{X^{2}}\big( \Delta_{j,n}\partial_{122} f(X^{1},X^{2})\Delta_{l,n}\partial_{111} f(X^{1},X^{2})\notag\\
&&  e^{i \langle \lambda, W_n(f,t) \rangle}\xi\big) , \delta^{\otimes 2}_{(j+1)2^{-n/2}} \bigg\rangle  I^{(1)}_2(\delta^{\otimes 2}_{(l+1)2^{-n/2}})\bigg)  \big\langle \delta_{(l+1)2^{-n/2}}  , \delta_{(j+1)2^{-n/2}} \big\rangle.\notag
\end{eqnarray}
We have shown before that
\begin{eqnarray*}
&& D^2_{X^{2}}\big( \Delta_{j,n}\partial_{122} f(X^{1},X^{2})\Delta_{l,n}\partial_{111} f(X^{1},X^{2})e^{i \langle \lambda, W_n(f,t) \rangle}\xi\big)\\
&=& D^2_{X^{2}}\big( \Delta_{j,n}\partial_{122} f(X^{1},X^{2})\Delta_{l,n}\partial_{111} f(X^{1},X^{2})\big)e^{i \langle \lambda, W_n(f,t) \rangle}\xi\\
&& + 2 D_{X^{2}}\big( \Delta_{j,n}\partial_{122} f(X^{1},X^{2})\Delta_{l,n}\partial_{111} f(X^{1},X^{2})\big)\tilde{\otimes}D_{X^{2}}(e^{i \langle \lambda, W_n(f,t) \rangle}\xi)\\
&& + \Delta_{j,n}\partial_{122} f(X^{1},X^{2})\Delta_{l,n}\partial_{111} f(X^{1},X^{2})D^2_{X^{2}}(e^{i \langle \lambda, W_n(f,t) \rangle}\xi).
\end{eqnarray*}
that
\begin{eqnarray*}
&&D_{X^{2}} \big( \Delta_{j,n}\partial_{122} f(X^{1},X^{2}) \Delta_{l,n}\partial_{111}f(X^{1},X^{2})\big)\\
&=& \Delta_{j,n}\partial_{1222} f(X^{1},X^{2})\Delta_{l,n}\partial_{111}f(X^{1},X^{2})\bigg(\frac{\varepsilon_{j2^{-n/2}} + \varepsilon_{(j+1)2^{-n/2} }}{2} \bigg)\notag\\
&& + \Delta_{j,n}\partial_{122} f(X^{1},X^{2})\Delta_{l,n}\partial_{1112}f(X^{1},X^{2})\bigg(\frac{\varepsilon_{l2^{-n/2}} + \varepsilon_{(l+1)2^{-n/2} }}{2} \bigg),\notag
\end{eqnarray*} 
that
\begin{eqnarray*}
&& D_{X^{2}}^2 \big( \Delta_{j,n}\partial_{122} f(X^{1},X^{2}) \Delta_{l,n}\partial_{111}f(X^{1},X^{2})\big)\\
&=&\Delta_{j,n}\partial_{12222} f(X^{1},X^{2})\Delta_{l,n}\partial_{111}f(X^{1},X^{2})\bigg(\frac{\varepsilon_{j2^{-n/2}} + \varepsilon_{(j+1)2^{-n/2} }}{2} \bigg)^{\otimes 2}\notag\\
&& + 2 \Delta_{j,n}\partial_{1222} f(X^{1},X^{2}) \Delta_{l,n}\partial_{1112}f(X^{1},X^{2})\bigg(\frac{\varepsilon_{j2^{-n/2}} + \varepsilon_{(j+1)2^{-n/2} }}{2} \bigg)\tilde{\otimes}\notag\\
&& \hspace{6cm}\bigg(\frac{\varepsilon_{l2^{-n/2}} + \varepsilon_{(l+1)2^{-n/2} }}{2} \bigg)\notag\\
&& + \Delta_{j,n}\partial_{122} f(X^{1},X^{2}) \Delta_{l,n}\partial_{11122}f(X^{1},X^{2})\bigg(\frac{\varepsilon_{l2^{-n/2}} + \varepsilon_{(l+1)2^{-n/2} }}{2} \bigg)^{\otimes 2},\notag
\end{eqnarray*}
that 
\begin{equation*}
D_{X^{2}}(e^{i \langle \lambda, W_n(f,t) \rangle}\xi)= ie^{i \langle \lambda, W_n(f,t) \rangle}\xi D_{X^{2}}\langle \lambda, W_n(f,t) \rangle + e^{i \langle \lambda, W_n(f,t) \rangle} D_{X^{2}}\xi,
\end{equation*}
and that
\begin{eqnarray*}
&& D^2_{X^{2}}(e^{i \langle \lambda, W_n(f,t) \rangle}\xi)\\
&=& -e^{i \langle \lambda, W_n(f,t) \rangle}\xi \big(D_{X^{2}}\langle \lambda, W_n(f,t) \rangle\big)^{\otimes 2} +ie^{i \langle \lambda, W_n(f,t) \rangle}\xi D^2_{X^{2}}\langle \lambda, W_n(f,t) \rangle\\
&& + 2ie^{i \langle \lambda, W_n(f,t) \rangle} D_{X^{2}}\langle \lambda, W_n(f,t) \rangle \tilde{\otimes}D_{X^{2}}\xi + e^{i \langle \lambda, W_n(f,t) \rangle}D^2_{X^{2}}\xi.
\end{eqnarray*}
We deduce that
\begin{eqnarray}
&& D^2_{X^{2}}\big( \Delta_{j,n}\partial_{122} f(X^{1},X^{2})\Delta_{l,n}\partial_{111} f(X^{1},X^{2})e^{i \langle \lambda, W_n(f,t) \rangle}\xi\big) \label{derivative2-Bn(1,1)}\\
&=&\Delta_{j,n}\partial_{12222} f(X^{1},X^{2})\Delta_{l,n}\partial_{111}f(X^{1},X^{2})e^{i \langle \lambda, W_n(f,t) \rangle}\xi\bigg(\frac{\varepsilon_{j2^{-n/2}} + \varepsilon_{(j+1)2^{-n/2} }}{2} \bigg)^{\otimes 2} \notag\\
&& + 2 \Delta_{j,n}\partial_{1222} f(X^{1},X^{2}) \Delta_{l,n}\partial_{1112}f(X^{1},X^{2})e^{i \langle \lambda, W_n(f,t) \rangle}\xi\bigg(\frac{\varepsilon_{j2^{-n/2}} + \varepsilon_{(j+1)2^{-n/2} }}{2} \bigg)\tilde{\otimes}\notag\\
&& \hspace{6cm}\bigg(\frac{\varepsilon_{l2^{-n/2}} + \varepsilon_{(l+1)2^{-n/2} }}{2} \bigg) \notag
\end{eqnarray}
\begin{eqnarray}
&& + \Delta_{j,n}\partial_{122} f(X^{1},X^{2}) \Delta_{l,n}\partial_{11122}f(X^{1},X^{2})e^{i \langle \lambda, W_n(f,t) \rangle}\xi\bigg(\frac{\varepsilon_{l2^{-n/2}} + \varepsilon_{(l+1)2^{-n/2} }}{2} \bigg)^{\otimes 2} \notag\\
&& +2i \Delta_{j,n}\partial_{1222} f(X^{1},X^{2})\Delta_{l,n}\partial_{111}f(X^{1},X^{2})e^{i \langle \lambda, W_n(f,t) \rangle}\xi\bigg(\frac{\varepsilon_{j2^{-n/2}} + \varepsilon_{(j+1)2^{-n/2} }}{2} \bigg)\tilde{\otimes} \notag\\
&& \hspace{5cm} D_{X^{2}}\langle \lambda, W_n(f,t) \rangle \notag\\
&& + 2i \Delta_{j,n}\partial_{122} f(X^{1},X^{2})\Delta_{l,n}\partial_{1112}f(X^{1},X^{2})e^{i \langle \lambda, W_n(f,t) \rangle}\xi\bigg(\frac{\varepsilon_{l2^{-n/2}} + \varepsilon_{(l+1)2^{-n/2} }}{2} \bigg)\tilde{\otimes} \notag\\
&& \hspace{5cm}D_{X^{2}}\langle \lambda, W_n(f,t) \rangle \notag\\
&& + 2 \Delta_{j,n}\partial_{1222} f(X^{1},X^{2})\Delta_{l,n}\partial_{111}f(X^{1},X^{2})e^{i \langle \lambda, W_n(f,t) \rangle}\bigg(\frac{\varepsilon_{j2^{-n/2}} + \varepsilon_{(j+1)2^{-n/2} }}{2} \bigg)\tilde{\otimes}D_{X^{2}}\xi \notag\\
&& + 2 \Delta_{j,n}\partial_{122} f(X^{1},X^{2})\Delta_{l,n}\partial_{1112}f(X^{1},X^{2})e^{i \langle \lambda, W_n(f,t) \rangle}\bigg(\frac{\varepsilon_{l2^{-n/2}} + \varepsilon_{(l+1)2^{-n/2} }}{2} \bigg)\tilde{\otimes}D_{X^{2}}\xi \notag\\
&& - \Delta_{j,n}\partial_{122} f(X^{1},X^{2})\Delta_{l,n}\partial_{111} f(X^{1},X^{2})e^{i \langle \lambda, W_n(f,t) \rangle}\xi \big(D_{X^{2}}\langle \lambda, W_n(f,t) \rangle\big)^{\otimes 2} \notag\\
&& +i \Delta_{j,n}\partial_{122} f(X^{1},X^{2})\Delta_{l,n}\partial_{111} f(X^{1},X^{2})e^{i \langle \lambda, W_n(f,t) \rangle}\xi D^2_{X^{2}}\langle \lambda, W_n(f,t) \rangle \notag\\
&& +2i \Delta_{j,n}\partial_{122} f(X^{1},X^{2})\Delta_{l,n}\partial_{111} f(X^{1},X^{2})e^{i \langle \lambda, W_n(f,t) \rangle} D_{X^{2}}\langle \lambda, W_n(f,t) \rangle \tilde{\otimes}D_{X^{2}}\xi \notag\\
&& +\Delta_{j,n}\partial_{122} f(X^{1},X^{2})\Delta_{l,n}\partial_{111} f(X^{1},X^{2})e^{i \langle \lambda, W_n(f,t) \rangle}D^2_{X^{2}}\xi. \notag
\end{eqnarray}
By plugging (\ref{derivative2-Bn(1,1)}) into (\ref{Bn(1,1)}),  we deduce that
\begin{eqnarray*}
&&B_n^{(1,1)}(t) \\
&=& -\frac{1}{64}\sum_{j,l=0}^{\lfloor 2^{\frac{n}{2}} t \rfloor -1} E\bigg( \Delta_{j,n}\partial_{12222} f(X^{1},X^{2})\Delta_{l,n}\partial_{111}f(X^{1},X^{2})e^{i \langle \lambda, W_n(f,t) \rangle}\xi I^{(1)}_2(\delta^{\otimes 2}_{(l+1)2^{-n/2}})\bigg) \\
&& \times \bigg\langle \bigg(\frac{\varepsilon_{j2^{-n/2}} + \varepsilon_{(j+1)2^{-n/2} }}{2} \bigg), \delta_{(j+1)2^{-n/2}} \bigg\rangle^2 \big\langle \delta_{(l+1)2^{-n/2}}  , \delta_{(j+1)2^{-n/2}} \big\rangle\\
&& -\frac{1}{32}\sum_{j,l=0}^{\lfloor 2^{\frac{n}{2}} t \rfloor -1} E\bigg( \Delta_{j,n}\partial_{1222} f(X^{1},X^{2}) \Delta_{l,n}\partial_{1112}f(X^{1},X^{2})e^{i \langle \lambda, W_n(f,t) \rangle}\xi I^{(1)}_2(\delta^{\otimes 2}_{(l+1)2^{-n/2}})\bigg) \\
&& \times \bigg\langle \bigg(\frac{\varepsilon_{j2^{-n/2}} + \varepsilon_{(j+1)2^{-n/2} }}{2} \bigg), \delta_{(j+1)2^{-n/2}} \bigg\rangle \bigg\langle \bigg(\frac{\varepsilon_{l2^{-n/2}} + \varepsilon_{(l+1)2^{-n/2} }}{2} \bigg), \delta_{(j+1)2^{-n/2}} \bigg\rangle\\
&& \times \big\langle \delta_{(l+1)2^{-n/2}}  , \delta_{(j+1)2^{-n/2}} \big\rangle\\
&&-\frac{1}{64}\sum_{j,l=0}^{\lfloor 2^{\frac{n}{2}} t \rfloor -1} E\bigg( \Delta_{j,n}\partial_{122} f(X^{1},X^{2})\Delta_{l,n}\partial_{11122}f(X^{1},X^{2})e^{i \langle \lambda, W_n(f,t) \rangle}\xi I^{(1)}_2(\delta^{\otimes 2}_{(l+1)2^{-n/2}})\bigg) \\
&& \times \bigg\langle \bigg(\frac{\varepsilon_{l2^{-n/2}} + \varepsilon_{(l+1)2^{-n/2} }}{2} \bigg), \delta_{(j+1)2^{-n/2}} \bigg\rangle^2 \big\langle \delta_{(l+1)2^{-n/2}}  , \delta_{(j+1)2^{-n/2}} \big\rangle\\
\end{eqnarray*}
\begin{eqnarray*}
&&  -\frac{i}{32}\sum_{j,l=0}^{\lfloor 2^{\frac{n}{2}} t \rfloor -1} E\bigg( \Delta_{j,n}\partial_{1222} f(X^{1},X^{2})\Delta_{l,n}\partial_{111}f(X^{1},X^{2})e^{i \langle \lambda, W_n(f,t) \rangle}\xi I^{(1)}_2(\delta^{\otimes 2}_{(l+1)2^{-n/2}})\\
&& \bigg\langle  D_{X^{2}}\langle \lambda, W_n(f,t) \rangle , \delta_{(j+1)2^{-n/2}} \bigg\rangle\bigg) \bigg\langle \bigg(\frac{\varepsilon_{j2^{-n/2}} + \varepsilon_{(j+1)2^{-n/2} }}{2} \bigg), \delta_{(j+1)2^{-n/2}} \bigg\rangle\\
&& \times \big\langle \delta_{(l+1)2^{-n/2}}  , \delta_{(j+1)2^{-n/2}} \big\rangle\\
&& -\frac{i}{32}\sum_{j,l=0}^{\lfloor 2^{\frac{n}{2}} t \rfloor -1} E\bigg( \Delta_{j,n}\partial_{122} f(X^{1},X^{2})\Delta_{l,n}\partial_{1112}f(X^{1},X^{2})e^{i \langle \lambda, W_n(f,t) \rangle}\xi I^{(1)}_2(\delta^{\otimes 2}_{(l+1)2^{-n/2}})\\
&& \bigg\langle  D_{X^{2}}\langle \lambda, W_n(f,t) \rangle , \delta_{(j+1)2^{-n/2}} \bigg\rangle\bigg) \bigg\langle \bigg(\frac{\varepsilon_{l2^{-n/2}} + \varepsilon_{(l+1)2^{-n/2} }}{2} \bigg), \delta_{(j+1)2^{-n/2}} \bigg\rangle\\
&& \times \big\langle \delta_{(l+1)2^{-n/2}}  , \delta_{(j+1)2^{-n/2}} \big\rangle\\
&& -\frac{1}{32}\sum_{j,l=0}^{\lfloor 2^{\frac{n}{2}} t \rfloor -1} E\bigg( \Delta_{j,n}\partial_{1222} f(X^{1},X^{2})\Delta_{l,n}\partial_{111}f(X^{1},X^{2})e^{i \langle \lambda, W_n(f,t) \rangle} I^{(1)}_2(\delta^{\otimes 2}_{(l+1)2^{-n/2}})\\
&& \bigg\langle  D_{X^{2}}\xi , \delta_{(j+1)2^{-n/2}} \bigg\rangle\bigg) \bigg\langle \bigg(\frac{\varepsilon_{j2^{-n/2}} + \varepsilon_{(j+1)2^{-n/2} }}{2} \bigg), \delta_{(j+1)2^{-n/2}} \bigg\rangle\\
&& \times \big\langle \delta_{(l+1)2^{-n/2}}  , \delta_{(j+1)2^{-n/2}} \big\rangle\\
&& -\frac{1}{32}\sum_{j,l=0}^{\lfloor 2^{\frac{n}{2}} t \rfloor -1} E\bigg( \Delta_{j,n}\partial_{122} f(X^{1},X^{2})\Delta_{l,n}\partial_{1112}f(X^{1},X^{2})e^{i \langle \lambda, W_n(f,t) \rangle} I^{(1)}_2(\delta^{\otimes 2}_{(l+1)2^{-n/2}})\\
&& \bigg\langle  D_{X^{2}}\xi , \delta_{(j+1)2^{-n/2}} \bigg\rangle\bigg) \bigg\langle \bigg(\frac{\varepsilon_{l2^{-n/2}} + \varepsilon_{(l+1)2^{-n/2} }}{2} \bigg), \delta_{(j+1)2^{-n/2}} \bigg\rangle \big\langle \delta_{(l+1)2^{-n/2}}  , \delta_{(j+1)2^{-n/2}} \big\rangle\\
&& + \frac{1}{64}\sum_{j,l=0}^{\lfloor 2^{\frac{n}{2}} t \rfloor -1} E\bigg(\Delta_{j,n}\partial_{122} f(X^{1},X^{2})\Delta_{l,n}\partial_{111} f(X^{1},X^{2})e^{i \langle \lambda, W_n(f,t) \rangle}\xi I^{(1)}_2(\delta^{\otimes 2}_{(l+1)2^{-n/2}})\\
&&\big\langle \big(D_{X^{2}}\langle \lambda, W_n(f,t) \rangle\big)^{\otimes 2},  \delta_{(j+1)2^{-n/2}}^{\otimes 2} \big \rangle \bigg) \big\langle \delta_{(l+1)2^{-n/2}}  , \delta_{(j+1)2^{-n/2}} \big\rangle\\ 
&& -\frac{i}{64}\sum_{j,l=0}^{\lfloor 2^{\frac{n}{2}} t \rfloor -1} E\bigg(\Delta_{j,n}\partial_{122} f(X^{1},X^{2})\Delta_{l,n}\partial_{111} f(X^{1},X^{2})e^{i \langle \lambda, W_n(f,t) \rangle}\xi I^{(1)}_2(\delta^{\otimes 2}_{(l+1)2^{-n/2}})\\
&& \big\langle D^2_{X^{2}}\langle \lambda, W_n(f,t) \rangle ,  \delta_{(j+1)2^{-n/2}}^{\otimes 2} \big \rangle \bigg) \big\langle \delta_{(l+1)2^{-n/2}}  , \delta_{(j+1)2^{-n/2}} \big\rangle\\ 
&& -\frac{i}{32}\sum_{j,l=0}^{\lfloor 2^{\frac{n}{2}} t \rfloor -1} E\bigg(\Delta_{j,n}\partial_{122} f(X^{1},X^{2})\Delta_{l,n}\partial_{111} f(X^{1},X^{2})e^{i \langle \lambda, W_n(f,t) \rangle} I^{(1)}_2(\delta^{\otimes 2}_{(l+1)2^{-n/2}})\\
&& \big\langle  \big(D_{X^{2}}\langle \lambda, W_n(f,t) \rangle\big)\tilde{\otimes} D_{X^{2}} \xi ,\delta_{(j+1)2^{-n/2}}^{\otimes 2} \big \rangle \bigg) \big\langle \delta_{(l+1)2^{-n/2}}  , \delta_{(j+1)2^{-n/2}} \big\rangle\\ 
\end{eqnarray*}
\begin{eqnarray*}
&& - \frac{1}{64}\sum_{j,l=0}^{\lfloor 2^{\frac{n}{2}} t \rfloor -1} E\bigg(\Delta_{j,n}\partial_{122} f(X^{1},X^{2})\Delta_{l,n}\partial_{111} f(X^{1},X^{2})e^{i \langle \lambda, W_n(f,t) \rangle} I^{(1)}_2(\delta^{\otimes 2}_{(l+1)2^{-n/2}})\\
&& \big\langle D^2_{X^{2}}\xi ,\delta_{(j+1)2^{-n/2}}^{\otimes 2} \big \rangle \bigg)  \big\langle \delta_{(l+1)2^{-n/2}}  , \delta_{(j+1)2^{-n/2}} \big\rangle, \\
&=& \sum_{i=1}^{11}B_{n,i}^{(1,1)}(t).
\end{eqnarray*}
Let us prove the convergence to 0 of $B_{n,i}^{(1,1)}(t)$ for all $i = 4,5,6,7$ and 8. For the other values of $i$, the proof of the convergence to 0 of $B_{n,i}^{(1,1)}(t)$ is similar.

\begin{itemize}

\item \underline{Convergence to 0 of $B_{n,4}^{(1,1)}(t)$ and $B_{n,5}^{(1,1)}(t)$.} Since $f \in C_b^{\infty}$, $e^{i \langle \lambda, W_n(f,t) \rangle}$ and $\xi$ are bounded and thanks to (\ref{12}), the Cauchy-Schwarz inequality, (\ref{isometry}), (\ref{lemma5-2}) and (\ref{13}), we have
\begin{eqnarray*}
&& |B_{n,4}^{(1,1)}(t)|\\
&\leq & C 2^{-n/6}\sum_{j,l=0}^{\lfloor 2^{\frac{n}{2}} t \rfloor -1} E\bigg( \big|I^{(1)}_2(\delta^{\otimes 2}_{(l+1)2^{-n/2}})\big|\big|\big\langle  D_{X^{2}}\langle \lambda, W_n(f,t) \rangle , \delta_{(j+1)2^{-n/2}} \big\rangle\big|\bigg)\\
&& \times \big|\big\langle \delta_{(l+1)2^{-n/2}}  , \delta_{(j+1)2^{-n/2}} \big\rangle\big|\\
& \leq & C 2^{-n/6}\sum_{p=1}^4 |\lambda_p| \sum_{j,l=0}^{\lfloor 2^{\frac{n}{2}} t \rfloor -1} E\bigg( \big|I^{(1)}_2(\delta^{\otimes 2}_{(l+1)2^{-n/2}})\big|\big|\big\langle  D_{X^{2}}(K_n^{(p)}(t)) , \delta_{(j+1)2^{-n/2}} \big\rangle\big|\bigg)\\
&& \times \big|\big\langle \delta_{(l+1)2^{-n/2}}  , \delta_{(j+1)2^{-n/2}} \big\rangle\big|\\
&\leq & C 2^{-n/6}\sum_{p=1}^4 |\lambda_p| \sum_{j,l=0}^{\lfloor 2^{\frac{n}{2}} t \rfloor -1}  \|I^{(1)}_2(\delta^{\otimes 2}_{(l+1)2^{-n/2}})\|_2\|\big\langle  D_{X^{2}}(K_n^{(p)}(t)) , \delta_{(j+1)2^{-n/2}} \big\rangle\|_2\\
&& \times \big|\big\langle \delta_{(l+1)2^{-n/2}}  , \delta_{(j+1)2^{-n/2}} \big\rangle\big|\\
&\leq & C(2^{-n/6})^22^{-n/6}(t^2+t+1)^{\frac12}2^{n/3}t = C 2^{-n/6}(t^2+t+1)^{\frac12}t.
\end{eqnarray*}
We can prove similarly that $|B_{n,5}^{(1,1)}(t)| \leq C 2^{-n/6}(t^2+t+1)^{\frac12}t$. We deduce that $B_{n,4}^{(1,1)}(t)$ and $B_{n,5}^{(1,1)}(t)$ converge to 0 as $n \to \infty$.

\item \underline{Convergence to 0 of $B_{n,6}^{(1,1)}(t)$ and $B_{n,7}^{(1,1)}(t)$.} Since $f \in C_b^{\infty}$, $e^{i \langle \lambda, W_n(f,t) \rangle}$ and $\xi$ are bounded and thanks to (\ref{12}), (\ref{lemma5-1}), the Cauchy-Schwarz inequality, (\ref{isometry}) and (\ref{13}), we get
\begin{eqnarray*}
&& |B_{n,6}^{(1,1)}(t)|\\
&\leq & C2^{-n/6}\sum_{j,l=0}^{\lfloor 2^{\frac{n}{2}} t \rfloor -1} E\bigg( \big|I^{(1)}_2(\delta^{\otimes 2}_{(l+1)2^{-n/2}})\big|\big|\big\langle D_{X^{2}}(\xi), \delta_{(j+1)2^{-n/2}} \big\rangle\big| \bigg)\\
&& \times \big|\big\langle \delta_{(l+1)2^{-n/2}}  , \delta_{(j+1)2^{-n/2}} \big\rangle\big|\\
&\leq & C2^{-n/3}\sum_{j,l=0}^{\lfloor 2^{\frac{n}{2}} t \rfloor -1}  \|I^{(1)}_2(\delta^{\otimes 2}_{(l+1)2^{-n/2}})\|_2\big|\big\langle \delta_{(l+1)2^{-n/2}}  , \delta_{(j+1)2^{-n/2}} \big\rangle\big|\\
&\leq & C2^{-n/3}2^{-n/6}2^{n/3}t=C2^{-n/6}t.
\end{eqnarray*}
We can prove similarly that $|B_{n,7}^{(1,1)}(t)| \leq C2^{-n/6}t.$ Thus, we get that $B_{n,6}^{(1,1)}(t)$ and $B_{n,7}^{(1,1)}(t)$ converge to 0 as $n \to \infty$.

\item \underline{Convergence to 0 of $B_{n,8}^{(1,1)}(t)$.} Thanks to (\ref{lemma5-4}) among other things, we deduce that
\begin{eqnarray*}
&&|B_{n,8}^{(1,1)}(t)|\\
&\leq & C \sum_{j,l=0}^{\lfloor 2^{\frac{n}{2}} t \rfloor -1} E\bigg(\big|I^{(1)}_2(\delta^{\otimes 2}_{(l+1)2^{-n/2}})\big| \big\langle D_{X^{2}}\langle \lambda, W_n(f,t) \rangle, \delta_{(j+1)2^{-n/2}} \big \rangle^2 \bigg)\\
&& \times  \big|\big\langle \delta_{(l+1)2^{-n/2}}  , \delta_{(j+1)2^{-n/2}} \big\rangle\big|\\ 
&\leq & C \sum_{p=1}^4 \lambda^2_p\sum_{j,l=0}^{\lfloor 2^{\frac{n}{2}} t \rfloor -1} E\bigg(\big|I^{(1)}_2(\delta^{\otimes 2}_{(l+1)2^{-n/2}})\big| \big\langle D_{X^{2}}(K_n^{(p)}(t)), \delta_{(j+1)2^{-n/2}} \big \rangle^2 \bigg)\\
&& \times  \big|\big\langle \delta_{(l+1)2^{-n/2}}  , \delta_{(j+1)2^{-n/2}} \big\rangle\big|\\
&\leq & C\sum_{p=1}^4 \lambda^2_p \sum_{j,l=0}^{\lfloor 2^{\frac{n}{2}} t \rfloor -1} \|I^{(1)}_2(\delta^{\otimes 2}_{(l+1)2^{-n/2}})\|_2 \|\big\langle D_{X^{2}}(K_n^{(p)}(t)), \delta_{(j+1)2^{-n/2}} \big \rangle^2 \|_2\\
&& \times  \big|\big\langle \delta_{(l+1)2^{-n/2}}  , \delta_{(j+1)2^{-n/2}} \big\rangle\big|\\
&\leq & C2^{-n/6}2^{-n/3}(t^3+t^2+t+1)^{\frac12}2^{n/3}t = C2^{-n/6}(t^3+t^2+t+1)^{\frac12}t.
\end{eqnarray*}
It is now clear that $B_{n,8}^{(1,1)}(t) \to 0$ as $n \to \infty$.

\end{itemize}
Finally we have shown that $B_n^{(1,1)}(t) \to 0$ as $n \to \infty$.

\item \underline{Convergence to 0 of $B_n^{(1,2)}(t)$.} The proof is very similar to the previous one and is left to the reader.
\end{enumerate}

\subsubsection*{Proof of the convergence to 0 of $B_n^{(2)}(t)$ and $B_n^{(4)}(t)$.}
 The motivated reader may check that  there is no additional difficulties to prove the convergence to 0 of $B_n^{(p)}(t)$ for $p \in \{2,4\}$. Indeed, all the arguments and techniques which are needed to this proof, were already introduced and used along the analysis of the asymptotic behaviour of $B_n^{(p)}(t)$ for $p \in \{1,3\}$.

Hence, we may consider that the proof of (\ref{convergence Bn}) is done.

\subsubsection{\underline{Step 5}: Convergence of $C_n(t)$ to 0.}
Since $e^{i \langle \lambda, W_n(f,t) \rangle}$ is bounded and $f \in C_b^{\infty}(\R^2)$, we deduce that
\begin{eqnarray*}
|C_n(t)|& =& \bigg|\frac{i}{8} \sum_{j=0}^{\lfloor 2^{\frac{n}{2}} t \rfloor -1} E\bigg( \Delta_{j,n}\partial_{122} f(X^{1},X^{2})e^{i \langle \lambda, W_n(f,t) \rangle} \: \big\langle D_{X^{1}} \xi , \delta_{(j+1)2^{-n/2}} \big\rangle I^{(2)}_2(\delta^{\otimes 2}_{(j+1)2^{-n/2}})\bigg)\bigg|\\
& \leq & C\sum_{j=0}^{\lfloor 2^{\frac{n}{2}} t \rfloor -1} E\big( \big|\big\langle D_{X^{1}} \xi , \delta_{(j+1)2^{-n/2}} \big\rangle\big| \big| I^{(2)}_2(\delta^{\otimes 2}_{(j+1)2^{-n/2}})\big|\big).\\
\end{eqnarray*}
Observe that 
\[
D_{X^{1}} \xi = \sum_{k=1}^r \frac{\partial \psi}{\partial x_{2k-1}}(X^{1}_{s_1},X^{2}_{s_1}, \ldots , X^{1}_{s_r},X^{2}_{s_r})\varepsilon_{s_k}.
\]
Since $\psi \in C_b^{\infty}(\R^{2r})$, we get
\[
\big| \big\langle D_{X^{1}} \xi , \delta_{(j+1)2^{-n/2}} \big\rangle\big| \leq C \sum_{k=1}^r \big|\big\langle \varepsilon_{s_k} , \delta_{(j+1)2^{-n/2}} \big\rangle\big|.
\]
We deduce that
\begin{eqnarray*}
|C_n(t)|& =& C\sum_{k=1}^r \sum_{j=0}^{\lfloor 2^{\frac{n}{2}} t \rfloor -1}\big|\big\langle \varepsilon_{s_k} , \delta_{(j+1)2^{-n/2}} \big\rangle\big| E\big( \big|I^{(2)}_2(\delta^{\otimes 2}_{(j+1)2^{-n/2}})\big|\big)\\
& \leq & C\sum_{k=1}^r \sum_{j=0}^{\lfloor 2^{\frac{n}{2}} t \rfloor -1}\big|\big\langle \varepsilon_{s_k} , \delta_{(j+1)2^{-n/2}} \big\rangle\big| \|I^{(2)}_2(\delta^{\otimes 2}_{(j+1)2^{-n/2}})\|_2 \\
&\leq & C2^{-n/6}\sum_{k=1}^r \sum_{j=0}^{\lfloor 2^{\frac{n}{2}} t \rfloor -1}\big|\big\langle \varepsilon_{s_k} , \delta_{(j+1)2^{-n/2}} \big\rangle\big| \leq C2^{-n/6}t^{1/3}, 
\end{eqnarray*}
where the third inequality follows from (\ref{isometry}) and the last one by (\ref{lemma2-3}). It is now clear that $C_n(t)$ converges to 0.

Finally, putting together the respective conclusions of Steps 1 to 5 lead to the end of the proof of Theorem \ref{f.d.d-h=1/6}. 
\qed
%\end{proof}

\begin{definition}\label{definition Vn3}
For $f \in C_b^{\infty}(\R^2)$, for all $t \geq 0$, we define $ V^3_n(f,t) $ as follows:
\begin{eqnarray*}
V^3_n(f,t) &:=&  \frac{1}{24} 2^{-\frac{3nH}{2}} \sum_{j=0}^{\lfloor 2^{\frac{n}{2}} t \rfloor -1} \Delta_{j,n}\partial_{111} f(X^{1},X^{2}) \big(X^{1,n}_{j+1} - X^{1,n}_j \big)^3 \\
&& + \frac{1}{24}2^{-\frac{3nH}{2}} \sum_{j=0}^{\lfloor 2^{\frac{n}{2}} t \rfloor -1} \Delta_{j,n}\partial_{222} f(X^{1},X^{2}) \big(X^{2,n}_{j+1} - X^{2,n}_j \big)^3 \notag\\
&& + \frac{1}{8} 2^{-\frac{3nH}{2}} \sum_{j=0}^{\lfloor 2^{\frac{n}{2}} t \rfloor -1} \Delta_{j,n}\partial_{122} f(X^{1},X^{2}) \big(X^{1,n}_{j+1} - X^{1,n}_j \big)\big(X^{2,n}_{j+1} - X^{2,n}_j \big)^2 \notag\\
&& + \frac{1}{8} 2^{-\frac{3nH}{2}} \sum_{j=0}^{\lfloor 2^{\frac{n}{2}} t \rfloor -1} \Delta_{j,n}\partial_{112} f(X^{1},X^{2}) \big(X^{1,n}_{j+1} - X^{1,n}_j \big)^2\big(X^{2,n}_{j+1} - X^{2,n}_j \big), \notag
\end{eqnarray*}
where, for $i \in \{1,2\}$, $X^{i,n}_j := 2^{\frac{nH}{2}}X^{i}_{j2^{-\frac{n}{2}}}$. 
\end{definition}
Since $x^3 = H_3(x) + 3x$, $x^2 = H_2(x) + 1$ and $x=H_1(x)$. We get, for $H=1/6$,
\begin{eqnarray*}
V^3_n(f,t)&=& \frac{1}{24} 2^{-\frac{n}{4}} \sum_{j=0}^{\lfloor 2^{\frac{n}{2}} t \rfloor -1} \Delta_{j,n}\partial_{111} f(X^{1},X^{2}) H_3 \big(X^{1,n}_{j+1} - X^{1,n}_j \big)\\
&& + \frac{1}{24}2^{-\frac{n}{4}} \sum_{j=0}^{\lfloor 2^{\frac{n}{2}} t \rfloor -1} \Delta_{j,n}\partial_{222} f(X^{1},X^{2}) H_3 \big(X^{2,n}_{j+1} - X^{2,n}_j \big)\\
&& + \frac{1}{8} 2^{-\frac{n}{4}} \sum_{j=0}^{\lfloor 2^{\frac{n}{2}} t \rfloor -1} \Delta_{j,n}\partial_{122} f(X^{1},X^{2}) H_1 \big(X^{1,n}_{j+1} - X^{1,n}_j \big)H_2 \big(X^{2,n}_{j+1} - X^{2,n}_j \big)\\
&& + \frac{1}{8} 2^{-\frac{n}{4}} \sum_{j=0}^{\lfloor 2^{\frac{n}{2}} t \rfloor -1} \Delta_{j,n}\partial_{112} f(X^{1},X^{2}) H_2 \big(X^{1,n}_{j+1} - X^{1,n}_j \big) H_1 \big(X^{2,n}_{j+1} - X^{2,n}_j \big)\\
&&+ \frac{1}{8} 2^{-\frac{n}{6}}\bigg( \sum_{j=0}^{\lfloor 2^{\frac{n}{2}} t \rfloor -1}\big(\Delta_{j,n}\partial_{111} f(X^{1},X^{2}) + \Delta_{j,n}\partial_{122} f(X^{1},X^{2})\big) (X^{1}_{j+1} - X^{1}_j )\\
&& \hspace{2cm} + \sum_{j=0}^{\lfloor 2^{\frac{n}{2}} t \rfloor -1}\big(\Delta_{j,n}\partial_{222} f(X^{1},X^{2}) + \Delta_{j,n}\partial_{112} f(X^{1},X^{2})\big) (X^{2}_{j+1} - X^{2}_j ) \bigg)
\end{eqnarray*}
Let us define $P_n(f,t)$ as follows:
\begin{eqnarray*}
&& P_n(f,t) := \frac{1}{8} 2^{-\frac{n}{6}}\bigg( \sum_{j=0}^{\lfloor 2^{\frac{n}{2}} t \rfloor -1}\big(\Delta_{j,n}\partial_{111} f(X^{1},X^{2}) + \Delta_{j,n}\partial_{122} f(X^{1},X^{2})\big) (X^{1}_{j+1} - X^{1}_j )\\
&& \hspace{2cm} + \sum_{j=0}^{\lfloor 2^{\frac{n}{2}} t \rfloor -1}\big(\Delta_{j,n}\partial_{222} f(X^{1},X^{2}) + \Delta_{j,n}\partial_{112} f(X^{1},X^{2})\big) (X^{2}_{j+1} - X^{2}_j ) \bigg).
\end{eqnarray*}
Then, thanks to (\ref{linear-isometry}) and to the Definition \ref{definition Kni}, we deduce that, for $H = 1/6$,
\begin{equation}
V^3_n(f,t) = K_n^{(1)}(f,t) + K_n^{(2)}(f,t) + K_n^{(3)}(f,t) + K_n^{(4)}(f,t) + P_n(f,t). \label{definition Vn3'}
\end{equation}
We have the following corollary of Theorem \ref{f.d.d-h=1/6}.
\begin{corollary}\label{corollary f.d.d.h=1/6} Suppose $H=1/6$. Fix $t \geq 0$. Then
\begin{equation}
\big ( X^{1}, X^{2}, V^3_n(f,t) \big) \overset{f.d.d.}{\longrightarrow} \big( X^{1}, X^{2}, \int_0^t D^3f(X_s)d^3X_s  \big),\label{equation corollary h=1/6}
\end{equation}
where $\int_0^t D^3f(X_s)d^3X_s$ is short-hand for
\begin{eqnarray*}
 \int_0^t D^3f(X_s)d^3X_s &=& \kappa_1\int_0^t\frac{\partial^3 f}{\partial x^3}\big(X^{1}_s,X^{2}_s\big)dB^{1}_s + \kappa_2\int_0^t\frac{\partial^3 f}{\partial y^3}\big(X^{1}_s,X^{2}_s\big)dB^{2}_s \\
 && + \kappa_3\int_0^t\frac{\partial^3 f}{\partial x^2 \partial y}\big(X^{1}_s,X^{2}_s\big)dB^{3}_s + \kappa_4\int_0^t\frac{\partial^3 f}{\partial x \partial y^2}\big(X^{1}_s,X^{2}_s\big)dB^{4}_s \notag
\end{eqnarray*}
with $B= (B^{1}, \ldots , B^{4})$  a 4-dimensional Brownian motion independent of $X$, $\kappa_1^2 = \kappa_2^2 = \frac{1}{96}\sum_{r\in \Z}\rho^3(r)$ and $\kappa_3^2 = \kappa_4^2 = \frac{1}{32}\sum_{r\in \Z}\rho^3(r)$ with $\rho$ defined in (\ref{rho}). Otherwise stated, (\ref{equation corollary h=1/6}) means that 
$V_n^{(3)}(f,t)$ converges stably in law  to the random variable  $\int_0^t D^3f(X_s)d^3X_s$.
\end{corollary}
\begin{proof}{}
Thanks to (\ref{definition Vn3'}), if we prove that
\begin{equation}
P_n(f,t) \overset{P}{\longrightarrow} 0 \:\:\text{\:as\:} \:\: n \to \infty, \label{convergence Pn}
\end{equation}
then we can deduce (\ref{equation corollary h=1/6}) immediately from Theorem \ref{f.d.d-h=1/6}. So, let us prove (\ref{convergence Pn}).\\
We define $g := \partial_{11}f + \partial_{22}f$. Thanks to Lemma \ref{taylor-expansion}, we have 
\begin{eqnarray*}
 && g(X^{1}_{(j+1)2^{-n/2}}, X^{2}_{(j+1)2^{-n/2}})- g(X^{1}_{j2^{-n/2}},X^{2}_{j2^{-n/2}} )\\
 &=& \Delta_{j,n} \partial_1 g\big(X^{1}, X^{2}\big)\big(X^{1}_{(j+1)2^{-n/2}}-X^{1}_{j2^{-n/2}}\big) +\Delta_{j,n}\partial_2 g \big(X^{1}, X^{2}\big)\big(X^{2}_{(j+1)2^{-n/2}}-X^{2}_{j2^{-n/2}}\big) \\
&&  + \sum_{i=2}^6\sum_{\alpha_1 +\alpha_2 = 2i-1}C(\alpha_1,\alpha_2)\: \Delta_{j,n}\partial_{1 \ldots 1 2 \ldots 2}^{\alpha_1,\alpha_2} g(X^{1}, X^{2})\big(X^{1}_{(j+1)2^{-n/2}}-X^{1}_{j2^{-n/2}}\big)^{\alpha_1}\\
  && \hspace{1cm}\times \big(X^{2}_{(j+1)2^{-n/2}}-X^{2}_{j2^{-n/2}}\big)^{\alpha_2}+ R_{13}\big( \big(X^{1}_{(j+1)2^{-n/2}},X^{2}_{(j+1)2^{-n/2}})\big), \big(X^{1}_{j2^{-n/2}},X^{2}_{j2^{-n/2}})\big) \big).
 \end{eqnarray*}

Since $f\in C^\infty$, we have in particular that $\partial_1g = \partial_{111}f + \partial_{122}f$ and $\partial_2g = \partial_{112}f + \partial_{222}f$. So, by combining this fact with the definition of $V_n^{\alpha_1, \alpha_2}(\cdot,t)$ given in (\ref{pq-varition}) and a telescoping argument, we get
\begin{eqnarray*}
 &&g(X^{1}_{\lfloor 2^{\frac{n}{2}} t \rfloor 2^{-n/2}}, X^{2}_{\lfloor 2^{\frac{n}{2}} t \rfloor 2^{-n/2}}) -g(0,0)\\
  &=& 82^{n/6} P_n(f,t) + \sum_{i=2}^6\sum_{\alpha_1 +\alpha_2 = 2i-1}C(\alpha_1,\alpha_2)\: V_n^{\alpha_1, \alpha_2}(\partial_{1 \ldots 1 2 \ldots 2}^{\alpha_1,\alpha_2} g,t)\notag\\
 && + \sum_{j=0}^{\lfloor 2^{\frac{n}{2}} t \rfloor -1} R_{13}\big( \big(X^{1}_{(j+1)2^{-n/2}},X^{2}_{(j+1)2^{-n/2}})\big), \big(X^{1}_{j2^{-n/2}},X^{2}_{j2^{-n/2}})\big) \big).
 \end{eqnarray*}
This way, we deduce that 
\begin{eqnarray*}
 P_n(f,t) &=& \frac18 2^{-n/6} \big(g(X^{1}_{\lfloor 2^{\frac{n}{2}} t \rfloor 2^{-n/2}}, X^{2}_{\lfloor 2^{\frac{n}{2}} t \rfloor 2^{-n/2}}) -g(0,0)\big) \label{first equality Pn}\\
 && - \frac18 2^{-n/6}\sum_{i=2}^6\sum_{\alpha_1 +\alpha_2 = 2i-1}C(\alpha_1,\alpha_2)\: V_n^{\alpha_1, \alpha_2}(\partial_{1 \ldots 1 2 \ldots 2}^{\alpha_1,\alpha_2} g,t)\notag\\
 && - \frac18 2^{-n/6}\sum_{j=0}^{\lfloor 2^{\frac{n}{2}} t \rfloor -1} R_{13}\big( \big(X^{1}_{(j+1)2^{-n/2}},X^{2}_{(j+1)2^{-n/2}})\big), \big(X^{1}_{j2^{-n/2}},X^{2}_{j2^{-n/2}})\big) \big)\notag\\
 &=& - \frac18 2^{-n/6}\sum_{\alpha_1 +\alpha_2 = 3}C(\alpha_1,\alpha_2)\: V_n^{\alpha_1, \alpha_2}(\partial_{1 \ldots 1 2 \ldots 2}^{\alpha_1,\alpha_2} g,t)\notag\\
 && + r_{n,1}(t),\notag
\end{eqnarray*}
with obvious notation at the last equality.
Thanks to Proposition \ref{proposition, H >1/6}, (\ref{H >1/6, reste}) and since, by continuity of $g(X^1,X^2)$, we have a.s. $g(X^{1}_{\lfloor 2^{\frac{n}{2}} t \rfloor 2^{-n/2}}, X^{2}_{\lfloor 2^{\frac{n}{2}} t \rfloor 2^{-n/2}}) \to g(X_t^1,X_t^2)$, we deduce that 
\begin{equation}
r_{n,1}(t) \overset{P}{\to} 0 \:\:\text{\: as \:} \:\: n \to \infty. \label{convergence rn1}
\end{equation}
So, it remains to prove that
\begin{equation}
2^{-n/6}\sum_{\alpha_1 +\alpha_2 = 3}C(\alpha_1,\alpha_2)\: V_n^{\alpha_1, \alpha_2}(\partial_{1 \ldots 1 2 \ldots 2}^{\alpha_1,\alpha_2} g,t) \overset{P}{\longrightarrow} 0. \label{second convergence Pn}
\end{equation}
By Lemma \ref{taylor-expansion}, we have $C(3,0)=C(0,3)=\frac{1}{24}$ and $C(2,1)=C(1,2)=\frac18$. As a result,
\begin{equation}
\sum_{\alpha_1 +\alpha_2 = 3}C(\alpha_1,\alpha_2)\: V_n^{\alpha_1, \alpha_2}(\partial_{1 \ldots 1 2 \ldots 2}^{\alpha_1,\alpha_2} g,t) = V_n^3(g,t), \label{transforming Vn3}
\end{equation}
with $V_n^3(g,t)$  given in Definition \ref{definition Vn3}. Thanks to (\ref{definition Vn3'}), we have 
\[
 2^{-n/6}V^3_n(g,t) = 2^{-n/6}\big(K_n^{(1)}(g,t) + K_n^{(2)}(g,t) + K_n^{(3)}(g,t) + K_n^{(4)}(g,t)\big) + 2^{-n/6}P_n(g,t).
 \]
 By Theorem \ref{f.d.d-h=1/6}, we have that $2^{-n/6}\big(K_n^{(1)}(g,t) + K_n^{(2)}(g,t) + K_n^{(3)}(g,t) + K_n^{(4)}(g,t)\big) \overset{P}{\longrightarrow} 0 $. So, in order to prove (\ref{second convergence Pn}), we have to show that, as $n \to \infty$
  \begin{equation}
  2^{-n/6}P_n(g,t)\overset{P}{\longrightarrow} 0. \label{third convergence Pn}
  \end{equation}
Set $h:= \partial_{11}g + \partial_{22}g$. Thanks to Lemma \ref{taylor-expansion}, we have 
\begin{eqnarray*}
 && h(X^{1}_{(j+1)2^{-n/2}}, X^{2}_{(j+1)2^{-n/2}})- h(X^{1}_{j2^{-n/2}},X^{2}_{j2^{-n/2}} )\\
 &=& \Delta_{j,n} \partial_1 h\big(X^{1}, X^{2}\big)\big(X^{1}_{(j+1)2^{-n/2}}-X^{1}_{j2^{-n/2}}\big) +\Delta_{j,n}\partial_2 h \big(X^{1}, X^{2}\big)\big(X^{2}_{(j+1)2^{-n/2}}-X^{2}_{j2^{-n/2}}\big) \\
&&  + \sum_{i=2}^6\sum_{\alpha_1 +\alpha_2 = 2i-1}C(\alpha_1,\alpha_2)\: \Delta_{j,n}\partial_{1 \ldots 1 2 \ldots 2}^{\alpha_1,\alpha_2} h(X^{1}, X^{2})\big(X^{1}_{(j+1)2^{-n/2}}-X^{1}_{j2^{-n/2}}\big)^{\alpha_1}\\
  && \hspace{1cm}\times \big(X^{2}_{(j+1)2^{-n/2}}-X^{2}_{j2^{-n/2}}\big)^{\alpha_2}+ R_{13}\big( \big(X^{1}_{(j+1)2^{-n/2}},X^{2}_{(j+1)2^{-n/2}})\big), \big(X^{1}_{j2^{-n/2}},X^{2}_{j2^{-n/2}})\big) \big).
 \end{eqnarray*}
 Observe that $\partial_1h = \partial_{111}g + \partial_{122}g$ and $\partial_2h = \partial_{112}g + \partial_{222}g$. By the same  arguments that has been used previously, we get
\begin{eqnarray*}
 &&h(X^{1}_{\lfloor 2^{\frac{n}{2}} t \rfloor 2^{-n/2}}, X^{2}_{\lfloor 2^{\frac{n}{2}} t \rfloor 2^{-n/2}}) -h(0,0)\\
  &=& 82^{n/6} P_n(g,t) + \sum_{i=2}^6\sum_{\alpha_1 +\alpha_2 = 2i-1}C(\alpha_1,\alpha_2)\: V_n^{\alpha_1, \alpha_2}(\partial_{1 \ldots 1 2 \ldots 2}^{\alpha_1,\alpha_2} h,t)\notag\\
 && + \sum_{j=0}^{\lfloor 2^{\frac{n}{2}} t \rfloor -1} R_{13}\big( \big(X^{1}_{(j+1)2^{-n/2}},X^{2}_{(j+1)2^{-n/2}})\big), \big(X^{1}_{j2^{-n/2}},X^{2}_{j2^{-n/2}})\big) \big).
 \end{eqnarray*}
 Hence, we have
 \begin{eqnarray*}
 P_n(g,t) &=& \frac18 2^{-n/6} \big(h(X^{1}_{\lfloor 2^{\frac{n}{2}} t \rfloor 2^{-n/2}}, X^{2}_{\lfloor 2^{\frac{n}{2}} t \rfloor 2^{-n/2}}) -h(0,0)\big) \\
 && - \frac18 2^{-n/6}\sum_{i=2}^6\sum_{\alpha_1 +\alpha_2 = 2i-1}C(\alpha_1,\alpha_2)\: V_n^{\alpha_1, \alpha_2}(\partial_{1 \ldots 1 2 \ldots 2}^{\alpha_1,\alpha_2} h,t)\notag\\
 && - \frac18 2^{-n/6}\sum_{j=0}^{\lfloor 2^{\frac{n}{2}} t \rfloor -1} R_{13}\big( \big(X^{1}_{(j+1)2^{-n/2}},X^{2}_{(j+1)2^{-n/2}})\big), \big(X^{1}_{j2^{-n/2}},X^{2}_{j2^{-n/2}})\big) \big)\notag\\
 &=& - \frac18 2^{-n/6}\sum_{\alpha_1 +\alpha_2 = 3}C(\alpha_1,\alpha_2)\: V_n^{\alpha_1, \alpha_2}(\partial_{1 \ldots 1 2 \ldots 2}^{\alpha_1,\alpha_2} h,t)\notag\\
 && + r_{n,2}(t),\notag
\end{eqnarray*}
with obvious notation at the last equality. Thus, we finally have 
\begin{eqnarray*}
 2^{-n/6}P_n(g,t)&=& - \frac18 2^{-n/3}\sum_{\alpha_1 +\alpha_2 = 3}C(\alpha_1,\alpha_2)\: V_n^{\alpha_1, \alpha_2}(\partial_{1 \ldots 1 2 \ldots 2}^{\alpha_1,\alpha_2} h,t)\notag\\
 && + 2^{-n/6}r_{n,2}(t).\notag
\end{eqnarray*}
By the same arguments that has been used to prove (\ref{convergence rn1}), we deduce that $r_{n,2}(t) \overset{P}{\to} 0 $ as $n \to \infty$. Hence, to prove (\ref{third convergence Pn}) it remains to show that, as $n \to \infty$
\begin{equation}
2^{-n/3}\sum_{\alpha_1 +\alpha_2 = 3}C(\alpha_1,\alpha_2)\: V_n^{\alpha_1, \alpha_2}(\partial_{1 \ldots 1 2 \ldots 2}^{\alpha_1,\alpha_2} h,t) \overset{P}{\longrightarrow} 0. \label{final convergence Pn}
\end{equation}
In fact, since $h \in C_b^{\infty}$ and by the definition of $V_n^{\alpha_1, \alpha_2}(\partial_{1 \ldots 1 2 \ldots 2}^{\alpha_1,\alpha_2} h,t)$ given in (\ref{pq-varition}), we deduce that
\begin{eqnarray}
&& 2^{-n/3}\big|\sum_{\alpha_1 +\alpha_2 = 3}C(\alpha_1,\alpha_2)\: V_n^{\alpha_1, \alpha_2}(\partial_{1 \ldots 1 2 \ldots 2}^{\alpha_1,\alpha_2} h,t)\big|\notag\\
&\leq & C  2^{-n/3} \sum_{j=0}^{\lfloor 2^{\frac{n}{2}} t \rfloor -1} \big|X^{1}_{j+1} - X^{1}_j \big|^3 + C2^{-n/3} \sum_{j=0}^{\lfloor 2^{\frac{n}{2}} t \rfloor -1}  \big|X^{2}_{j+1} - X^{2}_j \big|^3\label{last inequality corollary}\\
&& + C2^{-n/3} \sum_{j=0}^{\lfloor 2^{\frac{n}{2}} t \rfloor -1}  \big|X^{1}_{j+1} - X^{1}_j \big| \big|X^{2}_{j+1} - X^{2}_j \big|^2 + C 2^{-n/3} \sum_{j=0}^{\lfloor 2^{\frac{n}{2}} t \rfloor -1}  \big|X^{1}_{j+1} - X^{1}_j \big|^2  \big|X^{2}_{j+1} - X^{2}_j \big|.\notag
\end{eqnarray}
Recall the following notation: for $i \in \{1,2\}$, $X^{i,n}_j := 2^{\frac{n}{12}}X^{i}_{j2^{-\frac{n}{2}}}$. We deduce that, for all $p \in \N^*$, 
\[
E[|X^{i}_{j+1} - X^{i}_j |^p]= 2^{-\frac{np}{12}}E[|X^{i,n}_{j+1} - X^{i,n}_j |^p]= 2^{-\frac{np}{12}}E[|G|^p],
\]
 where $G\sim N(0,1)$. Thanks to this identity , to the independence of $X^1, X^2$ and to (\ref{last inequality corollary}),  we deduce that
 \begin{eqnarray*}
&& 2^{-n/3}E\bigg(\big|\sum_{\alpha_1 +\alpha_2 = 3}C(\alpha_1,\alpha_2)\: V_n^{\alpha_1, \alpha_2}(\partial_{1 \ldots 1 2 \ldots 2}^{\alpha_1,\alpha_2} h,t)\big|\bigg)\notag\\
&\leq & C  2^{-n/3} 2^{-n/4} E[|G|^3]\lfloor 2^{\frac{n}{2}} t \rfloor + C2^{-n/3} 2^{-n/4} E[|G|]E[|G|^2]\lfloor 2^{\frac{n}{2}} t \rfloor\\
& \leq & C2^{-n/12}t \underset{n \to \infty}{\longrightarrow} 0.
\end{eqnarray*}
Convergence (\ref{final convergence Pn}) follows immediately. Consequently, we have proved (\ref{third convergence Pn}), (\ref{second convergence Pn}) and (\ref{convergence Pn}). It finishes the proof of Corollary \ref{corollary f.d.d.h=1/6}.
\end{proof}

\underline{We are now ready to prove (\ref{first-main-2}).}
Thanks to Lemma \ref{taylor-expansion}, we have 
\begin{eqnarray*}
 && f(X^{1}_{(j+1)2^{-n/2}}, X^{2}_{(j+1)2^{-n/2}})-f(X^{1}_{j2^{-n/2}},X^{2}_{j2^{-n/2}} )\\
 &=& \Delta_{j,n}\frac{\partial f}{\partial x}\big(X^{1}, X^{2}\big)\big(X^{1}_{(j+1)2^{-n/2}}-X^{1}_{j2^{-n/2}}\big) +\Delta_{j,n}\frac{\partial f}{\partial y}\big(X^{1}, X^{2}\big)\big(X^{2}_{(j+1)2^{-n/2}}-X^{2}_{j2^{-n/2}}\big) \\
&&  + \sum_{i=2}^6\sum_{\alpha_1 +\alpha_2 = 2i-1}C(\alpha_1,\alpha_2)\: \Delta_{j,n}\partial_{1 \ldots 1 2 \ldots 2}^{\alpha_1,\alpha_2} f(X^{1}, X^{2})\big(X^{1}_{(j+1)2^{-n/2}}-X^{1}_{j2^{-n/2}}\big)^{\alpha_1}\\
  && \hspace{1cm}\times \big(X^{2}_{(j+1)2^{-n/2}}-X^{2}_{j2^{-n/2}}\big)^{\alpha_2}+ R_{13}\big( \big(X^{1}_{(j+1)2^{-n/2}},X^{2}_{(j+1)2^{-n/2}})\big), \big(X^{1}_{j2^{-n/2}},X^{2}_{j2^{-n/2}})\big) \big).
 \end{eqnarray*}
 Then, by Definition \ref{ivan-definition} and (\ref{pq-varition}),  we can  write
 \begin{eqnarray*}
 &&f(X^{1}_{\lfloor 2^{\frac{n}{2}} t \rfloor 2^{-n/2}}, X^{2}_{\lfloor 2^{\frac{n}{2}} t \rfloor 2^{-n/2}}) -f(0,0)\\
  &=& O_n(f,t) + \sum_{i=2}^6\sum_{\alpha_1 +\alpha_2 = 2i-1}C(\alpha_1,\alpha_2)\: V_n^{\alpha_1, \alpha_2}(\partial_{1 \ldots 1 2 \ldots 2}^{\alpha_1,\alpha_2} f,t)\notag\\
 && + \sum_{j=0}^{\lfloor 2^{\frac{n}{2}} t \rfloor -1} R_{13}\big( \big(X^{1}_{(j+1)2^{-n/2}},X^{2}_{(j+1)2^{-n/2}})\big), \big(X^{1}_{j2^{-n/2}},X^{2}_{j2^{-n/2}})\big) \big).\notag
 \end{eqnarray*}
 By the same arguments that has been used to show (\ref{transforming Vn3}), we get
 \[
 \sum_{\alpha_1 +\alpha_2 = 3}C(\alpha_1,\alpha_2)\: V_n^{\alpha_1, \alpha_2}(\partial_{1 \ldots 1 2 \ldots 2}^{\alpha_1,\alpha_2} f,t) = V_n^3(f,t).
 \]
Combining this fact with our Taylor's expansion, we deduce that
\begin{eqnarray}
&& O_n(f,t) \label{H=1/6 taylor}\\ 
 &=& f(X^{1}_{\lfloor 2^{\frac{n}{2}} t \rfloor 2^{-n/2}}, X^{2}_{\lfloor 2^{\frac{n}{2}} t \rfloor 2^{-n/2}}) -f(0,0) - V_n^3(f,t) \notag\\
 && - \sum_{i=3}^6\sum_{\alpha_1 +\alpha_2 = 2i-1}C(\alpha_1,\alpha_2)\: V_n^{\alpha_1, \alpha_2}(\partial_{1 \ldots 1 2 \ldots 2}^{\alpha_1,\alpha_2} f,t)\notag\\
 && - \sum_{j=0}^{\lfloor 2^{\frac{n}{2}} t \rfloor -1} R_{13}\big( \big(X^{1}_{(j+1)2^{-n/2}},X^{2}_{(j+1)2^{-n/2}})\big), \big(X^{1}_{j2^{-n/2}},X^{2}_{j2^{-n/2}})\big) \big).\notag
 \end{eqnarray} 
Thanks to Proposition \ref{proposition, H >1/6}, we have
\begin{equation}
\sum_{i=3}^6\sum_{\alpha_1 +\alpha_2 = 2i-1}C(\alpha_1,\alpha_2)\: V_n^{\alpha_1, \alpha_2}(\partial_{1 \ldots 1 2 \ldots 2}^{\alpha_1,\alpha_2} f,t) \overset{P}{\longrightarrow} 0. \label{final rest1}
\end{equation}
On the other hand, by (\ref{H >1/6, reste}) we have
\begin{equation}
 \sum_{j=0}^{\lfloor 2^{\frac{n}{2}} t \rfloor -1} R_{13}\big( \big(X^{1}_{(j+1)2^{-n/2}},X^{2}_{(j+1)2^{-n/2}})\big), \big(X^{1}_{j2^{-n/2}},X^{2}_{j2^{-n/2}})\big) \big) \overset{P}{\longrightarrow}0 \label{final rest2}
\end{equation} 
 Observe also that, by the almost sure continuity of $f\big(X^{1},X^{2}\big)$, one has, almost surely and
as $n\to\infty$,
\begin{equation}
 f(X^{1}_{\lfloor 2^{\frac{n}{2}} t \rfloor 2^{-n/2}}, X^{2}_{\lfloor 2^{\frac{n}{2}} t \rfloor 2^{-n/2}}) -f(0,0)
\to f\big(X^{1}_t, X^{2}_t \big)-f(0,0).\label{H =1/6, pathcontinuity}
\end{equation}
 Finally, the desired conclusion (\ref{first-main-2}) follows from  (\ref{final rest1}), (\ref{final rest2}), (\ref{H =1/6, pathcontinuity}) and the conclusion of Corollary \ref{corollary f.d.d.h=1/6}, plugged into (\ref{H=1/6 taylor}).

\subsection{Proof of (\ref{first-main-3})} Using $b^3 - a^3 = 3 \big(\frac{a+b}{2}\big)^2(b-a) + \frac14 (b-a)^3$, one can write
\begin{eqnarray*}
 O_n(x\mapsto x^3,t) -  (X_t^1)^3 &=& - V^3_n(x\mapsto x^3,t) + \sum_{j=0}^{\lfloor 2^{\frac{n}{2}} t \rfloor -1} \big( (X^{1}_{(j+1)2^{-n/2}})^3 - (X^{1}_{j2^{-n/2}})^3 \big) -  (X_t^1)^3,
\end{eqnarray*}
where $O_n(\cdot,t)$ is introduced in Definition \ref{ivan-definition} and $V_n^3(\cdot,t)$ is given in Definition \ref{definition Vn3}. As a result, and since $(X^1_{\lfloor 2^{\frac{n}{2}} t \rfloor2^{-n/2}})^3 \to (X^1_t)^3$ a.s. as $n \to \infty$, one deduces that if $O_n(x\mapsto x^3,t)$ converges stably in law, then $V^3_n(x\mapsto x^3,t)$ must converge as well. But it is known (see for example (1.8) in \cite{NNT}) that $2^{-n(\frac12 - 3H)}V^3_n(x\mapsto x^3,t)$ converges in law to a non degenerate limit. This fact being  in contradiction with the convergence of $V^3_n(x\mapsto x^3,t)$, we deduce that (\ref{first-main-3}) holds.

\section{Proof of Theorem \ref{second-main}}

We divide the proof of Theorem \ref{second-main} in several steps.

\subsection{\underline{Step 1}: A key algebraic lemma}
For each
integer $n\geq 1$, $k\in\Z$ and real number $t\geq 0$, let $U_{j,n}(t)$ (resp.
$D_{j,n}(t)$) denote the number of \textit{upcrossings} (resp.
\textit{downcrossings}) of the interval
$[j2^{-n/2},(j+1)2^{-n/2}]$ within the first $\lfloor 2^n
t\rfloor$ steps of the random walk $\{Y(T_{k,n})\}_{k\geq 0}$, where $(T_{k,n})_{k\geq 0}$ is introduced in (\ref{TN}). That is,
\begin{eqnarray}
U_{j,n}(t)=\sharp\big\{k=0,\ldots,\lfloor 2^nt\rfloor -1 :&&
\notag
\\ Y(T_{k,n})\!\!\!\!&=&\!\!\!\!j2^{-n/2}\mbox{ and }Y(T_{k+1,n})=(j+1)2^{-n/2}
\big\}; \notag\\
D_{j,n}(t)=\sharp\big\{k=0,\ldots,\lfloor 2^nt\rfloor -1:&&
\notag
\\ Y(T_{k,n})\!\!\!\!&=&\!\!\!\!(j+1)2^{-n/2}\mbox{ and }Y(T_{k+1,n})=j2^{-n/2}
\big\}.\notag
\end{eqnarray}
\begin{definition} For $f \in C_b^\infty$ and $t\geq 0$, for all $p,q \in \N$ such that $p+q$ is odd, we define $\tilde{V}_n^{p,q}(f,t)$ as follows:
\begin{eqnarray}
\tilde{V}_n^{p,q}(f,t)&=&\sum_{k=0}^{\lfloor 2^n t \rfloor -1} f\big(\frac12(Z^1_{T_{k,n}}+ Z^1_{T_{k+1,n}}),\frac12 (Z^2_{T_{k,n}}+ Z^2_{T_{k+1,n}}) \big)(Z^1_{T_{k+1,n}}-Z^1_{T_{k,n}})^p\notag\\
&& \hspace{7cm}\times (Z^2_{T_{k+1,n}}-Z^2_{T_{k,n}})^q. \label{definition tilde{V}pq}
\end{eqnarray}
\end{definition}
While easy, the following lemma taken from \cite[Lemma 2.4]{kh-lewis1} is going to be the key when
studying the asymptotic behavior $\tilde{V}_n^{p,q}(f,t)$. 
Its main feature is to separate $(X^1, X^2)$ from $Y$, thus providing a representation of
$\tilde{V}_n^{p,q}(f,t)$ which is amenable to analysis.

\begin{lemme}\label{lemme-kl}
Fix $f\in C^\infty_b$, $t\geq 0$ and for all $p,q \in \N$ such that $p+ q$ is odd, we have
then
\begin{eqnarray*}
&&\tilde{V}_n^{p,q}(f,t)\\
&=& \sum_{j\in\Z}
f\big(\frac12(X^1_{(j+1)2^{-n/2}}+ X^1_{j2^{-n/2}}),\frac12 (X^2_{(j+1)2^{-n/2}}+ X^2_{j2^{-n/2}}) \big)(X^1_{(j+1)2^{-n/2}} - X^1_{j2^{-n/2}})^p\\
&& \hspace{6cm}\times(X^2_{(j+1)2^{-n/2}} - X^2_{j2^{-n/2}})^q \big(U_{j,n}(t)-D_{j,n}(t)\big).
\end{eqnarray*}
\end{lemme}

\subsection{\underline{Step 2}: Transforming the 2D weighted power variations of odd order}

By \cite[Lemma 2.5]{kh-lewis1}, one has
\[
U_{j,n}(t) - D_{j,n}(t)= \left\{
\begin{array}{lcl}
 1_{\{0\leq j< j^*(n,t)\}} & &\mbox{if $j^*(n,t) > 0$}\\
 0 & &\mbox{if $j^{*} = 0$}\\
  -1_{\{j^*(n,t)\leq j<0\}} & &\mbox{if $j^*(n,t)< 0$}
  \end{array}
\right.,
\]
where $j^*(n,t)=2^{n/2}Y_{T_{\lfloor 2^n t\rfloor,n}}$.
As a consequence, we have  
\begin{enumerate}

\item \underline{If $j^*(n,t) > 0$ :}
\begin{eqnarray*}
&&\tilde{V}_n^{p,q}(f,t)\\
&=& \sum_{j=0}^{j^*(n,t)-1}f\big(\frac12(X^{1,+}_{(j+1)2^{-n/2}}+ X^{1,+}_{j2^{-n/2}}),\frac12 (X^{2,+}_{(j+1)2^{-n/2}}+ X^{2,+}_{j2^{-n/2}}) \big) \\
&& \hspace{1.5cm}\times \big(X^{1,+}_{(j+1)2^{-n/2}}- X^{1,+}_{j2^{-n/2}}\big)^p\big(X^{2,+}_{(j+1)2^{-n/2}}- X^{2,+}_{j2^{-n/2}}\big)^q.
\end{eqnarray*}

\item \underline{If $j^{*} = 0$ :}  $\tilde{V}_n^{p,q}(f,t) = 0$.

\item \underline{If $j^*(n,t) < 0$ :}
 \begin{eqnarray*}
 && \tilde{V}_n^{p,q}(f,t)\\
 &=& \sum_{j=0}^{|j^*(n,t)|-1} f\big(\frac12(X^{1,-}_{(j+1)2^{-n/2}}+ X^{1,-}_{j2^{-n/2}}),\frac12 (X^{2,-}_{(j+1)2^{-n/2}}+ X^{2,-}_{j2^{-n/2}}) \big) \\
&& \hspace{1.5cm}\times \big(X^{1,-}_{(j+1)2^{-n/2}}- X^{1,-}_{j2^{-n/2}}\big)^p\big(X^{2,-}_{(j+1)2^{-n/2}}- X^{2,-}_{j2^{-n/2}}\big)^q, 
 \end{eqnarray*}
 where, for $i \in \{1,2\}$, $X^{i,+}_t := X^i_t$ for $t\geq 0$, $X^{i,-}_{-t} :=X^i_t$ for $t<0$. 
 \end{enumerate}

Let us introduce the following sequence of processes $W_{\pm,n}^{p,q}$:
\begin{eqnarray}
W_{\pm, n}^{p,q}(f,t)&=&
  \sum_{j=0}^{\lfloor 2^{n/2} t \rfloor -1} f\big(\frac12(X^{1,\pm}_{(j+1)2^{-n/2}}+ X^{1,\pm}_{j2^{-n/2}}),\frac12 (X^{2,\pm}_{(j+1)2^{-n/2}}+ X^{2,\pm}_{j2^{-n/2}}) \big)\notag\\
  && \hspace{2cm}\times \big(X^{1,\pm}_{(j+1)2^{-n/2}}- X^{1,\pm}_{j2^{-n/2}}\big)^p  \big(X^{2,\pm}_{(j+1)2^{-n/2}}- X^{2,\pm}_{j2^{-n/2}}\big)^q, \quad t \geq 0 \notag\\
 W_n^{p,q}(f,t)&=&\left \{ \begin{array}{lc}
                      W_{+,n}^{p,q}(f,t) &\text{if $t \geq 0$}\notag\\
                      W_{-,n}^{p,q}(f,-t) &\text{if $t < 0$}
                      \end{array}
                      \right. .
\end{eqnarray}
We then have that  
   \begin{equation}
   \tilde{V}_n^{p,q}(f,t) = W_n^{p,q}(f,Y_{T_{\lfloor 2^n t\rfloor,n}}). \label{transforme tilde{V}pq}
   \end{equation}

\subsection{\underline{Step 3}: Known results for the 2D fractional Brownian motion}
\begin{itemize}

\item If $H > 1/6$, $p+q \geq 3$ and if $H = 1/6$, $p+q \geq 5$, then, thanks to (\ref{norme2 Vnpq}) and (\ref{norme2 Vn0q}), we have for all $t \geq 0$
\begin{eqnarray}
E\big[ \big(W_{\pm,n}^{p,q}(f,t)\big)^2 \big] &\leq & C\sum_{k'=1}^{\frac{q+1}{2}}\bigg(\big( \sum_{a=1}^{2k'-1}2^{-n[H(p+q + 2k'-1-a) - \frac12]}\big)t + 2^{-n[H(p+q+2k'-2) -\frac12]}t^{2H+1}\notag\\
 &&  + 2^{-nH[p+q +2k'-1]}\bigg). \label{norme2 Wnpq} 
\end{eqnarray}

\item If $H= 1/6$, for all $t \in \R$, we define $W_n^{(3)}(f,t)$ as follows:
\begin{equation}
W_n^{(3)}(f,t)= \sum_{p+q=3}C(p,q)W_n^{p,q}(\partial_{1 \ldots 1 2 \ldots 2}^{p,q} f,t),\label{definition Wn3}
\end{equation}
where $C(3,0)=C(0,3)= \frac{1}{24}$ and $C(2,1)=C(1,2)=\frac18$. Then, thanks to Theorem \ref{f.d.d-h=1/6} we have, for $H=1/6$, for any fixed $t \in \R$ and as $n \to \infty$
\begin{eqnarray}
(X^1, X^2, W_n^{(3)}(f,t)) \overset{fdd}{\longrightarrow} (X^1, X^2, \int_0^t D^3f(X_s)d^3X_s) \label{convergence Wn3}
\end{eqnarray}
where $\int_0^t D^3f(X_s)d^3X_s$ is short-hand for
\begin{eqnarray*}
 \int_0^t D^3f(X_s)d^3X_s &=& \kappa_1\int_0^t\frac{\partial^3 f}{\partial x^3}\big(X^{1}_s,X^{2}_s\big)dB^{1}_s + \kappa_2\int_0^t\frac{\partial^3 f}{\partial y^3}\big(X^{1}_s,X^{2}_s\big)dB^{2}_s \\
 && + \kappa_3\int_0^t\frac{\partial^3 f}{\partial x^2 \partial y}\big(X^{1}_s,X^{2}_s\big)dB^{3}_s + \kappa_4\int_0^t\frac{\partial^3 f}{\partial x \partial y^2}\big(X^{1}_s,X^{2}_s\big)dB^{4}_s \notag
\end{eqnarray*}
with $B= (B^{1}, \ldots , B^{4})$  a 4-dimensional two-sided Brownian motion independent of $(X^1, X^2)$, $\kappa_1^2 = \kappa_2^2 = \frac{1}{96}\sum_{r\in \Z}\rho^3(r)$ and $\kappa_3^2 = \kappa_4^2 = \frac{1}{32}\sum_{r\in \Z}\rho^3(r)$ with $\rho$ defined in (\ref{rho}).
\end{itemize}

\subsection{\underline{Step 4}: Moment bounds for $W_n^{(3)}(f,\cdot)$}
Fix $f\in C_b^\infty$ and set $H= 1/6$. We claim the existence of $C >0$ such that, for all real numbers $s <t$ and all $n \in \N$,
\begin{equation}\label{bounds for Wn3}
E[( W_n^{(3)}(f,t) - W_n^{(3)}(f,s))^2] \leq C \max\big(|s|^{1/3}, |t|^{1/3}\big)\big(2^{-n/2}+ |t-s| \big).
\end{equation}
\begin{proof}{} 
By the definition of $W_n^{(3)}(f,t)$ in (\ref{definition Wn3}), we deduce that
\begin{equation*}
E[( W_n^{(3)}(f,t) - W_n^{(3)}(f,s))^2] \leq C \sum_{p+q=3}E[ (W_n^{p,q}(\partial_{1 \ldots 1 2 \ldots 2}^{p,q} f,t)-W_n^{p,q}(\partial_{1 \ldots 1 2 \ldots 2}^{p,q} f,s))^2].
\end{equation*}
So, if we prove that for all $f \in C_b^\infty$ and for all $p,q \in \N$  such that $p+q=3$,
\[
E[ (W_n^{p,q}( f,t)-W_n^{p,q}( f,s))^2] \leq C \max\big(|s|^{1/3}, |t|^{1/3}\big)\big(2^{-n/2}+ |t-s| \big),
\]
then the conclusion  (\ref{bounds for Wn3}) follows immediately. In fact, we will prove the last inequality only for $p=1$ and $q=2$, the proof being similar for the other values\footnote{When $p=3,\, q=0$ or $p=0,\,q=3$ the reader will find a very similar result in Step 4 in \cite{ItoD1}} of $p$ and $q$.

For $p=1,\, q=2$, bearing the notation of Step 2 in mind, we have
\begin{eqnarray*}
W_{\pm,n}^{1,2}(f,t)&=& \frac18 2^{-\frac{n}{4}}\sum_{j=0}^{\lfloor 2^{\frac{n}{2}} t \rfloor -1}\Delta_{j,n} f(X^{1,\pm},X^{2,\pm})\big(2^{\frac{n}{12}}(X^{1,\pm}_{(j+1)2^{-n/2}} - X^{1,\pm}_{j2^{-n/2}})\big)\\
&& \times \big(2^{\frac{n}{12}}(X^{2,\pm}_{(j+1)2^{-n/2}} - X^{2,\pm}_{j2^{-n/2}})\big)^2\\
&=&  \frac18 2^{-\frac{n}{4}}\sum_{j=0}^{\lfloor 2^{\frac{n}{2}} t \rfloor -1}\Delta_{j,n} f(X^{1,\pm},X^{2,\pm})H_{1}\big(2^{\frac{n}{12}}(X^{1,\pm}_{(j+1)2^{-n/2}} - X^{1,\pm}_{j2^{-n/2}})\big)\\
&& + \frac18 2^{-\frac{n}{4}}\sum_{j=0}^{\lfloor 2^{\frac{n}{2}} t \rfloor -1}\Delta_{j,n} f(X^{1,\pm},X^{2,\pm})H_{1}\big(2^{\frac{n}{12}}(X^{1,\pm}_{(j+1)2^{-n/2}} - X^{1,\pm}_{j2^{-n/2}})\big)\\
&& \times H_{2}\big(2^{\frac{n}{12}}(X^{2,\pm}_{(j+1)2^{-n/2}} - X^{2,\pm}_{j2^{-n/2}})\big)\\
&&=: \widetilde{W}_{\pm,n}^{1,2}(f,t) + \overline{W}_{\pm,n}^{1,2}(f,t).
\end{eqnarray*}
 We claim that:
 \begin{eqnarray}\label{bounds for Wn12}
 E[ (\widetilde{W}_n^{1,2}( f,t)- \widetilde{W}_n^{1,2}( f,s))^2] &\leq & C \max\big(|s|^{1/3}, |t|^{1/3}\big)\big(2^{-n/2}+ |t-s| \big); \label{bounds for Wn12}\\
 E[ (\overline{W}_n^{1,2}( f,t)- \overline{W}_n^{1,2}( f,s))^2] &\leq & C \max\big(|s|^{1/3}, |t|^{1/3}\big)\big(2^{-n/2}+ |t-s| \big).\label{bounds for Wn12'}
 \end{eqnarray}
It suffices to prove (\ref{bounds for Wn12'}) which is representative of the difficulty. To do so, we distinguish two cases according to the signs of $s,t \in \R$ (and reducing the problem by symmetry):
\begin{enumerate}
\item[(1)] If $0 \leq s < t$ (the case $s < t \leq 0$  being similar), then
\begin{eqnarray*}
&& E[( \overline{W}_n^{1,2}(f,t) - \overline{W}_n^{1,2}(f,s))^2] = E[( \overline{W}_{+,n}^{1,2}(f,t) - \overline{W}_{+,n}^{1,2}(f,s))^2]\\
&=& \frac{1}{64} 2^{-n/2}\sum_{j,j'=\lfloor 2^{n/2} s \rfloor}^{\lfloor 2^{\frac{n}{2}} t \rfloor -1}E\bigg( \Delta_{j,n} f(X^{1,+},X^{2,+})\Delta_{j',n} f(X^{1,+},X^{2,+})I^{(1)}_{1}\big( (2^{\frac{n}{12}}\delta_{(j+1)2^{-n/2}})\big)\\
&& \times I^{(2)}_{2}\big( (2^{\frac{n}{12}}\delta_{(j+1)2^{-n/2}})^{\otimes 2}\big)I^{(1)}_{1}\big( (2^{\frac{n}{12}}\delta_{(j'+1)2^{-n/2}})\big)I^{(2)}_{2}\big( (2^{\frac{n}{12}}\delta_{(j'+1)2^{-n/2}})^{\otimes 2}\big)\bigg)\\
&=&\frac{1}{64} \sum_{j,j'=\lfloor 2^{n/2} s \rfloor}^{\lfloor 2^{\frac{n}{2}} t \rfloor -1}E\bigg( \Delta_{j,n} f(X^{1,+},X^{2,+})\Delta_{j',n} f(X^{1,+},X^{2,+})I^{(1)}_{1}\big( \delta_{(j+1)2^{-n/2}}\big)\\
&& \times I^{(1)}_{1}\big( \delta_{(j'+1)2^{-n/2}}\big) I^{(2)}_{2}\big( \delta^{\otimes 2}_{(j+1)2^{-n/2}}\big)I^{(2)}_{2}\big( \delta^{\otimes 2}_{(j'+1)2^{-n/2}}\big)\bigg),\\
\end{eqnarray*}
where we have the first equality by (\ref{linear-isometry}).  Relying to the product formula (\ref{product formula}), we deduce that this latter quantity is less than or equal to
\begin{eqnarray}
&&  \frac{1}{64} \sum_{j,j'=\lfloor 2^{n/2} s \rfloor}^{\lfloor 2^{\frac{n}{2}} t \rfloor -1}\bigg|E\bigg( \Delta_{j,n} f(X^{1,+},X^{2,+})\Delta_{j',n} f(X^{1,+},X^{2,+}) I^{(2)}_{2}\big( \delta^{\otimes 2}_{(j+1)2^{-n/2}}\big)\notag\\
&& \hspace{4cm}\times I^{(2)}_{2}\big( \delta^{\otimes 2}_{(j'+1)2^{-n/2}}\big)\bigg)\bigg||\langle \delta_{(j+1)2^{-n/2}}, \delta_{(j'+1)2^{-n/2}} \rangle| \notag\\
&& +\frac{1}{64} \sum_{j,j'=\lfloor 2^{n/2} s \rfloor}^{\lfloor 2^{\frac{n}{2}} t \rfloor -1}\bigg|E\bigg( \Delta_{j,n} f(X^{1,+},X^{2,+})\Delta_{j',n} f(X^{1,+},X^{2,+})\notag\\
&& \hspace{2cm}\times I^{(1)}_{2}\big( \delta_{(j+1)2^{-n/2}}\otimes\delta_{(j'+1)2^{-n/2}}\big)I^{(2)}_{2}\big( \delta^{\otimes 2}_{(j+1)2^{-n/2}}\big)I^{(2)}_{2}\big( \delta^{\otimes 2}_{(j'+1)2^{-n/2}}\big)\bigg)\bigg|\notag\\
&&=: \sum_{i=1}^2 Q_n^{+,i}(s,t). \label{definition of Qn+i}
\end{eqnarray}
We have then the following estimates.
\begin{itemize}
\item \underline{Case $i =1$.} Since $f \in C_b^\infty$, and thanks to the Cauchy-Schwarz inequality and to (\ref{isometry}), we have
\begin{eqnarray}
&&\bigg|E\bigg( \Delta_{j,n} f(X^{1,+},X^{2,+})\Delta_{j',n} f(X^{1,+},X^{2,+}) I^{(2)}_{2}\big( \delta^{\otimes 2}_{(j+1)2^{-n/2}}\big)\label{step4: inequality case i=1}\\
&& \hspace{4cm}\times I^{(2)}_{2}\big( \delta^{\otimes 2}_{(j'+1)2^{-n/2}}\big)\bigg)\bigg|\notag\\
& \leq & C \|I^{(2)}_{2}\big( \delta^{\otimes 2}_{(j+1)2^{-n/2}}\big)\|_2 \|I^{(2)}_{2}\big( \delta^{\otimes 2}_{(j'+1)2^{-n/2}}\big)\|_2 \notag\\
&\leq & C (2^{-n/6})^2.\notag
\end{eqnarray}
We deduce that
\begin{eqnarray*}
&& Q_n^{+,1}(s,t)\\
& \leq & C 2^{-n/3}\sum_{j,j'=\lfloor 2^{n/2} s \rfloor}^{\lfloor 2^{\frac{n}{2}} t \rfloor -1}|\langle \delta_{(j+1)2^{-n/2}}, \delta_{(j'+1)2^{-n/2}} \rangle|\\
&\leq& C 2^{-n/2}\sum_{j,j'= \lfloor 2^{n/2}s\rfloor}^{\lfloor 2^{n/2}t \rfloor -1}
\big|\frac12( |j-j'+1|^{1/3} + |j-j'-1|^{1/3} - 2 |j-j'|^{1/3}) \big|\notag \\
&=&  C2^{-n/2} \sum_{j = \lfloor 2^{n/2}s\rfloor }^{\lfloor 2^{n/2}t \rfloor -1}\sum_{q =j - \lfloor 2^{n/2}t \rfloor + 1}^{j - \lfloor 2^{n/2}s\rfloor } \big|\rho(q) \big|,\notag
\end{eqnarray*}
with $\rho(q)$ defined in (\ref{rho}). By a Fubini argument, it comes
\begin{eqnarray}
&&Q_n^{+,1}(s,t)\notag\\
&\leq& C2^{-n/2} \!\!\!\sum_{q = \lfloor 2^{n/2}s \rfloor  - \lfloor 2^{n/2}t \rfloor +1}^{\lfloor 2^{n/2}t \rfloor - \lfloor 2^{n/2}s \rfloor  -1} \!\!\!\!\!|\rho(q)| \bigg( (q+ \lfloor 2^{n/2}t \rfloor)\wedge
\lfloor 2^{n/2}t \rfloor\\
&& \hspace{5cm} - ( q+ \lfloor 2^{n/2}s\rfloor )\vee \lfloor 2^{n/2}s\rfloor \bigg)\notag\\
&\leq&  C2^{-n/2} \sum_{q = \lfloor 2^{n/2}s \rfloor  - \lfloor 2^{n/2}t \rfloor +1}^{\lfloor 2^{n/2}t \rfloor - \lfloor 2^{n/2}s \rfloor  -1} |\rho(q)|\big( \lfloor 2^{n/2}t \rfloor - \lfloor 2^{n/2}s \rfloor \big)\notag\\
&\leq& C2^{-n/2} \sum_{q \in \Z}|\rho(q)|\big|\lfloor 2^{n/2}t \rfloor - \lfloor 2^{n/2}s \rfloor \big|= C 2^{-n/2}\big|\lfloor 2^{n/2}t \rfloor - \lfloor 2^{n/2}s \rfloor \big|\notag\\
& \leq&  C2^{-n/2}\big(\big|\lfloor 2^{n/2}t \rfloor -2^{n/2}t \big| + 2^{n/2}\big|t-s\big| +  \big|\lfloor 2^{n/2}s \rfloor -2^{n/2}s\big|\big) \notag\\
&\leq& C ( 2^{-n/2} + |t-s|). \label{majoration of Qn1(s,t)}
\end{eqnarray}
Note that $\sum_{q \in \Z}|\rho(q)| < \infty$ since $H<\frac12$.

\item \underline{Case $i=2$.} \:\:\: Thanks to the duality formula (\ref{duality formula}) and to the Leibniz rule (\ref{Leibnitz1}), one has that
\begin{eqnarray*}
&&\bigg|E\bigg( \Delta_{j,n} f(X^{1,+},X^{2,+})\Delta_{j',n} f(X^{1,+},X^{2,+})I^{(1)}_{2}\big( \delta_{(j+1)2^{-n/2}}\otimes\delta_{(j'+1)2^{-n/2}}\big)\\
&& \hspace{4cm}\times I^{(2)}_{2}\big( \delta^{\otimes 2}_{(j+1)2^{-n/2}}\big)I^{(2)}_{2}\big( \delta^{\otimes 2}_{(j'+1)2^{-n/2}}\big)\bigg)\bigg|\\
&=&\bigg|E\bigg(\bigg \langle D_{X^{1}}^{2}\big(\Delta_{j,n} f(X^{1,+},X^{2,+})\Delta_{j',n} f(X^{1,+},X^{2,+})\big),\delta_{(j+1)2^{-n/2}}\otimes\delta_{(j'+1)2^{-n/2}} \bigg\rangle\\
&& \hspace{4cm}\times I^{(2)}_{2}\big( \delta^{\otimes 2}_{(j+1)2^{-n/2}}\big)I^{(2)}_{2}\big( \delta^{\otimes 2}_{(j'+1)2^{-n/2}}\big)\bigg)\bigg|\\
\end{eqnarray*}
\begin{eqnarray*}
&\leq & E\bigg(\bigg|\Delta_{j,n} \partial_{11}f(X^{1,+},X^{2,+})\Delta_{j',n} f(X^{1,+},X^{2,+})\big)I^{(2)}_{2}\big( \delta^{\otimes 2}_{(j+1)2^{-n/2}}\big)\\
&& \times I^{(2)}_{2}\big( \delta^{\otimes 2}_{(j'+1)2^{-n/2}}\big)\bigg|\bigg)\bigg|\bigg\langle \bigg(\frac{\varepsilon_{j2^{-n/2}}+ \varepsilon_{(j+1)2^{-n/2}}}{2}\bigg)^{\otimes 2},\delta_{(j+1)2^{-n/2}}\otimes\delta_{(j'+1)2^{-n/2}} \bigg\rangle\bigg|\\
&& + 2 E\bigg(\bigg|\Delta_{j,n} \partial_{1}f(X^{1,+},X^{2,+})\Delta_{j',n} \partial_1 f(X^{1,+},X^{2,+})\big)I^{(2)}_{2}\big( \delta^{\otimes 2}_{(j+1)2^{-n/2}}\big)\\
&& \times I^{(2)}_{2}\big( \delta^{\otimes 2}_{(j'+1)2^{-n/2}}\big)\bigg|\bigg) \bigg|\bigg\langle \bigg(\frac{\varepsilon_{j2^{-n/2}}+ \varepsilon_{(j+1)2^{-n/2}}}{2}\bigg)\tilde{\otimes}\bigg(\frac{\varepsilon_{j'2^{-n/2}}+ \varepsilon_{(j'+1)2^{-n/2}}}{2}\bigg),\\
&& \hspace{4cm} \delta_{(j+1)2^{-n/2}}\otimes\delta_{(j'+1)2^{-n/2}} \bigg\rangle\bigg|\\
&& + E\bigg(\bigg|\Delta_{j,n} f(X^{1,+},X^{2,+})\Delta_{j',n} \partial_{11}f(X^{1,+},X^{2,+})\big)I^{(2)}_{2}\big( \delta^{\otimes 2}_{(j+1)2^{-n/2}}\big)\\
&& \times I^{(2)}_{2}\big( \delta^{\otimes 2}_{(j'+1)2^{-n/2}}\big)\bigg|\bigg)\bigg|\bigg\langle \bigg(\frac{\varepsilon_{j'2^{-n/2}}+ \varepsilon_{(j'+1)2^{-n/2}}}{2}\bigg)^{\otimes 2},\delta_{(j+1)2^{-n/2}}\otimes\delta_{(j'+1)2^{-n/2}} \bigg\rangle\bigg|\\
&& =:d_n^1 + d_n^2 +d_n^3.
\end{eqnarray*}
Observe that, thanks to (\ref{step4: inequality case i=1}), we get
\begin{eqnarray*}
d_n^1&\leq & C2^{-n/3}\bigg|\bigg\langle \bigg(\frac{\varepsilon_{j2^{-n/2}}+ \varepsilon_{(j+1)2^{-n/2}}}{2}\bigg)^{\otimes 2},\delta_{(j+1)2^{-n/2}}\otimes\delta_{(j'+1)2^{-n/2}} \bigg\rangle\bigg|,\\
d_n^2 &\leq & C2^{-n/3}\bigg|\bigg\langle \bigg(\frac{\varepsilon_{j2^{-n/2}}+ \varepsilon_{(j+1)2^{-n/2}}}{2}\bigg)\tilde{\otimes}\bigg(\frac{\varepsilon_{j'2^{-n/2}}+ \varepsilon_{(j'+1)2^{-n/2}}}{2}\bigg),\\
&& \hspace{8cm} \delta_{(j+1)2^{-n/2}}\otimes\delta_{(j'+1)2^{-n/2}} \bigg\rangle\bigg|,\\
d_n^3 &\leq &  C2^{-n/3}\bigg|\bigg\langle \bigg(\frac{\varepsilon_{j'2^{-n/2}}+ \varepsilon_{(j'+1)2^{-n/2}}}{2}\bigg)^{\otimes 2},\delta_{(j+1)2^{-n/2}}\otimes\delta_{(j'+1)2^{-n/2}} \bigg\rangle\bigg|.
\end{eqnarray*}
By (\ref{12}), recall that $|\langle \varepsilon_u, \delta_{(j+1)2^{-n/2}} \rangle | \leq  2^{-n/6}$ for all $u\geq 0$ and all $j \in \N$. We thus get,
\begin{eqnarray*}
d_n^1 + d_n^2 + d_n^3 &\leq & C 2^{-n/2}\big( |\langle\varepsilon_{j2^{-n/2}}, \delta_{(j'+1)2^{-n/2}} \rangle| + |\langle\varepsilon_{(j+1)2^{-n/2}}, \delta_{(j'+1)2^{-n/2}} \rangle|\\
&& + |\langle\varepsilon_{j'2^{-n/2}}, \delta_{(j+1)2^{-n/2}} \rangle| + |\langle\varepsilon_{(j'+1)2^{-n/2}}, \delta_{(j+1)2^{-n/2}} \rangle|\\
&& + |\langle\varepsilon_{j'2^{-n/2}} + \varepsilon_{(j'+1)2^{-n/2}}, \delta_{(j'+1)2^{-n/2}} \rangle|\big).
\end{eqnarray*}
For instance, we can write
\begin{eqnarray*}
 && 2^{-n/2}\sum_{j,j'=\lfloor 2^{n/2}s\rfloor}^{\lfloor 2^{n/2}t\rfloor -1}
 | \langle \varepsilon_{(j'+1)2^{-n/2}}; \delta_{(j+1)2^{-n/2}}\rangle|\\
&  =&\frac12 2^{-2n/3} \sum_{j,j'=\lfloor 2^{n/2}s\rfloor}^{\lfloor 2^{n/2}t\rfloor -1}
 \big|
 (j+1)^{1/3}-j^{1/3}+|j'-j+1|^{1/3}-|j'-j|^{1/3}
 \big|\\
 \end{eqnarray*}
 \begin{eqnarray*}
 &\leq& \frac12 2^{-2n/3} \sum_{j,j'=\lfloor 2^{n/2}s\rfloor}^{\lfloor 2^{n/2}t\rfloor -1}
 \big((j+1)^{1/3}-j^{1/3}\big)\\
 &&+
 \frac12 2^{-2n/3} \sum_{\lfloor 2^{n/2}s\rfloor\leq j\leq j'\leq \lfloor 2^{n/2}t\rfloor-1}
 \big((j'-j+1)^{1/3}-(j'-j)^{1/3}\big)\\
  &&+
 \frac12 2^{-2n/3} \sum_{\lfloor 2^{n/2}s\rfloor\leq j'<j\leq \lfloor 2^{n/2}t\rfloor-1}
 \big((j-j')^{1/3}-(j-j'-1)^{1/3}\big)\\
 &\leq&\frac32\,2^{-2n/3}\big(\lfloor 2^{n/2}t\rfloor - \lfloor 2^{n/2}s\rfloor\big) \lfloor 2^{n/2}t\rfloor ^{1/3} \leq \frac{3t^{1/3}}2\,\big(2^{-n/2}+ |t-s|\big).
  \end{eqnarray*}
   Similarly,
  \begin{eqnarray*}
  2^{-n/2}\sum_{j,j'=\lfloor 2^{n/2}s\rfloor}^{\lfloor 2^{n/2}t\rfloor -1}
  |\langle \varepsilon_{j2^{-n/2}}; \delta_{(j'+1)2^{-n/2}}\rangle|&\leq&\frac{3t^{1/3}}2\,\big(2^{-n/2}+|t-s|\big);\\
   2^{-n/2}\sum_{j,j'=\lfloor 2^{n/2}s\rfloor}^{\lfloor 2^{n/2}t\rfloor -1}
  |\langle \varepsilon_{(j+1)2^{-n/2}}; \delta_{(j'+1)2^{-n/2}}\rangle|&\leq&\frac{3t^{1/3}}2\,\big(2^{-n/2}+|t-s|\big);\\
   2^{-n/2}\sum_{j,j'=\lfloor 2^{n/2}s\rfloor}^{\lfloor 2^{n/2}t\rfloor -1}
  |\langle \varepsilon_{j'2^{-n/2}}; \delta_{(j+1)2^{-n/2}}\rangle|&\leq&\frac{3t^{1/3}}2\,\big(2^{-n/2}+|t-s|\big);\\
     2^{-n/2}\sum_{j,j'=\lfloor 2^{n/2}s\rfloor}^{\lfloor 2^{n/2}t\rfloor -1}
  |\langle \varepsilon_{j'2^{-n/2}}+\varepsilon_{(j'+1)2^{-n/2}}; \delta_{(j'+1)2^{-n/2}}\rangle|&\leq&\frac{3t^{1/3}}2\,\big(2^{-n/2}+|t-s|\big).
  \end{eqnarray*}
  As a consequence, we deduce 
  \begin{equation}
  Q_n^{+,2}(s,t) \leq C t^{1/3} \big(2^{-n/2}+|t-s|\big). \label{majoration Qn2(s,t)}
  \end{equation}
\end{itemize}
Combining (\ref{definition of Qn+i}), (\ref{majoration of Qn1(s,t)}) and (\ref{majoration Qn2(s,t)}) finally shows our claim (\ref{bounds for Wn12'}).

\item[(2)] If $s < 0 \leq t$, then 
\begin{eqnarray*}
&& E[( \overline{W}_n^{1,2}(f,t) - \overline{W}_n^{1,2}(f,s))^2] = E[( \overline{W}_{+,n}^{1,2}(f,t) - \overline{W}_{-,n}^{1,2}(f,-s))^2]\\
&\leq & 2E[( \overline{W}_{+,n}^{1,2}(f,t))^2] + 2E[( \overline{W}_{-,n}^{1,2}(f,-s))^2].
\end{eqnarray*}
By (1) with $s=0$, one can write
\[
E[( \overline{W}_{+,n}^{1,2}(f,t))^2] \leq C t^{1/3}\big(2^{-n/2}+t \big).
\]
Similarly
\[
E[( \overline{W}_{-,n}^{1,2}(f,-s))^2] \leq C (-s)^{1/3}\big(2^{-n/2}+ (-s) \big).
\]
We deduce that
\[
E[( \overline{W}_n^{1,2}(f,t) - \overline{W}_n^{1,2}(f,s))^2] \leq C \max\big(t^{1/3}, (-s)^{1/3}\big)\big(2^{-n/2}+ (t-s) \big).
\]
That is, (\ref{bounds for Wn12'}) holds true in this case.
\end{enumerate}
\end{proof}

\subsection{\underline{Step 5}: Limits of the 2D weighted power variations of odd order}
Fix $f \in C_b^{\infty}$ and $t\geq 0$. We claim that, if $H\in\big[\frac16,\frac12\big)$ and $p+q \geq 5$ then, as $n\to\infty$,
\begin{equation}\label{first convergence 2D}
\tilde{V}_n^{p,q}(f,t) \overset{\rm prob}{\longrightarrow}  0.
\end{equation}
Moreover, if $H\in\big(\frac16,\frac12\big)$ and $p+q = 3$ then, as $n\to\infty$,
\begin{equation}\label{second convergence 2D}
\tilde{V}_n^{p,q}(f,t)  \overset{\rm prob}{\longrightarrow}  0.
\end{equation}
For all $t \geq 0$, we define $\tilde{V}_n^{(3)}(f,t)$ as follows
\begin{equation}
\tilde{V}_n^{(3)}(f,t) = \sum_{p+q =3}C(p,q)\tilde{V}_n^{p,q}(\partial_{1 \ldots 1 2 \ldots 2}^{p,q} f,t), \label{definition tilde{V}n3}
\end{equation}
with $C(3,0)=C(0,3)=\frac{1}{24}$ and $C(2,1)=C(1,2)=\frac18$.  Observe that thanks to (\ref{transforme tilde{V}pq}) and (\ref{definition Wn3}), we have 
\begin{equation}
\tilde{V}_n^{(3)}(f,t) = W_n^{(3)}(f,Y_{T_{\lfloor 2^n t\rfloor,n}}). \label{decomposition tilde{V}n3}
\end{equation}
Then, we claim that, for $H=1/6$, for any fixed $t\geq 0$, as $n \to \infty$ 
\begin{eqnarray}
(X^1, X^2, Y, \tilde{V}_n^{(3)}(f,t)) \overset{fdd}{\longrightarrow} (X^1, X^2, Y, \int_0^{Y_t} D^3f(X_s)d^3X_s), \label{convergence tilde{V}n3}
\end{eqnarray}
where $\int_0^t D^3f(X_s)d^3X_s$ is short-hand for
\begin{eqnarray*}
 \int_0^t D^3f(X_s)d^3X_s &=& \kappa_1\int_0^t\frac{\partial^3 f}{\partial x^3}\big(X^{1}_s,X^{2}_s\big)dB^{1}_s + \kappa_2\int_0^t\frac{\partial^3 f}{\partial y^3}\big(X^{1}_s,X^{2}_s\big)dB^{2}_s \\
 && + \kappa_3\int_0^t\frac{\partial^3 f}{\partial x^2 \partial y}\big(X^{1}_s,X^{2}_s\big)dB^{3}_s + \kappa_4\int_0^t\frac{\partial^3 f}{\partial x \partial y^2}\big(X^{1}_s,X^{2}_s\big)dB^{4}_s \notag
\end{eqnarray*}
 with $B= (B^{1}, \ldots , B^{4})$  a 4-dimensional two-sided Brownian motion independent of $(X^1, X^2)$ and also independent of $Y$. The constants $\kappa_1, \ldots, \kappa_4$ are the same as in (\ref{convergence Wn3}). Otherwise stated, (\ref{convergence tilde{V}n3}) means that 
$\tilde{V}_n^{(3)}(f,t)$ converges stably in law  to the random variable  $\int_0^{Y_t} D^3f(X_s)d^3X_s$. 

Indeed, combining (\ref{transforme tilde{V}pq}), (\ref{norme2 Wnpq}) together with the independence of $Y$ and $(X^1, X^2)$ (by the definition of $Z$ in (\ref{fBmBt})), we deduce that
\begin{eqnarray*}
E\big[ \big(\tilde{V}_n^{p,q}(f,t)\big)^2 \big] &\leq & C\sum_{k'=1}^{\frac{q+1}{2}}\bigg(\big( \sum_{a=1}^{2k'-1}2^{-n[H(p+q + 2k'-1-a) - \frac12]}\big)E(|Y_{T_{\lfloor 2^n t\rfloor,n}}|)\\
&& + 2^{-n[H(p+q+2k'-2) -\frac12]}E\big(|Y_{T_{\lfloor 2^n t\rfloor,n}}|^{2H+1}\big)  + 2^{-nH[p+q +2k'-1]}\bigg). 
\end{eqnarray*}
 On the other hand, recall from \cite[Lemma 2.3]{kh-lewis1} that $Y_{T_{\lfloor 2^n t\rfloor,n}} \overset{L^2}{\longrightarrow} Y_t$ as $n \to \infty$. So, combining this fact with the last inequality, we deduce that (\ref{first convergence 2D}) and (\ref{second convergence 2D}) hold true.

Now, using the decomposition (\ref{decomposition tilde{V}n3}), the conclusion of Step 4 (to pass from $Y_{T_{\lfloor 2^n t\rfloor,n}}$ to $Y_t$)  and the convergence: $Y_{T_{\lfloor 2^n t\rfloor,n}} \overset{L^2}{\longrightarrow} Y_t$, we deduce that the limit
of
$\tilde{V}_n^{(3)}(f,t) $
is the same as that of $ W_{n}^{(3)}(f,Y_t).$
Thus, the proof of (\ref{convergence tilde{V}n3}) then follows directly from  (\ref{convergence Wn3}) and the fact that $Y$ is independent of $(X^1, X^2)$ and independent of $(B^1, \ldots, B^4)$.

\subsection{\underline{Step 6}: Proving (\ref{second-main-1}) and (\ref{second-main-2})}
Let us introduce the following notation: for $f\in C_b^\infty$, for $j \in \N$, $\Delta_{j,n} f(Z^{1},Z^{2}):= f\big(\frac12(Z^1_{T_{j,n}}+ Z^1_{T_{j+1,n}}),\frac12 (Z^2_{T_{j,n}}+ Z^2_{T_{j+1,n}}) \big)$. Then, 
thanks to Lemma \ref{taylor-expansion}, we have 
\begin{eqnarray*}
 && f(Z^{1}_{T_{j+1,n}}, Z^{2}_{T_{j+1,n}})-f(Z^{1}_{T_{j,n}},Z^{2}_{T_{j,n}} )\\
 &=& \Delta_{j,n}\frac{\partial f}{\partial x}\big(Z^{1}, Z^{2}\big)\big(Z^{1}_{T_{j+1,n}}-Z^{1}_{T_{j,n}}\big) +\Delta_{j,n}\frac{\partial f}{\partial y}\big(Z^{1}, Z^{2}\big)\big(Z^{2}_{T_{j+1,n}}-Z^{2}_{T_{j,n}}\big) \\
&&  + \sum_{i=2}^6\sum_{\alpha_1 +\alpha_2 = 2i-1}C(\alpha_1,\alpha_2)\: \Delta_{j,n}\partial_{1 \ldots 1 2 \ldots 2}^{\alpha_1,\alpha_2} f(Z^{1}, Z^{2})\big(Z^{1}_{T_{j+1,n}}-Z^{1}_{T_{j,n}}\big)^{\alpha_1}\\
  && \hspace{1cm}\times \big(Z^{2}_{T_{j+1,n}}- Z^{2}_{T_{j,n}}\big)^{\alpha_2}+ R_{13}\big( \big(Z^{1}_{T_{j+1,n}},Z^{2}_{T_{j+1,n}})\big), \big(Z^{1}_{T_{j,n}},Z^{2}_{T_{j,n}})\big) \big).
 \end{eqnarray*}
 Then, by the Definition \ref{ivan-definition'}   and (\ref{definition tilde{V}pq}),  we can  write
 \begin{eqnarray*}
 &&f(Z^{1}_{T_{\lfloor 2^n t \rfloor ,n}}, Z^{2}_{T_{\lfloor 2^n t \rfloor , n}}) -f(0,0)\\
  &=& \tilde{O}_n(f,t) + \sum_{i=2}^6\sum_{\alpha_1 +\alpha_2 = 2i-1}C(\alpha_1,\alpha_2)\: \tilde{V}_n^{\alpha_1, \alpha_2}(\partial_{1 \ldots 1 2 \ldots 2}^{\alpha_1,\alpha_2} f,t)\notag\\
 && + \sum_{j=0}^{\lfloor 2^n t \rfloor -1} R_{13}\big( \big(Z^{1}_{T_{j+1,n}},Z^{2}_{T_{j+1,n}})\big), \big(Z^{1}_{T_{j,n}},Z^{2}_{T_{j,n}})\big) \big).\notag
 \end{eqnarray*}
Thanks to (\ref{definition tilde{V}n3}), we can write
\begin{eqnarray}
&& \tilde{O}_n(f,t) \label{second H>=1/6 taylor}\\
&=& f(Z^{1}_{T_{\lfloor 2^n t \rfloor ,n}}, Z^{2}_{T_{\lfloor 2^n t \rfloor , n}}) -f(0,0) - \tilde{V}_n^{(3)}(f,t)\notag\\
&& -\sum_{i=3}^6\sum_{\alpha_1 +\alpha_2 = 2i-1}C(\alpha_1,\alpha_2)\: \tilde{V}_n^{\alpha_1, \alpha_2}(\partial_{1 \ldots 1 2 \ldots 2}^{\alpha_1,\alpha_2} f,t)\notag\\
&& -\sum_{j=0}^{\lfloor 2^n t \rfloor -1} R_{13}\big( \big(Z^{1}_{T_{j+1,n}},Z^{2}_{T_{j+1,n}})\big), \big(Z^{1}_{T_{j,n}},Z^{2}_{T_{j,n}})\big) \big). \notag
\end{eqnarray}
By Lemma \ref{taylor-expansion}, we have, with $G\sim N(0,1)$,
\begin{eqnarray}
&& \sum_{j=0}^{\lfloor 2^{\frac{n}{2}} t \rfloor -1}E \bigg( \bigg| R_{13}\big( \big(Z^{1}_{T_{j+1,n}},Z^{2}_{T_{j+1,n}})\big), \big(Z^{1}_{T_{j,n}},Z^{2}_{T_{j,n}})\big) \big)\bigg| \bigg) \notag\\
& \leq & C_f \sum_{\alpha_1 + \alpha_2 =13}\sum_{j=0}^{\lfloor 2^n t \rfloor -1}E\bigg(\big|Z^{1}_{T_{j+1,n}} - Z^{1}_{T_{j,n}} \big|^{\alpha_1}\big| Z^{2}_{T_{j+1,n}} - Z^{2}_{T_{j,n}} \big|^{\alpha_2} \bigg) \notag\\
& \leq & C_f \sum_{\alpha_1 + \alpha_2 =13}\sum_{j=0}^{\lfloor 2^{\frac{n}{2}} t \rfloor -1}\||Z^{1}_{T_{j+1,n}} - Z^{1}_{T_{j,n}} |^{\alpha_1}\|_2 \|| Z^{2}_{T_{j+1,n}} - Z^{2}_{T_{j,n}} |^{\alpha_2} \|_2 \notag\\
&=& C_f 2^{-\frac{13nH}{2}}\sum_{\alpha_1 + \alpha_2 =13}\sum_{j=0}^{\lfloor 2^n t \rfloor -1}\big(E[G^{2\alpha_1}]E[G^{2\alpha_2}]\big)^{1/2} \leq Ct 2^{-n(\frac{13H}{2} -1)}. \label{second R13}
\end{eqnarray}

On the other hand, by continuity of $f\circ Z$  and due to (\ref{lemma2.2}), one has, almost surely and
as $n\to\infty$,
\begin{equation}\label{pm1}
 f(Z^{1}_{T_{\lfloor 2^n t \rfloor ,n}}, Z^{2}_{T_{\lfloor 2^n t \rfloor , n}}) -f(0,0)
\to f(Z^1_t, Z^2_t)-f(0,0).
\end{equation}

Finally, when $H>\frac16$ the desired conclusion (\ref{second-main-1}) follows from
(\ref{first convergence 2D}), (\ref{second convergence 2D}), (\ref{second R13}) and (\ref{pm1}) plugged into (\ref{second H>=1/6 taylor}).
The proof of (\ref{second-main-2}) when $H=\frac16$ is similar, the only difference being that one has (\ref{convergence tilde{V}n3}) instead of (\ref{second convergence 2D}), thus leading
to the bracket term  $\int_0^t D^3f(Z_s)d^3Z_s$ in (\ref{second-main-2}).

\subsection{\underline{Step 7}: Proving (\ref{second-main-3})}
Using $b^3 - a^3 = 3 \big(\frac{a+b}{2}\big)^2(b-a) + \frac14 (b-a)^3$, one can write,
\begin{eqnarray*}
 \tilde{O}_n(x\mapsto x^3,t) -  (Z_t^1)^3 &=& - \tilde{V}^3_n(x\mapsto x^3,t) + \sum_{j=0}^{\lfloor 2^n t \rfloor -1} \big( (Z^{1}_{T_{j+1,n}})^3 - (Z^{1}_{T_{j,n}})^3 \big) -  (Z_t^1)^3.
\end{eqnarray*}
 As a result and since, by (\ref{lemma2.2}), $(Z^{1}_{T_{\lfloor 2^n t \rfloor ,n}} )^3 \to (Z^1_t)^3$ a.s. as $n \to \infty$, one deduces that if $\tilde{O}_n(x\mapsto x^3,t)$ converges stably in law, then $\tilde{V}^3_n(x\mapsto x^3,t)$ must converge as well. But
it is shown in \cite[Corollary 1.2]{zeineddine} that
$2^{-n(1-6H)/4}\tilde{V}_n^{(3)}(x \mapsto x^3,t)$ converges in law
to a non degenerate limit. This fact being  in contradiction with the convergence of $\tilde{V}^3_n(x\mapsto x^3,t)$, we deduce that (\ref{second-main-3}) holds.

\section{Proof of Lemma \ref{biglemma}}

\subsection{Proof of (\ref{lemma3})}
 We will consider only the case $(i,j) = (1,2)$ (by symmetry, the proof is very similar for $(i,j) =(2,1)$ and is left to the reader.) By the product formula (\ref{product formula}), we have
 \begin{eqnarray}
 I^{(1)}_1 \big(\delta_{(i_3+1)2^{-n/2}} \big)I^{(1)}_1 \big(\delta_{(i_4+1)2^{-n/2}} \big)&=& I^{(1)}_2 \big(\delta_{(i_3+1)2^{-n/2}}\otimes \delta_{(i_4+1)2^{-n/2}}\big)\label{prod1}\\
 && + \langle \delta_{(i_3+1)2^{-n/2}}, \delta_{(i_4+1)2^{-n/2}} \rangle \notag\\
 I^{(2)}_2 \big(\delta_{(i_3+1)2^{-n/2}}^{\otimes 2} \big)I^{(2)}_2 \big(\delta_{(i_4+1)2^{-n/2}}^{\otimes 2} \big)&=& 
 I^{(2)}_4 \big(\delta_{(i_3+1)2^{-n/2}}^{\otimes 2}\otimes\delta_{(i_4+1)2^{-n/2}}^{\otimes 2} \big) \label{prod2}\\
 && + 4 I^{(2)}_2 \big(\delta_{(i_3+1)2^{-n/2}}\otimes\delta_{(i_4+1)2^{-n/2}} \big)\langle \delta_{(i_3+1)2^{-n/2}}, \delta_{(i_4+1)2^{-n/2}} \rangle \notag\\
 && + 2 \langle \delta_{(i_3+1)2^{-n/2}}, \delta_{(i_4+1)2^{-n/2}} \rangle^2. \notag
 \end{eqnarray}
 Thanks to (\ref{prod1}) we deduce that, for all $ i_1, i_2 \in \{0, \ldots, \lfloor 2^{\frac{n}{2}} t \rfloor -1\}$,
 \begin{eqnarray*}
 &&\sum_{i_3,i_4=0}^{\lfloor 2^{\frac{n}{2}} t \rfloor -1}\bigg|E\bigg(\Delta_{i_1,n}f_1(X^{1},X^{2})\Delta_{i_2,n}f_2(X^{1},X^{2})\Delta_{i_3,n}f_3(X^{1},X^{2})\\
 && \times \Delta_{i_4,n}f_4(X^{1},X^{2})I^{(1)}_1 \big(\delta_{(i_3+1)2^{-n/2}} \big)I^{(2)}_2 \big(\delta_{(i_3+1)2^{-n/2}}^{\otimes 2} \big)I^{(1)}_1 \big(\delta_{(i_4+1)2^{-n/2}} \big)I^{(2)}_2 \big(\delta_{(i_4+1)2^{-n/2}}^{\otimes 2} \big)\bigg)\bigg|\\
 & \leq & \sum_{i_3,i_4=0}^{\lfloor 2^{\frac{n}{2}} t \rfloor -1}\bigg|E\bigg(\Delta_{i_1,n}f_1(X^{1},X^{2})\Delta_{i_2,n}f_2(X^{1},X^{2})\Delta_{i_3,n}f_3(X^{1},X^{2})\\
 && \times \Delta_{i_4,n}f_4(X^{1},X^{2})I^{(1)}_2 \big(\delta_{(i_3+1)2^{-n/2}}\otimes  \delta_{(i_4+1)2^{-n/2}}\big)I^{(2)}_2 \big(\delta_{(i_3+1)2^{-n/2}}^{\otimes 2} \big)I^{(2)}_2 \big(\delta_{(i_4+1)2^{-n/2}}^{\otimes 2} \big)\bigg)\bigg|\\
 && + \sum_{i_3,i_4=0}^{\lfloor 2^{\frac{n}{2}} t \rfloor -1}\bigg|E\bigg(\Delta_{i_1,n}f_1(X^{1},X^{2})\Delta_{i_2,n}f_2(X^{1},X^{2})\Delta_{i_3,n}f_3(X^{1},X^{2})\\
 && \times \Delta_{i_4,n}f_4(X^{(1)},X^{(2)})I^{(2)}_2 \big(\delta_{(i_3+1)2^{-n/2}}^{\otimes 2} \big)I^{(2)}_2 \big(\delta_{(i_4+1)2^{-n/2}}^{\otimes 2} \big)\bigg)\bigg|\big|\langle \delta_{(i_3+1)2^{-n/2}}, \delta_{(i_4+1)2^{-n/2}} \rangle \big|\\
 &=& M_{n,1}(i_1,i_2,t) + M_{n,2}(i_1,i_2,t),
 \end{eqnarray*}
with obvious notation at the last line. 
 Set  
 \[
 \phi(i_1,i_2,i_3,i_4):=\Delta_{i_1,n}f_1(X^{1},X^{2})\Delta_{i_2,n}f_2(X^{1},X^{2})\Delta_{i_3,n}f_3(X^{1},X^{2})
  \Delta_{i_4,n}f_4(X^{1},X^{2}).
 \]
 Let us prove that, for $i \in \{1,2\}$,  $\exists C >0$ such that :
 \begin{equation}\label{boundedness}
 \sup_{n\geq 0}\sup_{i_1,i_2 \in \{0, \ldots, \lfloor 2^{\frac{n}{2}} t \rfloor -1\}} M_{n,i}(i_1,i_2,t) \leq C(t + t^2).
 \end{equation}
 \begin{enumerate}
 
 \item \underline{for $i=1$ :}  thanks to the duality formula (\ref{duality formula}), we have
 \begin{eqnarray*}
 M_{n,1}(i_1,i_2,t)&=&\sum_{i_3,i_4=0}^{\lfloor 2^{\frac{n}{2}} t \rfloor -1}\bigg|E\bigg(\bigg \langle D_{X^{1}}\big(\phi(i_1,i_2,i_3,i_4) \big) , \delta_{(i_3+1)2^{-n/2}}\otimes  \delta_{(i_4+1)2^{-n/2}} \bigg\rangle \\
 && \hspace{4cm} \times I^{(2)}_2 \big(\delta_{(i_3+1)2^{-n/2}}^{\otimes 2} \big)I^{(2)}_2 \big(\delta_{(i_4+1)2^{-n/2}}^{\otimes 2} \big)\bigg)\bigg|.\\
 \end{eqnarray*}
Observe that, thanks to (\ref{Leibnitz0}) and  (\ref{Leibnitz1}), we have
\begin{eqnarray*}
&& D_{X^{1}}\big(\phi(i_1,i_2,i_3,i_4) \big) = \sum_{j=1}^4\phi_j(i_1,i_2,i_3,i_4)\bigg(\frac{\varepsilon_{i_j2^{-n/2}} + \varepsilon_{(i_j+1)2^{-n/2} }}{2}\bigg),
\end{eqnarray*}
where $\phi_j(i_1,i_2,i_3,i_4)$ is a quantity having a similar form as $\phi(i_1,i_2,i_3,i_4)$ and arising when one differentiates $\Delta_{i_j,n}f_j(X^{1},X^{2})$ in $\phi(i_1,i_2,i_3,i_4)$ with respect to $X^{1}$. 
 By combining this fact with (\ref{12}), we get
 \begin{eqnarray*}
 &&M_{n,1}(i_1,i_2,t)\\
 &\leq & (2^{-n/6})^2 \sum_{j=1}^4 \sum_{i_3,i_4=0}^{\lfloor 2^{\frac{n}{2}} t \rfloor -1}\bigg|E\bigg( \phi_j(i_1,i_2,i_3,i_4)I^{(2)}_2 \big(\delta_{(i_3+1)2^{-n/2}}^{\otimes 2} \big)I^{(2)}_2 \big(\delta_{(i_4+1)2^{-n/2}}^{\otimes 2} \big)\bigg)\bigg|\\
 &=&  \sum_{j=1}^4 M^{(j)}_{n,1}(i_1,i_2, t)
 \end{eqnarray*}
 with obvious notation at the last line. 
 We have to prove that, for all $j \in \{1, \ldots , 4 \}$, one has $\sup_{n \geq 0}\sup_{i_1,i_2 \in \{0, \ldots,\lfloor 2^{\frac{n}{2}} t \rfloor -1\} } M^{(j)}_{n,1}(i_1,i_2,t)$ $\leq C(t+t^2)$. Let us do it.
 Thanks to (\ref{prod2}), we have
 
 \begin{eqnarray*}
 M^{(j)}_{n,1}(i_1,i_2,t)&=& 2^{-n/3} \sum_{i_3,i_4=0}^{\lfloor 2^{\frac{n}{2}} t \rfloor -1}\bigg|E\bigg( \phi_j(i_1,i_2,i_3,i_4)I^{(2)}_4 \big(\delta_{(i_3+1)2^{-n/2}}^{\otimes 2}\otimes\delta_{(i_4+1)2^{-n/2}}^{\otimes 2} \big)\bigg)\bigg|\\
 && + 42^{-n/3} \sum_{i_3,i_4=0}^{\lfloor 2^{\frac{n}{2}} t \rfloor -1}\bigg|E\bigg( \phi_j(i_1,i_2,i_3,i_4) I^{(2)}_2 \big(\delta_{(i_3+1)2^{-n/2}}\otimes\delta_{(i_4+1)2^{-n/2}} \big)\bigg)\bigg|\\
 && \hspace{4cm} \times |\langle \delta_{(i_3+1)2^{-n/2}}, \delta_{(i_4+1)2^{-n/2}} \rangle |\\
 && + 22^{-n/3} \sum_{i_3,i_4=0}^{\lfloor 2^{\frac{n}{2}} t \rfloor -1}\bigg|E\bigg( \phi_j(i_1,i_2,i_3,i_4) \bigg)\bigg|\langle \delta_{(i_3+1)2^{-n/2}}, \delta_{(i_4+1)2^{-n/2}} \rangle^2.
 \end{eqnarray*}
 Thanks to the duality formula (\ref{duality formula}), to (\ref{Leibnitz1}) and to (\ref{12}) and since $\phi_j$ is bounded, we deduce that
 
 \begin{eqnarray*}
 && \bigg|E\bigg( \phi_j(i_1,i_2,i_3,i_4)I^{(2)}_4 \big(\delta_{(i_3+1)2^{-n/2}}^{\otimes 2}\otimes\delta_{(i_4+1)2^{-n/2}}^{\otimes 2} \big)\bigg)\bigg|\\
 &=& \bigg|E\bigg( \big \langle D^4_{X^{2}}\big(\phi_j(i_1,i_2,i_3,i_4)\big), \delta_{(i_3+1)2^{-n/2}}^{\otimes 2}\otimes\delta_{(i_4+1)2^{-n/2}}^{\otimes 2} \big\rangle\bigg)\bigg|\\
 && \leq  C_j (2^{-n/6})^4,\\
 && \bigg|E\bigg( \phi_j(i_1,i_2,i_3,i_4) I^{(2)}_2 \big(\delta_{(i_3+1)2^{-n/2}}\otimes\delta_{(i_4+1)2^{-n/2}} \big)\bigg)\bigg|\\
 &=& \bigg|E\bigg(  \big\langle D^2_{X^{2}}\big(\phi_j(i_1,i_2,i_3,i_4)\big), \delta_{(i_3+1)2^{-n/2}}\otimes\delta_{(i_4+1)2^{-n/2}} \big\rangle \bigg)\bigg| \\
 && \leq C_j (2^{-n/6})^2.
 \end{eqnarray*}
By combining these inequalities with (\ref{13}), we get
 \begin{eqnarray*}
 M^{(j)}_{n,1}(i_1,i_2,t)&\leq & C_j 2^{-n} 2^{n} t^2 + C_j t 2^{-2n/3}2^{n/3} + C_j t 2^{-n/3} 2^{n/6}\\
 & \leq & C_j (t +t^2).
 \end{eqnarray*}
Hence $\exists C >0$ such that for all $j \in \{1, \ldots , 4 \}$, $\sup_{n \geq 0}\sup_{i_1,i_2 \in \{0, \ldots,\lfloor 2^{\frac{n}{2}} t \rfloor -1 \}} M^{(j)}_{n,1}(i_1,i_2,t) \leq C(t +t^2)$. So, we have the desired conclusion (\ref{boundedness}) for $i=1$.

 \item \underline{for $i=2$ :} Thanks to (\ref{prod2}), we have
 
 \begin{eqnarray*}
 M_{n,2}(i_1,i_2,t) &=&\sum_{i_3,i_4=0}^{\lfloor 2^{\frac{n}{2}} t \rfloor -1}\bigg|E\bigg(\phi(i_1,i_2,i_3,i_4)I^{(2)}_4 \big(\delta_{(i_3+1)2^{-n/2}}^{\otimes 2}\otimes\delta_{(i_4+1)2^{-n/2}}^{\otimes 2} \big)\bigg)\bigg|\\
 && \hspace{4cm} \times \big|\langle \delta_{(i_3+1)2^{-n/2}}, \delta_{(i_4+1)2^{-n/2}} \rangle \big|\\
 && + 4 \sum_{i_3,i_4=0}^{\lfloor 2^{\frac{n}{2}} t \rfloor -1}\bigg|E\bigg(\phi(i_1,i_2,i_3,i_4)I^{(2)}_2 \big(\delta_{(i_3+1)2^{-n/2}}\otimes\delta_{(i_4+1)2^{-n/2}} \big)\bigg)\bigg|\\
 && \hspace{4cm} \times \langle \delta_{(i_3+1)2^{-n/2}}, \delta_{(i_4+1)2^{-n/2}} \rangle^2\\
 && + 2 \sum_{i_3,i_4=0}^{\lfloor 2^{\frac{n}{2}} t \rfloor -1}\bigg|E\big(\phi(i_1,i_2,i_3,i_4)\big)\bigg| \big|\langle \delta_{(i_3+1)2^{-n/2}}, \delta_{(i_4+1)2^{-n/2}} \rangle\big|^3.
 \end{eqnarray*}
 By the same arguments as used in the previous case, we deduce that
 \[
 M_{n,2}(i_1,i_2,t) \leq C (2^{-n/6})^4 t 2^{n/3} + C (2^{-n/6})^2 t 2^{n/6} + C t \leq C t.
 \]
 Hence, $\exists C >0$ such that $\sup_{n \geq 0} \sup_{i_1,i_2 \in \{0,\ldots, \lfloor 2^{\frac{n}{2}} t \rfloor -1\}} M_{n,2}(i_1,i_2,t) \leq C(t+t^2)$. So, we have the desired conclusion (\ref{boundedness}) for $i=2$. This ends the proof of  (\ref{lemma3}).
 \end{enumerate}

\subsection{Proof of (\ref{lemma4})} 
We will consider only the case $i=2$ (by symmetry, the proof is very similar for $i=1$). 
Set  
 \[
 \phi(i_1,i_2,i_3,i_4):=\Delta_{i_1,n}f_1(X^{1},X^{2})\Delta_{i_2,n}f_2(X^{1},X^{2})\Delta_{i_3,n}f_3(X^{1},X^{2})\Delta_{i_4,n}f_4(X^{1},X^{2}).
 \]
Using the product formula (\ref{product formula}), we have that $I^{(2)}_3 \big(\delta_{(i_3+1)2^{-n/2}}^{\otimes 3} \big)I^{(2)}_3 \big(\delta_{(i_4+1)2^{-n/2}}^{\otimes 3} \big)$ equals
\begin{eqnarray*}
&& I^{(2)}_6 \big(\delta_{(i_3+1)2^{-n/2}}^{\otimes 3}\otimes\delta_{(i_4+1)2^{-n/2}}^{\otimes 3} \big) + 9I^{(2)}_4 \big(\delta_{(i_3+1)2^{-n/2}}^{\otimes 2}\otimes\delta_{(i_4+1)2^{-n/2}}^{\otimes 2} \big)\langle \delta_{(i_3+1)2^{-n/2}}, \delta_{(i_4+1)2^{-n/2}} \rangle\\
&& + 18 I^{(2)}_2 \big(\delta_{(i_3+1)2^{-n/2}}\otimes\delta_{(i_4+1)2^{-n/2}} \big)\langle \delta_{(i_3+1)2^{-n/2}}, \delta_{(i_4+1)2^{-n/2}} \rangle^2 + 6\langle \delta_{(i_3+1)2^{-n/2}}, \delta_{(i_4+1)2^{-n/2}} \rangle^3.
\end{eqnarray*} 
So, we get
\begin{eqnarray*}
&& \sum_{i_3,i_4=0}^{\lfloor 2^{\frac{n}{2}} t \rfloor -1}\bigg|E\bigg(\phi(i_1,i_2,i_3,i_4)I^{(2)}_3 \big(\delta_{(i_3+1)2^{-n/2}}^{\otimes 3} \big)I^{(2)}_3 \big(\delta_{(i_4+1)2^{-n/2}}^{\otimes 3} \big)\bigg)\bigg|\\
&=& \sum_{i_3,i_4=0}^{\lfloor 2^{\frac{n}{2}} t \rfloor -1}\bigg|E\bigg(\phi(i_1,i_2,i_3,i_4)I^{(2)}_6 \big(\delta_{(i_3+1)2^{-n/2}}^{\otimes 3}\otimes\delta_{(i_4+1)2^{-n/2}}^{\otimes 3} \big)\bigg)\bigg|\\
&& + 9 \sum_{i_3,i_4=0}^{\lfloor 2^{\frac{n}{2}} t \rfloor -1}\bigg|E\bigg(\phi(i_1,i_2,i_3,i_4)I^{(2)}_4 \big(\delta_{(i_3+1)2^{-n/2}}^{\otimes 2}\otimes\delta_{(i_4+1)2^{-n/2}}^{\otimes 2} \big)\bigg)\bigg| |\langle \delta_{(i_3+1)2^{-n/2}}, \delta_{(i_4+1)2^{-n/2}} \rangle |\\
&& + 18 \sum_{i_3,i_4=0}^{\lfloor 2^{\frac{n}{2}} t \rfloor -1}\bigg|E\bigg(\phi(i_1,i_2,i_3,i_4)I^{(2)}_2 \big(\delta_{(i_3+1)2^{-n/2}}\otimes\delta_{(i_4+1)2^{-n/2}} \big)\bigg)\bigg| \langle \delta_{(i_3+1)2^{-n/2}}, \delta_{(i_4+1)2^{-n/2}} \rangle^2 \\
&& + 6 \sum_{i_3,i_4=0}^{\lfloor 2^{\frac{n}{2}} t \rfloor -1}\bigg|E\bigg(\phi(i_1,i_2,i_3,i_4)\bigg)\bigg| |\langle \delta_{(i_3+1)2^{-n/2}}, \delta_{(i_4+1)2^{-n/2}} \rangle|^3.
\end{eqnarray*}
Thanks to the duality formula (\ref{duality formula}), we get
\begin{eqnarray*}
&& \sum_{i_3,i_4=0}^{\lfloor 2^{\frac{n}{2}} t \rfloor -1}\bigg|E\bigg(\phi(i_1,i_2,i_3,i_4)I^{(2)}_3 \big(\delta_{(i_3+1)2^{-n/2}}^{\otimes 3} \big)I^{(2)}_3 \big(\delta_{(i_4+1)2^{-n/2}}^{\otimes 3} \big)\bigg)\bigg|\\
&=& \sum_{i_3,i_4=0}^{\lfloor 2^{\frac{n}{2}} t \rfloor -1}\bigg|E\bigg( \big \langle D^6_{X^{2}}\big(\phi(i_1,i_2,i_3,i_4)\big),\delta_{(i_3+1)2^{-n/2}}^{\otimes 3}\otimes\delta_{(i_4+1)2^{-n/2}}^{\otimes 3} \big\rangle\bigg)\bigg|\\
\end{eqnarray*}
\begin{eqnarray*}
&& + 9 \sum_{i_3,i_4=0}^{\lfloor 2^{\frac{n}{2}} t \rfloor -1}\bigg|E\bigg(\big \langle D^4_{X^{2}}\big(\phi(i_1,i_2,i_3,i_4)\big), \delta_{(i_3+1)2^{-n/2}}^{\otimes 2}\otimes\delta_{(i_4+1)2^{-n/2}}^{\otimes 2} \big\rangle\bigg)\bigg|\\
&& \hspace{2cm}\times |\langle \delta_{(i_3+1)2^{-n/2}}, \delta_{(i_4+1)2^{-n/2}} \rangle |\\
&& + 18 \sum_{i_3,i_4=0}^{\lfloor 2^{\frac{n}{2}} t \rfloor -1}\bigg|E\bigg(\langle D^2_{X^{2}}\big(\phi(i_1,i_2,i_3,i_4)\big), \delta_{(i_3+1)2^{-n/2}}\otimes\delta_{(i_4+1)2^{-n/2}} \big\rangle\bigg)\bigg| \\
&& \hspace{2cm} \times \langle \delta_{(i_3+1)2^{-n/2}}, \delta_{(i_4+1)2^{-n/2}} \rangle^2 \\
&& + 6 \sum_{i_3,i_4=0}^{\lfloor 2^{\frac{n}{2}} t \rfloor -1}\bigg|E\bigg(\phi(i_1,i_2,i_3,i_4)\bigg)\bigg| |\langle \delta_{(i_3+1)2^{-n/2}}, \delta_{(i_4+1)2^{-n/2}} \rangle|^3.
\end{eqnarray*}
Observe that, thanks to (\ref{Leibnitz0}) and (\ref{Leibnitz1}), for any $k \in \{1,2,3\}$, we have
\begin{eqnarray}
&&D^{2k}_{X^{2}}\big(\phi(i_1,i_2,i_3,i_4)\big)\label{bigderivative}\\
&=&  \sum_{a_1+a_2+a_3+a_4=2k}\phi_{(a_1,a_2,a_3,a_4)}(i_1,i_2,i_3,i_4)\bigg(\frac{\varepsilon_{i_12^{-n/2}} + \varepsilon_{(i_1+1)2^{-n/2} }}{2}\bigg)^{\otimes a_1}\notag\\
&& \tilde{\otimes} \bigg(\frac{\varepsilon_{i_22^{-n/2}} + \varepsilon_{(i_2+1)2^{-n/2} }}{2}\bigg)^{\otimes a_2}\tilde{\otimes}\bigg(\frac{\varepsilon_{i_32^{-n/2}} + \varepsilon_{(i_3+1)2^{-n/2} }}{2}\bigg)^{\otimes a_3}\tilde{\otimes}\bigg(\frac{\varepsilon_{i_42^{-n/2}} + \varepsilon_{(i_4+1)2^{-n/2} }}{2}\bigg)^{\otimes a_4}\notag
\end{eqnarray}
where $(a_1,a_2,a_3,a_4) \in \N^4$ and $\phi_{(a_1,a_2,a_3,a_4)}(i_1,i_2,i_3,i_4)$ is a quantity having a similar form as $\phi(i_1,i_2,i_3,i_4)$ and arising when one differentiates $\Delta_{i_j,n}f_j(X^{1},X^{2})$ in $\phi(i_1,i_2,i_3,i_4)$ $a_j$-times  with respect to $X^{2}$. Thanks to (\ref{bigderivative}), (\ref{12}), (\ref{13}) and since $\phi_{(a_1,a_2,a_3,a_4)}(i_1,i_2,i_3,i_4)$ is bounded, we deduce that

\begin{eqnarray*}
&&\sum_{i_3,i_4=0}^{\lfloor 2^{\frac{n}{2}} t \rfloor -1}\bigg|E\bigg(\phi(i_1,i_2,i_3,i_4)I^{(2)}_3 \big(\delta_{(i_3+1)2^{-n/2}}^{\otimes 3} \big)I^{(2)}_3 \big(\delta_{(i_4+1)2^{-n/2}}^{\otimes 3} \big)\bigg)\bigg|\\
& \leq & C(2^{-n/6})^6\: 2^n \: t^2 + C (2^{-n/6})^4\sum_{i_3,i_4=0}^{\lfloor 2^{\frac{n}{2}} t \rfloor -1}|\langle \delta_{(i_3+1)2^{-n/2}}, \delta_{(i_4+1)2^{-n/2}} \rangle | \\
&& + C(2^{-n/6})^2\sum_{i_3,i_4=0}^{\lfloor 2^{\frac{n}{2}} t \rfloor -1}\langle \delta_{(i_3+1)2^{-n/2}}  \delta_{(i_4+1)2^{-n/2}} \rangle^2 + C \sum_{i_3,i_4=0}^{\lfloor 2^{\frac{n}{2}} t \rfloor -1}|\langle \delta_{(i_3+1)2^{-n/2}}  \delta_{(i_4+1)2^{-n/2}} \rangle|^3 \\
&\leq & C t^2 + C 2^{-n/3} t + C 2^{-n/6} t + Ct.
\end{eqnarray*}
Hence, we deduce immediately that
\begin{equation*}
  \sup_{n \geq 1}\sup_{i_1,i_2 \in \{0, \ldots, \lfloor 2^{\frac{n}{2}} t \rfloor -1\}} \sum_{i_3,i_4=0}^{\lfloor 2^{\frac{n}{2}} t \rfloor -1}\bigg|E\bigg(\phi(i_1,i_2,i_3,i_4)I^{(2)}_3 \big(\delta_{(i_3+1)2^{-n/2}}^{\otimes 3} \big)I^{(2)}_3 \big(\delta_{(i_4+1)2^{-n/2}}^{\otimes 3} \big)\bigg)\bigg| \leq C(t + t^2),
\end{equation*}
which end the proof of (\ref{lemma4}).

\subsection{Proof of (\ref{lemma4'})}
We will consider only the case $(i,j) = (1,2)$ (by symmetry, the proof is very similar for $(i,j) =(2,1)$ and is left to the reader.)
Thanks to (\ref{product formula}), we have
\begin{eqnarray*}
&& I^{(1)}_1 \big(\delta_{(i_3+1)2^{-n/2}} \big)I^{(2)}_2 \big(\delta_{(i_3+1)2^{-n/2}}^{\otimes 2} \big)I^{(1)}_1 \big(\delta_{(i_4+1)2^{-n/2}} \big)I^{(2)}_2 \big(\delta_{(i_4+1)2^{-n/2}}^{\otimes 2} \big)\\
&=& I^{(1)}_2 \big(\delta_{(i_3+1)2^{-n/2}}\otimes\delta_{(i_4+1)2^{-n/2}} \big)I^{(2)}_2 \big(\delta_{(i_3+1)2^{-n/2}}^{\otimes 2} \big)I^{(2)}_2 \big(\delta_{(i_4+1)2^{-n/2}}^{\otimes 2} \big)\\
&& + \langle  \delta_{(i_3+1)2^{-n/2}} , \delta_{(i_4+1)2^{-n/2}} \rangle  I^{(2)}_2 \big(\delta_{(i_3+1)2^{-n/2}}^{\otimes 2} \big)I^{(2)}_2 \big(\delta_{(i_4+1)2^{-n/2}}^{\otimes 2} \big)\\
&=& I^{(1)}_2 \big(\delta_{(i_3+1)2^{-n/2}}\otimes\delta_{(i_4+1)2^{-n/2}} \big)I^{(2)}_4 \big(\delta_{(i_3+1)2^{-n/2}}^{\otimes 2} \otimes\delta_{(i_4+1)2^{-n/2}}^{\otimes 2} \big)\\
&& + 4 I^{(1)}_2 \big(\delta_{(i_3+1)2^{-n/2}}\otimes\delta_{(i_4+1)2^{-n/2}} \big)I^{(2)}_2 \big(\delta_{(i_3+1)2^{-n/2}}\otimes\delta_{(i_4+1)2^{-n/2}} \big)\langle  \delta_{(i_3+1)2^{-n/2}} , \delta_{(i_4+1)2^{-n/2}} \rangle \\
&& + 2 I^{(1)}_2 \big(\delta_{(i_3+1)2^{-n/2}}\otimes\delta_{(i_4+1)2^{-n/2}} \big)\langle  \delta_{(i_3+1)2^{-n/2}} , \delta_{(i_4+1)2^{-n/2}} \rangle^2\\
&& + I^{(2)}_4 \big(\delta_{(i_3+1)2^{-n/2}}^{\otimes 2} \otimes\delta_{(i_4+1)2^{-n/2}}^{\otimes 2} \big)\langle  \delta_{(i_3+1)2^{-n/2}} , \delta_{(i_4+1)2^{-n/2}} \rangle\\
&& + 4 I^{(2)}_2 \big(\delta_{(i_3+1)2^{-n/2}} \otimes \delta_{(i_4+1)2^{-n/2}} \big)\langle  \delta_{(i_3+1)2^{-n/2}} , \delta_{(i_4+1)2^{-n/2}} \rangle^2\\
&& + 2 \langle  \delta_{(i_3+1)2^{-n/2}} , \delta_{(i_4+1)2^{-n/2}} \rangle^3.
\end{eqnarray*}
For $i_1, i_2, i_3, i_4 \in \N$, set  $\phi(i_1, i_2, i_3, i_4):= \prod_{a=1}^4 \Delta_{i_a,n}f_a(X^{1},X^{2})$. Then, thanks to the previous estimate, we get
\begin{eqnarray*}
 && \sum_{i_1,i_2,i_3,i_4=0}^{\lfloor 2^{\frac{n}{2}} t \rfloor -1}\bigg|E\bigg(\prod_{a=1}^4 \Delta_{i_a,n}f_a(X^{1},X^{2})I^{(1)}_1 \big(\delta_{(i_a+1)2^{-n/2}} \big)I^{(2)}_2 \big(\delta_{(i_a+1)2^{-n/2}}^{\otimes 2} \big)\bigg)\bigg|\\
 &\leq & \sum_{i_1,i_2,i_3,i_4=0}^{\lfloor 2^{\frac{n}{2}} t \rfloor -1}\bigg|E\bigg(\phi(i_1, i_2, i_3, i_4)\prod_{a=1}^2 I^{(1)}_1 \big(\delta_{(i_a+1)2^{-n/2}} \big)I^{(2)}_2 \big(\delta_{(i_a+1)2^{-n/2}}^{\otimes 2} \big)\\
 &&\hspace{2cm} \times I^{(1)}_2 \big(\delta_{(i_3+1)2^{-n/2}}\otimes\delta_{(i_4+1)2^{-n/2}} \big)I^{(2)}_4 \big(\delta_{(i_3+1)2^{-n/2}}^{\otimes 2} \otimes\delta_{(i_4+1)2^{-n/2}}^{\otimes 2} \big)\bigg)\bigg|\\
 && + 4\sum_{i_1,i_2,i_3,i_4=0}^{\lfloor 2^{\frac{n}{2}} t \rfloor -1}\bigg|E\bigg(\phi(i_1, i_2, i_3, i_4)\prod_{a=1}^2 I^{(1)}_1 \big(\delta_{(i_a+1)2^{-n/2}} \big)I^{(2)}_2 \big(\delta_{(i_a+1)2^{-n/2}}^{\otimes 2} \big)\\
 &&\hspace{2cm} \times I^{(1)}_2 \big(\delta_{(i_3+1)2^{-n/2}}\otimes\delta_{(i_4+1)2^{-n/2}} \big)I^{(2)}_2 \big(\delta_{(i_3+1)2^{-n/2}}\otimes\delta_{(i_4+1)2^{-n/2}} \big)\bigg)\bigg|\\
 && \hspace{2cm} \times |\langle  \delta_{(i_3+1)2^{-n/2}} , \delta_{(i_4+1)2^{-n/2}} \rangle| \\
&& + 2 \sum_{i_1,i_2,i_3,i_4=0}^{\lfloor 2^{\frac{n}{2}} t \rfloor -1}\bigg|E\bigg(\phi(i_1, i_2, i_3, i_4)\prod_{a=1}^2 I^{(1)}_1 \big(\delta_{(i_a+1)2^{-n/2}} \big)I^{(2)}_2 \big(\delta_{(i_a+1)2^{-n/2}}^{\otimes 2} \big)\\
 &&\hspace{2cm}\times I^{(1)}_2 \big(\delta_{(i_3+1)2^{-n/2}}\otimes\delta_{(i_4+1)2^{-n/2}} \big)\bigg)\bigg|\langle  \delta_{(i_3+1)2^{-n/2}} , \delta_{(i_4+1)2^{-n/2}} \rangle^2 \\
 \end{eqnarray*}
\begin{eqnarray*}
 && + \sum_{i_1,i_2,i_3,i_4=0}^{\lfloor 2^{\frac{n}{2}} t \rfloor -1}\bigg|E\bigg(\phi(i_1, i_2, i_3, i_4)\prod_{a=1}^2 I^{(1)}_1 \big(\delta_{(i_a+1)2^{-n/2}} \big)I^{(2)}_2 \big(\delta_{(i_a+1)2^{-n/2}}^{\otimes 2} \big)\\
 &&\hspace{2cm}\times I^{(2)}_4 \big(\delta_{(i_3+1)2^{-n/2}}^{\otimes 2} \otimes\delta_{(i_4+1)2^{-n/2}}^{\otimes 2} \big) \bigg)\bigg||\langle  \delta_{(i_3+1)2^{-n/2}} , \delta_{(i_4+1)2^{-n/2}} \rangle|\\
 && + 4 \sum_{i_1,i_2,i_3,i_4=0}^{\lfloor 2^{\frac{n}{2}} t \rfloor -1}\bigg|E\bigg(\phi(i_1, i_2, i_3, i_4)\prod_{a=1}^2 I^{(1)}_1 \big(\delta_{(i_a+1)2^{-n/2}} \big)I^{(2)}_2 \big(\delta_{(i_a+1)2^{-n/2}}^{\otimes 2} \big)\\
 &&\hspace{2cm}\times I^{(2)}_2 \big(\delta_{(i_3+1)2^{-n/2}} \otimes \delta_{(i_4+1)2^{-n/2}} \big)\bigg)\bigg|\langle  \delta_{(i_3+1)2^{-n/2}} , \delta_{(i_4+1)2^{-n/2}} \rangle^2 \\
 && + 2 \sum_{i_1,i_2,i_3,i_4=0}^{\lfloor 2^{\frac{n}{2}} t \rfloor -1}\bigg|E\bigg(\phi(i_1, i_2, i_3, i_4)\prod_{a=1}^2 I^{(1)}_1 \big(\delta_{(i_a+1)2^{-n/2}} \big)I^{(2)}_2 \big(\delta_{(i_a+1)2^{-n/2}}^{\otimes 2} \big)\bigg)\bigg|\\
 &&\hspace{2cm}\times |\langle  \delta_{(i_3+1)2^{-n/2}} , \delta_{(i_4+1)2^{-n/2}} \rangle|^3 \\
 && =: \sum_{i=1}^6 L_{n,i}(t).
\end{eqnarray*}
Let us prove that, for any $i\in \{1,\ldots,6\}$ there exists $C >0$ (depending only on $f$) such that
\begin{equation}
L_{n,i}(t)\leq C(t+t^2+t^3+t^4). \label{last-desired conclusion}
\end{equation}
Then the desired conclusion of (\ref{lemma4'}) will follow immediately.

Thanks to the duality formula (\ref{duality formula}), we have
\begin{eqnarray*}
L_{n,1}(t) &=& \sum_{i_1,i_2,i_3,i_4=0}^{\lfloor 2^{\frac{n}{2}} t \rfloor -1}\bigg|E\bigg(\bigg\langle D_{X^{(1)}}^2\bigg( \phi(i_1, i_2, i_3, i_4)\prod_{a=1}^2 I^{(1)}_1 \big(\delta_{(i_a+1)2^{-n/2}} \big)I^{(2)}_2 \big(\delta_{(i_a+1)2^{-n/2}}^{\otimes 2} \big)\bigg), \\
 &&\hspace{2cm} \delta_{(i_3+1)2^{-n/2}}\otimes\delta_{(i_4+1)2^{-n/2}} \bigg\rangle I^{(2)}_4 \big(\delta_{(i_3+1)2^{-n/2}}^{\otimes 2} \otimes\delta_{(i_4+1)2^{-n/2}}^{\otimes 2} \big)\bigg)\bigg|\\ 
 &=& \sum_{i_1,i_2,i_3,i_4=0}^{\lfloor 2^{\frac{n}{2}} t \rfloor -1}\bigg|E\bigg(\bigg\langle D_{X^{(1)}}^2\bigg( \phi(i_1, i_2, i_3, i_4)\prod_{a=1}^2 I^{(1)}_1 \big(\delta_{(i_a+1)2^{-n/2}} \big)\bigg), \\
 && \delta_{(i_3+1)2^{-n/2}}\otimes \delta_{(i_4+1)2^{-n/2}} \bigg\rangle \prod_{a=1}^2I^{(2)}_2 \big(\delta_{(i_a+1)2^{-n/2}}^{\otimes 2} \big) I^{(2)}_4 \big(\delta_{(i_3+1)2^{-n/2}}^{\otimes 2} \otimes \delta_{(i_4+1)2^{-n/2}}^{\otimes 2} \big)\bigg)\bigg|\\
\end{eqnarray*}
When computing the second Malliavin derivative 
\[
D_{X^{(1)}}^2\bigg( \phi(i_1, i_2, i_3, i_4)\prod_{a=1}^2 I^{(1)}_1 \big(\delta_{(i_a+1)2^{-n/2}} \big)\bigg),
\]
 there are three types of terms:
\begin{enumerate}
\item[(1)] The first type consists in terms arising when one only differentiates $\phi(i_1, i_2, i_3, i_4)$. By (\ref{12}), these terms are all bounded by
\begin{eqnarray*}
&& C2^{-n/3}\sum_{i_1,i_2,i_3,i_4=0}^{\lfloor 2^{\frac{n}{2}} t \rfloor -1}\bigg|E\bigg( \widetilde{\phi}(i_1, i_2, i_3, i_4)\prod_{a=1}^2 I^{(1)}_1 \big(\delta_{(i_a+1)2^{-n/2}} \big)I^{(2)}_2 \big(\delta_{(i_a+1)2^{-n/2}}^{\otimes 2} \big) \\
 &&\hspace{2cm} \times I^{(2)}_4 \big(\delta_{(i_3+1)2^{-n/2}}^{\otimes 2} \otimes\delta_{(i_4+1)2^{-n/2}}^{\otimes 2} \big)\bigg)\bigg|,
\end{eqnarray*}
where $\widetilde{\phi}(i_1, i_2, i_3, i_4)$ is a quantity having a similar form as $\phi(i_1, i_2, i_3, i_4)$. By the duality formula (\ref{duality formula}), we deduce that the last quantity is equal to
\begin{eqnarray*}
&& C2^{-n/3}\sum_{i_1,i_2,i_3,i_4=0}^{\lfloor 2^{\frac{n}{2}} t \rfloor -1}\bigg|E\bigg(\bigg\langle D_{X^{(2)}}^4\bigg( \widetilde{\phi}(i_1, i_2, i_3, i_4)\prod_{a=1}^2 I^{(1)}_1 \big(\delta_{(i_a+1)2^{-n/2}} \big)I^{(2)}_2 \big(\delta_{(i_a+1)2^{-n/2}}^{\otimes 2} \big)\bigg), \\
 &&\hspace{4cm} \delta_{(i_3+1)2^{-n/2}}^{\otimes 2} \otimes\delta_{(i_4+1)2^{-n/2}}^{\otimes 2} \bigg\rangle\bigg)\bigg|\\
 &=& C2^{-n/3}\sum_{i_1,i_2,i_3,i_4=0}^{\lfloor 2^{\frac{n}{2}} t \rfloor -1}\bigg|E\bigg(\bigg\langle D_{X^{(2)}}^4\bigg( \widetilde{\phi}(i_1, i_2, i_3, i_4)\prod_{a=1}^2 I^{(2)}_2 \big(\delta_{(i_a+1)2^{-n/2}}^{\otimes 2} \big)\bigg), \\
 &&\hspace{4cm} \delta_{(i_3+1)2^{-n/2}}^{\otimes 2} \otimes\delta_{(i_4+1)2^{-n/2}}^{\otimes 2} \bigg\rangle\prod_{a=1}^2I^{(1)}_1 \big(\delta_{(i_a+1)2^{-n/2}} \big)\bigg)\bigg|.\\
\end{eqnarray*}
When computing the fourth Malliavin derivative  
\[
D_{X^{(2)}}^4\bigg( \widetilde{\phi}(i_1, i_2, i_3, i_4)
\prod_{a=1}^2  I^{(2)}_2 \big(\delta_{(i_a+1)2^{-n/2}}^{\otimes 2} \big)\bigg),
\] there are three types of terms:
\begin{enumerate}
\item The first type consists in terms arising when one only differentiates $\widetilde{\phi}(i_1, i_2, i_3, i_4)$. Thanks to (\ref{12}), these terms are all bounded by 
\begin{eqnarray*}
&& C2^{-n}\sum_{i_1,i_2,i_3,i_4=0}^{\lfloor 2^{\frac{n}{2}} t \rfloor -1}\bigg|E\bigg( \bar{\phi}(i_1, i_2, i_3, i_4)\prod_{a=1}^2 I^{(1)}_1 \big(\delta_{(i_a+1)2^{-n/2}} \big)I^{(2)}_2 \big(\delta_{(i_a+1)2^{-n/2}}^{\otimes 2} \big)\bigg)\bigg|,\\
\end{eqnarray*}
where $\bar{\phi}(i_1, i_2, i_3, i_4)$ is a quantity having a similar form as $\widetilde{\phi}(i_1, i_2, i_3, i_4)$. Observe that the last quantity is less than
\begin{eqnarray}
 && Ct^2 \sup_{i_3,i_4 \in \{0, \ldots, \lfloor 2^{\frac{n}{2}} t \rfloor -1\}}\sum_{i_1,i_2=0}^{\lfloor 2^{\frac{n}{2}} t \rfloor -1}\bigg|E\bigg( \bar{\phi}(i_1, i_2, i_3, i_4)\prod_{a=1}^2 I^{(1)}_1 \big(\delta_{(i_a+1)2^{-n/2}} \big)\notag\\
 && \hspace{6cm} \times I^{(2)}_2 \big(\delta_{(i_a+1)2^{-n/2}}^{\otimes 2} \big)\bigg)\bigg|\notag\\
 &\leq & C(t^3+t^4)\label{Lemma2.10-1},
\end{eqnarray}
where the last inequality is a consequence of (\ref{lemma3}).

\item The second type consists in terms arising when one  differentiates $\widetilde{\phi}(i_1, i_2, i_3,i_4)$ and $I^{(2)}_2 \big(\delta_{(i_1+1)2^{-n/2}}^{\otimes 2} \big)$ but not $I^{(2)}_2 \big(\delta_{(i_2+1)2^{-n/2}}^{\otimes 2} \big)$ (the case when one differentiates $\widetilde{\phi}(i_1, i_2, i_3, i_4)$ and $I^{(2)}_2 \big(\delta_{(i_2+1)2^{-n/2}}^{\otimes 2} \big)$ but not $I^{(2)}_2 \big(\delta_{(i_1+1)2^{-n/2}}^{\otimes 2} \big)$ is completely similar). In this case, with $\rho$ defined in (\ref{rho}) and $\alpha \in \{0,1\}$, the corresponding terms are bounded either by
\begin{eqnarray}
&&C2^{-n}\sum_{i_1,i_2,i_3,i_4=0}^{\lfloor 2^{\frac{n}{2}} t \rfloor -1}\bigg|E\bigg( \bar{\phi}(i_1, i_2, i_3, i_4)I^{(2)}_{\alpha} \big(\delta_{(i_1+1)2^{-n/2}}^{\otimes\alpha} \big)I^{(2)}_2 \big(\delta_{(i_2+1)2^{-n/2}}^{\otimes 2} \big)\notag\\
&& \hspace{4cm}\times \prod_{a=1}^2 I^{(1)}_1 \big(\delta_{(i_a+1)2^{-n/2}} \big)\bigg)\bigg||\rho(i_1-i_3)|\label{Lemma2.10-2}
\end{eqnarray}
or by the same quantity with $|\rho(i_1-i_4)|$ instead of $|\rho(i_1-i_3)|$. We have obtained the previous estimate by using (\ref{derivative- multiple,integral}) and (\ref{12}). Observe that, by the duality formula (\ref{duality formula}), we have
\begin{eqnarray*}
&&\bigg|E\bigg( \bar{\phi}(i_1, i_2, i_3, i_4)I^{(2)}_{\alpha} \big(\delta_{(i_1+1)2^{-n/2}}^{\otimes\alpha} \big)I^{(2)}_2 \big(\delta_{(i_2+1)2^{-n/2}}^{\otimes 2} \big)\prod_{a=1}^2 I^{(1)}_1 \big(\delta_{(i_a+1)2^{-n/2}} \big)\bigg)\bigg|\\
&=& \bigg|E\bigg( \bigg\langle D_{X^{(2)}}^2\bigg(\bar{\phi}(i_1, i_2, i_3, i_4)I^{(2)}_{\alpha} \big(\delta_{(i_1+1)2^{-n/2}}^{\otimes\alpha} \big)\bigg), \delta_{(i_2+1)2^{-n/2}}^{\otimes 2} \bigg\rangle\\
&& \hspace{4cm}\times \prod_{a=1}^2 I^{(1)}_1 \big(\delta_{(i_a+1)2^{-n/2}} \big)\bigg)\bigg|\\
&=:& F(i_1,i_2,i_3,i_4,\alpha).
\end{eqnarray*}
We have
\begin{itemize}
\item \underline{For $\alpha = 0$:}
\begin{eqnarray*}
&& F(i_1,i_2,i_3,i_4,0)\\
&=&\bigg|E\bigg( \bigg\langle D_{X^{(2)}}^2\big(\bar{\phi}(i_1, i_2, i_3, i_4)\big), \delta_{(i_2+1)2^{-n/2}}^{\otimes 2} \bigg\rangle \prod_{a=1}^2 I^{(1)}_1 \big(\delta_{(i_a+1)2^{-n/2}} \big)\bigg)\bigg|\\
\end{eqnarray*}
\begin{eqnarray*}
&\leq & C(2^{-n/6})^2\|I^{(1)}_1 \big(\delta_{(i_1+1)2^{-n/2}} \big)\|_2\|I^{(1)}_1 \big(\delta_{(i_2+1)2^{-n/2}} \big)\|_2 \\
&\leq & C2^{-n/2},
\end{eqnarray*}
where we have the first inequality since $f\in C_b^\infty$ and thanks to (\ref{12}) and to the Cauchy-Schwarz inequality. The second inequality follows from (\ref{isometry}). 

\item \underline{For $\alpha =1$:} Thanks to (\ref{Leibnitz0}),(\ref{Leibnitz1}), (\ref{derivative- multiple,integral}),(\ref{12}) and (\ref{isometry}), we have
\begin{eqnarray*}
&& F(i_1,i_2,i_3,i_4,1)\\
&\leq & C2^{-n/3}E\big(\big|I^{(1)}_1 \big(\delta_{(i_1+1)2^{-n/2}} \big)\big|\big|I^{(1)}_1 \big(\delta_{(i_2+1)2^{-n/2}} \big)\big|\big)\\
&\leq & C2^{-n/3}\|I^{(1)}_1 \big(\delta_{(i_1+1)2^{-n/2}} \big)\|_2\|I^{(1)}_1 \big(\delta_{(i_2+1)2^{-n/2}} \big)\|_2 \leq C2^{-n/2}.
\end{eqnarray*}
\end{itemize}
For $\alpha \in \{0,1\}$, by plugging  $F(i_1,i_2,i_3,i_4,\alpha)$ into (\ref{Lemma2.10-2}) we deduce that the quantity given in (\ref{Lemma2.10-2}) is bounded by 
\begin{eqnarray}
 Ct^22^{-n/2}\sum_{i_1, i_3=0}^{\lfloor 2^{\frac{n}{2}} t \rfloor -1}|\rho(i_1-i_3)| \leq Ct^3(\sum_{r\in \Z}|\rho(r)|)\leq Ct^3. \label{Lemma2.10-3}
\end{eqnarray}
Note that $\sum_{r\in \Z}|\rho(r)| < \infty$ because $H=1/6 < 1/2$.

\item The third type consists in terms arising when one  differentiates $\widetilde{\phi}(i_1, i_2, i_3, i_4)$, $I^{(2)}_2 \big(\delta_{(i_1+1)2^{-n/2}}^{\otimes 2} \big)$ and $I^{(2)}_2 \big(\delta_{(i_2+1)2^{-n/2}}^{\otimes 2} \big)$. In this case, thanks to  (\ref{derivative- multiple,integral}) and (\ref{12}), for $\alpha, \beta \in \{0,1\}$ the corresponding terms can be bounded either by
\begin{eqnarray}
&&C2^{-n}\sum_{i_1,i_2,i_3,i_4=0}^{\lfloor 2^{\frac{n}{2}} t \rfloor -1}\bigg|E\bigg( \bar{\phi}(i_1, i_2, i_3, i_4)I^{(2)}_{\alpha} \big(\delta_{(i_1+1)2^{-n/2}}^{\otimes\alpha} \big)I^{(2)}_{\beta} \big(\delta_{(i_2+1)2^{-n/2}}^{\otimes \beta} \big)\notag\\
&& \hspace{2cm}\times \prod_{a=1}^2 I^{(1)}_1 \big(\delta_{(i_a+1)2^{-n/2}} \big)\bigg)\bigg||\rho(i_1-i_3)||\rho(i_2-i_3)|,\label{Lemma2.10-4}
\end{eqnarray} 
or by the same quantity with $|\rho(i_1-i_4)||\rho(i_2-i_4)|$ or $|\rho(i_1-i_3)||\rho(i_2-i_4)|$ or $|\rho(i_2-i_3)||\rho(i_1-i_4)|$ instead of $|\rho(i_1-i_3)||\rho(i_2-i_3)|$. Observe that 
\[
\bigg|E\bigg( \bar{\phi}(i_1, i_2, i_3, i_4)I^{(2)}_{\alpha} \big(\delta_{(i_1+1)2^{-n/2}}^{\otimes\alpha} \big)I^{(2)}_{\beta} \big(\delta_{(i_2+1)2^{-n/2}}^{\otimes \beta} \big)\prod_{a=1}^2 I^{(1)}_1 \big(\delta_{(i_a+1)2^{-n/2}} \big)\bigg)\bigg|
\]
is uniformly bounded in $n$. So, the quantity given in (\ref{Lemma2.10-4}) is bounded by
\begin{eqnarray}
 Ct2^{-n/2}\sum_{i_1,i_2,i_3 =0}^{\lfloor 2^{\frac{n}{2}} t \rfloor -1}|\rho(i_1-i_3)||\rho(i_2-i_3)|\leq  Ct^2(\sum_{r\in \Z}|\rho(r)|)^2 \leq Ct^2. \label{Lemma2.10-5}
\end{eqnarray}
\end{enumerate}
Thanks to (\ref{Lemma2.10-5}), (\ref{Lemma2.10-3}) and (\ref{Lemma2.10-1}), we deduce that the terms of the first type in $L_{n,1}(t)$ agree with the desired conclusion (\ref{last-desired conclusion}). 

\item[(2)] The second type consists in terms arising when one  differentiates $\phi(i_1, i_2, i_3, i_4)$ and $I^{(1)}_1 \big(\delta_{(i_1+1)2^{-n/2}} \big)$, but not $I^{(1)}_1 \big(\delta_{(i_2+1)2^{-n/2}} \big)$ (the case where one differentiates $\phi(i_1, i_2, i_3, i_4)$ and $I^{(1)}_1 \big(\delta_{(i_2+1)2^{-n/2}} \big)$, but not $I^{(1)}_1 \big(\delta_{(i_1+1)2^{-n/2}} \big)$ is completely similar). In this case, thanks to (\ref{12}), the corresponding terms are all bounded either by 
\begin{eqnarray*}
&& C2^{-n/3}\sum_{i_1,i_2,i_3,i_4=0}^{\lfloor 2^{\frac{n}{2}} t \rfloor -1}\bigg|E\bigg( \widetilde{\phi}(i_1, i_2, i_3, i_4)I^{(1)}_1 \big(\delta_{(i_2+1)2^{-n/2}} \big)\prod_{a=1}^2 I^{(2)}_2 \big(\delta_{(i_a+1)2^{-n/2}}^{\otimes 2} \big) \\
 &&\hspace{2cm} \times I^{(2)}_4 \big(\delta_{(i_3+1)2^{-n/2}}^{\otimes 2} \otimes\delta_{(i_4+1)2^{-n/2}}^{\otimes 2} \big)\bigg)\bigg||\rho(i_1-i_3)|,
\end{eqnarray*}
or by the same quantity with $|\rho(i_1-i_4)|$ instead of $|\rho(i_1-i_3)|$. By the duality formula (\ref{duality formula}), the previous quantity is equal to
\begin{eqnarray*}
&& C2^{-n/3}\sum_{i_1,i_2,i_3,i_4=0}^{\lfloor 2^{\frac{n}{2}} t \rfloor -1}\bigg|E\bigg(\bigg\langle D^4_{X^{(2)}}\bigg( \widetilde{\phi}(i_1, i_2, i_3, i_4)\prod_{a=1}^2 I^{(2)}_2 \big(\delta_{(i_a+1)2^{-n/2}}^{\otimes 2} \big)\bigg), \\
 &&\hspace{2cm} \delta_{(i_3+1)2^{-n/2}}^{\otimes 2} \otimes\delta_{(i_4+1)2^{-n/2}}^{\otimes 2} \bigg\rangle I^{(1)}_1 \big(\delta_{(i_2+1)2^{-n/2}} \big)\bigg)\bigg||\rho(i_1-i_3)|.
\end{eqnarray*}
When computing the fourth Malliavin derivative  
\[
D_{X^{(2)}}^4\bigg( \widetilde{\phi}(i_1, i_2, i_3, i_4)
\prod_{a=1}^2  I^{(2)}_2 \big(\delta_{(i_a+1)2^{-n/2}}^{\otimes 2} \big)\bigg),
\] there are three types of terms, exactly as it has been proved previously:
\begin{enumerate}
\item The first type consists in terms arising when one only differentiates $\widetilde{\phi}(i_1, i_2, i_3, i_4)$. Thanks to (\ref{12}), these terms are all bounded by 
\begin{eqnarray}
&& C2^{-n}\sum_{i_1,i_2,i_3,i_4=0}^{\lfloor 2^{\frac{n}{2}} t \rfloor -1}\bigg|E\bigg( \bar{\phi}(i_1, i_2, i_3, i_4)\prod_{a=1}^2 I^{(2)}_2 \big(\delta_{(i_a+1)2^{-n/2}}^{\otimes 2} \big) \notag\\
 &&\hspace{2cm} \times I^{(1)}_1 \big(\delta_{(i_2+1)2^{-n/2}} \big)\bigg)\bigg||\rho(i_1-i_3)|. \label{Lemma2.10-6}
\end{eqnarray}
Observe that  since $f \in C_b^\infty$ and thanks to (\ref{duality formula}), (\ref{12}) and (\ref{isometry}), we have
\begin{eqnarray*}
&&\bigg|E\bigg( \bar{\phi}(i_1, i_2, i_3, i_4)\prod_{a=1}^2 I^{(2)}_2 \big(\delta_{(i_a+1)2^{-n/2}}^{\otimes 2} \big)  I^{(1)}_1 \big(\delta_{(i_2+1)2^{-n/2}} \big)\bigg)\bigg|\\
&=& \bigg|E\bigg( \big\langle D_{X^{(1)}}\big(\bar{\phi}(i_1, i_2, i_3, i_4)\prod_{a=1}^2 I^{(2)}_2 \big(\delta_{(i_a+1)2^{-n/2}}^{\otimes 2} \big)\big) ,\delta_{(i_2+1)2^{-n/2}} \big\rangle\bigg)\bigg|\\
&\leq & C2^{-n/6}E\bigg( \prod_{a=1}^2 \big| I^{(2)}_2 \big(\delta_{(i_a+1)2^{-n/2}}^{\otimes 2} \big)\big|\bigg) \\
&\leq & C2^{-n/6} \|I^{(2)}_2 \big(\delta_{(i_1+1)2^{-n/2}}^{\otimes 2} \big)\|_2\|I^{(2)}_2 \big(\delta_{(i_2+1)2^{-n/2}}^{\otimes 2} \big)\|_2 \leq C2^{-n/2}.
\end{eqnarray*}
Hence, we deduce that the quantity given in (\ref{Lemma2.10-6}) is bounded by 
\begin{eqnarray}
Ct^22^{-n/2}\sum_{i_1,i_3=0}^{\lfloor 2^{\frac{n}{2}} t \rfloor -1}|\rho(i_1-i_3)|\leq Ct^3(\sum_{r\in \Z}|\rho(r)|) \leq Ct^3. \label{Lemma2.10-7}
\end{eqnarray}

\item The second type consists in terms arising when one  differentiates $\widetilde{\phi}(i_1, i_2, i_3, i_4)$ and $I^{(2)}_2 \big(\delta_{(i_1+1)2^{-n/2}}^{\otimes 2} \big)$ but not $I^{(2)}_2 \big(\delta_{(i_2+1)2^{-n/2}}^{\otimes 2} \big)$ (the case when one differentiates $\widetilde{\phi}(i_1, i_2, i_3, i_4)$ and $I^{(2)}_2 \big(\delta_{(i_2+1)2^{-n/2}}^{\otimes 2} \big)$ but not $I^{(2)}_2 \big(\delta_{(i_1+1)2^{-n/2}}^{\otimes 2} \big)$ is completely similar). In this case, thanks to (\ref{12}) and for $\alpha \in \{0,1\}$, the corresponding terms are all bounded either by
\begin{eqnarray}
&&C2^{-n}\sum_{i_1,i_2,i_3,i_4=0}^{\lfloor 2^{\frac{n}{2}} t \rfloor -1}\bigg|E\bigg( \bar{\phi}(i_1, i_2, i_3, i_4)I^{(2)}_{\alpha} \big(\delta_{(i_1+1)2^{-n/2}}^{\otimes\alpha} \big)I^{(2)}_2 \big(\delta_{(i_2+1)2^{-n/2}}^{\otimes 2} \big)\notag\\
&& \hspace{2cm}\times  I^{(1)}_1 \big(\delta_{(i_2+1)2^{-n/2}} \big)\bigg)\bigg||\rho(i_1-i_3)||\rho(i_1-i_4)|\label{Lemma2.10-8}
\end{eqnarray}
or by the same quantity with $|\rho(i_1-i_3)|$ instead of $|\rho(i_1-i_4)|$. Observe that, by (\ref{duality formula}) and (\ref{12}) among other things and since $f \in C_b^\infty$, we have
\begin{eqnarray*}
&&\bigg|E\bigg( \bar{\phi}(i_1, i_2, i_3, i_4)I^{(2)}_{\alpha} \big(\delta_{(i_1+1)2^{-n/2}}^{\otimes\alpha} \big)I^{(2)}_2 \big(\delta_{(i_2+1)2^{-n/2}}^{\otimes 2} \big)I^{(1)}_1 \big(\delta_{(i_2+1)2^{-n/2}} \big)\bigg)\bigg|\\
&= & \bigg|E\bigg( \big\langle D_{X^{(1)}}\big(\bar{\phi}(i_1, i_2, i_3, i_4)\big), \delta_{(i_2+1)2^{-n/2}} \big\rangle I^{(2)}_{\alpha} \big(\delta_{(i_1+1)2^{-n/2}}^{\otimes\alpha} \big)I^{(2)}_2 \big(\delta_{(i_2+1)2^{-n/2}}^{\otimes 2} \big)\bigg)\bigg|\\
&\leq & C2^{-n/6}\bigg|E\bigg( \chi(i_1, i_2, i_3, i_4) I^{(2)}_{\alpha} \big(\delta_{(i_1+1)2^{-n/2}}^{\otimes\alpha} \big)I^{(2)}_2 \big(\delta_{(i_2+1)2^{-n/2}}^{\otimes 2} \big)\bigg)\bigg|\\
&=& C2^{-n/6}\bigg|E\bigg(\big\langle D^2_{X^{(2)}}\big( \chi(i_1, i_2, i_3, i_4) I^{(2)}_{\alpha} \big(\delta_{(i_1+1)2^{-n/2}}^{\otimes\alpha} \big)\big), \delta_{(i_2+1)2^{-n/2}}^{\otimes 2} \big\rangle\bigg)\bigg|\\
&\leq & C2^{-n/2},
\end{eqnarray*}
where $\chi(i_1, i_2, i_3, i_4)$ is a quantity having a similar form as $\bar{\phi}(i_1, i_2, i_3, i_4)$. Thus, we get that the quantity given by (\ref{Lemma2.10-8}) is bounded by
\begin{eqnarray}
C2^{-n}t \sum_{i_1,i_3,i_4=0}^{\lfloor 2^{\frac{n}{2}} t \rfloor -1}|\rho(i_1-i_3)||\rho(i_1-i_4)|&\leq & Ct^22^{-n/2}(\sum_{r\in \Z}|\rho(r)|)^2 \notag\\
&\leq & C2^{-n/2} t^2. \label{Lemma2.10-9}
\end{eqnarray}

\item The third type consists in terms arising when one  differentiates $\widetilde{\phi}(i_1, i_2, i_3, i_4)$, $I^{(2)}_2 \big(\delta_{(i_1+1)2^{-n/2}}^{\otimes 2} \big)$ and $I^{(2)}_2 \big(\delta_{(i_2+1)2^{-n/2}}^{\otimes 2} \big)$. In this case, thanks to (\ref{12}), for $\alpha, \beta \in \{0,1\}$ the corresponding terms can be bounded either by
\begin{eqnarray}
&&C2^{-n}\sum_{i_1,i_2,i_3,i_4=0}^{\lfloor 2^{\frac{n}{2}} t \rfloor -1}\bigg|E\bigg( \bar{\phi}(i_1, i_2, i_3, i_4)I^{(2)}_{\alpha} \big(\delta_{(i_1+1)2^{-n/2}}^{\otimes\alpha} \big)I^{(2)}_{\beta} \big(\delta_{(i_2+1)2^{-n/2}}^{\otimes \beta} \big)\notag\\
&& \hspace{1cm}\times  I^{(1)}_1 \big(\delta_{(i_2+1)2^{-n/2}} \big)\bigg)\bigg||\rho(i_1-i_3)||\rho(i_1-i_4)||\rho(i_2-i_4)|,\label{Lemma2.10-10}
\end{eqnarray} 
or by the same quantity with $|\rho(i_1-i_3)||\rho(i_2-i_3)|$ or $|\rho(i_1-i_3)||\rho(i_2-i_4)|$ or $|\rho(i_2-i_3)||\rho(i_1-i_4)|$ instead of $|\rho(i_1-i_4)||\rho(i_2-i_4)|$. Observe that
\begin{eqnarray*}
\bigg|E\bigg( \bar{\phi}(i_1, i_2, i_3, i_4)I^{(2)}_{\alpha} \big(\delta_{(i_1+1)2^{-n/2}}^{\otimes\alpha} \big)I^{(2)}_{\beta} \big(\delta_{(i_2+1)2^{-n/2}}^{\otimes \beta} \big) I^{(1)}_1 \big(\delta_{(i_2+1)2^{-n/2}} \big)\bigg)\bigg|
\end{eqnarray*}
is uniformly bounded in $n$. So, we deduce that the quantity given by (\ref{Lemma2.10-10}) is bounded by
\begin{eqnarray}
C2^{-n}\sum_{i_1,i_2,i_3,i_4=0}^{\lfloor 2^{\frac{n}{2}} t \rfloor -1}|\rho(i_1-i_3)||\rho(i_1-i_4)||\rho(i_2-i_4)|&\leq & C2^{-n/2}t (\sum_{r\in \Z}|\rho(r)|)^3 \notag\\
& \leq & C2^{-n/2}t. \label{Lemma2.10-11}
\end{eqnarray}
\end{enumerate}
Combining (\ref{Lemma2.10-11}), (\ref{Lemma2.10-9}) and (\ref{Lemma2.10-7}), we deduce that the terms of the second type in $L_{n,1}(t)$ agree with the desired conclusion (\ref{last-desired conclusion}).

\item[(3)] The third type consists in terms arising when one only differentiates $\prod_{a=1}^2I^{(1)}_1 \big(\delta_{(i_a+1)2^{-\frac{n}{2}}} \big)$.
 In this case, thanks to (\ref{Leibnitz0}) and (\ref{derivative- multiple,integral}), the corresponding term is equal to:
 \begin{eqnarray*}
 &&\sum_{i_1,i_2,i_3,i_4=0}^{\lfloor 2^{\frac{n}{2}} t \rfloor -1}\bigg|E\bigg(\phi(i_1, i_2, i_3, i_4)\prod_{a=1}^2 I^{(2)}_2 \big(\delta_{(i_a+1)2^{-n/2}}^{\otimes 2} \big)I^{(2)}_4 \big(\delta_{(i_3+1)2^{-n/2}}^{\otimes 2} \otimes\delta_{(i_4+1)2^{-n/2}}^{\otimes 2} \big)\bigg)\bigg|\\
 &&\hspace{2cm} \times \big|\langle \delta_{(i_1+1)2^{-n/2}}\tilde{\otimes}\delta_{(i_2+1)2^{-n/2}}, \delta_{(i_3+1)2^{-n/2}}\otimes\delta_{(i_4+1)2^{-n/2}}\rangle\big|\\
 &\leq &  2^{-n/3}\sum_{i_1,i_2,i_3,i_4=0}^{\lfloor 2^{\frac{n}{2}} t \rfloor -1}\bigg|E\bigg(\phi(i_1, i_2, i_3, i_4)\prod_{a=1}^2 I^{(2)}_2 \big(\delta_{(i_a+1)2^{-n/2}}^{\otimes 2} \big)\\
 && \hspace{2cm}\times I^{(2)}_4 \big(\delta_{(i_3+1)2^{-n/2}}^{\otimes 2} \otimes\delta_{(i_4+1)2^{-n/2}}^{\otimes 2} \big)\bigg)\bigg||\rho(i_1 -i_3)||\rho(i_2-i_4)|\\
 && +  2^{-n/3}\sum_{i_1,i_2,i_3,i_4=0}^{\lfloor 2^{\frac{n}{2}} t \rfloor -1}\bigg|E\bigg(\phi(i_1, i_2, i_3, i_4)\prod_{a=1}^2 I^{(2)}_2 \big(\delta_{(i_a+1)2^{-n/2}}^{\otimes 2} \big)\\
 &&\hspace{2cm}\times I^{(2)}_4 \big(\delta_{(i_3+1)2^{-n/2}}^{\otimes 2} \otimes\delta_{(i_4+1)2^{-n/2}}^{\otimes 2} \big)\bigg)\bigg||\rho(i_2 -i_3)||\rho(i_1-i_4)|.
 \end{eqnarray*}
 It suffices to prove that the second quantity agree with the desired conclusion (\ref{last-desired conclusion}) (similarly, the first quantity agree as well with (\ref{last-desired conclusion})). Thanks to the duality formula (\ref{duality formula}), we have that the last quantity is equal to
 \begin{eqnarray}
 && 2^{-n/3}\sum_{i_1,i_2,i_3,i_4=0}^{\lfloor 2^{\frac{n}{2}} t \rfloor -1}\bigg|E\bigg(\bigg\langle D^4_{X^{(2)}}\bigg(\phi(i_1, i_2, i_3, i_4)\prod_{a=1}^2 I^{(2)}_2 \big(\delta_{(i_a+1)2^{-n/2}}^{\otimes 2} \big)\bigg), \label{Lemma2.10-12}\\
 && \hspace{3cm}\delta_{(i_3+1)2^{-n/2}}^{\otimes 2} \otimes\delta_{(i_4+1)2^{-n/2}}^{\otimes 2} \bigg\rangle\bigg)\bigg||\rho(i_2 -i_3)||\rho(i_1-i_4)|. \notag
 \end{eqnarray}
 Observe that one can prove, thanks to (\ref{12}) among other things (and following the approach already used several times previously) that 
 \begin{eqnarray*}
 && \bigg|E\bigg(\bigg\langle D^4_{X^{(2)}}\bigg(\phi(i_1, i_2, i_3, i_4)\prod_{a=1}^2 I^{(2)}_2 \big(\delta_{(i_a+1)2^{-n/2}}^{\otimes 2} \big)\bigg),\delta_{(i_3+1)2^{-n/2}}^{\otimes 2} \otimes\delta_{(i_4+1)2^{-n/2}}^{\otimes 2} \bigg\rangle\bigg)\bigg|\\
 &\leq & C2^{-2n/3}.
 \end{eqnarray*}
 Hence, we get that the quantity given in (\ref{Lemma2.10-12}) is bounded by
 \begin{eqnarray*}
 C2^{-n}\sum_{i_1,i_2,i_3,i_4=0}^{\lfloor 2^{\frac{n}{2}} t \rfloor -1}|\rho(i_2 -i_3)||\rho(i_1-i_4)|\leq Ct^2(\sum_{r\in \Z}|\rho(r)|)^2 \leq Ct^2,
 \end{eqnarray*}
 which agrees with the desired conclusion (\ref{last-desired conclusion}).
\end{enumerate}
Finally, we have proved that $L_{n,1}(t)$ agrees with the desired conclusion (\ref{last-desired conclusion}).

The motivated reader may check that there is no additional difficulties to prove that for all $i\in \{2, \ldots, 5\}$, $L_{n,i}(t)$ agrees with the desired conclusion (\ref{last-desired conclusion}). Indeed, all the arguments and techniques which are needed to prove this claim, were already introduced and used along the previous proof. 

It remains to prove that $L_{n,6}(t)$ agrees with the desired conclusion (\ref{last-desired conclusion}).
Observe that 
\begin{eqnarray*}
&& L_{n,6}(t)\\
&=& 22^{-n/2}\sum_{i_1,i_2,i_3,i_4=0}^{\lfloor 2^{\frac{n}{2}} t \rfloor -1}\bigg|E\bigg(\phi(i_1, i_2, i_3, i_4)\prod_{a=1}^2 I^{(1)}_1 \big(\delta_{(i_a+1)2^{-n/2}} \big)I^{(2)}_2 \big(\delta_{(i_a+1)2^{-n/2}}^{\otimes 2} \big)\bigg)\bigg|\\
 &&\hspace{3cm}\times |\rho(i_3 - i_4)|^3\\
 &\leq & 22^{-n/2}\sum_{i_3,i_4=0}^{\lfloor 2^{\frac{n}{2}} t \rfloor -1}\bigg(\sup_{i_3,i_4 \in \{0, \ldots, \lfloor 2^{\frac{n}{2}} t \rfloor -1\}}\sum_{i_1,i_2=0}^{\lfloor 2^{\frac{n}{2}} t \rfloor -1}\bigg|E\bigg(\phi(i_1, i_2, i_3, i_4)\prod_{a=1}^2 I^{(1)}_1 \big(\delta_{(i_a+1)2^{-n/2}} \big)\\
 && \hspace{3cm}\times I^{(2)}_2 \big(\delta_{(i_a+1)2^{-n/2}}^{\otimes 2} \big)\bigg)\bigg| \bigg)|\rho(i_3 - i_4)|^3\\
 &\leq & C(t+t^2)2^{-n/2} \sum_{i_3,i_4=0}^{\lfloor 2^{\frac{n}{2}} t \rfloor -1}|\rho(i_3-i_4)|^3 \leq C(t+t^2)t\big(\sum_{r\in \Z}|\rho(r)|^3\big)\leq C(t^2+t^3),
\end{eqnarray*}
where the second inequality is a consequence of (\ref{lemma3}). Thanks to the previous estimate, it is clear that $L_{n,6}(t)$ agrees with the desired conclusion (\ref{last-desired conclusion}). The proof of  (\ref{lemma4'}) is now complete.

\subsection{Proof of (\ref{lemma4''})}
The proof is similar to the proof of (\ref{lemma4'}) and is left to the reader. See also \cite[Lemma 3.5]{NRS} for a very similar result.
\bigskip 

\textbf{Acknowledgements.} This paper is part of my PhD thesis. I thank my supervisor Ivan Nourdin for several interesting discussions about this article and for his careful reading and helpful comments. Also,  I thank an anonymous referee for his/her many valuable comments and remarks on a previous version of this work.

\end{document}